\documentclass[oneside,openany,12pt]{uofathesis}

\usepackage[all,cmtip]{xy}
\usepackage{bbm}
\usepackage{mathrsfs}
\usepackage{dsfont}
\usepackage{enumitem}
\usepackage{fouridx}
\usepackage[stable]{footmisc}
\usepackage{array}
\usepackage{caption}
\usepackage{latexsym}
\usepackage{extarrows}
\usepackage{amsfonts}
\usepackage{amsmath}
\usepackage{amsthm}
\usepackage{amssymb}
\usepackage{latexsym}
\allowdisplaybreaks
\newcommand{\C}{\mathbb{C}}
\newcommand{\Z}{\mathbb{Z}}
\newcommand{\N}{\mathbb{N}}

\newcommand{\g}{\mathfrak{g}}
\newcommand{\W}{\mathcal{W}}
\newcommand{\h}{\mathfrak{h}}
\newcommand{\yn}{Y(\mathfrak{sl}_{l+1})}
\newcommand{\yg}{Y(\mathfrak{g})}
\newcommand{\ysl}{Y(\mathfrak{sl}_2)}
\newcommand{\nysl}{\mathfrak{sl}_2}
\newcommand{\nyn}{\mathfrak{sl}_{l+1}}
\newcommand{\nyo}{\mathfrak{so}(2l,\C)}
\newcommand{\yo}{Y\big(\mathfrak{so}(2l,\C)\big)}
\newcommand{\ysp}{Y\big(\mathfrak{sp}(2l,\C)\big)}
\newcommand{\nysp}{\mathfrak{sp}(2l,\C)}
\newcommand{\yso}{Y\big(\mathfrak{so}(2l+1,\C)\big)}
\newcommand{\nyso}{\mathfrak{so}(2l+1,\C)}
\newcommand{\ygg}{Y(\g)}
\numberwithin{equation}{section}

\setlist[enumerate]{itemsep=0mm}
\setlist[itemize]{itemsep=0mm}

\newenvironment{stretchtable}{\begin{minipage}[c]{\textwidth}\centering}{\end{minipage}\vspace*{3pt}}

\begin{document}
\title{Finite-dimensional representations of Yangians}
\author{Yilan Tan}
\degree{Doctor of Philosophy}
\specialization{Mathematics}
\department{Department of Mathematical and Statistical Science}
\convocation{Fall 2014}
\titlepage
{\frontmatter
\begin{abstract}
In this thesis, we study local Weyl modules of Yangians and a cyclicity condition for a tensor product of fundamental representations of a Yangian.
%

Let $\g$ be a simple Lie algebra over $\C$ with rank $l$ and $\pi$ be a generic $l$-tuple of polynomials in $u$. We show that there exists a universal representation $W(\pi)$ of the Yangian $\yg$, called the local Weyl module associated to $\pi$, such that every finite-dimensional highest-weight representation associated to $\pi$ is a quotient of $W(\pi)$. We prove that the dimension of $W(\pi)$ is bounded by the dimension of some local Weyl module of the current algebra $\g[t]$. Let  $L=V_{a_1}(\omega_{b_1})\otimes V_{a_2}(\omega_{b_2})\otimes\ldots\otimes V_{a_k}(\omega_{b_k})$, where $V_{a_i}(\omega_{b_i})$ is the $b_i$-th fundamental representation of $\yg$. We prove that if $\operatorname{Re}(a_1)\geq\operatorname{Re}(a_2)\geq \ldots \geq \operatorname{Re}(a_k)$, then $L$ is a highest weight representation. By comparing the dimensions of $L$ and the upper bound of $W(\pi)$, we have
$W(\pi)\cong L$.

A cyclicity condition of the tensor product $L$ is also studied: $L$ is a highest weight representation if $a_j-a_i\notin S(b_i, b_j)$ for $1\leq i<j\leq k$ where $S(b_i, b_j)$ is a finite set of positive rational numbers. The cyclicity condition implies an irreducibility criterion for $L$: $L$ is  irreducible if $a_j-a_i\notin S(b_i, b_j)$ for $1\leq i\neq j\leq k$. Especially, when $\g=\nyn$, $L$ is  irreducible if and only if $a_j-a_i\notin S(b_i, b_j)$ for $1\leq i\neq j\leq k$.

\newpage\

\begin{center}
\Large{\textbf{Preface}}
\end{center}
Quantum groups, introduced by V. Drinfeld and M. Jimbo independently, arose in the first half of 1980s. The lecture of V. Drinfeld
at the International Congress of Mathematicians in 1986 brought quantum groups to the attention of mathematicians
worldwide. The theory of quantum groups has applications in many mathematical branches: topology, harmonic analysis, and number theory, to name a few.
Yangians and quantum affine algebras form two of the most important classes of quantum groups. For any finite-dimensional simple Lie algebra $\g$ over $\C$ with rank $l$, V. Drinfeld \cite{Dr1} defined an infinite-dimensional Hopf algebra $\yg$, called the Yangian of $\mathfrak{g}$.
A few years later(1988), V. Drinfeld gave another definition \cite{Dr2} which is the one we use in this thesis.

Let $A=\left(a_{ij}\right)_{i,j\in I}$ be the Cartan matrix of $\g$, where $I=\{1,2,\ldots, l\}$. Let $D=\operatorname{diag}\left(d_1,\ldots, d_{l}\right)$, $d_i\in \N$, such that $d_1, d_2,\ldots, d_l$ are co-prime and $DA$ is symmetric. The Yangian $\yg$ is isomorphic to the
associative algebra with generators $x_{i,r}^{{}\pm{}}$,
$h_{i,r}$, $i\in I$, $r\in\Z_{\geq 0}$, and the following
defining relations:
\begin{equation*}\label{}
[h_{i,r},h_{j,s}]=0, \qquad [h_{i,0},\ x_{j,s}^{\pm}]={}\pm
d_ia_{ij}x_{j,s}^{\pm}, \qquad [x_{i,r}^+, x_{j,s}^-]=\delta_{i,j}h_{i,r+s},
\end{equation*}
\begin{equation*}\label{}
[h_{i,r+1}, x_{j,s}^{\pm}]-[h_{i,r}, x_{j,s+1}^{\pm}]=
\pm\frac{1}{2}d_i
a_{ij}\left(h_{i,r}x_{j,s}^{\pm}+x_{j,s}^{\pm}h_{i,r}\right),
\end{equation*}
\begin{equation*}\label{}
[x_{i,r+1}^{\pm}, x_{j,s}^{\pm}]-
[x_{i,r}^{\pm}, x_{j,s+1}^{\pm}]=\pm\frac12
d_ia_{ij}\left(x_{i,r}^{\pm}x_{j,s}^{\pm}
+x_{j,s}^{\pm}x_{i,r}^{\pm}\right),
\end{equation*}
\begin{equation*}\label{}
\sum_\pi
[x_{i,r_{\pi\left(1\right)}}^{\pm},
[x_{i,r_{\pi\left(2\right)}}^{\pm}, \ldots,
[x_{i,r_{\pi\left(m\right)}}^{\pm},
x_{j,s}^{\pm}]\cdots]]=0, i\neq j,
\end{equation*}
for all sequences of non-negative integers $r_1,\ldots,r_m$, where
$m=1-a_{ij}$ and the sum is over all permutations $\pi$ of $\{1,\dots,m\}$.


The above definition of $\yg$ allows us to define highest weight representations of $\yg$. 
Let $Y^{\pm}$ be the subalgebras of $\yg$ generated by the generators $x_{i, k}^{\pm}$. Set $N^{\pm}=\sum\limits_{i,k} x_{i,k}^{\pm}Y^{\pm}.$
Let $\mu=\Big(\mu_1\left(u\right),\mu_2\left(u\right),\ldots,\mu_l(u)\Big)$, where $\mu_i\left(u\right)=1+\mu_{i,0}u^{-1}+\mu_{i,1}u^{-2}+\ldots$ is a formal series in $u^{-1}$ for $i\in I$.
A representation $V(\mu)$ of $\yg$ is said to be highest weight if it is generated by a vector $v^{+}$ such that
$ x_{i,k}^{+}v^{+}=0$ and $ h_{i,k}v^{+}= \mu_{i,k} v^{+}.$                                                                                                               
The \textbf{Verma module} $M\left(\mu\right)$ is defined to be the quotient of $\yg$ by the left ideal generated by $N^{+}$ and the elements $h_{i,k}-\mu_{i,k}1$.
$\yg$ acts on $M\left(\mu\right)$ by left multiplication.  A highest weight vector of $M\left(\mu\right)$ is $1_{\mu}$ which is the image of the element $1\in \yg$ in the quotient. The Verma module $M(\mu)$ is a universal highest weight representation in the sense that: If $V(\mu)$ is another highest weight representation with a highest weight vector $v$, then the mapping $1_{\mu}\mapsto v$ defines a surjective $\yg$-module homomorphism $M\left(\mu\right)\rightarrow V\left(\mu\right)$.
Pulling back by the imbedding $U(\g)\hookrightarrow \yg$, $V(\mu)$ is a $\g$-module. The weight subspace $V_{\mu^{\left(0\right)}}$ with
$\mu^{\left(0\right)}=\left(\mu_{1,0},\dots,\mu_{l,0}\right)$ is one-dimensional and spanned
by the highest weight vector of $V(\mu)$. All other nonzero weight subspaces
correspond to the weights $\eta$ of the form
$\eta=\mu^{\left(0\right)}-k_1\alpha_1-\ldots-k_l\alpha_l,$
where all $k_i$ are nonnegative integers, not all of them are zero.
A standard argument shows that $M\left(\mu\right)$ has a unique irreducible quotient $L\left(\mu\right)$.

V. Drinfeld described in \cite{Dr2} the finite-dimensional irreducible representation of $\yg$: an irreducible representation $L(\mu)$
is finite-dimensional if and only if there exists an $l$-tuple of polynomials $\pi=\big(\pi_1(u),\ldots,\pi_l(u)\big)$ such that
$\mu_i\left(u\right)=\frac{\pi_i\left(u+d_i\right)}{\pi_i\left(u\right)}$, for $i\in I$. A standard argument shows that $\pi$ is unique if we impose the
conditions that $\pi_i$ is monic. An irreducible representation is called fundamental if there is an $i\in I$ such that $\pi_i(u)=u-a$ and $\pi_{j}(u)=1$, where $j\neq i$ and $a\in \C$. In this case, the fundamental representation is denoted by $V_{a}(\omega_i)$. If $L(\mu)$ is finite-dimensional, $\pi$
is called Drinfeld polynomials of $V$. Let $V(\mu)$ be a finite-dimensional highest weight representation of $\yg$, there exists an $l$-tuple of polynomials $\pi$ associated to its minimal quotient since this quotient of $V(\mu)$ is irreducible and finite-dimensional. We call
$\pi$ the associated polynomials of $V(\mu)$. We will write $V\left(\pi\right)$ instead of
$V(\mu)$ if the latter is finite-dimensional.
V. Chari and A. Pressley shown in \cite{ChPr4} that every finite-dimensional irreducible representation of $\yg$ is a subquotient of a tensor product of fundamental representations.
An explicit realization, however, of these modules is still unknown in general but when $\g=\nysl$. 

The finite-dimensional representation theory of the quantum affine algebra $U_q(\hat{\g})$ over $\C(q)$ is an analogue of the one of $\yg$, where $q$ is an indeterminate and $\C(q)$ is the field of rational functions in $q$ with complex coefficients. Every finite-dimensional irreducible representation $V$ of the quantum affine algebra $U_q(\hat{\g})$  is parameterized by an $l$-tuple of polynomials $P=\big(P_1(u),\ldots, P_l(u)\big)$, where $P_i(u)\in \C(q)[u]$.  The fundamental representations are defined similarly. Denote by $p_{i,j}$ the roots of $P_{i}(u)$ and the fundamental representation associated to $\omega_i$ by $L_{p}(\omega_i)$. It is well known that the finite-dimensional irreducible representation associated to $P$ is a subquotient of $\widetilde{W}=\bigotimes\limits_{ij}L_{p_{i,j}}(\omega_{i})$. In \cite{AkKa}, the authors conjectured that if for any ${p_{i,j}}$ and ${p_{s,t}}$, $\frac{p_{i,j}}{p_{s,t}}$ does not have a pole at $q=0$, then $\widetilde{W}$ is irreducible, and proved this conjecture in the case of type $A_n^{(1)}$ and $C_n^{(1)}$. This conjecture was also proved by E. Frenkel and E. Mukhin \cite{FrMu} using the $q$-characters method, by M. Varagnolo and Vasserot \cite{VaVa} via the quiver varieties method when $\g$ is simply-laced, and by M. Kashiwara in \cite{Ka} through the crystal basis method. This result was generalized in \cite{Ch3} by V. Chari who gave a sufficient condition for the cyclicity of the tensor product of irreducible finite-dimensional representations associated with $P$ such that the roots of $P_i(u)$ form a `q-string' and $P_j(u)=1$, for all $j\neq i$. This condition is obtained by considering a braid group action on the imaginary root vectors. However, there is no braid group action available for the Yangian $\yg$.

In order to have a better understanding of the category of finite-dimensional highest weight representations of quantum affine algebras associated to $P$, the local Weyl module $W(P)$ was introduced in \cite{ChPr1}. The module $W(P)$ has a nice universal property:  any finite-dimensional highest weight representations of $U_q(\hat{\g})$ associated to $P$ is a quotient of $W(P)$. It is known that $W(P)$ is isomorphic to an ordered tensor product of fundamental representations of $U_q(\hat{\g})$. A proof of this fact can be found in \cite{ChMo2}. The notion of a local Weyl module has been extended to the finite-dimensional representations of current algebras \cite{ChLo, FoLi, Na}, twisted loop algebras \cite{ChFoSe}, and current Lie algebras on affine varieties \cite{FeLo1}. We are about to extend the notion to the finite-dimensional representations of Yangians, which are quantization of the enveloping algebras of the current algebras.


Inspired by the notion of the local Weyl modules of quantum affine algebras $U_q(\hat{\g})$, it is reasonable to ask: does there exist a finite-dimensional highest weight representation $W(\pi)$ of $\yg$ such that every finite-dimensional highest weight representation $V\left(\pi\right)$ is a quotient of $W(\pi)$? If so, is $W(\pi)$ finite-dimensional and isomorphic to an ordered tensor product of fundamental representations of $\yg$? In this thesis, we will answer these questions when $\g$ is either a finite-dimensional classical simple Lie algebra or $\g=G_2$.

We now describe the content of this dissertation in more detail. We devote Chapter 1 to an exposition of the background information concerning the local Weyl modules of Yangians. We begin this chapter by introducing classical simple Lie algebras and their fundamental representations, which is related to the finite-dimensional representations of $\yg$.
We next introduce the definition of Yangians,  their highest weight representations and the fundamental representations of Yangians. $\yg$ admits a filtration and its associated graded algebra is isomorphic to $U(\g[t])$, the universal enveloping algebra of the current algebra $\g[t]$. Any highest weight representation $V$ of $\yg$ inherits the filtration, then the corresponding associated graded vector space $\operatorname{gr}(V)$ is a highest weight representation of $U(\g[t])$. Therefore, it is reasonable to introduce the current algebra and its finite-dimensional representations. In the end of this chapter, we introduce the local Weyl module $W(\lambda)$, where $\lambda=\sum\limits_{i\in I} m_i\omega_i$ is a dominant integral weight of $\g$, $\omega_i$ is a fundamental weight and $m_i\in \Z_{\geq 0}$.  The main result in Chapter 1 is that
\begin{center}
$\operatorname{Dim}\big(W(\lambda)\big)=\prod\limits_{i\in I}\Big(\operatorname{Dim}\big(W(\omega_i)\big)\Big)^{m_i}$ and $W(\omega_i)\cong_{\g}  V_{a}(\omega_i)$,
where $a\in \C$.
\end{center}

In Chapter 2, we give a definition of the local Weyl module $W(\pi)$ of Yangian $\yg$ in terms of generator and relations.  The main results in this chapter are that the local Weyl module $W(\pi)$ is finite-dimensional, and $\operatorname{Dim}\big(W\left(\pi\right)\big)\leq \operatorname{Dim}\big(W(\lambda)\big),$ where $W(\lambda)$ is the local Weyl module of the current algebra $\g[t]$ associated to the dominant integral weight $\lambda=\sum\limits_{i\in I} m_i\omega_i$ and $m_i$ is the degree of the polynomial $\pi_i$.

In Chapters 3 to 6, we obtain a description of the local Weyl module $W(\pi)$ of Yangians of finite-dimensional classical simple Lie algebras over $\C$. The ideas in these chapters can be applied to characterize the local Weyl modules of $Y(\g)$ when $\g$ is simple Lie algebra of type $G_2$ in Chapter 7. The main proofs are similar, so we can discuss them together. The main results in Chapter 1
lead us to consider a tensor product of fundamental representations of $\yg$. Let $L=V_{a_1}(\omega_{b_1})\otimes V_{a_2}(\omega_{b_2})\otimes\ldots\otimes V_{a_k}(\omega_{b_k}).$ Our main objective is to show that $L$ is a highest weight representation under certain conditions on the order of the tensor factors.

Let $W_1\left(a\right)$, $a\in \C$, be an evaluation representation of $\ysl$ which is isomorphic to $\C^2$ as an $\nysl$-module. It was proved in \cite{ChPr3} that, as a $\ysl$-module,
$W_1\left(a_1\right)\otimes W_1\left(a_2\right)\otimes\ldots\otimes W_1\left(a_r\right)$ is a highest weight representation if and only if $a_j-a_i\neq 1$ for $1\leq i<j\leq r$. Therefore some restriction on the subscripts $a_i$ in $L$ are expected. By arranging the orders of $V_{a_i}\left(\omega_{b_i}\right)$ in $L$ if necessary, we may assume that $$\operatorname{Re}(a_j)-\operatorname{Re}(a_i)\leq 0$$ for $1\leq i<j\leq l$. We show that $L$ is a highest weight representation of $\yg$ by induction on $k$. Denote by $v_i^{+}$ the highest weight vectors in $V_{a_i}\left(\omega_{b_i}\right)$ and by $v_1^{-}$ the lowest weight vectors in $V_{a_1}\left(\omega_{b_1}\right)$. Without loss of generality, we may assume that $k\geq 2$ and $V_{a_2}(\omega_{b_2})\otimes V_{a_3}(\omega_{b_3})\otimes\ldots\otimes V_{a_k}(\omega_{b_k})$ is a highest weight representation with a highest weight vector $v^{+}=v_2^{+}\otimes v_3^{+}\otimes \ldots \otimes v_k^{+}$.  It is shown in Corollary \ref{v-w+gvtw}, that $v_1^{-}\otimes v^{+}$ generates $L$. If we can show $v_1^{-}\otimes v^{+}\in \yg\left(v_1^{+}\otimes v^{+}\right)$, then $L$ is generated by $v_1^{+}\otimes v^{+}$. The other two conditions on $L$ to be a highest weight representation are easy to check; then $L$ is a highest weight representation of $\yg$.

We now show that $v_1^{-}\otimes v^{+}\in \yg\left(v_1^{+}\otimes v^{+}\right)$. Note that the fundamental representations of $\yg$ are finite-dimensional and $U(\g)\subseteq \yg$.  As $\g$-modules, the fundamental representations of $\yg$ are
completely reducible by Weyl's theorem. Their decomposition as $\g$-modules is known, see \cite{ChPr4}. There is a ``path" from the highest
weight vector $v_{1}^{+}$ to the lowest weight vector $v_{1}^{-}$ by applying negative root vectors in $U(\g)$ to $v_{1}^{+}$. Suppose that
$v^{-}_{1}=x_{n_s,0}^{-}x_{n_{s-1},0}^{-}\ldots x_{n_2,0}^{-}x_{n_1,0}^{-}v^{+}_{1}.$ Define $v_0=v^{+}_{1}$ and $v_j=x_{n_{j},0}^{-}x_{n_{j-1},0}^{-}\ldots
x_{n_1,0}^{-}v^{+}_{1}$ for $1\leq j\leq s$. Let $Y_{n_{j+1}}$ be the subalgebra of $\yg$ generated by $x_{n_{j+1},r}^{\pm}$ and $h_{n_{j+1},r}$, $r\in \Z_{\geq 0}$, which is isomorphic to $\ysl$.
We prove that, as $\ysl$-modules, $$Y_{n_{j+1}}\left(v_j\right)\otimes Y_{n_{j+1}}\left(v_2^{+}\right)\otimes \ldots\otimes Y_{n_{j+1}}\left(v_{k}^{+}\right)=Y_{n_{j+1}}\Big(v_j\otimes \left(v_2^{+}\otimes \ldots\otimes v_{k}^{+}\right)\Big)
.$$ It is obvious that
$v_{j+1}\otimes \left(v_2^{+}\otimes \ldots\otimes v_{k}^{+}\right)$ is contained in $Y_{n_{j+1}}\left(v_j\right)\otimes
Y_{n_{j+1}}\left(v_2^{+}\right)\otimes \ldots\otimes Y_{n_{j+1}}\left(v_{k}^{+}\right)$, so is in $Y_{n_{j+1}}\big(v_j\otimes \left(v_2^{+}\otimes
\ldots\otimes v_{k}^{+}\right)\big)$. Using induction on $j$ downward, we have $$v^{-}_{1}\otimes v^{+}\in \yg\left(v^{+}_{1}\otimes v^{+}\right).$$


Let S be the multiset of all roots of $\pi_i\left(u\right)$ for all $i\in I$. Let $a_{m,n}$ be one of the numbers in $S$ with the maximal real parts, and denote $a_1=a_{m,n}$ and $b_1=m$. Inductively, let $a_{s,t} $ be one of the numbers in $S-\{a_1, a_2,\ldots, a_{r-1}\}$ with the maximal real part, and denote $a_r=a_{s,t}$ and $b_r=s$ $\left(r\geq 2\right)$. Then $L=V_{a_1}(\omega_{b_1})\otimes V_{a_2}(\omega_{b_2})\otimes\ldots\otimes V_{a_k}(\omega_{b_k})$ is a highest weight representation, and its associated polynomial is $\pi$.  Since $L$ is a quotient of $W(\pi)$, $\operatorname{Dim}(L)\leq \operatorname{Dim}\big(W(\pi)\big)$. However, $$\operatorname{Dim}(L)=\operatorname{Dim}\big(W(\lambda)\big)\geq \operatorname{Dim}\big(W(\pi)\big).$$ Thus both $L$ and $W(\pi)$ have the same dimension, and then they must be isomorphic. Therefore the structure of the local Weyl module $W(\pi)$ is obtained.

On the way to prove that $L$ is a highest weight representation, we can impose a finite set to replace the constraint that $\operatorname{Re}(a_j)-\operatorname{Re}(a_i)\leq 0$ for $1\leq i<j\leq k$, and then a sufficient condition for $L$ to be a highest weight representation is obtained: $L=V_{a_1}(\omega_{b_1})\otimes V_{a_2}(\omega_{b_2})\otimes\ldots\otimes V_{a_k}(\omega_{b_k})$ is a highest weight representation if $a_j-a_i\notin S(b_i, b_j)$ for $1\leq i<j\leq k$, where $S(b_i, b_j)$ is a finite set of positive rational numbers.  Note that $L$ is irreducible if and only if both $L$ and the left dual $\ ^tL$ are highest weight representations of $\yg$. It is known that the left dual of a fundamental representation of $\yg$ is also a fundamental representation, thus an irreducibility criterion for $L$ follows immediately from the cyclicity condition. Indeed, if $a_j-a_i\notin S(b_i,b_j)$ for $1\leq i\neq j\leq k$, then $L$ is irreducible. When $\g$ is of type $A$, $L$ is irreducible if and only if $a_j-a_i\notin S(b_i,b_j)$ for $1\leq i\neq j\leq k$.

\end{abstract}

\begin{acknowledgement}
I am deeply grateful to my supervisor, Nicolas Guay, for all the time spent on interpreting basic concepts, checking the detail of the thesis, and making comments and suggestions through the content. He is a dedicated person and sets up an example for me as a supervisor with principles.

I am deeply grateful to Dr. Sudarshan Sehgal for so many valuable advices in the past five years. I truly enjoyed those conversations!

I would like to thank my supervisory committee: Dr. G. Cliff, Dr. N. Guay and Dr. A. Pianzola; and my final examiners: Dr. M. De Montigny, Dr. J. kuttler. Special thanks goes to my external examiner, Dr. A. Moura, who spent time to read my thesis, made valuable comments and suggestions.

I would like to thank the Department of Mathematical and Statistical Science for the supports in teaching and researching.

Many thanks go to Lucy Chen for accompanying and encouraging.
\end{acknowledgement}

\tableofcontents

\begin{symbols}
\begin{stretchtable}
\begin{tabular}{lp{4in}}
$\Bbb{Z}$& the set of integers\\
$\Bbb{Z}_{\geq 0}$& the set of nonnegative integers\\
$\Bbb{C}$& the field of complex numbers\\
$\delta_{ij}$& the Kronecker symbol, which is $1$ if $i=j$ and is $0$ if $i\neq j$\\
$\mathfrak{g}$& finite-dimensional simple Lie algebra over $\C$ \\
$\mathfrak{h}$ & the Cartan subalgebra of $\mathfrak{g}$\\
$l$ & rank of $\g$\\
$I=\{1,2,\ldots, l\}$& the nodes of the Dynkin diagram\\
$\W$& the Weyl group of $\g$\\
$w_0$ & the longest element in $\W$\\
$U\left(\mathfrak{g}\right)$& the universal enveloping algebra of $\g$\\
$\overline{k}$& $\overline{k}=0$ if $k$ is even; $\overline{k}=1$ if $k$ is odd\\
$\pi$ & $l$-tuple of polynomials in $u$\\
$\kappa$ & $\frac{1}{2}\times \text{dual Coxeter number of $\g$}$
\end{tabular}
\end{stretchtable}
\end{symbols}
}
\mainbody

\chapter{Background}
\label{chap:intro}

In this chapter we review some important definitions and theorems concerning the local Weyl modules of Yangians of finite-dimensional classical simple Lie algebras $\g$ over the complex numbers $\C$. The objects of this chapter are the Yangians and their representations, the current algebras and their local Weyl modules.  

The universal enveloping algebra $U(\g)$ can be identified with the subalgebra of $\yg$. Thus every finite-dimensional representation of the Yangian is also a $\g$-module, hence it is isomorphic to a direct sum of irreducible representations of $\g$. So the dimensions of the fundamental representations of Yangians are determined by how they decompose as irreducible $\g$-modules. Therefore we begin this chapter by a review of the finite-dimensional classical simple Lie algebras and their finite-dimensional irreducible representations. 
\section{Classical simple Lie algebras and their representations}


Although there are many ways to choose a Cartan subalgebra of $\g$, we only take the widely used one, which we give in below. Let Dim $\h=l$ and $I=\{1,2,\ldots, l\}$ be nodes of the Dynkin diagram of $\g$. Let $\Delta$ be the root system corresponding to the Cartan subalgebra $\mathfrak{h}$, and $\Delta^{+}$ be the set of all positive roots.  Let $\prod=\{\alpha_1,\ldots,\alpha_l\}$ be the set of simple roots. Denote by $\theta\in \Delta^{+}$ the highest root.  Let $Q=\bigoplus\limits_{i=1}^{l}\Z\alpha_i$  be the root lattice, and Let $Q^{+}=\bigoplus\limits_{i=1}^{l}\Z_{\geq 0}\alpha_i$. 
Let $B\left(\cdot,\cdot\right)$ be the Killing form on $\g$. $B\left(\cdot,\cdot\right)$ is non-degenerate on $\h$. For any $\alpha\in \h^{\ast}$, there exists a unique $h_\alpha\in \h$ such that $\alpha\left(h_\beta\right)=B\left(h_\alpha,h_\beta\right)$. Define a bilinear form $\left(\cdot,\cdot\right)$ on $\h^{\ast}$ by restriction of the Killing form of $\g$ to $\h$. 
The lattice $P$ of integral weights is the set of elements $\lambda\in\h^{\ast}$ such that
$\lambda\left(h_\alpha\right)\in\Z$ for all $\alpha\in \Delta$, and let $P^+$ be the set of dominant
integral such that $\lambda\left(h_\alpha\right)\ge 0$. For $i\in I$,  the
fundamental weight
$\omega_i$ of $\g$ is given by $$\langle \omega_i,\alpha_j\rangle:=\frac{2\left(\omega_i,\alpha_j\right)}{\left(\alpha_j,\alpha_j\right)}=\delta_{ij}.$$ 
Let
$\W\subset {\text{Aut}}\left(\frak h^*\right)$ be the Weyl group of $\frak g$, which is generated by simple reflections $s_i$ ($i\in I$).

We list a theorem concerning finite-dimensional irreducible representations of a simple Lie algebra $\g$ over $\C$.
\begin{theorem}[Theorem 21.2, \cite{Hu}]
If $\lambda\in \h$ is a dominant integral weight, then the irreducible $\g$-module $L=L(\lambda)$ is finite-dimensional, and its set of weights $\prod(\lambda)$ is permuted by $\W$, with $\operatorname{Dim}\left(L_{\mu}\right)=\operatorname{Dim}\left(L_{\sigma\left(u\right)}\right)$ for $\mu\in \prod(\lambda)$ and $\sigma\in \W$.
\end{theorem}

Especially, when $\lambda=\omega_i$ ($i\in I$), the module $L(\omega_i)$ is called the \textbf{$i$-th fundamental representation} of $\g$.


\subsection{Simple Lie algebras $A_{l}$}

Cartan Subalgebra: $\mathfrak{h}=\left\{diag\left(h_1,\ldots, h_{l+1}\right)|h_i\in \C, h_1+\ldots+ h_{l+1}=0\right\}$.\\
For $i=1,2,\ldots, l+1$, $\mu_i$ are the functions defined by $$\mu_i\Big(diag\left(h_1,\ldots, h_{l+1}\right)\Big)=h_i.$$

Root system: $\Phi=\{\mu_i-\mu_j|1\leq i\neq j\leq l+1\}$.

Simple roots: $\{\alpha_1=\mu_1-\mu_2,\ldots, \alpha_l=\mu_{l}-\mu_{l+1}\}$.

Positive roots: $\{\mu_i-\mu_j| 1\leq i< j\leq l+1\}$.

Fundamental weights: $\{\omega_i=\mu_1+\ldots+\mu_{i}|i=1,2\ldots, l\}$.

Longest root: $\mu_1-\mu_{l+1}=\alpha_1+\alpha_2+\ldots+\alpha_l$.

Weyl group: $\W=S_l=\langle s_1,\ldots, s_l\rangle$, where $s_i$ is defined by
\begin{center}
$s_i\left(\mu_i\right) = \mu_{i+1},\quad s_i\left(\mu_{i+1}\right)=\mu_{i}, \quad s_i\left(\mu_j\right) = \mu_j\ \text{for}\ j\neq i, i+1.$
\end{center}

One reduced expression of the longest element in the Weyl group is:
$$w_0=s_l\left(s_{l-1}s_{l}\right)\left(s_{l-2}s_{l-1}s_{l}\right)\ldots\left(s_1s_2\ldots s_{l}\right).$$

Indecomposable Cartan matrix of type $A_l$: $$\left(
                                                \begin{array}{ccccccc}
                                                  2 & -1 &  &  &  &  &  \\
                                                 -1 & 2 & -1 &  &  &  &  \\
                                                   & -1 & 2 & -1 &  &  &  \\
                                                  &  & \cdot & \cdot & \cdot & & \\
                                                  & &  & -1 &2 & -1   &  \\
                                                   &  &  &  & -1 & 2 & -1 \\
                                                   &  &  &  &  & -1 & 2\\
                                                \end{array}
                                              \right).$$
\begin{theorem}\label{frosllc1}
The $i$-fundamental module for the simple Lie algebra of type $A_l$ isomorphic to $\bigwedge^{i} V$, where $V=\C^{l+1}$; its dimension is ${l+1\choose i}$.
\end{theorem}

\subsection{Simple Lie algebras $B_{l}$}

Cartan Subalgebra: $\mathfrak{h}=\{diag\left(0, h_1,\ldots, h_{l},-h_{1},\ldots,-h_{l}\right)|h_i\in \C\}$.\\
For $i=1,2,\ldots, l$, $\mu_i$ are the functions defined by $$\mu_i\Big(diag\left(0, h_1,\ldots, h_{l},-h_{1},\ldots,-h_{l}\right)\Big)=h_i.$$

Root system: $\Phi=\{\pm\mu_i\pm\mu_j|1\leq i\neq j\leq l\}\bigcup \{\pm\mu_i\}$.

Positive roots: $\{\mu_i\pm\mu_j|1\leq i< j\leq l\}\bigcup \{\mu_i\}$.

Simple roots: $\{\alpha_1=\mu_1-\mu_2,\ldots, \alpha_{l-1}=\mu_{l-1}-\mu_{l},\alpha_{l}=\mu_{l}\}$.

Longest root: $\mu_1+\mu_2=\alpha_1+2\alpha_2+2\alpha_3+\ldots+2\alpha_{l-1}+2\alpha_{l}$.

Weyl group: $\W=\langle s_1,\ldots, s_l\rangle$, where $s_i$, $1\leq i\leq l-1$, is defined by
$$s_i\left(\mu_i\right)=\mu_{i+1}, \quad s_i\left(\mu_{i+1}\right)=\mu_{i}, \quad s_i\left(\mu_j\right) = \mu_j \ \text{for}\ j\neq i, i+1;$$
for $i=l$, $$s_l\left(\mu_{l}\right)=-\mu_{l}, \quad s_l\left(\mu_j\right) = \mu_j\ \text{for}\ j\neq l.$$\

One reduced expression of the longest element in the Weyl group is:
\begin{eqnarray*}
  w_0=-1&=&s_{l}\left(s_{l-1}s_ls_{l-1}\right)\left(s_{l-2}s_{l-1}s_{l}s_{l-1}s_{l-2}\right)\ldots \\
   &&  \left(s_2\ldots s_{l-1}s_{l}s_{l-1}\ldots s_2\right)\left(s_1s_2\ldots s_{l-1}s_{l}s_{l-1}\ldots s_2 s_1\right).
\end{eqnarray*}

Fundamental weights: $\{\omega_i|\omega_i=\mu_1+\ldots+\mu_{i}$ for $1\leq i\leq l-1$, and\\ $\omega_{l}=\frac{1}{2}\left(\mu_1+\ldots+\mu_{l-1}+\mu_{l}\right)\}$.

Indecomposable Cartan matrix of type $B_l$: $$\left(
                                                \begin{array}{cccccccc}
                                                  2 & -1 &  &  &  &  & &\\
                                                 -1 & 2 & -1 &  &  &  & & \\
                                                   & -1 & 2 & -1 &  &  & & \\
                                                  &  & \cdot & \cdot & \cdot & && \\
                                                  & &  & -1 &2 & -1   & & \\
                                                   &  &  &  & -1 & 2 & -1& \\
                                                   &  &  &  &  & -1 & 2& -1\\
                                                      &  &  &  &  &  & -2&2\\
                                                \end{array}
                                              \right).$$
\begin{theorem}\label{fmotbc1}
For $1\leq i\leq l-1$,  the i-fundamental module for the simple Lie algebra $B_l$ is isomorphic to $\bigwedge^{i} V$ with dimension ${2l+1\choose i}$, where $V=\C^{2l+1}$; $L\left(\omega_{l}\right)\cong U$, where $U$ is the spin representation with dimension $2^{l}$.
\end{theorem}

\subsection{Simple Lie algebras $C_{l}$}
%

Cartan Subalgebra: $\mathfrak{h}=\{diag\left( h_1,\ldots, h_{l},-h_{1},\ldots,-h_{l}\right)|h_i\in \C\}$.\\
For $i=1,2,\ldots, l$, $\mu_i$ are the functions defined by $$\mu_i\Big(diag\left( h_1,\ldots, h_{l},-h_{1},\ldots,-h_{l}\right)\Big)=h_i.$$

Root system: $\Phi=\{\pm\mu_i\pm\mu_j|1\leq i\neq j\leq l\}\bigcup \{\pm 2\mu_i\}$.

Positive roots: $\{\mu_i\pm\mu_j|1\leq i< j\leq l\}\bigcup \{2\mu_i\}$.

Simple roots: $\{\alpha_1=\mu_1-\mu_2,\ldots, \alpha_{l-1}=\mu_{l-1}-\mu_{l},\alpha_{l}=2\mu_{l}\}$.

Longest root: $2\mu_1=2\alpha_1+2\alpha_2+2\alpha_3+\ldots+2\alpha_{l-1}+\alpha_{l}$.

Weyl group: $\W=\langle s_1,\ldots, s_l\rangle $, where $s_i$, $1\leq i\leq l-1$, is defined by $$s_i\left(\mu_i\right)=\mu_{i+1}, \quad s_i\left(\mu_{i+1}\right)=\mu_{i}, \quad s_i\left(\mu_j\right) = \mu_j \ \text{for}\ j\neq i, i+1;$$
for $i=l$, $$s_l\left(\mu_{l}\right)=-\mu_{l}, \quad s_l\left(\mu_j\right) = \mu_j\ \text{for}\ j\neq l.$$\

One reduced expression of the longest element, $-1$, in the Weyl group is:
\begin{eqnarray*}
  w_0=-1&=&s_{l}\left(s_{l-1}s_ls_{l-1}\right)\left(s_{l-2}s_{l-1}s_{l}s_{l-1}s_{l-2}\right)\ldots \\
   &&  \left(s_2\ldots s_{l-1}s_{l}s_{l-1}\ldots s_2\right)\left(s_1s_2\ldots s_{l-1}s_{l}s_{l-1}\ldots s_2 s_1\right).
\end{eqnarray*}

Fundamental weights: $\{\omega_i|\omega_i=\mu_1+\ldots+\mu_{i}$ for $1\leq i\leq l\}$.

Indecomposable Cartan matrix of type $C_l$: $$\left(
                                                \begin{array}{cccccccc}
                                                  2 & -1 &  &  &  &  & &\\
                                                 -1 & 2 & -1 &  &  &  & & \\
                                                   & -1 & 2 & -1 &  &  & & \\
                                                  &  & \cdot & \cdot & \cdot & && \\
                                                   &  &  &  & -1 & 2 & -1& \\
                                                   &  &  &  &  & -1 & 2& -2\\
                                                      &  &  &  &  &  & -1&2\\
                                                \end{array}
                                              \right).$$

\begin{theorem}\label{fmsimlieC}
The i-th fundamental modules $L(\omega_i)$ for the simple Lie algebra of type $C_l$ are given as follows.
\begin{enumerate}
  \item $L\left(\omega_1\right)$ is the natural $2l$-dimensional $C_l$-module $V\cong \C^{2l}$.
  \item For $2\leq i\leq l$, $L(\omega_i)$ is the submodule of $\bigwedge ^{i} V$ given by the kernel of the contraction map $\theta: \bigwedge^{i}V\rightarrow \bigwedge^{i-2} V$; and $Dim\Big(L\left(\omega_j\right)\Big)={2l\choose j}-{2l\choose j-2}$.
\end{enumerate}
\end{theorem}

\subsection{Simple Lie algebras $D_{l}$}
Cartan Subalgebra: $\mathfrak{h}=\{diag\left( h_1,\ldots, h_{l},-h_{1},\ldots,-h_{l}\right)|h_i\in \C\}$.\\
For $i=1,2,\ldots, l$, $\mu_i$ are the functions defined by $$\mu_i\Big(diag\left( h_1,\ldots, h_{l},-h_{1},\ldots,-h_{l}\right)\Big)=h_i.$$

Root system: $\Phi=\{\pm\mu_i\pm\mu_j|1\leq i\neq j\leq l\}$.

Simple roots: $\{\alpha_1=\mu_1-\mu_2,\ldots, \alpha_{l-1}=\mu_{l-1}-\mu_{l},\alpha_{l}=\mu_{l-1}+\mu_{l}\}$.

Positive roots: $\{\mu_i\pm\mu_j| 1\leq i< j\leq l\}$.

Longest root: $\mu_1+\mu_2=\alpha_1+2\alpha_2+2\alpha_3+\ldots+2\alpha_{l-2}+\alpha_{l-1}+\alpha_{l}$.

Fundamental weights: $\{\omega_i|\omega_i=\mu_1+\ldots+\mu_{i}$ for $1\leq i\leq l-2, \\ \qquad\omega_{l-1}=\frac{1}{2}\left(\mu_1+\ldots+\mu_{l-1}-\mu_{l}\right),\ \omega_{l}=\frac{1}{2}\left(\mu_1+\ldots+\mu_{l-1}+\mu_{l}\right)\}$.

Weyl group: $\W=\langle s_1,\ldots, s_l\rangle$, where $s_i$, $1\leq i\leq l-1$, is defined by
$$s_i\left(\mu_i\right) = \mu_{i+1}, \quad s_i\left(\mu_{i+1}\right) = \mu_{i}, \quad s_i\left(\mu_j\right) = \mu_j \quad\text{for}\quad j\neq i, i+1;$$
for $i=l$,
$$ s_l\left(\mu_{l-1}\right) = -\mu_{l}, \quad  s_l\left(\mu_{l}\right) = -\mu_{l-1}, \quad s_l\left(\mu_j\right) = \mu_j \  \text{for}\ j\neq l-1, l.
$$
One reduced expression of the longest element in the Weyl group:
\begin{align*}
 w_0&=s_{l}s_{l-1}\left(s_{l-2}s_{l}s_{l-1}s_{l-2}\right)\ldots \left(s_3\ldots s_{l-2}s_{l}s_{l-1}s_{l-2}\ldots s_3\right)\\
   &\qquad\left(s_2\ldots s_{l-2}s_{l}s_{l-1}s_{l-2}\ldots s_2\right)\left(s_1s_2\ldots s_{l-2}s_{l}s_{l-1}s_{l-2}\ldots s_2 s_1\right).
\end{align*}
Indecomposable Cartan matrix of type $D_l$: $$\left(
                                                \begin{array}{cccccccc}
                                                  2 & -1 &  &  &  &  & &\\
                                                 -1 & 2 & -1 &  &  &  & & \\
                                                   & -1 & 2 & -1 &  &  & & \\
                                                  &  & \cdot & \cdot & \cdot & && \\
                                                  & &  & -1 &2 & -1   & & \\
                                                   &  &  &  & -1 & 2 & -1&-1 \\
                                                   &  &  &  &  & -1 & 2& 0\\
                                                      &  &  &  &  &  -1& 0&2\\
                                                \end{array}
                                              \right).$$
\begin{theorem}\label{dofrdl}
For $1\leq i\leq l-2$,  the i-fundamental module for the simple Lie algebra of type $D_l$ is isomorphic to $\bigwedge^{i} V$ with dimension ${2l\choose i}$, where $V=\C^{l+1}$; $L\left(\omega_{l-1}\right)\cong U^{-}$ and $L\left(\omega_{l}\right)\cong U^{+}$, where $U^{-}$ and $U^{+}$ are the spin representations with dimension $2^{l-1}$.
\end{theorem}

\section{Yangians and their representations}

\subsection{Definition of Yangians}


Let $A=\left(a_{ij}\right)_{i,j\in I}$ be the Cartan matrix of $\g$, and $D=\operatorname{diag}\left(d_1,\ldots, d_{l}\right)$, $d_i\in \N$, such that $d_1, d_2,\ldots, d_l$ are co-prime and $DA$ is symmetric. The Yangian $Y\left(\frak g\right)$ is isomorphic to the
associative algebra with generators $x_{i,r}^{{}\pm{}}$,
$h_{i,r}$, $i\in I$, $r\in\Z_{\geq 0}$, and the following
defining relations:
\begin{equation*}\label{}
[h_{i,r},h_{j,s}]=0, \qquad [h_{i,0},\ x_{j,s}^{\pm}]={}\pm
d_ia_{ij}x_{j,s}^{\pm}, \qquad [x_{i,r}^+, x_{j,s}^-]=\delta_{i,j}h_{i,r+s},
\end{equation*}
\begin{equation*}\label{}
[h_{i,r+1}, x_{j,s}^{\pm}]-[h_{i,r}, x_{j,s+1}^{\pm}]=
\pm\frac{1}{2}d_i
a_{ij}\left(h_{i,r}x_{j,s}^{\pm}+x_{j,s}^{\pm}h_{i,r}\right),
\end{equation*}
\begin{equation*}\label{}
[x_{i,r+1}^{\pm}, x_{j,s}^{\pm}]-
[x_{i,r}^{\pm}, x_{j,s+1}^{\pm}]=\pm\frac12
d_ia_{ij}\left(x_{i,r}^{\pm}x_{j,s}^{\pm}
+x_{j,s}^{\pm}x_{i,r}^{\pm}\right),
\end{equation*}
\begin{equation*}\label{}
\sum_\pi
[x_{i,r_{\pi\left(1\right)}}^{\pm},
[x_{i,r_{\pi\left(2\right)}}^{\pm}, \ldots,
[x_{i,r_{\pi\left(m\right)}}^{\pm},
x_{j,s}^{\pm}]\cdots]]=0, i\neq j,
\end{equation*}
for all sequences of non-negative integers $r_1,\ldots,r_m$, where
$m=1-a_{ij}$ and the sum is over all permutations $\pi$ of $\{1,\dots,m\}$.

 The Yangian $\yg$ admits a filtration defined by setting the degree of $x_{i,r}^{\pm}$ and $h_{i,r}$ to be $r$. Let $\yg_r$, for $r\geq 0$, be the linear span of all monomials in the generators $x_{i,s}^{\pm}$, $h_{i,s}$ for which the sum of the indices $s$ is at most $r$. It follows from the definition that $\yg_r\subset \yg_{r+1}$, and $\yg_r\cdot\yg_s\subset \yg_{r+s}$. The associated graded algebra is denoted by gr$\yg$.

 Suppose in $\g$ that $$x_{\beta}^{\pm}=c[x_{i_1}^{\pm},[x_{i_2}^{\pm},\ldots,[x_{i_{k-1}}^{\pm},x_{i_{k}}^{\pm}]\ldots]],$$ where $0\neq c\in\C$. For each $r\in \N$, let $r=r_1+\ldots+r_k$ be a partition of $r$ into a sum of $k$ non-negative integers, and define $$x_{\beta,r}^{\pm}=c[x_{i_1,r_1}^{\pm},[x_{i_2,r_2}^{\pm},\ldots,[x_{i_{k-1},r_{k-1}}^{\pm},x_{i_{k}, r_k}^{\pm}]\ldots]].$$
If $\tilde{x}_{\beta,r}^{\pm}$ is obtained by using another different partition of $r$, then by the defining relations of the Yangian $$\tilde{x}_{\beta,r}^{\pm}-x_{\beta,r}^{\pm}\in \yg_{r-1}.$$

\begin{proposition}[Proposition 12.1.6, \cite{ChPr2}]
The Yangian $\yg$ has the structure of a filtered algebra, such that the associated graded algebra is isomorphic to $U\left(\g\otimes_\C\C[t]\right)$.
\end{proposition}

We state the PBW theorem for Yangians.

\begin{proposition}[Proposition 12.1.8, \cite{ChPr2}]
Fix a total ordering on the set $$\sum=\{x_{\beta,r}^{\pm}|\beta\in \Delta^{+}, r\in\Z_{\geq 0}\}\bigcup \{h_{i,r}|i\in I, r\in\Z_{\geq 0}\}.$$ Then the set of ordered monomials in the elements of $\sum$ is a vector space basis of $\yg$.
\end{proposition}
In this thesis, we choose the order as $$x_{\beta,r_1}^{-}\preceq h_{\beta,r_2}\preceq x_{\beta,r_3}^{+}.$$

The universal enveloping algebra $U(\g)$ can be identified with the subalgebra of $\yg$ generated by the elements $x^{\pm}_{i,0}$ and $h_{i,0}$ for $i=1,\ldots, n$. Therefore $$U(\g)=\yg_0.$$
For a fixed $i$ and $r\in\Z_{\geq 0}$, both $x^{\pm}_{i,r}$ and $h_{i,r}$ generate a subalgebra of $\yg$ which isomorphic to $Y\left(\mathfrak{sl}_2\right)$.

Let $Y^{\pm}$, $H$ be the subalgebras of $\yg$ generated by the $x_{\beta, k}^{\pm}$, $h_{ik}$ respectively. Set $$N^{\pm}=\sum_{i,k} x_{ik}^{\pm}Y^{\pm}.$$

We next introduce an automorphism of $\yg$, which plays an important role in the representation theory of Yangians.
\begin{proposition}[Proposition 2.6, \cite{ChPr4}]\label{s1.2.1p26}
The assignment $$\tau_{a}(h_{i,k})=\sum_{r=0}^{k}{k\choose r}a^{k-r}h_{i,r},\qquad \tau_{a}(x_{i,k}^{\pm})=\sum_{r=0}^{k}{k\choose r}a^{k-r}x_{i,r}^{\pm}$$ extends a Hopf algebra automorphism of $\yg$.
\end{proposition}
\subsection{Coproduct of Yangians}
In this subsection, we denote by $\Delta_{\ysl}$ and $\Delta_{\yg}$ the coproduct of $\ysl$ and $\yg$, respectively. Define $N^{\pm}_{i}$ to be the subalgebra of $\yg$ generated by all monomials in $x_{\alpha,k}^{\pm}$ with at least one factor with $\alpha\neq\alpha_i$. Let $M^{\pm}_i$ and $H_i$ be the subalgebra of $\yg$ generated by $x_{i,k}^{\pm}$ and $h_{i,k}$, respectively. When $\g=\nysl$, $Y^{\pm}$ would be a better notation than $M_i^{\pm}$.

There is no explicit formula for the coproduct for the realization used in this thesis. In \cite{ChPr5}, part of the coproduct of $\ysl$ is defined as:
$$\Delta_{\ysl}(x_{k}^{-})=x_{k}^{-}\otimes 1 +1\otimes x_{k}^{-}+\sum_{s=1}^{k}x_{k-s}^{-}\otimes h_{s-1}\text{\ modulo\ }\sum_{p,q,r} Yx_{p}^{-}x_{q}^{-}\otimes Yx_{r}^{+};$$
$$\Delta_{\ysl}(x_{k}^{+})=x_{k}^{+}\otimes 1 +1\otimes x_{k}^{+}+\sum_{s=1}^{k}h_{s-1}\otimes x_{k-s}^{+}\text{\ modulo\ }\sum_{p,q,r} Yx_{p}^{-}\otimes Y x_{q}^{+}x_{r}^{+}.$$
For the former formula, it follows from the proof in this paper that we can replace the first $Y$ in $\sum_{p,q,r} Yx_{p}^{-}x_{q}^{-}\otimes Yx_{r}^{+}$ by $HY^{-}$ and the second $Y$ by $H$. For the latter formula, we can replace the first $Y$ in $\sum_{p,q,r} Yx_{p}^{-}\otimes Y x_{q}^{+}x_{r}^{+}$ by $H$ and the second $Y$ by $HY^{+}$.
In \cite{ChPr4}, the following proposition was proved.
\begin{proposition}[Proposition 2.8,\cite{ChPr4}]\label{Delta}\

Define $Y:=\yg$. Modulo $\overline{Y}\equiv \left(N^{-}Y\otimes YN^{+}\right)\bigcap \left(YN^{-}\otimes N^{+}Y\right)$, we have
\begin{enumerate}
  \item $\Delta_{\yg}\left(x_{i,k}^{+}\right)=x_{i,k}^{+}\otimes 1+1\otimes x_{i,k}^{+}+\sum\limits_{j=1}^{k}h_{i,j-1}\otimes x_{i,k-j}^{+}$,
  \item $\Delta_{\yg}\left(x_{i,k}^{-}\right)=x_{i,k}^{-}\otimes 1+1\otimes x_{i,k}^{-}+\sum\limits_{j=1}^{k}x_{i,k-j}^{-}\otimes h_{i,j-1}$,
  \item $\Delta_{\yg}\left(h_{i,k}\right)=h_{i,k}\otimes 1+1\otimes h_{i,k}+\sum\limits_{j=1}^{k}h_{i,j-1}\otimes h_{i,k-j}$.
\end{enumerate}
\end{proposition}

In $\yg$, for a fixed $i\in I$, $\{x_{i,r}^{\pm}, h_{i,r}|r\geq 0\}$ generates $Y_i$, and the assignment $\frac{\sqrt{d_i}}{d_i^{k+1}}x_{i,k}^{\pm}\rightarrow x_{k}^{\pm}$ and $\frac{1}{d_i^{k+1}}h_{i,k}\rightarrow h_{k}$ extends an automorphism from $Y_i$ to $\ysl$. Let $\overline{h}_{1}=h_{1}-\frac{1}{2}h_{0}^2$ and $\overline{h}_{i,1}=h_{i,1}-\frac{1}{2}h_{i,0}^2$. It is proved in \cite{ChPr5} that $\Delta_{\ysl}\left(\bar{h}_{1}\right)=\bar{h}_{1}\otimes 1+1\otimes \bar{h}_{1}-2x^{-}_{0}\otimes x^{+}_{0}.$ The Hopf algebra homomorphism $Y_i\cong \ysl$ tells that $$\Delta_{Y_i}\left(\bar{h}_{i,1}\right)=\bar{h}_{i,1}\otimes 1+1\otimes \bar{h}_{i,1}-\left(\alpha_i,\alpha_i\right)x^{-}_{i,0}\otimes x^{+}_{i,0}.$$

There is a difference between the coproduct of the Yangian $Y_i$ and the coproduct of Yangians $\yg$. For instance,
$$\Delta_{Y_i}\left(\bar{h}_{i,1}\right)=\bar{h}_{i,1}\otimes 1+1\otimes \bar{h}_{i,1}-\left(\alpha_i,\alpha_i\right)x^{-}_{i,0}\otimes x^{+}_{i,0};$$
$$\Delta_{\yg}\left(\bar{h}_{i,1}\right)=\bar{h}_{i,1}\otimes 1+1\otimes \bar{h}_{i,1}-\left(\alpha_i,\alpha_i\right)x^{-}_{i,0}\otimes x^{+}_{i,0}-\sum\limits_{\alpha\succ 0 , \alpha\neq \alpha_i}\left(\alpha,\alpha_i\right)x^{-}_{\alpha,0}\otimes x^{+}_{\alpha,0}.$$ In this subsection, we will show that $\Delta_{\yg}\left(x_{ik_1}^{-}\ldots x_{ik_s}^{-}\right)$ acts on any tensor product of highest weight vectors in the same way as $\Delta_{Y_i}\left(x_{ik_1}^{-}\ldots x_{ik_s}^{-}\right)$.
We first show the difference between $\Delta_{\yg}\left(x_{ik_1}^{-}\right)$ and $\Delta_{Y_i}\left(x_{ik_1}^{-}\right)$.
\begin{proposition}\
$\Delta_{\yg}\left(x^{-}_{i,k}\right)-\Delta_{Y_i}\left(x^{-}_{i,k}\right)\in H_iN^{-}_i\otimes H_iN^{+}_i.$
\end{proposition}
\begin{proof} We prove this proposition by induction. Note that $[\bar{h}_{i,1}, x^{\pm}_{i,k}]=\pm\left(\alpha_i,\alpha_i\right)x^{\pm}_{i,k+1}.$
It is shown in \cite{ChPr4} that
\begin{align*}\label{C6delta2}
\Delta_{\yg}\left(\bar{h}_{i,1}\right)&=\bar{h}_{i,1}\otimes 1+1\otimes \bar{h}_{i,1}-\left(\alpha_i,\alpha_i\right)x^{-}_{i,0}\otimes x^{+}_{i,0}-\sum\limits_{\alpha\succ 0 , \alpha\neq \alpha_i}\left(\alpha,\alpha_i\right)x^{-}_{\alpha,0}\otimes x^{+}_{\alpha,0}\\
&=\Delta_{Y_i}\left(\bar{h}_{i,1}\right)-\sum\limits_{\alpha\succ 0 , \alpha\neq \alpha_i}\left(\alpha,\alpha_i\right)x^{-}_{\alpha,0}\otimes x^{+}_{\alpha,0}
\end{align*}
and
\begin{equation*}
\Delta_{\yg}\left(x^{-}_{i,0}\right)=x^{-}_{i,0}\otimes 1+ 1\otimes x^{-}_{i,0}=\Delta_{Y_i}\left(x^{-}_{i,0}\right).
\end{equation*}
\begin{align*}
\Delta_{\yg}&\left(x^{-}_{i,1}\right)\\
&=\frac{-1}{\left(\alpha_i,\alpha_i\right)}\Delta_{\yg}\left([\bar{h}_{i,1},x^{-}_{i,0}]\right)\\
&=\frac{-1}{\left(\alpha_i,\alpha_i\right)}\left([\Delta_{\yg}\left(\bar{h}_{i,1}\right),\Delta_{\yg}\left(x^{-}_{i,0}\right)]\right)\\
&=\frac{-1}{\left(\alpha_i,\alpha_i\right)}[\Delta_{Y_i}\left(\bar{h}_{i,1}\right)-\sum\limits_{\alpha\succ 0 , \alpha\neq \alpha_i}\left(\alpha,\alpha_i\right)x^{-}_{\alpha,0}\otimes x^{+}_{\alpha,0}, \Delta_{Y_i}\left(x^{-}_{i,0}\right)]\\
&=\frac{-1}{\left(\alpha_i,\alpha_i\right)}\Delta_{Y_i}[\bar{h}_{i,1},x^{-}_{i,0}]\\
&+[\sum\limits_{\alpha\succ 0 ,\alpha\neq \alpha_i}\frac{\left(\alpha,\alpha_i\right)}{\left(\alpha_i,\alpha_i\right)}x^{-}_{\alpha,0}\otimes x^{+}_{\alpha,0},x^{-}_{i,0}\otimes 1+ 1\otimes x^{-}_{i,0}]\\
&=\frac{-1}{\left(\alpha_i,\alpha_i\right)}\Delta_{Y_i}[\left(\bar{h}_{i,1}\right),\left(x^{-}_{i,0}\right)]+[\sum\limits_{\alpha\succ 0 ,\alpha\neq \alpha_i}\frac{\left(\alpha,\alpha_i\right)}{\left(\alpha_i,\alpha_i\right)}x^{-}_{\alpha,0}\otimes x^{+}_{\alpha,0},x^{-}_{i,0}\otimes 1]\\
&+[\sum\limits_{\alpha\succ 0 ,\alpha\neq \alpha_i}\frac{\left(\alpha,\alpha_i\right)}{\left(\alpha_i,\alpha_i\right)}x^{-}_{\alpha,0}\otimes x^{+}_{\alpha,0}, 1\otimes x^{-}_{i,0}]\\
&\equiv \Delta_{Y_i}\left(x^{-}_{i,1}\right)+N^{-}_{i}\otimes H_iN^{+}_{i}.
\end{align*}
The only difficulty for the above computations is to show that $$[\sum\limits_{\alpha\succ 0 ,\alpha\neq \alpha_i}\frac{\left(\alpha,\alpha_i\right)}{\left(\alpha_i,\alpha_i\right)}x^{-}_{\alpha,0}\otimes x^{+}_{\alpha,0}, 1\otimes x^{-}_{i,0}]\in H_iN^{-}_{i}\otimes H_iN^{+}_{i}.$$
Indeed, for $\alpha\succ 0$ and $\alpha\neq \alpha_i$, by induction, we can show that $$[[\ldots[[x^{+}_{\alpha,0},x^{-}_{i,r_1}],x^{-}_{i,r_2}]\ldots]x^{-}_{i,r_m}]\in HN_{i}^{+}\quad \text{for any $m\geq 1$.}$$
Thus $[x^{-}_{\alpha,0}\otimes x^{+}_{\alpha,0}, 1\otimes x^{-}_{i,0}]=x^{-}_{\alpha,0}\otimes[ x^{+}_{\alpha,0}, x^{-}_{i,0}]\in H_iN^{-}_{i}\otimes H_iN^{+}_{i}.$

Suppose that $\Delta_{\yg}\left(x^{-}_{i,k-1}\right)=\Delta_{Y_i}\left(x^{-}_{i,k-1}\right)+H_iN^{-}_{i}\otimes H_iN^{+}_{i}$. 
We next show: $\Delta_{\yg}\left(x^{-}_{i,k}\right)=\Delta_{Y_i}\left(x^{-}_{i,k}\right)+H_iN^{-}_{i}\otimes H_iN^{+}_{i}$. 
Note that $\Delta_{Y_i}\left(x^{-}_{i,k-1}\right)=x_{i,k-1}^{-}\otimes 1+1\otimes x_{i,k-1}^{-}+\sum\limits_{s=1}^{k-1}x_{i,k-1-s}^{-}\otimes h_{i,s-1}+H_i\left(M^{-}_{i}\right)^2\otimes H_iM^{+}_{i}$.
\begin{align*}
\Delta_{\yg}&\left(x^{-}_{i,k}\right)\\
&=\frac{-1}{\left(\alpha_i,\alpha_i\right)}\Delta_{\yg}\left([\bar{h}_{i,1},x^{-}_{i,k-1}]\right)\\
&\equiv \frac{-1}{\left(\alpha_i,\alpha_i\right)}\left([\Delta_{\yg}\left(\bar{h}_{i,1}\right),\Delta_{\yg}\left(x^{-}_{i,k-1}\right)]\right)\\
&\equiv \frac{-1}{\left(\alpha_i,\alpha_i\right)}[\Delta_{Y_i}\left(\bar{h}_{i,1}\right)-\sum\limits_{\alpha\succ 0 , \alpha\neq \alpha_i}\left(\alpha,\alpha_i\right)x^{-}_{\alpha,0}\otimes x^{+}_{\alpha,0},\\
&\qquad\qquad\qquad\Delta_{Y_i}\left(x^{-}_{i,k-1}\right)+H_iN^{-}_{i}\otimes H_iN^{+}_{i}]\\
&\equiv \frac{-1}{\left(\alpha_i,\alpha_i\right)}\Delta_{Y_i}[\bar{h}_{i,1},x^{-}_{i,k-1}]+\frac{-1}{\left(\alpha_i,\alpha_i\right)}[\Delta_{Y_i}\left(\bar{h}_{i,1}\right),H_iN^{-}_{i}\otimes H_iN^{+}_{i}]\\
&+[\sum\limits_{\alpha\succ 0 ,\alpha\neq \alpha_i}\frac{\left(\alpha,\alpha_i\right)}{\left(\alpha_i,\alpha_i\right)}x^{-}_{\alpha,0}\otimes x^{+}_{\alpha,0},\Delta_{Y_i}\left(x^{-}_{i,k-1}\right)]\\
&+[\sum\limits_{\alpha\succ 0 ,\alpha\neq \alpha_i}\frac{\left(\alpha,\alpha_i\right)}{\left(\alpha_i,\alpha_i\right)}x^{-}_{\alpha,0}\otimes x^{+}_{\alpha,0},H_iN^{-}_{i}\otimes H_iN^{+}_{i}]\\
&\equiv \Delta_{Y_i}\left(x^{-}_{i,k}\right)+H_iN^{-}_{i}\otimes H_iN^{+}_{i}\\
&+[\sum\limits_{\alpha\succ 0 ,\alpha\neq \alpha_i}\frac{\left(\alpha,\alpha_i\right)}{\left(\alpha_i,\alpha_i\right)}x^{-}_{\alpha,0}\otimes x^{+}_{\alpha,0},\\
&\qquad x_{i,k-1}^{-}\otimes 1+1\otimes x_{i,k-1}^{-}+\sum\limits_{s=1}^{k-1}x_{i,k-1-s}^{-}\otimes h_{i,s-1}+H_i\left(M^{-}_{i}\right)^2\otimes H_iM^{+}_{i}]\\
&+[\sum\limits_{\alpha\succ 0 ,\alpha\neq \alpha_i}\frac{\left(\alpha,\alpha_i\right)}{\left(\alpha_i,\alpha_i\right)}x^{-}_{\alpha,0}\otimes x^{+}_{\alpha,0},H_iN^{-}_{i}\otimes H_iN^{+}_{i}]\\
&\equiv \Delta_{Y_i}\left(x^{-}_{i,k}\right)+H_iN^{-}_{i}\otimes H_iN^{+}_{i}.
\end{align*}
By induction, the proposition is proved.
\end{proof}
Similarly, we can show that
\begin{proposition}\label{deltaxpc1}
 $\Delta_{\yg}\left(x^{+}_{i,k}\right)-\Delta_{Y_i}\left(x^{+}_{i,k}\right)\in HN^{-}\otimes HN^{+}.$
\end{proposition}




Next, we prove the following proposition which will be crucial for this thesis.
\begin{proposition}\label{c1dgd2d}
$\Delta_{\yg}\left(x_{ik_s}^{-}\ldots x_{ik_1}^{-}\right)-\Delta_{Y_i}\left(x_{ik_s}^{-}\ldots x_{ik_1}^{-}\right)\in HN^{-}_{i}\otimes M_i^{-}HN^{+}_{i}$.
\end{proposition}
\begin{proof} We proof this proposition by induction. When $s=2$, the basis of induction.
 \begin{align*}
&\Delta_{\yg}\left(x^{-}_{i,k_2}x^{-}_{i,k_1}\right)\\
&=\Delta_{\yg}\left(x^{-}_{i,k_2}\right)\Delta_{\yg}\left(x^{-}_{i,k_1}\right)\\
&\equiv\left( \Delta_{Y_i}\left(x^{-}_{i,k_2}\right)+H_iN^{-}_{i}\otimes H_iN^{+}_{i}\right)\left(\Delta_{Y_i}\left(x^{-}_{i,k_1}\right)+H_iN^{-}_{i}\otimes H_iN^{+}_{i}\right)\\
&\equiv\Delta_{Y_i}\left(x^{-}_{i,k_2}\right)\cdot \Delta_{Y_i}\left(x^{-}_{i,k_1}\right)+ \Delta_{Y_i}\left(x^{-}_{i,k_2}\right)\cdot H_iN^{-}_{i}\otimes H_iN^{+}_{i}\\
&+H_iN^{-}_{i}\otimes H_iN^{+}_{i}\cdot \Delta_{Y_i}\left(x^{-}_{i,k_1}\right)+H_iN^{-}_{i}\otimes H_iN^{+}_{i}\cdot H_iN^{-}_{i}\otimes H_iN^{+}_{i}\\
&\equiv\Delta_{Y_i}\left(x^{-}_{i,k_2}\right)\Delta_{Y_i}\left(x^{-}_{i,k_1}\right)+H_iN^{-}_{i}\otimes x_{i,k_1}^{-}H_iN^{+}_{i}\\
&+HN^{-}_{i}\otimes x_{i,k_2}^{-}HN^{+}_{i}+H_iN^{-}_{i}\otimes H_iN^{+}_{i}\\
&\equiv\Delta_{Y_i}\left(x^{-}_{i,k_2}\right)\Delta_{Y_i}\left(x^{-}_{i,k_1}\right)+H_iN^{-}_{i}\otimes H_iN^{+}_{i}+H_iN^{-}_{i}\otimes M_i^{-}H_iN^{+}_{i}\\
&\equiv\Delta_{Y_i}\left(x^{-}_{i,k_2}\right)\Delta_{Y_i}\left(x^{-}_{i,k_1}\right)+H_iN^{-}_{i}\otimes M_i^{-}HN^{+}_{i}.
\end{align*}
Similar proof as above, $\left(H_iN^{-}_{i}\otimes M_i^{-}H_iN^{+}_{i}\right)\Delta_{Y_i}\left(x^{-}_{i,k_1}\right)\subset H_iN^{-}_{i}\otimes M_i^{-}H_iN^{+}_{i}.$

Suppose that $$\Delta_{\yg}\left(x^{-}_{i,k_{s-1}}\ldots x^{-}_{i,k_2}x^{-}_{i,k_1}\right)=\Delta_{Y_i}\left(x^{-}_{i,k_{s-1}}\ldots x^{-}_{i,k_2}x^{-}_{i,k_1}\right)+H_iN^{-}_{i}\otimes M_i^{-}H_iN^{+}_{i}.$$ We next compute $\Delta_{\yg}\left(x^{-}_{i,k_{s}}x^{-}_{i,k_{s-1}}\ldots x^{-}_{i,k_2}x^{-}_{i,k_1}\right)$.
\begin{align*}
\Delta_{\yg}&\left(x^{-}_{i,k_{s}}x^{-}_{i,k_{s-1}}\ldots x^{-}_{i,k_2}x^{-}_{i,k_1}\right)\\
&=\Delta_{\yg}\left(x^{-}_{i,k_{s}}\right)\Delta_{\yg}\left(x^{-}_{i,k_{s}}x^{-}_{i,k_{s-1}}\ldots x^{-}_{i,k_2}x^{-}_{i,k_1}\right)\\
&\equiv\left(\Delta_{Y_i}\left(x^{-}_{i,k_s}\right)+H_iN^{-}_{i}\otimes H_iN^{+}_{i}\right)\\
&\qquad \Big(\Delta_{Y_i}\left(x^{-}_{i,k_{s-1}}\ldots x^{-}_{i,k_2}x^{-}_{i,k_1}\right)+HN^{-}_{i}\otimes M_i^{-}HN^{+}_{i}\Big)\\
&\equiv\Delta_{Y_i}\left(x^{-}_{i,k_s}\right)\Delta_{Y_i}\left(x^{-}_{i,k_{s-1}}\ldots x^{-}_{i,k_2}x^{-}_{i,k_1}\right)\\
&+\Delta_{Y_i}\left(x^{-}_{i,k_s}\right)\left(HN^{-}_{i}\otimes M_i^{-}HN^{+}_{i}\right)\\
&+\left(H_iN^{-}_{i}\otimes H_iN^{+}_{i}\right)\Delta_{Y_i}\left(x^{-}_{i,k_{s-1}}\ldots x^{-}_{i,k_2}x^{-}_{i,k_1}\right)\\
&+\left(H_iN^{-}_{i}\otimes H_iN^{+}_{i}\right)\left(HN^{-}_{i}\otimes M_i^{-}HN^{+}_{i}\right)\\
&\equiv\Delta_{Y_i}\left(x^{-}_{i,k_s}x^{-}_{i,k_{s-1}}\ldots x^{-}_{i,k_2}x^{-}_{i,k_1}\right)+HN^{-}_{i}\otimes M_i^{-}HN^{+}_{i}.
\end{align*}
By induction, the proposition is proved.
\end{proof}
It follows from Proposition \ref{c1dgd2d} that $\Delta_{\yg}\left(x_{ik_s}^{-}\ldots x_{ik_1}^{-}\right)$ acts on a tensor product of highest weight representations in the same way as $\Delta_{\ysl}\left(x_{ik_s}^{-}\ldots x_{ik_1}^{-}\right)$.

\subsection{Representations of Yangians}

The representation theory of Yangians has many applications in mathematics and physics. For instance, from any finite-dimensional representation of a Yangian, one can construct a $R$-matrix which is a  solution of the quantum Yang-Baxter equations:
\begin{equation*}\label{117}
    R_{12}\left(u\right)R_{13}\left(u+v\right)R_{23}\left(v\right)=R_{23}\left(v\right)R_{13}\left(u+v\right)R_{12}\left(u\right).
\end{equation*}
As stated in \cite{NT}, ``the physical data such as mass formula, fusion angle, and the spins of integrals of motion can be extracted from the Yangian highest weight representations." The representation theory of $\yg$ have found applications to the representation theory of classical Lie algebras \cite{Mo1}.


\begin{definition}
A representation $V(\mu)$ of $\yg$ is said to be highest weight if it is generated by a vector $v^{+}$ such that
\begin{center}
  $x_{i,k}^{+}v^{+}= 0$ and $h_{i,k}v^{+} =\mu_{i,k} v^{+}$,
\end{center}
where $\mu=\Big(\mu_1\left(u\right),\mu_2\left(u\right),\ldots,\mu_l(u)\Big)$ and $\mu_i\left(u\right)=1+\mu_{i,0}u^{-1}+\mu_{i,1}u^{-2}+\ldots$ is a formal series in $u^{-1}$ for $i\in I$.
\end{definition}

We call a weight vector $v$  \textbf{maximal} in an $\yg$-module $V$ if $N^{+}v=0.$
The defining relations of the Yangian allow one to define analogously the notion of Verma module $M(\mu)$ of $\yg$. The Verma module $M(\mu)$ is defined to be the quotient of $\yg$ by the left ideal generated by $N^{+}$ and the elements $h_{i,k}-\mu_{i,k}1$; $\yg$ acts on $M(\mu)$ by left multiplication.  A highest weight vector of $M(\mu)$ is $1_{\mu}$ which is the image of the element $1\in \yg$ in the quotient. The Verma module $M(\mu)$ is a universal highest weight representation in the sense that: if $V(\mu)$ is another highest weight representation with a highest weight vector $v$, then the mapping $1_{\mu}\mapsto v$ defines a surjective $\yg$-module homomorphism $M\left(\mu\right)\rightarrow V\left(\mu\right)$.

Let $V=V(\mu)$ be a highest weight representation of $\yg$. Pulling back by the imbedding $U(\g)\hookrightarrow \yg$, $V$ is a $\g$-module. The weight subspace $V_{\mu^{\left(0\right)}}$ with
$\mu^{\left(0\right)}=\left(\mu_{1,0},\dots,\mu_{l,0}\right)$ is one-dimensional and spanned
by the highest weight vector of $V$. All other nonzero weight subspaces
correspond to the weights $\eta$ of the form
\begin{equation*}
\eta=\mu^{\left(0\right)}-k_1\alpha_1-\ldots-k_l\alpha_l,
\end{equation*}
where all $k_i$ are nonnegative integers, not all of them are zero.
A standard argument shows that $M(\mu)$ has a unique irreducible quotient $L\left(\mu\right)$.
In \cite{Dr2}, Drinfeld gave a classification of the finite-dimensional irreducible representations of $Y\left(\frak g\right)$.

\begin{theorem}\label{cfdihwr}\
\begin{enumerate}
  \item[(a)] Every finite-dimensional irreducible representation of $\yg$ is highest weight.
  \item [(b)] The representation $L(\mu)$ is finite dimensional if and only if there exist polynomials $\pi_i\left(u\right)$, $i\in I$, such that $$\frac{\pi_i\left(u+d_i\right)}{\pi_i\left(u\right)}=1+\sum_{k=0}^{\infty} \mu_{i,k}u^{-k-1},$$ in the sense that the right-hand side is the Laurent expansion of the left-hand side about $u=\infty$.
\end{enumerate}
The polynomials $\pi_i\left(u\right)$ above are called Drinfeld polynomials.
\end{theorem}

\begin{lemma}\label{g2puisd3}
Let $\frac{Q\left(u+1\right)}{Q\left(u\right)}=1+du^{-1}+dau^{-2}+da^2u^{-3}+\ldots$ and $Q(u)$ be a monic polynomial. Then $Q(u)=\big(u-(a-d+1)\big)\ldots\big(u-(a-1)\big)\big(u-a\big)$.
\end{lemma}
\begin{proof}
We first show the uniqueness of $Q\left(u\right)$. If $\frac{Q\left(u+1\right)}{Q\left(u\right)}=\frac{R\left(u+1\right)}{R\left(u\right)}$, then $\frac{Q\left(u\right)}{R\left(u\right)}=\frac{Q\left(u+1\right)}{R\left(u+1\right)}$. Therefore $\frac{Q\left(u\right)}{R\left(u\right)}$ is periodic in $u$, which implies $Q\left(u\right)=R\left(u\right)$.

Let $Q\left(u\right)=\left(u-\left(a-d+1\right)\right)\ldots\left(u-\left(a-1\right)\right)\left(u-a\right)$.
\begin{align*}
   \frac{Q(u+1)}{Q(u)}
   &=\frac{\big(u-(a-d)\big)\big(u-(a-d+1)\big)\ldots\big(u-(a-2)\big)\big(u-(a-1)\big)}{\big(u-(a-d+1)\big)\ldots\big(u-(a-1)\big)\big(u-a\big)}\\
   &=\frac{u-(a-d)}{u-a}\\
     &= \frac{1-\left(a-d\right)u^{-1}}{1-au^{-1}}\\
                      &= \left(1-\left(a-d\right)u^{-1}\right)\left(1+au^{-1}+a^2u^{-2}+\ldots+a^nu^{-n}+\ldots\right)\\
                      &= 1+du^{-1}+dau^{-2}+da^2u^{-3}+\cdots.
\end{align*}
\end{proof}


Among all the finite-dimensional irreducible representations of $\yg$, one type of representation, called fundamental, is important: each finite-dimensional irreducible representation is a subquotient of some tensor product of fundamental representations.
\begin{definition}
A finite-dimensional irreducible $Y\left(\mathfrak{\g}\right)$-module is called fundamental if its Drinfeld polynomials are given by
\begin{equation*}
\pi_j\left(u\right)=\begin{cases}
1\qquad \qquad\text{if}\qquad j\neq i, \\
u-a\qquad \text{if}\qquad j=i
\end{cases}
\end{equation*}
for some $i\in I$; the corresponding representation will be denoted by $V_{a}(\omega_i)$.
\end{definition}
\begin{theorem}[Theorem2.16, \cite{ChPr4}]\label{TC1irisq}
 Every finite-dimensional irreducible representation $V$ of $\yg$ is a sub-quotient of a tensor product $W=V_1\otimes\ldots\otimes V_n$ of the fundamental representations. In fact, $V$ is a quotient of the cyclic sub-representation of $W$ generated by the tensor product of the highest weight vectors in the $V_i$.
\end{theorem}

In \cite{ChPr3} and \cite{ChPr5}, an explicit realization of finite dimensional irreducible representations of $Y\left(\mathfrak{sl}_2\right)$ is given: every finite dimensional irreducible representation of $Y\left(\mathfrak{sl}_2\right)$ is a tensor product of representations  which are irreducible under $\mathfrak{sl}_2$. However, this property is not true in general. A. Molev has a counterexample in Remark 3.4.10 of his book \cite{Mo1}.  The $Y\left(\mathfrak{sl}_3\right)$-module $L\left(\mu\right)$ is 8-dimensional, where $\mu_1\left(u\right) = \left(1 + 3u^{-1}\right)\left(1 + u^{-1}\right), $ $\mu_2\left(u\right) = 1 + 3u^{-1}$ and $\mu_3\left(u\right) = 1 + 2u^{-1}$. On the other hand, the possible dimensions of the evaluation modules are 1, 3, 6, 8 \ldots,  so that $L\left(\mu\right)$ can not be
isomorphic to a non-trivial tensor product of such modules. The structure of the general finite-dimensional irreducible $\yg$-module still remains unknown.
In this thesis, a cyclicity criterion of the tensor product $W$ as in Theorem \ref{TC1irisq} is given when $\g$ is a classical simple Lie algebra or $\g=G_2$. Moreover, a sufficient condition for $W$ to be irreducible can follow from the cyclicity condition.

We next introduce the fundamental representations of $\yg$.
\begin{proposition}[Proposition 12.1.17, \cite{ChPr2}]\label{tati}
  Let $m_i$ be the multiplicity of the simple root $\alpha_i$ in the highest root $\theta$ of $\g$, and let $d_{\theta}=d_i$ if $\theta$ is conjugate to $\alpha_i$ under the Weyl group of $\g$. If $m_i=1$ or $m_i=d_\theta/d_i$,
the fundamental representation $L(\omega_i)$ of $\g$ can be made into a fundamental representation of $\yg$.
\end{proposition}
\begin{remark}\label{rofry}
The Proposition tells that
\begin{enumerate}
  \item When $\g$ is a simple Lie algebra of type $A$ or $C$, $L(\omega_i)\cong_\g V_{a}(\omega_i)$ for all $i\in I$;
  \item When $\g$ is a simple Lie algebra of type $B$, $L(\omega_i)\cong_\g V_{a}(\omega_i)$ for $i=1,l$;
  \item When $\g$ is a simple Lie algebra of type $D$, $L(\omega_i)\cong_\g V_{a}(\omega_i)$ for $i=1,l-1, l$.
\end{enumerate}
\end{remark}
\noindent The remaining fundamental representations $V_a(\omega_i)$ of $\g$ are given by the following proposition.
\begin{proposition}[Proposition 12.1.18, \cite{ChPr2}]\label{dcofbdl}\
\begin{enumerate}
  \item If $\g$ is a simple Lie algebra of type $B_l$, then, for $2\leq i\leq l-1$, $$V_{a}(\omega_i)\cong_{\g} \bigoplus\limits_{j=0}^{[\frac{i}{2}]}L\left(\omega_{i-2j}\right).$$
  \item If $\g$ is a simple Lie algebra of type $D_l$, then, for $2\leq i\leq l-2$, $$V_{a}(\omega_i)\cong_{\g} \bigoplus\limits_{j=0}^{[\frac{i}{2}]}L\left(\omega_{i-2j}\right).$$
\end{enumerate}
\end{proposition}

The fundamental representations of Yangians of exceptional simple Lie algebras over $\C$ can be found in \cite{ChPr4}.

Let $V$ be an finite-dimensional irreducible representation of $\yg$ with associated polynomial $\pi$. We are going to define the following associated $\yg$-representations.

\begin{enumerate}
  \item Pulling back through $\tau_{a}$ as in Proposition \ref{s1.2.1p26}, the representation $V(a)$ has associated polynomial $\{\pi_i(u-a)\}$.
  \item The left dual $^t V$ of $V$ and right dual $V^t$ of $V$ are the representations of $\yg$ on the vector space dual of $V$ defined as follows:
\begin{center}
$\left(y\cdot f\right)\left(v\right)=f\left(S\left(y\right)\cdot v\right)$, \qquad $y\in \yg$, $f\in\ ^tV$, $v\in V$.\\
$\left(y\cdot f\right)\left(v\right)=f\left(S^{-1}\left(y\right)\cdot v\right)$, \qquad $y\in \yg$, $f\in V^t$, $v\in V$.
\end{center}

The dual of an irreducible representation is irreducible.
\end{enumerate}
\begin{lemma}[Lemma 3.3, \cite{ChPr4}]\label{dualfrc1}
Let $V_a(\omega_i)$ be a fundamental representation of $\yg$. Then
\begin{center}
$\ ^tV_a(\omega_i)\cong V_{a-\kappa}\left(\omega_{-w_0\left(i\right)}\right),$ where $\kappa=\frac{1}{2}\times \text{dual Coxeter number of $\g$}$.

\end{center}

\end{lemma}
\begin{proposition}[Proposition 2.15, \cite{ChPr4}]\label{vtv'hwv}
Let both $V$ and $\widetilde{V}$ be finite-dimensional irreducible representations of $\yg$ with associated polynomials $\pi$ and $\widetilde{\pi}$, respectively. Let $v^{+}\in V$, $\widetilde{v}^{+}\in \widetilde{V}$ be their highest weight vectors. Then $v^{+}\otimes \widetilde{v}^{+}$ is a highest weight vector in $V\otimes \widetilde{V}$ and its associated polynomials are $\pi\widetilde{\pi}$.
\end{proposition}

\begin{proposition}[Proposition 3.8, \cite{ChPr8}]\label{VoWWoVhi}
Let $V$ be a finite-dimensional representation of $\yg$. $V$ is irreducible if and only if $V$ and $\ ^tV$ \big(respectively, $V$ and $V^t$\big) are both highest weight $\yg$-modules.
\end{proposition}

\begin{lemma}[Lemma 3.1, \cite{ChPr4}]
Let $V$ and $W$ be finite-dimensional irreducible representations of $Y(\g)$ with lowest and highest weight vectors $v^{-}$ and $w^{+}$, respectively. Then $v^{-}\otimes w^{+}$generates $V\otimes W$.
\end{lemma}

The following corollary is analogue of Lemma 4.2 of \cite{Ch3}. Note that every finite-dimensional highest-weight representation is also a lowest-weight representation and vice-versa. We can not find a proof for the case of the representation theory of Yangians, so we prove it.
\begin{corollary}\label{v-w+gvtw}
Let $V$ and $W$ be finite-dimensional highest weight representations of $\yg$ with lowest and highest weight vectors $v^{-}$ and $w^{+}$, respectively. Then $v^{-}\otimes w^{+}$ generates $V\otimes W$.
\end{corollary}
\begin{proof}


We first show that $v^{-}\otimes W\subseteq \yg\left(v^{-}\otimes w^+\right)$. $\forall w\in W$, there exists $X^{-}\in Y^{-}$ such that $X^{-}w^{+}=w$ since $W$ is a highest weight representation. Then it follows from (ii) of Proposition \ref{Delta} that $$X^{-}\left(v^{-}\otimes w^+\right)=\Delta\left(X^{-}\right)\left(v^{-}\otimes w^+\right)=v^{-}\otimes X^{-}w^+=v^{-}\otimes w.$$
Thus $v^{-}\otimes w\in \yg\left(v^{-}\otimes w^+\right)$. Therefore $v^{-}\otimes W\subseteq \yg\left(v^{-}\otimes w^+\right)$.

To prove this corollary, it is equivalent to show that $v\otimes w\in \yg\left(v^{-}\otimes w^+\right)$, where $v\in V$ and $w\in W$. Since $V$ is a lowest weight representation, it is enough to show that $\left(\prod\limits_{i=1}^{m} \left(x_{j_i,k_i}^{+}\right)^{a_i}v^{-}\right)\otimes w\in \yg\left(v^{-}\otimes w^+\right)$.
We use induction on the degree $N=k_1a_1+\ldots+k_ma_m$ of $\prod\limits_{i=1}^{m} \left(x_{j_i,k_i}^{+}\right)^{a_i}$. For fixed $N$, we use induction on $M=a_1+\ldots+a_m$.

For $N=0$, $M=1$, Without loss of generality, we may assume that $a_1\neq 0$.
 $$\left(x_{j_1,0}^{+}v^{-}\right)\otimes w=x_{j_1,0}^{+}\left(v^{-}\otimes w\right)-v^{-}\otimes \left(x_{j_1,0}^{+}w\right)\in \yg\left(v^{-}\otimes w^+\right).$$
Therefore the claim is true. Suppose $M\geq 2$, and the claim is true for $M-1$. By the induction hypothesis, we may assume that for $a_1+\ldots+a_m=M-1$, $$\left(\prod\limits_{i=1}^{m} \left(x_{j_i,0}^{+}\right)^{a_i}v^{-}\right)\otimes w\in \yg\left(v^{-}\otimes w^+\right) \text{for all $w\in W$}.$$
\begin{align*}
   \left(\prod\limits_{i=1}^{m} \left(x_{j_i,0}^{+}\right)^{a_i}v^{-}\right)\otimes w
   &=\left(x_{j_1,0}^{+}\left(x_{j_1,0}^{+}\right)^{a_1-1}\prod\limits_{i=2}^{m} \left(x_{j_i,0}^{+}\right)^{a_i}v^{-}\right)\otimes w \\
      &= x_{j_1,0}^{+}\left(\left(x_{j_1,0}^{+}\right)^{a_1-1}\prod\limits_{i=2}^{m} \left(x_{j_i,0}^{+}\right)^{a_i}v^{-}\otimes w\right) \qquad\qquad \qquad\qquad \qquad\qquad\\
&-\left(\left(x_{j_1,0}^{+}\right)^{a_1-1}\prod\limits_{i=2}^{m} \left(x_{j_i,0}^{+}\right)^{a_i}v^{-}\right)\otimes \Bigg(x_{j_1,0}^{+}w\Bigg).
\end{align*}
The claim follows from the induction hypothesis.

Suppose that $N\geq 1$, and the claim is true for all value $k\leq N-1$. Without loss of generality, we may assume that both $k_1>0$ and $a_1>0$. 
\begin{align*}
\Big(\prod\limits_{i=1}^{m} &\left(x_{j_i,k_i}^{+}\right)^{a_i}v^{-}\Big)\otimes w\\
   &= \Big(x_{j_i,k_1}^{+}\left(x_{j_i,k_1}^{+}\right)^{a_1-1}\prod\limits_{i=2}^{m} \left(x_{j_i,k_i}^{+}\right)^{a_i}v^{-}\Big)\otimes w \\
   &= x_{j_1,k_1}^{+}\Big(\left(x_{j_1,k_1}^{+}\right)^{a_1-1}\prod\limits_{i=2}^{m} \left(x_{j_i,k_i}^{+}\right)^{a_i}v^{-}\otimes w\Big)\\
   &-\Big( \left(x_{j_1,k_1}^{+}\right)^{a_1-1}\prod\limits_{i=2}^{m}\left(x_{j_i,k_i}^{+}\right)^{a_i}v^{-}\Big)\otimes \Big(x_{j_1,k_1}^{+}w\Big) \\
   &- \sum_{j=1}^{k_1} h_{j_1,j-1}\otimes x_{j_1,k-j}^{+}\Big(\Big(\left(x_{j_1,k_1}^{+}\right)^{a_1-1}\prod\limits_{i=2}^{m} \left(x_{j_i,k_i}^{+}\right)^{a_i}v^{-}\Big)\otimes w\Big)\\
   &-\overline{Y}\Big(\Big(x_{j_1,k_1}^{+}\Big)^{a_1-1}\prod\limits_{i=2}^{m} \left(x_{j_i,k_i}^{+}\right)^{a_i}v^{-}\Big)\otimes w\Big). 
\end{align*}
Note that $\overline{Y}\in H_iN^{-}_i\otimes H_iN^{+}_i$ by Proposition \ref{deltaxpc1}. By the induction hypothesis, most terms but one are obviously in $\yg\left(v^{-}\otimes w^{+}\right)$: $$\sum\limits_{j=1}^{k_1} h_{j_1,j-1}\otimes x_{j_1,k-j}^{+}\bigg(\Big(\left(x_{j_1,k_1}^{+}\right)^{a_1-1}\prod\limits_{i=2}^{m} \left(x_{j_i,k_i}^{+}\right)^{a_i}v^{-}\Big)\otimes w\bigg).$$ 
By the defining relations of Yangians, $\mathrm{\operatorname{Deg}}\bigg(h_{j_1,j-1}\Big(\left(x_{j_1,k_1}^{+}\right)^{a_1-1}\prod\limits_{i=2}^{m} \left(x_{j_i,k_i}^{+}\right)^{a_i}v^{-}\Big)\bigg)$ is strictly less than $N$. Therefore $$\sum\limits_{j=1}^{k_1} h_{j_1,j-1}\otimes x_{j_1,k-j}^{+}\bigg(\Big(\left(x_{j_1,k_1}^{+}\right)^{a_1-1}\prod\limits_{i=2}^{m} \left(x_{j_i,k_i}^{+}\right)^{a_i}v^{-}\Big)\otimes w\bigg) \in \yg\left(v^{-}\otimes w^{+}\right),$$
and then $$\left(\prod\limits_{i=1}^{m} \left(x_{j_i,k_i}^{+}\right)^{a_i}v^{-}\right)\otimes w\in \yg\left(v^{-}\otimes w^+\right).$$
This finished the proof. \end{proof}

\section{The current algebras and their representations}

Let $t$ be an indeterminate. Denote by $\C[t]$ the polynomial ring in $t$. Let $\mathfrak{g}[t]=\mathfrak{g}\otimes_{\C}\C[t]$ be a Lie algebra with commutator $$[x\otimes_\C f, y\otimes_\C g]=[x,y]\otimes_{\C} fg,\quad x,y\in \mathfrak{g}, \quad f,g\in \C[t].$$

\begin{definition}\label{c1lwmotca}
Let $\lambda=\sum m_i \omega_i$ be a dominant integral weight of $\g$. Denote by $W(\lambda)$ the $\g[t]$-module generated by an element $v_{\lambda}$ with the relations:
$$
\mathfrak{n}^{+}\otimes \C[t] v_{\lambda} = 0, \;h \otimes t\C[t] v_{\lambda} = 0, \; hv_{\lambda} = \lambda\left(h\right)v_{\lambda},  \; \left(x_{\alpha_i}^{-}\otimes 1\right)^{m_i +1}v_{\lambda}= 0
$$
for all $h \in \mathfrak{h}$ and all simple roots $\alpha_i$. This module is called
the Weyl module for $\g[t]$ associated to $\lambda\in P^+$.
\end{definition}
\begin{theorem}[Theorem 1.2.2, \cite{ChLo}]\label{mpocac1}
For all $\lambda\in P^{+}$, the local Weyl modules $W(\lambda)$ are finite dimensional. Moreover, any finite-dimensional $\g[t]$-module $V$ generated by an element $v\in V$ satisfying the relations:
$$
\mathfrak{n}^{+}\otimes \C[t] v = 0, \;h \otimes t\C[t] v= 0, \; hv = \lambda\left(h\right)v
$$ is a quotient of $W(\lambda)$.
\end{theorem}
The structure of Weyl mdoules for $\g[t]$ associated to a dominant integral weight $\lambda$ is known, see \cite{ChLo, FoLi, Na}.

\begin{theorem}[Corollary B, \cite{Na}]
Let $\lambda=\sum m_i \omega_i$ be a dominant integral weight of $\g$. We have
{\footnotesize{$$ W(\lambda)\cong W\left(\omega_1\right)_{a_1}\ast\ldots\ast W\left(\omega_1\right)_{a_{m_1}}\ast W\left(\omega_2\right)_{b_1}\ast\ldots\ast W\left(\omega_2\right)_{b_{m_2}}\ast\ldots\ast W\left(\omega_l\right)_{l_1}\ast\ldots\ast W\left(\omega_l\right)_{l_{m_l}},$$}}
    where $*$ denotes the fusion product introduced in \cite{FeLo2},
    and $a_1,\dots,a_{m_1}$,\\ $b_1,\dots,b_{m_2},\ldots, l_1,\dots, l_{m_l}$ are parameters used to define the fusion product.
\end{theorem}
\begin{remark}
As a $\g$-module, the fusion products are isomorphic to tensor products.
\end{remark}
\begin{theorem}[Corollary A, \cite{Na}]
$$\operatorname{Dim}\ W(\lambda)=\prod_{i\in I} \Big(\operatorname{Dim}\big(W(\omega_i)\big)\Big)^{m_i}.$$
\end{theorem}
We will not give the definition of Kirillov-Reshetkhin modules $\operatorname{KR}(m\omega_i)$ in this thesis because we only require the fundamental ones (i.e. when $m=1$) and it is well known that $\operatorname{KR}(\omega_i)\cong_{\g[t]} W(\omega_i)$. From the first part of the main theorem of Section 2.2 in \cite{ChMo}, we obtain that $\operatorname{KR}\left(\omega_i\right)\cong_{\g} V_a\left(\omega_i\right)$, where $\g$ is a classical simple Lie algebra, any $i\in I$ and $a\in \C$. Consequently, we have the following important corollary.
\begin{corollary}\label{dkrvawocsp}
$\operatorname{Dim}\big(V_a(\omega_i)\big)=\operatorname{Dim}\big(W(\omega_i)\big).$
\end{corollary}

\chapter{Finiteness of local Weyl modules} 

In this chapter, we first show the existence of the local Weyl module $W(\pi)$ of $\yg$ associated to a generic $l$-tuple of polynomials $\pi=\big(\pi_1(u),\ldots,\pi_l(u)\big)$, where $\pi_i\left(u\right)=\prod\limits_{j=1}^{m_i} \left(u-a_{ij}\right)$. Then we prove that $W(\pi)$ are finite-dimensional. In the end, we obtain an upper bound for the dimension of $W(\pi)$.
\begin{definition}
The local Weyl module $W\left(\mathbf{\pi}\right)$ for $\yg$ is defined as the module generated by a highest weight vector $w_\pi$ with the following relations: \begin{equation}
x_{i,k}^{+}w_\pi=0, \quad \left(x_{i,0}^{-}\right)^{m_i+1}w_\pi=0,\quad \left(h_{i}\left(u\right)-\frac{\pi_i\left(u+d_i\right)}{\pi_i\left(u\right)}\right)w_\pi=0.\end{equation}
\end{definition}

\begin{remark}\label{mi=lambdahi}\
\begin{enumerate}
  \item The local Weyl module $W(\pi)$ can be considered as the quotient of the Verma module $M(\mu)$ by the submodule generated by $\left(x_{i,0}^{-}\right)^{m_i+1}w_\pi$, where $\mu_i(u)=\frac{\pi_i(u+d_i)}{\pi_i(u)}$.
  \item It follows from Remark [2], part C of section 12.1 of \cite{ChPr2} that $m_i=\lambda(h_i)$, where $h_i=\frac{h_{i,0}}{d_i}$. Thus for any finite-dimensional highest weight representation $V(\pi)$ of $\yg$, $\left(x_{i,0}^{-}\right)^{m_i+1}w_\pi=0$. Therefore $V(\pi)$ is a quotient of $W(\pi)$.
\end{enumerate}
\end{remark}
We are now in the position ready to prove the main objective of this Chapter.
\begin{theorem}\label{wmifd}
The local Weyl module $W(\pi)$ is finite-dimensional.
\end{theorem}
\begin{proof} Let $\lambda=\sum\limits_{i\in I}m_i\omega_i$, where $m_i$ is the degree of the polynomial $\pi_i\left(u\right)$.
We divide the proof into steps.

Step 1: The $\N$-filtration on $\yg$ induces a filtration on $W(\pi)$ as follows: let $W(\pi)_s$ be the subspace of $W(\pi)$ spanned by elements of the form $yw_{\pi}$, where $y\in \yg_{s}$.
Denote by gr$\big(W(\pi)\big)=\bigoplus\limits_{r=0}^{\infty} W(\pi)_{r}/W(\pi)_{r-1}$ the associated graded module, where $W(\pi)_{-1}=0$. Let $\overline{w_{\pi}}$ be the image of $w_{\pi}$ in gr$\big(W(\pi)\big)$.



Step 2: show $\text{gr}\big(W(\pi)\big)=\text{gr}\big(\yg\big)
\overline{w_{\pi}}=U\left(\g[t]\right)\overline{w_{\pi}}$.

Proof: The action of $\text{gr}\big(\yg\big)$ on gr$\big(W(\pi)\big)$ is given by $\overline{y}\ \overline{w_{\pi}}=\overline{y w_{\pi}},$ where $\bar{y}$ is the image of $y\in \yg$ in $\text{gr}\big(\yg\big).$ It is obvious that $\text{gr}\big(W(\pi)\big)\supseteq \text{gr}\big(\yg\big)\overline{w_{\pi}}.$ Let $\overline{v}\in \text{gr}\big(W(\pi)\big)$. Without loss of generality, we may assume that $v\in W(\pi)_r\backslash W(\pi)_{r-1}.$ Thus $v=yw_{\pi}$ for some $y\in \yg_r\backslash \yg_{r-1}$ by Step 1 and hence $\bar{v}=\bar{y}\ \overline{w_{\pi}}.$ So $\text{gr}\big(W(\pi)\big)\subseteq \text{gr}\big(\yg\big)\overline{w_{\pi}}.$

Step 3: show $\text{gr}\big(W(\pi)\big)$ is a highest weight representation of $\text{gr}\big(\yg\big)$.

Proof: Note that $\left(x_{\alpha}^{+}\otimes t^r\right)\overline{w_{\pi}}=\overline{x_{\alpha,r}^{+}w_{\pi}}= \overline{0}$ and $h_i\overline{w_{\pi}}=\overline{h_{i,0}w_{\pi}}=m_i\overline{w_{\pi}}$. Since $w_\pi\in W(\pi)_0\subseteq W(\pi)_{r}$ for all $r\geq 1$, $\left(h_i\otimes t^r\right)\overline{w_{\pi}}=\overline{h_{i,r}w_{\pi}}=
\overline{c.w_{\pi}}=\overline{0}$, where $c\in\C$.  Moreover, $(x_{i}^-)^{m_i+1}\overline{w_{\pi}}=\overline{(x_{i,0}^-)^{m_i+1}w_{\pi}}=0$. Therefore Step 3 holds.

Step 4: $W(\pi)$ is finite-dimensional.

Proof: There is a surjective homomorphism $\varphi: W\left(\lambda\right)\rightarrow \text{gr}\big(W(\pi)\big)$, by Theorem \ref{mpocac1}. From Theorem \ref{PS3TADOLWM}, we know that $W(\lambda)$ is finite-dimensional and hence $\text{gr}\big(W(\pi)\big)$ is as well since it is a quotient of $W(\lambda)$. As $\operatorname{Dim}\big(W(\pi)\big)= \operatorname{Dim}\Big(\text{gr}\big(W(\pi)\big)\Big)$, $W(\pi)$ is finite-dimensional.
\end{proof}

Since $\text{gr}\big(W(\pi)\big)$ is a quotient of the Weyl module $W\left(\lambda\right)$ of $\g[t]$, we obtain an upper bound of the dimension of the local Weyl module.
\begin{theorem}\label{ubodowm} Using the notation as in the above theorem,
$$\text{Dim}\big(W(\pi)\big)\leq\text{Dim}\big(W(\lambda)\big).$$
\end{theorem}




\chapter{Local Weyl modules of $\yn$}
In this chapter, the local Weyl modules of $\yn$ are studied. We first look into the local Weyl modules of $\ysl$. A lot of useful information is contained in this case. Next we consider the general case: the local Weyl modules for $\yn$. The ideas in this chapter will inspire us to characterize the local Weyl modules of Yangians of simple Lie algebras of types $B$, $C$, $D$ and $G_2$.

Let  $\pi=\big(\pi_1(u),\pi_2(u),\ldots, \pi_l(u)\big)$ be a generic $l$-tuple of monic polynomials in $u$, and $\pi_i\left(u\right)=\prod\limits_{j=1}^{m_i}\left(u-a_{i,j}\right)$. Let $k=m_1+m_2+\ldots+m_l$, $S=\{a_{i,j}|i=1,\ldots,l; j=1,\ldots,m_i\}$, and $\lambda=\sum\limits_{i=1}^{l}m_i\omega_i$.
Let $a_{m,n}$ be one of the numbers in $S$ with the maximal real part. Then define $a_1=a_{m,n}$ and $b_1=m$. Inductively, let $a_{s,t}$ be one of the numbers in $S-\{a_1,\ldots, a_{i-1}\}$ ($i\geq 2$) with the maximal real part. Then define $a_i=a_{s,t}$ and $b_i=s$. We construct a tensor product $L=V_{a_1}\left(\omega_{b_1}\right)\otimes V_{a_2}\left(\omega_{b_2}\right)\otimes \ldots \otimes V_{a_k}\left(\omega_{b_k}\right),$ where $V_{a_i}(\omega_{b_i})$ are fundamental representations of $\yn$. We prove that $L$ is a highest weight representation associated to $\pi$.
Note that the dimension of $V_{a_i}\left(\omega_{b_i}\right)$ is ${l+1\choose b_i}$. Then $\operatorname{Dim}(L)=\prod\limits_{i=1}^{l} {l+1\choose i}^{m_i}$. Since $L$ is a quotient of $W(\pi)$, a lower bound on the dimension of the local Weyl module $W(\pi)$ is obtained.

Let $W(\lambda)$ be the local Weyl module associated to $\lambda$ of the current algebra $\nyn[t]$. It follows from main theorem of \cite{ChLo} that $\operatorname{Dim}\big(W(\lambda)\big)=\prod\limits_{i=1}^{l} {l+1\choose i}^{m_i}$. Comparing the dimensions of $W(\lambda)$ and $L$, the structure of the local Weyl module $W(\pi)$ of $\yn$ is obtained, namely, $$W\left(\pi\right)\cong V_{a_1}\left(\omega_{b_1}\right)\otimes V_{a_2}\left(\omega_{b_2}\right)\otimes \ldots \otimes V_{a_k}\left(\omega_{b_k}\right),\ \text{where}\ a_i\in S.$$

Similar to the proof that $L$ is a highest weight representation (Proposition \ref{Lhwr}), we obtain a sufficient condition for a tensor product of fundamental representations of the form $V_{a_1}(\omega_{b_1})\otimes V_{a_2}(\omega_{b_2})\otimes\ldots\otimes V_{a_k}(\omega_{b_k})$ to be a highest weight representation. If $a_j-a_i\notin S(b_i, b_j)$ for $1\leq i<j\leq k$, then $L$ is a highest weight representation, where $S(b_i, b_j)$ is a finite set of positive rational numbers as in Lemma \ref{C3l317ss}. A sufficient and necessary condition on the irreducibility of the tensor product is obtained: $L$ is irreducible if and only if $a_j-a_i\notin S(b_i, b_j)$ for $1\leq i\neq j\leq k$.

\section{On the local Weyl modules for Yangian $Y\left(\mathfrak{sl}_2\right)$}
In this section, we denote the generators of $\ysl$ by $x_{r}^{\pm}$,
$h_{r}$, $r\in\Z_{\geq 0}$.
\begin{lemma}
The assignment 
$\rho(x_{k}^{\pm})=\delta_{k,0}x_{k}^{\pm}$ and $\rho(h_{k})=\delta_{k,0}h_{k}$ extends linearly to a homomorphism from $\ysl$ to $U(\nysl)$ for $k\geq 0$.
\end{lemma}

Let $W_m$ be the irreducible representation of $\nysl$ with highest weight $m\in Z_{\geq 1}$. Pulling it back by $\rho$, $W_m$ becomes to a $\ysl$-module. Let $W_m\left(a\right)$ be the irreducible representation of $\ysl$ associated to the Drinfeld polynomial $\pi(u)=\Big(u-a\Big)\Big(u-(a_1+1)\Big)\ldots \Big(u-(a+m-1)\Big)$.

\begin{proposition}[Proposition 3.5,\cite{ChPr3}]\label{wra}
For any $m\geq 1$, $a\in\C$, $W_m\left(a\right)$ has a basis
$\{w_0,w_1,\ldots, w_m\}$ on which the action of $\ysl$ is given by
\begin{center}
$ x_k^+.w_s=\left(s+a\right)^k\left(s+1\right)w_{s+1},\ \ x_k^-.w_s=\left(s+a-1\right)^k\left(m-s+1\right)w_{s-1},$\\

$h_k.w_s=\Big(\left(s+a-1\right)^ks\left(m-s+1\right)-\left(s+a\right)^k\left(s+1\right)\left(m-s\right)\Big)w_s$.
\end{center}
\end{proposition}
The cases when $m=1$ and $m=2$ are important for the characterization of the local Weyl modules of Yangians. Calculations are listed in the two corollaries below.
\begin{corollary}\label{slw1a}
In $W_1\left(a\right)$,
\begin{equation}\label{Cor3.3}
  h_{k}w_1=a^kw_1,\quad x_{k}^{-}w_1=a^kx_{0}^{-}w_1,\quad h_{k}x_{i,0}^{-}w_1=-a^kx_{0}^{-}w_1.
\end{equation}
\end{corollary}

\begin{corollary}\label{w2ihw}
Let $W_2\left(a\right)=\ysl\left(w_2\right)=span\{w_0, w_1, w_2\}$.
\begin{enumerate}
  \item $x_{k}^{-}w_2=\left(a+1\right)^kw_1$, $x_{k}^{-}w_1=2a^kw_0$.
  \item $x_{0}^{-}x_{1}^{-}\left(w_2\right)=\left(a+1\right)\left(x_{0}^{-}\right)^2\left(w_2\right)$.
  \item $x_{1}^{-}x_{0}^{-}\left(w_2\right)=a\left(x_{0}^{-}\right)^2\left(w_2\right)$.
    \item $\left(x_{1}^{-}x_{0}^{-}+x_{0}^{-}x_{1}^{-}\right)\left(w_2\right)=\left(2a+1\right)\left(x_{0}^{-}\right)^2\left(w_2\right)$.
    \item $\left(x_{2}^{-}x_{0}^{-}+x_{0}^{-}x_{2}^{-}\right)\left(w_2\right)=\left(2a^2+2a+1\right)\left(x_{0}^{-}\right)^2\left(w_2\right)$.
    \item $\left(x_{1}^{-}\right)^2\left(w_2\right)=a\left(a+1\right)\left(x_{0}^{-}\right)^2\left(w_2\right)$;
\end{enumerate}
\end{corollary}
It follows from Propositions \ref{Delta} and \ref{wra} that
\begin{corollary}\label{w1bw1a}
Let $v^{+}\left(v^{-}\right)$ and $w^{+}\left(w^{-}\right)$ be the highest (lowest) weight vectors in $W_1\left(b\right)$ and $W_1\left(a\right)$ respectively.
In $W_1\left(b\right)\otimes W_1\left(a\right)$,
\begin{enumerate}
  \item $x_{0}^{-}\left(v^{+}\otimes w^{+}\right)$ and $x_{1}^{-}\left(v^{+}\otimes w^{+}\right)$ are linear independent.
  \item $x_{0}^{-}x_{1}^{-}\left(v^{+}\otimes w^{+}\right)=\frac{\left(a+b+1\right)}{2}\left(x_{0}^{-}\right)^2\left(v^{+}\otimes w^{+}\right)$.
  \item $x_{1}^{-}x_{0}^{-}\left(v^{+}\otimes w^{+}\right)=\frac{\left(a+b-1\right)}{2}\left(x_{0}^{-}\right)^2\left(v^{+}\otimes w^{+}\right)$.
  \item $\left(x_{1}^{-}x_{0}^{-}+x_{0}^{-}x_{1}^{-}\right)\left(v^{+}\otimes w^{+}\right)=\left(a+b\right)\left(x_{0}^{-}\right)^2\left(v^{+}\otimes w^{+}\right)$.
  \item $x_{2}^{-}\left(v^{+}\otimes w^{+}\right)=\left(b^2+b+a\right)v^{-}\otimes w^{+}+a^2 v^{+}\otimes w^{-}$, \text{and}\\ $x_0^{-}x_{2}^{-}\left(v^{+}\otimes w^{+}\right)=\frac{\left(b^2+b+a+a^2\right)}{2}\left(x_{0}^{-}\right)^2\left(v^{+}\otimes w^{+}\right),$ \text{and} \\ $x_1^{-}x_{2}^{-}\left(v^{+}\otimes w^{+}\right)=\frac{ab\left(a+b+1\right)}{2}\left(x_{0}^{-}\right)^2\left(v^{+}\otimes w^{+}\right)$.
  \item $x_1^{-}x_{1}^{-}\left(v^{+}\otimes w^{+}\right)=ab\left(x_{0}^{-}\right)^2\left(v^{+}\otimes w^{+}\right)$.
  \item $x_{3}^{-}\left(v^{+}\otimes w^{+}\right)=\left(b^3+b^2+ab+a^2\right)v^{-}\otimes w^{+}+a^3 v^{+}\otimes w^{-}$ \text{and}\\ $x_0^{-}x_{3}^{-}\left(v^{+}\otimes w^{+}\right)=\frac{b^3+b^2+ab+a^2+a^3}{2}\left(x_{0}^{-}\right)^2\left(v^{+}\otimes w^{+}\right)$.

\end{enumerate}
\end{corollary}


\begin{proposition}[Proposition 3.6, \cite{ChPr3}]
Let $a\in\C$ and let $v_0=v^+\otimes w^{-} -v^-\otimes
w^+$, where $v^+$ and $w^+$ ($v^-$ and $w^{-}$) are highest (lowest) weight
vectors in $W_1\left(a+1\right)$ and $W_1\left(a\right)$, respectively.
$W_1\left(a+1\right)\otimes W_1\left(a\right)=\ysl\left(v^+\otimes v^-\right)$, and we have a short
exact sequence of $\ysl$-modules
$$0\rightarrow\ysl v_0\rightarrow W_1\left(a+1\right)\otimes W_1\left(a\right)\rightarrow W_2\left(a\right)\rightarrow 0,$$
where $\ysl v_0$ is the one-dimensional trivial module.
\end{proposition}

\begin{proposition}[Proposition 3.7, \cite{ChPr3}]\label{ctpihw}
Let $a_1,a_2,\ldots,a_m\in\C$, $m\geq 1$. Then if $a_j-a_i\neq 1$ when $i<j$,  $$W_1\left(a_1\right)\otimes\ldots\otimes W_1\left(a_m\right)$$ is a highest weight $Y\left(\mathfrak{sl}_2\right)$-module.
\end{proposition}

It follows from the above proposition that we have
\begin{corollary}\label{tpihw}
Let $a_1,a_2,\ldots,a_m\in\C$, $m\geq 1$, and $\operatorname{Re}\left(a_1\right)\geq \ldots \geq \operatorname{Re}\left(a_m\right)$. Then $$W_1\left(a_1\right)\otimes\ldots\otimes W_1\left(a_m\right)$$ is a highest weight $Y\left(\mathfrak{sl}_2\right)$-module.
\end{corollary}

\begin{lemma}[Corollary 3.8.\cite{ChPr3}]\label{beautiful}
Let $P_1, P_2,\ldots, P_m$ be polynomials, and let $V\left(P_i\right)$ be the irreducible representation whose Drinfeld polynomial is $P_i$. Assume that if $a_i$ is a root of $P_i$ and $a_j$ a root of $P_j$, where $i<j$, then $a_j-a_i\neq 1$. Then
$$V\left(P_1\right)\otimes V\left(P_2\right)\otimes \ldots\otimes V\left(P_m\right)$$ is a highest weight $\ysl$-module.
\end{lemma}
\begin{corollary}\label{tpihw2}
Let $a_1,a_2,\ldots,a_m\in\C$, $m\geq 1$, and $\operatorname{Re}\left(a_1\right)\geq \ldots \geq \operatorname{Re}\left(a_m\right)$. Then $$W_2\left(a_1\right)\otimes W_1\left(a_2\right)\otimes W_1\left(a_3\right)\otimes \ldots\otimes W_1\left(a_m\right)$$ is a highest weight $Y\left(\mathfrak{sl}_2\right)$-module.
\end{corollary}
\begin{proof}

Let $P_1\left(u\right)=\big(u-\left(a_1+1\right)\big)\big(u-a_1\big)$, $P_2\left(u\right)=u-a_2$, $P_3\left(u\right)=u-a_3$, $\ldots$, $P_m\left(u\right)=u-a_m$. It is easy to check that if $a_i$ is a root of $P_i$ and $a_j$ a root of $P_j$, where $i<j$, then $a_j-a_i\neq 1$. $V\left(P_1\right)=W_2\left(a_1\right)$, $V\left(P_2\right)=W_1\left(a_1\right)$, $V\left(P_3\right)=W_1\left(a_3\right)$, $\ldots$, $V\left(P_m\right)=W_1\left(a_m\right)$. Therefore the corollary follows from the above lemma.
\end{proof}

\begin{theorem}\label{wmfsl2} Let $\pi(u)=(u-a_1)(u-a_2)\ldots(u-a_m)$ be a decomposition of $\pi$ over $\C$ such that $\operatorname{Re}\left(a_1\right)\geq \ldots \geq \operatorname{Re}\left(a_m\right)$.
Then $W\left(\pi\right)\cong W_1\left(a_1\right)\otimes\ldots\otimes W_1\left(a_m\right)$.
\end{theorem}
\begin{proof}
It follows from Corollary \ref{tpihw} that $W_1\left(a_1\right)\otimes\ldots\otimes W_1\left(a_m\right)$ is a highest weight module of $Y\left(\mathfrak{sl}_2\right)$. By Proposition \ref{vtv'hwv}, the associated polynomial is $\pi(u)$. By the maximality of the local Weyl module $W(\pi)$, $\operatorname{Dim}\Big(W\left(\pi\right)\Big)\geq 2^m$. It follows from the main theorem of \cite{ChLo} that
$\operatorname{Dim}\Big(W\left(\pi\right)\Big)\leq 2^m$.
Thus it follows from Theorem \ref{ubodowm} that $\operatorname{Dim}\Big(W\left(\pi\right)\Big)=2^m$. Therefore $$W\left(\pi\right)=W_1\left(a_1\right)\otimes\ldots\otimes W_1\left(a_m\right).$$
\end{proof}
\section{On the local Weyl modules of $\yn$ for $l\geq 2$}
In this section, we characterize the local Weyl modules of $\yn$ for $l\geq 2$. Let $\pi$ be an $l$-tuple of polynomials $\pi=\Big(\pi_1(u),\ldots, \pi_{l}(u)\Big)$, and $W(\pi)$ be the local Weyl module associated to $\pi$.

The next corollary is a consequence of Theorem \ref{ubodowm} and the main theorem of \cite{ChLo}.
\begin{corollary}\label{dowmsl} Let $\pi$ be an $l$-tuple of polynomials $\pi=\Big(\pi_1(u),\ldots, \pi_{l}(u)\Big)$, and $m_i$ be the degree of $\pi_i(u)$. Then
$\mathrm{Dim}\Big(W\left(\pi\right)\Big)\leq \prod\limits_{i=1}^{l} {l+1\choose i}^{m_i}$.
\end{corollary}
\subsection{On a lower bound of the local Weyl modules of $\yn$}

\begin{proposition}\label{Lhwr}
Let $L=V_{a_1}(\omega_{b_1})\otimes V_{a_2}(\omega_{b_2})\otimes\ldots\otimes V_{a_k}(\omega_{b_k})$, where $b_{i}\in\{1,\ldots, l\}$. If $\operatorname{Re}\left(a_1\right)\geq \operatorname{Re}\left(a_2\right)\geq \ldots \geq \operatorname{Re}\left(a_k\right)$,  then $L$ is a highest weight representation of $\yn$.
\end{proposition}

\begin{proof}
Let $Y_{i}$ be the Yangian subalgebra generated by $x_{i, r}^{\pm}$ and $h_{i,r}$, where $r\in \Z_{\geq 0}$. Let $v^{+}_i$ be a highest weight vector of $V_{a_{i}}\left(\omega_{b_{i}}\right)$, and let $v^{-}_1$ be a lowest weight vector of $V_{a_{1}}(\omega_{b_{1}})$.

We prove this proposition by induction on $k$. It is obvious that for $k=1$, $L=V_{a_{1}}(\omega_{b_{1}})$ is irreducible, hence it is a highest weight representation. We assume that the claim is true for any integer less than or equal $k-1$ ($k\geq 2$).
By the induction hypothesis, $V_{a_2}(\omega_{b_2})\otimes V_{a_3}(\omega_{b_3})\otimes\ldots\otimes V_{a_k}(\omega_{b_k})$ is a highest weight representation of $\yn$, and a highest weight vector is $v^{+}=v^{+}_{2}\otimes \ldots\otimes v^{+}_k$.

Let $\{e_1,e_2,\ldots,e_{l+1}\}$ be the standard basis of $\C^{l+1}$, and $\{\mu_1,\mu_2,\ldots,\mu_{l+1}\}$ be the coordinate functions on the Cartan subalgebra of $\yn$. Then in $V_{a_{1}}(\omega_{b_{1}})$,  $$v^{+}_1=e_1\wedge e_2\wedge\ldots\wedge e_{b_1}$$ and $$v^{-}_1=e_{l+2-b_1}\wedge e_{l+3-b_1}\wedge\ldots\wedge e_{l+1}.$$ The corresponding weights of $v^{+}_1$ and $v^{-}_1$ are $\mu_1+\ldots+\mu_{b_1}$ and $\mu_{l+2-b_1}+\ldots+\mu_{l+1}$, respectively.
Note that $$v^{-}_1=\left(x_{l+1-b_1,0}^{-}\ldots x_{1,0}^{-}\right)\left(x_{l+2-b_1,0}^{-}\ldots x_{2,0}^{-}\right)\ldots \left(x_{l,0}^{-}\ldots x_{b_1,0}^{-}\right)v_1^{+}.$$ 
Let
$\sigma=\left(s_{l+1-b_1}\ldots s_{1}\right)\left(s_{l+2-b_1}\ldots s_{2}\right)\ldots\left(s_{l}\ldots s_{b_1}\right)$, where $s_{k}=\left(k,k+1\right)$.
Define $\sigma_0=1$, $\sigma_1=s_{b_1}\sigma_0$, $\sigma_2=s_{b_1+1}\sigma_1$, and so on (based on the expression of $\sigma$). Thus there exists some $i'\in\{1,\ldots,l\}$ such that $\sigma_{i+1}=s_{i'}\sigma_i$. For a fixed $i$, there exist unique nonnegative integers $r$ and $s$ for $0\leq s< l-b_1+1$ such that $i=r\left(l-b_1+1\right)+s$. Then $i'=b_1-r+s$.
Define $v_{\omega_{b_1}}=v^{+}_{1}$ and $v_{\sigma_{i+1}(\omega_{b_1})} = x_{i',0}^- v_{\sigma_{i}(\omega_{b_1})}$ inductively. Moreover, these vectors are non-zero.

It follows from Proposition \ref{Delta} that $v^{+}_1\otimes v^{+}$ is a maximal vector and $h_{i,k}$ acts on $v^{+}_1\otimes v^{+}$ by a scalar multiple. Thus we only have to show that $v^{+}_1\otimes v^{+}$ generates $L$. It follows from Proposition \ref{v-w+gvtw} that it is enough to prove that $$v^{-}_{1}\otimes v^{+}\in Y\left(\mathfrak{sl}_{l+1}\right)\left(v^{+}_1\otimes v^{+}\right).$$ We divide the proof into the following steps.

Step 1: $\sigma_i^{-1}\left(\alpha_{i'}\right)\in \Delta^{+}$.

Proof: It is routine to check.

Step 2: $Y_{i'}\left(v_{\sigma_i\left(\omega_{b_1}\right)}\right)$ is a highest weight module of $Y_{i'}$.

Proof: Since the weight $\sigma_i\left(\omega_{b_1}\right)$ is on the Weyl group orbit of the highest weight and the representation $V_{a_1}(\omega_{b_1})$ is finite-dimensional, the weight space of weight $\sigma_i\left(\omega_{b_1}\right)$ is 1-dimensional. The elements $h_{j,s}$ form a commutative subalgebra $H$ of $\yn$, so $v_{\sigma_i\left(\omega_{b_1}\right)}$ is an eigenvector of $h_{i',r}$. Therefore we only have to show that $v_{\sigma_i\left(\omega_{b_1}\right)}$ is a maximal vector. Suppose to the contrary that $x_{i',k}^{+}v_{\sigma_i\omega_{b_i}}\neq 0$.  Then $x_{i',k}^{+}v_{\sigma_i\omega_{b_i}}$ is a weight vector of weight $\sigma_i\omega_{b_i}+\alpha_{i'}$, and then $\omega_{b_1}+\sigma_i^{-1}(\alpha_{i'})$ is a weight of $V_{a_1}\left(\omega_{b_1}\right)$. Therefore $\omega_{b_1}$ precedes $\omega_{b_1}+\sigma_i^{-1}(\alpha_{i'})$  by Step 1, which contradicts that $\omega_i$ is the highest weight of $V_{a_{1}}(\omega_{b_{1}})$.



Step 3:$wt\left(v_{\sigma_i\left(\omega_{b_1}\right)}\right)=\left(\mu_1+\ldots+\mu_{b_1-r-1}\right)+\mu_{b_1-r+s}+\left(\mu_{l-r+2}+\ldots \mu_{l+1}\right)$.

Proof: Recall that $v_{\omega_{b_1}}=e_1\wedge\ldots\wedge e_{b_1}$. $$v_{\sigma_i\left(\omega_{b_1}\right)}=\left(e_1\wedge e_2\wedge\ldots\wedge e_{b_1-r-1}\right)\wedge e_{b_1-r+s}\wedge \left(e_{l-r+2}\wedge e_{l-r+3}\wedge\ldots \wedge e_{l+1}\right).$$ Thus $wt\left(v_{\sigma_i\left(\omega_{b_1}\right)}\right)=\left(\mu_1+\ldots+\mu_{b_1-r-1}\right)+\mu_{b_1-r+s}+\left(\mu_{l-r+2}+\ldots \mu_{l+1}\right)$.

Step 4: The associated polynomial $P\left(u\right)$ of $Y_{i'}\left(v_{\sigma_i\left(\omega_{b_1}\right)}\right)$ has of degree 1.

Proof: Note that $i'=b_1-r+s$. It follows from step 3 that $$h_{i',0}v_{\sigma_i\left(\omega_{b_1}\right)}=v_{\sigma_i \left(\omega_{b_1}\right)}.$$
Thus the degree of $P\left(u\right)$ equals 1.

Step 5: $Y_{i'}\left(v_{\sigma_i\left(\omega_{b_1}\right)}\right)$ is a 2-dimensional $Y_{i'}$-highest weight module which is isomorphic to $W_1\left(a\right)$. Moreover, if $i\geq 1$, then $\operatorname{Re}\left(a\right)>\operatorname{Re}\left(a_1\right)$. The value of $a$ will be explicitly computed in the next section.

Proof: Let ${P\left(u\right)}=\left(u-a\right)$.
It follows from the definition of the local Weyl module of $Y\left(\mathfrak{sl}_2\right)$ that $Y_{i'}\left(v_{\sigma_i\left(\omega_{b_1}\right)}\right)$ is a sub-quotient of $W_1\left(a\right)$.
Since $W_{1}\left(a\right)$ has dimension 2, $Y_{i'}\left(v_{\sigma_i\left(\omega_{b_1}\right)}\right)=W_1\left(a\right)$ by the maximality of $W_1\left(a\right)$. By Proposition \ref{wra}, the value of $a$ can be determined by the eigenvalue of $v_{\sigma_i\left(\omega_{b_1}\right)}$ under $h_{i',1}$. We claim $a=\left(a_1+\frac{r+s}{2}\right)$, i.e., $$h_{i',1}v_{\sigma_i\left(\omega_{b_1}\right)}=\left(a_1+\frac{r+s}{2}\right)v_{\sigma_i\left(\omega_{b_1}\right)},$$ where $r, s$ are decided uniquely from $i=r\left(l-b_1+1\right)+s$ for $0\leq s<l-b_1+1$.  Let us assume for the moment that this is done (the proof will be given in the next section), and let us proceed to the next step.

Step 6:  $Y_{i'}\left(v_{\sigma_i\left(\omega_{b_1}\right)}\otimes v_2^{+}\otimes\ldots \otimes v_k^{+}\right)=Y_{i'}\left(v_{\sigma_i\left(\omega_{b_1}\right)}\right)\otimes Y_{i'}\left(v_2^{+}\right)\otimes\ldots\otimes Y_{i'}\left(v_k^{+}\right)$.

Proof:  $Y_{i'}\left(v_m^{+}\right)$ is either trivial or isomorphic to $W_{1}\left(a_m\right)$ if $b_m=i'$.  Suppose in $\{b_1,\ldots,b_k\}$ that $b_{j_1}=b_{j_2}=\ldots=b_{j_{m}}=i'$ with $j_1<\ldots<j_{m}$, and $b_n\neq i'$ if $n\notin \{j_1, j_2,\ldots, j_m\}$.
Note in Step 5 that $Y_{i'}\left(v_{\sigma_i\left(\omega_{b_1}\right)}\right)\cong W_1\left(a\right)$. Thus
\begin{align*}
   Y_{i'}&\left(v_{\sigma_i\left(\omega_{b_1}\right)}\right)\otimes Y_{i'}\left(v_2^{+}\right)\otimes\ldots\otimes Y_{i'}\left(v_k^{+}\right) \\
   &\cong \begin{cases}
 W_1\left(a\right)\otimes W_1\left(a_{j_1}\right)\otimes \ldots\otimes W_{1}\left(a_{j_m}\right)\ \mathrm{if}\ b_1\neq i' \\
 W_1\left(a\right)\otimes W_1\left(a_{j_2}\right)\otimes \ldots\otimes W_{1}\left(a_{j_m}\right) \ \mathrm{if}\ b_1=i'.
\end{cases}
\end{align*}

Since $\operatorname{Re}\left(a\right)\geq \operatorname{Re}\left(a_1\right)\geq \ldots\geq \operatorname{Re}\left(a_{j_1}\right)\geq \ldots\geq \operatorname{Re}\left(a_{j_m}\right)$, it follows from Corollary \ref{tpihw} that $Y_{i'}\left(v_{\sigma_i\left(\omega_{b_1}\right)}\right)\otimes Y_{i'}\left(v_2^{+}\right)\otimes\ldots\otimes Y_{i'}\left(v_k^{+}\right)$ is a highest weight module with highest weight vector $v_{\sigma_i\left(\omega_{b_1}\right)}\otimes v_2^{+}\otimes\ldots \otimes v_k^{+}$. Thus $$Y_{i'}\left(v_{\sigma_i\left(\omega_{b_1}\right)}\otimes v_2^{+}\otimes\ldots \otimes v_k^{+}\right)\supseteq Y_{i'}\left(v_{\sigma_i\left(\omega_{b_1}\right)}\right)\otimes Y_{i'}\left(v_2^{+}\right)\otimes\ldots\otimes Y_{i'}\left(v_k^{+}\right).$$ It follows from the coproduct of the Yangian  and Proposition \ref{c1dgd2d} that $$Y_{i'}\left(v_{\sigma_i\left(\omega_{b_1}\right)}\otimes v_2^{+}\otimes\ldots \otimes v_k^{+}\right)\subseteq Y_{i'}\left(v_{\sigma_i\left(\omega_{b_1}\right)}\right)\otimes Y_{i'}\left(v_2^{+}\right)\otimes\ldots\otimes Y_{i'}\left(v_k^{+}\right).$$ Therefore the claim is true.

Step 7: $v_{\sigma_{i+1}(\omega_{b_1})}\otimes v^{+}\in Y_{i'}\left(v_{\sigma_i\left(\omega_{b_1}\right)}\otimes v^{+}\right)$.

Proof: The claim follows from $$v_{\sigma_{i+1}(\omega_{b_1})}\otimes v^{+}\in Y_{i'}\left(v_{\sigma_{i}\omega_{b_1}}\right)\otimes Y_{i'}\left(v^{+}\right)=Y_{i'}\left(v_{\sigma_i\left(\omega_{b_1}\right)}\otimes v^{+}\right).$$

Step 8: $v_1^{-}\otimes v^{+}\in Y\left(\mathfrak{sl}_{l+1}\right)\left(v_1^{+}\otimes v^{+}\right)$.

Proof: It follows from step 7 immediately by induction on the subscript of $\sigma_i$.

It follows from Step 8 and Corollary \ref{v-w+gvtw} that $L=Y\left(\mathfrak{sl}_{l+1}\right)\left(v_1^{+}\otimes v^{+}\right)$.
\end{proof}

\begin{remark}
In what follows, these are the records of values of $a$ as in Step 5 of Proposition \ref{Lhwr}, which will be calculated explicitly in the next section.

In what follows, the notation $a_1\xlongrightarrow{x_{b_1,0}^{-}}a_1+\frac{1}{2}\xlongrightarrow{x_{b_1+1,0}^{-}} a_1+1$ means that $Y_{b_1}\left(v_1^{+}\right)\cong W_1\left(a_1\right)$ and $Y_{b_1+1}\left(x_{b_1,0}^{-}v_1^{+}\right)=W_1\left(a_1+\frac{1}{2}\right)$. 

\noindent $a_1\xlongrightarrow{x_{b_1,0}^{-}}a_1+\frac{1}{2}\xlongrightarrow{x_{b_1+1,0}^{-}}\ldots\xlongrightarrow{x_{l-1,0}^{-}} a_1+\frac{l-b_1}{2}\xlongrightarrow{x_{l,0}^{-}}a_1+\frac{1}{2}\xlongrightarrow{x_{b_1-1,0}^{-}}\left(a_1+\frac{1}{2}\right)+\frac{1}{2} \xlongrightarrow{x_{b_1,0}^{-}}\ldots \xlongrightarrow{x_{l-2,0}^{-}}\left(a_1+\frac{1}{2}\right)+\frac{l-b_1}{2}\xlongrightarrow{x_{l-1,0}^{-}} \left(a_1+1\right)\xlongrightarrow{x_{b_1-2,0}^{-}}\left(a_1+1\right)+\frac{1}{2}\xlongrightarrow{x_{b_1-1,0}^{-}}\ldots\xlongrightarrow{x_{l-3,0}^{-}} \left(a_1+1\right)+\frac{l-b_1}{2}\xlongrightarrow{x_{l-2,0}^{-}} \left(a_1+\frac{3}{2}\right)\xlongrightarrow{x_{b_1-3,0}^{-}}\ldots \xlongrightarrow{x_{l-b_1+2,0}^{-}} \left(a_1+\frac{b_1-1}{2}\right)\xlongrightarrow{x_{1,0}^{-}}\left(a_1+\frac{b_1-1}{2}\right)+\frac{1}{2}\xlongrightarrow{x_{2,0}^{-}}\ldots\xlongrightarrow{x_{l-b_1,0}^{-}} \left(a_1+\frac{b_1-1}{2}\right)+\frac{l-b_1}{2}=a_1+\frac{l-1}{2}\xlongrightarrow{x_{l+1-b_1,0}^{-}}\text{the lowest weight vector reached}$.
\end{remark}
Replacing the condition $\operatorname{Re}\left(a_1\right)\geq \operatorname{Re}\left(a_2\right)\geq \ldots \geq \operatorname{Re}\left(a_k\right)$, a much broader condition on $L$ to be a highest weight representation can be obtained, which allows us to obtain a sufficient condition on $L$ to be irreducible. Note that the first item of next lemma was proved by other methods in \cite{ChPr8}.

\begin{lemma}\label{C3l317ss}
$V_{a_m}\left(\omega_{b_m}\right)\otimes V_{a_n}\left(\omega_{b_n}\right)$ is a highest weight representation if $a_n-a_m\notin S\left(b_m,b_n\right)$, where the set $S\left(b_m,b_n\right)$ is defined as follows:
\begin{enumerate}
  \item $S\left(b_m,b_n\right)=\left\{\frac{b_n-b_m}{2}+k|1\leq k\leq \text{min}\left\{b_m, l-b_n+1\right\}\right\}$ for $b_m\leq b_n$.
  \item $S\left(b_m,b_n\right)=\left\{\frac{b_n-b_m}{2}+k|b_m-b_n+1\leq k\leq \text{min}\left\{b_m, l-b_n+1\right\}\right\}$\\ for $b_m>b_n$.
\end{enumerate}
\end{lemma}
\begin{proof}
By Proposition \ref{Lhwr}, if $i'\neq b_n$, then $Y_{i'}\left(v_{b_n}\right)$ is 1-dimensional. Thus $Y_{i'}\left(v_{\sigma_{i}\left(\omega_{b_m}\right)}\right)\otimes Y_{i'}\left(v_{b_n}\right)\cong Y_{i'}\left(v_{\sigma_{i}\left(\omega_{b_m}\right)}\right)$, which is a highest weight representation. So we may assume that $i'=b_n$.
Let's compare the values of $l-b_m+1$ and $b_n$.

Case 1: $b_m\leq b_n$ and $l-b_m+1\leq b_n$.

Note that in this case
\begin{align*}
   v^{-}_1&= \left(x_{l+1-b_m,0}^{-}\ldots x_{1,0}^{-}\right)\left(x_{l+2-b_m,0}^{-}\ldots x_{2,0}^{-}\right)\ldots \left(x_{b_n-1,0}^{-}\ldots  x_{b_m-\left(l-b_n\right)-1,0}^{-}\right)\\
   &  \left(x_{b_n,0}^{-}\ldots  x_{b_m-\left(l-b_n\right),0}^{-}\right)\ldots \left(x_{l-1,0}^{-}\ldots x_{b_n,0}^{-} \ldots x_{b_m-1,0}^{-}\right)\\
   & \left(x_{l,0}^{-}\ldots x_{b_n,0}^{-} \ldots x_{b_m,0}^{-}\right)v_1^{+}.
\end{align*}
The position \big(counting from the back\big) of $x_{b_n,0}^{-}$ in the $k$-th parenthesis \big(counting from the back\big) is $b_n-b_m+k$. Moreover $1\leq k\leq \left(l-b_n\right)+1$ (this is the number of parenthesis which contain the term $x_{b_n,0}^{-}$).

The only possible values for $i$ are $(k-1)\left(l-b_m+1\right)+\left(b_n-b_m+k-1\right)$. Therefore $$Y_{i'}\left(v_{\sigma_{i}\left(\omega_{b_m}\right)}\right)\cong W_1\left(a_m+\frac{k+b_n-b_m+k-2}{2}\right).$$
For $1\leq k\leq \left(l-b_n\right)+1$, if $a_n-\left(a_m+\frac{2k+b_n-b_m-2}{2}\right)\neq 1$, or $a_n-a_m\neq \frac{2k+b_n-b_m}{2}$, then $Y_{i'}\left(v_{\sigma_{i}\left(\omega_{b_m}\right)}\right)\otimes Y_{i'}\left(v_{b_n}\right)$ is a highest weight $Y_{i'}$-module.

Case 2: $b_m\leq b_n$ and $l-b_m+1>b_n$.

Note that in this case
\begin{align*}
   v^{-}_1=& \left(x_{l+1-b_m,0}^{-}\ldots x_{b_n,0}^{-} \ldots x_{1,0}^{-}\right)\left(x_{l+2-b_m,0}^{-}\ldots x_{b_n,0}^{-}\ldots x_{2,0}^{-}\right) \ldots\\
   &\qquad \left(x_{l-1,0}^{-}\ldots x_{b_n,0}^{-} \ldots x_{b_m-1,0}^{-}\right) \left(x_{l,0}^{-}\ldots x_{b_n,0}^{-} \ldots x_{b_m,0}^{-}\right)v_1^{+}.
\end{align*}
The position (counting from the back) of $x_{b_n,0}^{-}$ in the $k$-th parenthesis (counting from the back) is $b_n-b_m+k$ for $1\leq k\leq b_m$.

The possible values for $i$ are $(k-1)\left(l-b_m+1\right)+\left(b_n-b_m+k-1\right)$. Therefore $$Y_{i'}\left(v_{\sigma_{i}\left(\omega_{b_m}\right)}\right)\cong W_1\left(a_m+\frac{k+b_n-b_m+k-2}{2}\right).$$
For $1\leq k\leq b_m$, if $a_n-\left(a_m+\frac{2k+b_n-b_m-2}{2}\right)\neq 1$, or $a_n-a_m\neq \frac{2k+b_n-b_m}{2}$,\\ $Y_{i'}\left(v_{\sigma_{i}\left(\omega_{b_m}\right)}\right)\otimes Y_{i'}\left(v_{b_n}\right)$ is a highest weight $Y_{i'}$-module.

Case 3: $b_m>b_n$ and $l-b_m+1\leq b_n$.

In this case
\begin{align*}
   v^{-}_1=& \left(x_{l+1-b_m,0}^{-}\ldots x_{1,0}^{-}\right)\ldots \left(x_{b_n-1,0}^{-}\ldots  x_{b_m-\left(l-b_n\right)-1,0}^{-}\right)\\
   &\left(x_{b_n,0}^{-}\ldots  x_{b_m-\left(l-b_n\right),0}^{-}\right)\ldots 
\left(x_{l-b_m+b_n,0}^{-}\ldots x_{b_n,0}^{-}\right)\\
   & \left(x_{l-b_m+b_n+1,0}^{-}\ldots x_{b_n+1,0}^{-}\right) \ldots \left(x_{l-1,0}^{-} \ldots x_{b_m-1,0}^{-}\right)\left(x_{l,0}^{-} \ldots x_{b_m,0}^{-}\right)v_1^{+}.
\end{align*}
The position of $x_{b_n,0}^{-}$ in the $k$-th parenthesis (counting from the back) is $k-\left(b_m-b_n\right)$, where $b_m-b_n+1\leq k\leq l-b_n+1$.

The possible values of $i$ are $(k-1)\left(l-b_m+1\right)+\Big(k-\left(b_m-b_n\right)-1\Big)$. Therefore $$Y_{i'}\left(v_{\sigma_{i}\left(\omega_{b_m}\right)}\right)\cong W_1\left(a_m+\frac{2k-b_m+b_n-2}{2}\right).$$
For $b_m-b_n+1\leq k\leq l-b_n+1$, if $a_n-\left(a_m+\frac{2k-b_m+b_n-2}{2}\right)\neq 1$, or $a_n-a_m\neq \frac{2k-b_m+b_n}{2}$, then $Y_{i'}\left(v_{\sigma_{i}\left(\omega_{b_m}\right)}\right)\otimes Y_{i'}\left(v_{b_n}\right)$ is a highest weight $Y_{i'}$-module.

Case 4: $b_m>b_n$ and $l-b_m+1> b_n$.

In this case
\begin{align*}
   v^{-}_1=& \left(x_{l+1-b_m,0}^{-}\ldots x_{1,0}^{-}\right)\ldots \left(x_{l-b_m+b_n-1,0}^{-}\ldots x_{b_n,0}^{-} x_{b_n-1,0}^{-}\right)\left(x_{l-b_m+b_n,0}^{-}\ldots x_{b_n,0}^{-}\right)\\
   & \left(x_{l-b_m+b_n+1,0}^{-}\ldots x_{b_n+1,0}^{-}\right) \ldots \left(x_{l-1,0}^{-} \ldots x_{b_m-1,0}^{-}\right)\left(x_{l,0}^{-} \ldots x_{b_m,0}^{-}\right)v_1^{+}.
\end{align*}
The position of $x_{b_n,0}^{-}$ in the $k$-th parenthesis (counting from the back) is $k-\left(b_m-b_n\right)$, where $b_m-b_n+1\leq k\leq b_m$.

The possible values of $i$ are $(k-1)\left(l-b_m+1\right)+\Big(k-\left(b_m-b_n\right)-1\Big)$. Therefore $$Y_{i'}\left(v_{\sigma_{i}\left(\omega_{b_m}\right)}\right)\cong W_1\left(a_m+\frac{2k-b_m+b_n-2}{2}\right).$$
For $b_m-b_n+1\leq k\leq b_m$, if $a_n-\left(a_m+\frac{2k-b_m+b_n-2}{2}\right)\neq 1$, or $a_n-a_m\neq \frac{2k-b_m+b_n}{2}$, then $Y_{i'}\left(v_{\sigma_{i}\left(\omega_{b_m}\right)}\right)\otimes Y_{i'}\left(v_{b_n}\right)$ is a highest weight $Y_{i'}$-module.
\end{proof}

\begin{theorem}\label{3mt2icl}
The tensor product $L=V_{a_1}(\omega_{b_1})\otimes V_{a_2}(\omega_{b_2})\otimes\ldots\otimes V_{a_k}(\omega_{b_k})$ is a highest weight representation of $\yn$ if $a_j-a_i\notin S\left(b_i,b_j\right)$ for any pair $(i,j)$ with $1\leq i<j\leq k$.
\end{theorem}
\begin{proof} Let us keep the notations used in Proposition \ref{Lhwr}. The proof is similar to the proof of Proposition \ref{Lhwr}. The only difference is Step 6. If $a_j-a_i\notin S\left(b_i,b_j\right)$ with $1\leq i<j\leq k$, then the difference of the root of any two associated polynomials of $Y_{i'}\left(v_{\sigma_i\left(\omega_{b_1}\right)}\right)$, $Y_{i'}\left(v_2^{+}\right),\ldots, Y_{i'}\left(v_k^{+}\right)$ never equals 1.   Therefore $Y_{i'}\left(v_{\sigma_i\left(\omega_{b_1}\right)}\right)\otimes Y_{i'}\left(v_2^{+}\right)\otimes\ldots\otimes Y_{i'}\left(v_k^{+}\right)$ is a highest weight representation by Proposition \ref{ctpihw}.
\end{proof}

We close this section by giving a sufficient and necessary condition for the irreducibility of $L$.
\begin{theorem}[Theorem 6.2, \cite{ChPr8}]\label{cp8c3t}
Let $1\leq b_m\leq b_n\leq l$. $V_{a_m}(\omega_{b_m})\otimes V_{a_n}(\omega_{b_n})$ is reducible as a $\yn$-module if and only if
\begin{center}
 $a_n-a_m=\pm\Big(\frac{b_n-b_m}{2}+r\Big)$, where $0<r\leq$ min $(b_m,l+1-b_n).$
\end{center}
\end{theorem}
\begin{remark}\label{C3ysllrat}\
 Theorem \ref{cp8c3t} can be paraphrased as the following. Let $1\leq b_m\leq b_n\leq l$. $V_{a_m}(\omega_{b_m})\otimes V_{a_n}(\omega_{b_n})$ is reducible as a $\yn$-module if and only if $a_n-a_m\notin S(b_m, b_n)$ or $a_m-a_n\notin S(b_m, b_n)$.
\end{remark}
\begin{theorem}\label{cp8c3tng2}
$L=V_{a_1}(\omega_{b_1})\otimes V_{a_2}(\omega_{b_2})\otimes\ldots\otimes V_{a_k}(\omega_{b_k})$ is an irreducible representation of $\yn$ if and only if for any $a_i$ and $a_j$, $a_j-a_i\notin S\left(b_i,b_j\right)$ with $1\leq i\neq j\leq k$.
\end{theorem}
\begin{proof}
$``\Leftarrow"$ It follows from Theorem \ref{3mt2icl} that if $a_j-a_i\notin S(b_i,b_j)$, then $L$ is a highest weight representation. We next show that $\ ^tL$ is a highest weight representation if $a_i-a_j\notin  S(b_i,b_j)$.
Note that $\ ^tV_{a}\left(\omega_{b}\right)=V_{a-\kappa}\left(\omega_{l+1-b}\right)$ and $\ ^t(V\otimes W)=\ ^tW\otimes \ ^tV$ as $\yn$-module, where $\kappa=\frac{1}{2}\times$dual Coxeter number of $\nyn$. Therefore
\begin{align*}
\ ^tL=V_{a_{k}-\kappa}&\left(\omega_{l+1-b_{k}}\right)\otimes \ldots \otimes V_{a_{j}-\kappa}\left(\omega_{l+1-b_{j}}\right)\\
& \otimes \ldots \otimes V_{a_{i}-\kappa}\left(\omega_{l+1-b_{i}}\right)\otimes \ldots\otimes V_{a_{1}-\kappa}\left(\omega_{l+1-b_{1}}\right).
\end{align*}
It follows from Theorem \ref{3mt2icl} that $\ ^tL$ is a highest weight representation if $a_i-a_j\notin S(l+1-b_j, l+1-b_i)$. A straight forward computation shows that $S(l+1-b_j, l+1-b_i)=S(b_i,b_j)$. Therefore $a_i-a_j\notin S(b_i, b_j)$ implies that $\ ^tL$ is a highest weight representation. Then $L$ is irreducible by Proposition \ref{VoWWoVhi}.

$``\Rightarrow"$ Suppose that $L$ is irreducible and there exists $a_i$ and $a_j$ such that  $a_j-a_i\in S(b_i,b_j)$. Since $L$ is irreducible, any permutation of tensor factors $V_{a_s}(\omega_{b_s})$ gives an isomorphic representation of $\yn$. Arranging the order if necessary, we may assume that $a_2-a_1\in S(b_1, b_2)$. A straight computation shows that $S(b_1,b_2)=S(b_2, b_1)$. If $b_1\leq b_2$, then $V_{a_1}(\omega_{b_1})\otimes V_{a_2}(\omega_{b_2})$ is reducible by Theorem \ref{cp8c3t}, so is $L$, contradicting to the assumption that $L$ is irreducible. Thus $b_1> b_2$. Since $a_2-a_1\in S(b_1,b_2)=S(b_2, b_1)$, $V_{a_2}(\omega_{b_2})\otimes V_{a_1}(\omega_{b_1})$ is reducible by Theorem \ref{cp8c3t}, and then $L$ is reducible. Therefore, if $a_j-a_i\in S(b_i,b_j)$, $L$ is reducible.
\end{proof}
\section{Supplement Step 5 in Proposition \ref{Lhwr}}
It remains to show that $\operatorname{Re}\left(a\right)\geq \operatorname{Re}\left(a_1\right)$. Before we prove it, we would like to introduce a notation to denote $x_{l,0}^{-}\ldots x_{b_1+1,0}^{-}x_{b_1,0}^{-}$ for the simplicity of formulas.
Let $X_{b_1}^{\left(m+1\right)}=x_{b_1+m,0}^{-}\ldots x_{b_1+1,0}^{-}x_{b_1,0}^{-}$ with the following conditions: $0\leq m\leq l-b_1$, the number of terms in $x_{b_1+m,0}^{-}\ldots x_{b_1+1,0}^{-}x_{b_1,0}^{-}$ is $m+1$, and there is a consecutive increase from $b_1$ to $b_1+m$. For example, $X_{b_1-r}^{\left(s\right)}=x_{b_1-r+s-1,0}^{-}\ldots x_{b_1-r+1,0}^{-}x_{b_1-r,0}^{-}$ and $X_{b_1-r}^{\left(l-b_1+1\right)}=x_{l-r,0}^{-}\ldots x_{b_1-r+1,0}^{-}x_{b_1-r,0}^{-}$.


We want to point out that the irreducible representation $W_1\left(a\right)$ of $\ysl$ is very important in the calculations in this section.
It follows from Proposition 2.6 of \cite{ChPr3} that in the 2-dimensional irreducible module $W_1\left(a\right)=\mathrm{span}\{v_1, x_{i,0}^{-}v_1\}$ of $Y\left(\mathfrak{sl}_2\right)$,
\begin{center}
if $h_{i,1}v_1=av_1$, then $x_{i,1}^{-}v=ax_{i,0}^{-}v$ and $h_{i,1}x_{i,0}^{-}v_1=-ax_{i,0}^{-}v_1.\qquad\qquad (\ref{Cor3.3})$
\end{center}

Let $i=r(l-b_1+1)+s$. We claim that
\begin{align}\label{hi1v}
   h_{i',1}X_{b_1-r}^{\left(s\right)}&X_{b_1-r+1}^{\left(l-b_1+1\right)}\ldots X_{b_1}^{\left(l-b_1+1\right)}v^{+}_1 \nonumber\\
&= h_{b_1-r+s,1}X_{b_1-r}^{\left(s\right)}X_{b_1-r+1}^{\left(l-b_1+1\right)}\ldots X_{b_1}^{\left(l-b_1+1\right)}v^{+}_1 \nonumber\\
   &= \left(a_1+\frac{r+s}{2}\right)X_{b_1-r}^{\left(s\right)}X_{b_1-r+1}^{\left(l-b_1+1\right)}\ldots X_{b_1}^{\left(l-b_1+1\right)}v^{+}_1.
\end{align}

We prove it by induction on $r$. For a fixed $r$, we use induction on $s$.
Let $r=0, s=1$.
\begin{align*}
  h_{b_1+1,1}x_{b_1,0}^{-}v^{+}_1 &= [h_{b_1+1,1},x_{b_1,0}^{-}]v^{+}_1  \\
   &= \left([h_{b_1+1,0},x_{b_1,1}^{-}]+\frac{1}{2}\left(h_{b_1+1,0}x_{b_1,0}^{-}+x_{b_1,0}^{-}h_{b_1+1,0}\right)\right)v^{+}_1 \\
   &= \left(x_{b_1,1}^{-}+\frac{1}{2}x_{b_1,0}^{-}+x_{b_1,0}^{-}h_{b_1+1,0}\right)v^{+}_1 \\
  &= x_{b_1,1}^{-}v^{+}_1+\frac{1}{2}x_{b_1,0}^{-}v^{+}_1 \\
  &= \left(a_1+\frac{1}{2}\right)x_{b_1,0}^{-}v^{+}_1.
\end{align*}
The claim is true.

Given a pair $\left(r,s\right)$, suppose that the claim (\ref{hi1v}) is true for all pairs $\left(m,n\right)$ such that  $m<r$ and all possible values of $n$.

We next show that (\ref{hi1v}) is true for the pair $\left(r,0\right)$.
By the induction hypothesis we may assume that
\begin{align*}
h_{b_1-r+1,1}&X_{b_1-r+2}^{\left(l-b_1+1\right)}X_{b_1-r+3}^{\left(l-b_1+1\right)}\ldots X_{b_1}^{\left(l-b_1+1\right)}v^{+}_1 \\
 &= \left(a_1+\frac{r-1}{2}\right)X_{b_1-r+2}^{\left(l-b_1+1\right)}X_{b_1-r+2}^{\left(l-b_1+1\right)}\ldots X_{b_1}^{\left(l-b_1+1\right)}v^{+}_1.
\end{align*}
From equations \ref{Cor3.3}, we have
\begin{align*}
 x_{b_1-r+1,1}^{-}&X_{b_1-r+2}^{\left(l-b_1+1\right)}X_{b_1-r+3}^{\left(l-b_1+1\right)}\ldots X_{b_1}^{\left(l-b_1+1\right)}v^{+}_1 \\
 &= \left(a_1+\frac{r-1}{2}\right)x_{b_1-r+1,0}^{-}X_{b_1-r+2}^{\left(l-b_1+1\right)}X_{b_1-r+3}^{\left(l-b_1+1\right)}\ldots X_{b_1}^{\left(l-b_1+1\right)}v^{+}_1.
\end{align*}
It is easy to see by the defining relations of Yangians that
\begin{align*}
h_{b_1-r,0}^{-}X_{b_1-r+2}^{\left(l-b_1+1\right)}\ldots X_{b_1}^{\left(l-b_1+1\right)}v^{+}_1&=0.
\end{align*}
Thus
\begin{align*}
        [h_{b_1-r,1},& x_{b_1-r+1,0}^{-}] X_{b_1-r+2}^{\left(l-b_1+1\right)}\ldots X_{b_1}^{\left(l-b_1+1\right)}v^{+}_1 \\
        &=\left(x_{b_1-r+1,1}^{-}+\frac{1}{2}x_{b_1-r+1,0}^{-}+x_{b_1-r+1,0}^{-}h_{b_1-r,0}\right)X_{b_1-r+2}^{\left(l-b_1+1\right)}\ldots X_{b_1}^{\left(l-b_1+1\right)}v^{+}_1\\
       &=\left(a_1+\frac{r}{2}\right)X_{b_1-r+1}^{\left(l-b_1+1\right)}X_{b_1-r+2}^{\left(l-b_1+1\right)}\ldots X_{b_1}^{\left(l-b_1+1\right)}v^{+}_1.
     \end{align*}
Therefore
\begin{align*}
h_{b_1-r,1}&X_{b_1-r+1}^{\left(l-b_1+1\right)}X_{b_1-r+2}^{\left(l-b_1+1\right)}\ldots X_{b_1}^{\left(l-b_1+1\right)}v^{+}_1 \\
&=[h_{b_1-r,1},X_{b_1-r+1}^{\left(l-b_1+1\right)}X_{b_1-r+2}^{\left(l-b_1+1\right)}\ldots X_{b_1}^{\left(l-b_1+1\right)}]v^{+}_1 \\
&=[h_{b_1-r,1},X_{b_1-r+1}^{\left(l-b_1+1\right)}]X_{b_1-r+2}^{\left(l-b_1+1\right)}\ldots X_{b_1}^{\left(l-b_1+1\right)}v^{+}_1\\
&=X_{b_1-r+1}^{\left(l-b_1\right)}[h_{b_1-r,1}, x_{b_1-r+1,0}^{-}]X_{b_1-r+2}^{\left(l-b_1+1\right)}\ldots X_{b_1}^{\left(l-b_1+1\right)}v^{+}_1\\
  &=\left(a_1+\frac{r}{2}\right)X_{b_1-r+1}^{\left(l-b_1+1\right)}X_{b_1-r+2}^{\left(l-b_1+1\right)}\ldots X_{b_1}^{\left(l-b_1+1\right)}v^{+}_1.
\end{align*}
We next show that (\ref{hi1v}) is true for the pair $\left(r,s\right)$ by induction on $s$. We may assume that $(\ref{hi1v})$ is true for all number less than or equals $s-1$.
Note that it was shown in Proposition \ref{Lhwr} that $$\mathrm{wt}\left(v_{\sigma_i\left(\omega_{b_1}\right)}\right)=\left(\mu_1+\ldots+\mu_{b_1-r-1}\right)+\mu_{b_1-r+s}+\left(\mu_{l-r+2}+\ldots +\mu_{l+1}\right).$$

Thus \begin{align*}
        \mathrm{wt}&\left(X_{b_1-r}^{\left(s-1\right)}X_{b_1-r+1}^{\left(l-b_1+1\right)}X_{b_1-r+2}^{\left(l-b_1+1\right)}\ldots X_{b_1}^{\left(l-b_1+1\right)}v^{+}_1\right) \\
        &=\left(\mu_1+\ldots+\mu_{b_1-r-1}\right)+\mu_{b_1-r+s-1}+\left(\mu_{l-r+2}+\ldots +\mu_{l+1}\right)
     \end{align*}
and
\begin{align*}
        \mathrm{wt}&\left(X_{b_1-r+1}^{\left(s\right)}X_{b_1-r+2}^{\left(l-b_1+1\right)}\ldots X_{b_1}^{\left(l-b_1+1\right)}v^{+}_1\right) \\
        &=\left(\mu_1+\ldots+\mu_{b_1-r}\right)+\mu_{b_1-r+s-1}+\left(\mu_{l-r+3}+\ldots +\mu_{l+1}\right).
     \end{align*}

It follows from induction hypothesis and \ref{Cor3.3} that
\begin{align*}
   x_{b_1-r+s-1,1}^{-}& X_{b_1-r}^{\left(s-1\right)}X_{b_1-r+1}^{\left(l-b_1+1\right)}X_{b_1-r+2}^{\left(l-b_1+1\right)}\ldots X_{b_1}^{\left(l-b_1+1\right)}v^{+}_1 \\
   &=\left(a_1+\frac{r+\left(s-1\right)}{2}\right)X_{b_1-r}^{\left(s-1\right)}X_{b_1-r+1}^{\left(l-b_1+1\right)}X_{b_1-r+2}^{\left(l-b_1+1\right)}\ldots X_{b_1}^{\left(l-b_1+1\right)}v^{+}_1,
\end{align*}
and
\begin{align*}
   x_{b_1-r+s+1,1}^{-}&X_{b_1-r+1}^{\left(s\right)}X_{b_1-r+2}^{\left(l-b_1+1\right)}\ldots X_{b_1}^{\left(l-b_1+1\right)}v^{+}_1\\
   &=\left(a_1+\frac{\left(r-1\right)+s}{2}\right)X_{b_1-r+1}^{\left(s+1\right)}X_{b_1-r+2}^{\left(l-b_1+1\right)}\ldots X_{b_1}^{\left(l-b_1+1\right)}v^{+}_1.
\end{align*}

Now we are ready to prove (\ref{hi1v}) for the pair $\left(r,s\right)$.
\begin{align*}
h_{b_1-r+s,1}&X_{b_1-r}^{\left(s\right)}X_{b_1-r+1}^{\left(l-b_1+1\right)}X_{b_1-r+2}^{\left(l-b_1+1\right)}\ldots X_{b_1}^{\left(l-b_1+1\right)}v^{+}_1 \\
  &= h_{b_1-r+s,1}\left(x_{b_1-r+s-1,0}^{-}\ldots x_{b_1-r,0}^{-}\right)X_{b_1-r+1}^{\left(l-b_1+1\right)}\ldots X_{b_1}^{\left(l-b_1+1\right)}v^{+}_1\\
  &=[h_{b_1-r+s,1},x_{b_1-r+s-1,0}^{-}]X_{b_1-r}^{\left(s-1\right)}X_{b_1-r+1}^{\left(l-b_1+1\right)}X_{b_1-r+2}^{\left(l-b_1+1\right)}\ldots X_{b_1}^{\left(l-b_1+1\right)}v^{+}_1\\
   &+X_{b_1-r}^{\left(s\right)}h_{b_1-r+s,1}X_{b_1-r+1}^{\left(l-b_1+1\right)}X_{b_1-r+2}^{\left(l-b_1+1\right)}\ldots X_{b_1}^{\left(l-b_1+1\right)}v^{+}_1\\
   &= \left(x_{b_1-r+s-1,1}^{-}+\frac{1}{2}x_{b_1-r+s-1,0}^{-}+x_{b_1-r+s-1,0}^{-}h_{b_1-r+s,0}\right)X_{b_1-r}^{\left(s-1\right)}X_{b_1-r+1}^{\left(l-b_1+1\right)}\\
   &\qquad \qquad \cdot X_{b_1-r+2}^{\left(l-b_1+1\right)}\ldots X_{b_1}^{\left(l-b_1+1\right)}v^{+}_1\\
  &+X_{b_1-r}^{\left(s\right)}h_{b_1-r+s,1}\left(x_{l-r+1,0}^{-}\ldots x_{b_1-r+2,0}^{-}x_{b_1-r+1,0}^{-}\right)X_{b_1-r+2}^{\left(l-b_1+1\right)}\ldots X_{b_1}^{\left(l-b_1+1\right)}v^{+}_1\\
  &= \left(a_1+\frac{r+\left(s-1\right)}{2}+\frac{1}{2}\right)X_{b_1-r}^{\left(s\right)}X_{b_1-r+1}^{\left(l-b_1+1\right)}X_{b_1-r+2}^{\left(l-b_1+1\right)}\ldots X_{b_1}^{\left(l-b_1+1\right)}v^{+}_1\\
  &+X_{b_1-r}^{\left(s\right)}h_{b_1-r+s,1}\left(x_{l-r+1,0}^{-}\ldots x_{b_1-r+2,0}^{-}x_{b_1-r+1,0}^{-}\right)X_{b_1-r+2}^{\left(l-b_1+1\right)}\ldots X_{b_1}^{\left(l-b_1+1\right)}v^{+}_1\\
  &= \left(a_1+\frac{r+s}{2}\right)X_{b_1-r}^{\left(s\right)}X_{b_1-r+1}^{\left(l-b_1+1\right)}X_{b_1-r+2}^{\left(l-b_1+1\right)}\ldots X_{b_1}^{\left(l-b_1+1\right)}v^{+}_1\\%
  &+X_{b_1-r}^{\left(s\right)}X_{b_1-r+s+2}^{\left(l-b_1-s\right)}[h_{b_1-r+s,1},x_{b_1-r+s+1,0}^{-}]X_{b_1-r+1}^{\left(s\right)}X_{b_1-r+2}^{\left(l-b_1+1\right)}\ldots X_{b_1}^{\left(l-b_1+1\right)}v^{+}_1\\%
  &+X_{b_1-r}^{\left(s\right)}X_{b_1-r+s+1}^{\left(l-b_1-s+1\right)} h_{b_1-r+s,1}x_{b_1-r+s,0}^{-} X_{b_1-r+1}^{\left(s-1\right)}X_{b_1-r+2}^{\left(l-b_1+1\right)}\ldots X_{b_1}^{\left(l-b_1+1\right)}v^{+}_1\\%
   &= \left(a_1+\frac{r+s}{2}\right)X_{b_1-r}^{\left(s\right)}X_{b_1-r+1}^{\left(l-b_1+1\right)}X_{b_1-r+2}^{\left(l-b_1+1\right)}\ldots X_{b_1}^{\left(l-b_1+1\right)}v^{+}_1\\%
  &+X_{b_1-r}^{\left(s\right)}X_{b_1-r+s+2}^{\left(l-b_1-s\right)}\left(x_{b_1-r+s+1,1}^{-}+\frac{1}{2}x_{b_1-r+s+1,0}^{-}+x_{b_1-r+s+1,0}^{-}h_{b_1-r+s,0}\right)\\
  &\qquad\qquad \cdot X_{b_1-r+1}^{\left(s\right)}X_{b_1-r+2}^{\left(l-b_1+1\right)}\ldots X_{b_1}^{\left(l-b_1+1\right)}v^{+}_1\\%
  &-X_{b_1-r}^{\left(s\right)}X_{b_1-r+s+1}^{\left(l-b_1-s+1\right)} h_{b_1-r+s,1}X_{b_1-r+1}^{\left(s-1\right)}X_{b_1-r+2}^{\left(l-b_1+1\right)}\ldots X_{b_1}^{\left(l-b_1+1\right)}v^{+}_1\\%
 &= \left(a_1+\frac{r+s}{2}\right)X_{b_1-r}^{\left(s\right)}X_{b_1-r+1}^{\left(l-b_1+1\right)}X_{b_1-r+2}^{\left(l-b_1+1\right)}\ldots X_{b_1}^{\left(l-b_1+1\right)}v^{+}_1\\
  &+\left(a_1+\frac{\left(r-1\right)+s}{2}+\frac{1}{2}-1\right) X_{b_1-r}^{\left(s\right)}X_{b_1-r+1}^{\left(l-b_1+1\right)}X_{b_1-r+2}^{\left(l-b_1+1\right)}\ldots X_{b_1}^{\left(l-b_1+1\right)}v^{+}_1\\
  &-\left(a_1+\frac{(r-1)+(s-1)}{2}\right)X_{b_1-r}^{\left(s\right)}X_{b_1-r+1}^{\left(l-b_1+1\right)}X_{b_1-r+2}^{\left(l-b_1+1\right)}\ldots X_{b_1}^{\left(l-b_1+1\right)}v^{+}_1\\
 &= \left(a_1+\frac{r+s}{2}\right)X_{b_1-r}^{\left(s\right)}X_{b_1-r+1}^{\left(l-b_1+1\right)}X_{b_1-r+2}^{\left(l-b_1+1\right)}\ldots X_{b_1}^{\left(l-b_1+1\right)}v^{+}_1.
\end{align*}
By induction, the claim is proved.
\section{On the local Weyl modules of $\yn$}
Our main theorem in this chapter follows from Theorem \ref{frosllc1} and Propositions \ref{tati}, \ref{Lhwr} and \ref{vtv'hwv}.
\begin{theorem}\label{wmiatp}
Let $\pi=\Big(\pi_1\left(u\right),\ldots, \pi_{l}\left(u\right)\Big)$, where $\pi_i\left(u\right)=\prod\limits_{j=1}^{m_i}\left(u-a_{i,j}\right)$ . Let $S=\{a_{1,1},\ldots, a_{1,m_1},\ldots, a_{l,1}\ldots, a_{l,m_l}\}$ be the multiset of roots of these polynomials. Let  $a_1=a_{m,n}$ be one of the numbers in $S$ with the maximal real part and let $b_1=m$. Similarly let $a_j=a_{s,t}\left(j\geq 2\right)$ be one of the numbers in $S\setminus\{a_1, \ldots, a_{j-1}\}$ with the maximal real part, and let $b_j=s$. Let $L=V_{a_1}(\omega_{b_1})\otimes V_{a_2}(\omega_{b_2})\otimes\ldots\otimes V_{a_k}(\omega_{b_k})$, where $k=m_1+\ldots+m_l$. Then $L$ is a highest weight representation of dimension $\prod\limits_{i=1}^{l} {l+1\choose i}^{m_i}$, and its associated polynomial is $\pi$.
\end{theorem}
Comparing the upper bound of dimension of $W(\pi)$ and the dimension of $L$ above, we can determine explicitly the local Weyl module $W(\pi)$ and its dimension.
\begin{theorem}
The local Weyl module $W(\pi)$ of $\yn$ has dimension $\prod\limits_{i=1}^{l} {l+1\choose i}^{m_i}$, and is isomorphic to $L$ as in Theorem \ref{wmiatp}.
\end{theorem}
\chapter{Local Weyl modules of $\yo$}
In this chapter, the local Weyl modules of $\yo$ are studied. The structure of the local Weyl modules is determined, and the dimensions of the local Weyl modules are obtained.  In the process of characterizing the local Weyl modules, a sufficient condition for a tensor product of fundamental representations of the Yangian to be highest weight representation is obtained, which shall lead to an irreducibility criterion for the tensor product.

Let $\pi=\big(\pi_1(u),\pi_2(u),\ldots, \pi_l(u)\big)$ be a generic $l$-tuple of monic polynomials in $u$, and $\pi_i\left(u\right)=\prod\limits_{j=1}^{m_i}\left(u-a_{i,j}\right)$. Let $k=m_1+m_2+\ldots+m_l$, $S=\{a_{i,j}|i=1,\ldots,l; j=1,\ldots,m_i\}$, and $\lambda=\sum\limits_{i=1}^{l}m_i\omega_i$.
Let $a_{m,n}$ be one of the numbers in $S$ with the maximal real part. Then define $a_1=a_{m,n}$ and $b_1=m$. Inductively, let $a_{s,t}$ be one of the numbers in $S-\{a_1,\ldots, a_{r-1}\}$ ($r\geq 2$) with the maximal real part. Then define $a_r=a_{s,t}$ and $b_r=s$. We prove that the ordered tensor product $L=V_{a_1}\left(\omega_{b_1}\right)\otimes V_{a_2}\left(\omega_{b_2}\right)\otimes \ldots \otimes V_{a_k}\left(\omega_{b_k}\right)$ is a highest weight representation, where $V_{a_i}(\omega_{b_i})$ are fundamental representations of $\yo$. A standard argument shows that the associated $l$-tuple of polynomials of $L$ is $\pi$.
Since $L$ is a quotient of $W(\pi)$, a lower bound on the dimension of $W(\pi)$ is obtained.

Let $W(\lambda)$ be the Weyl module associated to $\lambda$ of the current algebra $\nyo[t]$. It is showed in Corollary B of \cite{FoLi} that $\operatorname{Dim}\big(W(\lambda)\big)=\prod\limits_{i\in I} \Big(\operatorname{Dim}\big(W(\omega_i)\big)\Big)^{m_i}$. It follows from Corollary \ref{dkrvawocsp} that $\operatorname{Dim}(W(\lambda))=\operatorname{Dim}(L)$. Since $\operatorname{Dim}\big(W(\lambda)\big)\geq \operatorname{Dim}(W(\pi))\geq \operatorname{Dim}(L)$,  $$W\left(\pi\right)\cong V_{a_1}(\omega_{b_1})\otimes V_{a_2}(\omega_{b_2})\otimes\ldots\otimes V_{a_k}(\omega_{b_k}).$$
The dimension of $W(\pi)$ can be recovered from Theorem \ref{dofrdl}, Remark \ref{rofry} and Proposition \ref{dcofbdl}.

Similar to the proof that $L$ is a highest weight representation (Proposition \ref{v+=gv-}), we can obtain a sufficient condition for a tensor product of fundamental representations of the form $V_{a_1}(\omega_{b_1})\otimes V_{a_2}(\omega_{b_2})\otimes\ldots\otimes V_{a_k}(\omega_{b_k})$ to be a highest weight representation. If $a_j-a_i\notin S(b_i, b_j)$ for $1\leq i<j\leq k$, then $L$ is a highest weight representation, where $S(b_i, b_j)$ is a finite set of positive rational numbers as in Lemma \ref{ycdlsbij}. By Proposition \ref{VoWWoVhi} and Lemma \ref{dualfrc1}, an irreducible criterion for a tensor product of fundamental representations of $\yo$ is obtained: if $a_j-a_i\notin S(b_i, b_j)$ for $1\leq i\neq j\leq k$, then $L$ is irreducible.

\section{From the highest weight vector to the lowest weight vector in $V_{a}(\omega_i)$}

Let $s_i$ be the fundamental reflections of the Weyl group of $\nyo$ for $i=1,\ldots, l$. Denote by $\{\mu_1,\mu_2,\ldots,\mu_{l}\}$ the coordinate functions on the Cartan subalgebra of $\nyo$. When $1\leq i\leq l-1$, $s_i\left(\mu_i\right)=\mu_{i+1}$, $s_i\left(\mu_{i+1}\right)=\mu_{i}$ and $s_i\left(\mu_j\right)=\mu_j$ for $j\neq i, i+1$.
When $i=l$, $s_l\left(\mu_{l-1}\right)=-\mu_{l}$, $s_l\left(\mu_{l}\right)=-\mu_{l-1}$ and $s_l\left(\mu_j\right)=\mu_j$ for $j\neq l-1, l$.
The fundamental weights of $\nyo$ are given by $\omega_i=\mu_1+\ldots+\mu_i$ for $1\leq i\leq l-2$, $\omega_{l-1}=\frac{1}{2} \left(\mu_1+\ldots+\mu_{l-1}-\mu_l\right)$ and $\omega_{l}=\frac{1}{2} \left(\mu_1+\ldots+\mu_{l-1}+\mu_l\right)$. It follows that $\mu_1=\omega_1$, $\mu_2=-\omega_1+\omega_2$, $\ldots$, $\mu_{l-2}=-\omega_{l-3}+\omega_{l-2}$, $\mu_{l-1}=-\omega_{l-2}+\omega_{l-1}+\omega_{l}$ and $\mu_{l}=-\omega_{l-1}+\omega_{l}$. Thus $\alpha_1=2\omega_1-\omega_2$, $\alpha_i=-\omega_{i-1}+2\omega_{i}-\omega_{i+1}$ for $2\leq i\leq l-3$, $\alpha_{l-2}=-\omega_{l-3}+2\omega_{l-2}-\omega_{l-1}-\omega_{l}$, $\alpha_{l-1}=-\omega_{l-2}+2\omega_{l-1}$, and
$\alpha_{l}=-\omega_{l-2}+2\omega_{l}$.
All the above information can be checked in Section 13.4 \cite{Ca}.
\begin{proposition}\label{soomegasio}
$s_i\left(\omega_j\right)=\omega_j$ for $i\neq j$. $s_1\left(\omega_1\right)=-\omega_1+\omega_2$, $s_2\left(\omega_2\right)=\omega_1-\omega_2+\omega_3,\ldots,$ $s_{l-3}\left(\omega_{l-3}\right)=\omega_{l-4}-\omega_{l-3}+\omega_{l-2}$, $s_{l-2}\left(\omega_{l-2}\right)=\omega_{l-3}-\omega_{l-2}+\omega_{l-1}+\omega_{l}$, $s_{l-1}\left(\omega_{l-1}\right)=\omega_{l-2}-\omega_{l-1}$ and $s_{l}\left(\omega_{l}\right)=\omega_{l-2}-\omega_{l}$.
\end{proposition}
\begin{proof}
It is routine to check. For example, $s_{l-3}\left(\omega_{l-3}\right)=s_{l-3}\left(\mu_1+\ldots+\mu_{l-3}\right)=\mu_1+\ldots+\mu_{l-4}+\mu_{l-2}=\omega_4+\mu_{l-2}=\omega_{l-4}-\omega_{l-3}+\omega_{l-2}$; $s_{l-1}\left(\omega_{l-1}\right)=\frac{1}{2}s_{l-1}\left(\mu_1+\ldots+\mu_{l-2}+\mu_{l-1}-\mu_{l}\right)=\frac{1}{2}\left(\mu_1+\ldots+\mu_{l-3}+\mu_{l-2}+\mu_{l}-\mu_{l-1}\right)=\frac{1}{2}\left(\omega_{l-2}-\omega_{l-1}+\omega_{l}+\omega_{l-2}-\omega_{l-1}-\omega_{l}\right)
=\omega_{l-2}-\omega_{l-1}$.
\end{proof}
Let $w_0$ be the longest element of the Weyl group of $\nyo$. One reduced expression of $w_0$ is given by $$\left(s_l\right)\left(s_{l-1}\right)\left(s_{l-2}s_ls_{l-1}s_{l-2}\right)\left(s_{l-3}s_{l-2}s_ls_{l-1}s_{l-2}s_{l-3}\right)\ldots \left(s_{1}\ldots s_{l-2}s_ls_{l-1}s_{l-2}\ldots s_{1}\right).$$ 
Let $\overline{k}=0$ if $k$ is even and let $\overline{k}=1$ if $k$ is odd.
Define for $1\leq i\leq l-2$,
\begin{align*}
  w_i &= \left(s_{i}\ldots s_{l-2}s_ls_{l-1}\ldots s_{i}\right)\left(s_{i-1}s_{i}\ldots s_{l-2}s_ls_{l-1}\ldots s_{i}\right)\ldots \\
   &\left(s_{2}\ldots s_{l-2}s_{l}s_{l-1}\ldots s_{i}\right)\left(s_{1}\ldots s_{l-2}s_ls_{l-1}\ldots s_{i+1}s_{i}\right);
\end{align*}
$$w_{l-1}=s_{l-\overline{l-1}}\left(s_{l-2}s_{l-\overline{l-2}}\right)\left(s_{l-3}s_{l-2}s_{l-\overline{l-3}}\right)\ldots \left(s_{2}\ldots s_{l-2}s_{l}\right)\left(s_{1}\ldots s_{l-2}s_{l-1}\right);$$
$$w_{l}=s_{l-\overline{l}}\left(s_{l-2}s_{l-\overline{l-1}}\right)\left(s_{l-3}s_{l-2}s_{l-\overline{l-2}}\right)\ldots \left(s_{2}\ldots s_{l-2}s_{l-1}\right)\left(s_{1}\ldots s_{l-2}s_{l}\right).$$

According to the expression of $w_j$ for a fixed $j\in I$, we define $\sigma_k$ to be a product of the last $k$ terms in $w_j$ and keep the same orders as in $w_j$.  There exists $k'\in\{1,2,\ldots, l\}$ such that $\sigma_{k+1}=s_{k'}\sigma_k$.
Denote $v_{\sigma_k(\omega_i)}$ be a weight vector of weight space of weight $\sigma_k(\omega_i)$.
Since the weight space of $V_a(\omega_i)$ of weight $\omega_i$ is 1-dimensional, the weight space of weight $\sigma_k(\omega_i)$ is 1-dimensional, and then $v_{\sigma_k(\omega_i)}$ is unique, up to scalar.

\begin{proposition}\label{rileq2}
Write $\sigma_k(\omega_i)$ as $r_{k'}\omega_{k'}+\sum\limits_{j\neq k'} r_j\omega_j$. Then $r_{k'}\in\{1,2\}$.
\end{proposition}
\begin{proof}
It is enough to compute $w_0\left(\omega_k\right)$ for $1\leq k\leq l$; see Remark \ref{ycdr1} for the reason. In what follows, for example, $\omega_{l-3}-\omega_{l-2}\xlongrightarrow{s_2\ldots s_{l-4}s_{l-3}}\omega_{1}-\omega_{2}$ means $$s_2\ldots s_{l-4}s_{l-3}\left(\omega_{l-3}-\omega_{l-2}\right)=s_2\Big(\ldots s_{l-4}\big(s_{l-3}\left(\omega_{l-3}-\omega_{l-2}\right)\big)\ldots\Big) =\omega_{1}-\omega_{2}.$$ 
This fact can be easily proved by induction on $l$ and using Proposition \ref{soomegasio}. \

Case 1: $i=1$.
\begin{flushleft}
$\omega_1\xlongrightarrow{s_1}-\omega_1+\omega_2\xlongrightarrow{s_2}-\omega_2+\omega_3\xlongrightarrow{s_3}-\omega_3+\omega_4\xlongrightarrow{s_4} \ldots\xlongrightarrow{s_{l-3}}-\omega_{l-3}+\omega_{l-2}\xlongrightarrow{s_{l-2}}-\omega_{l-2}+\omega_{l-1}+\omega_{l}\xlongrightarrow{s_{l-1}}-\omega_{l-1}+\omega_{l}
\xlongrightarrow{s_{l}}\omega_{l-2}-\omega_{l-1}-\omega_{l}\xlongrightarrow{s_{l-2}}\omega_{l-3}-\omega_{l-2}\xlongrightarrow{s_2\ldots s_{l-4}s_{l-3}}\omega_{1}-\omega_{2}\xlongrightarrow{s_{1}}-\omega_{1}\xlongrightarrow{\left(s_l\right)\left(s_{l-1}\right)\left(s_{l-2}s_ls_{l-1}s_{l-2}\right)\left(s_{l-3}s_{l-2}s_ls_{l-1}s_{l-2}s_{l-3}\right)\ldots \left(s_{2}\ldots s_{l-2}s_ls_{l-1}s_{l-2}\ldots s_{2}\right)}-\omega_{1}$.
\end{flushleft}

Case 2: $2\leq i\leq l-2$.
\begin{flushleft}
$\omega_i\xlongrightarrow{s_{i-1}\ldots s_2s_1}\omega_i\xlongrightarrow{s_i}\omega_{i-1}-\omega_{i}+\omega_{i+1}\xlongrightarrow{s_{l-3}s_{l-4}\ldots s_{i+1}}\omega_{i-1}-\omega_{l-3}+\omega_{l-2}\xlongrightarrow{s_{l-2}}\omega_{i-1}-\omega_{l-2}+\omega_{l-1}+\omega_{l}\xlongrightarrow{s_{l-1}} \omega_{i-1}-\omega_{l-1}+\omega_{l}\xlongrightarrow{s_{l}}\omega_{i-1}+\omega_{l-2}-\omega_{l-1}-\omega_{l}\xlongrightarrow{s_{l-2}}\omega_{i-1}+
\omega_{l-3}-\omega_{l-2}\xlongrightarrow{s_{l-3}}\omega_{i-1}+\omega_{l-4}-\omega_{l-3}
\xlongrightarrow{s_{i+1}s_{i+2}\ldots s_{l-4}}\omega_{i-1}+\omega_{i}-\omega_{i+1}\xlongrightarrow{s_{i}}2\omega_{i-1}-\omega_{i}
\xlongrightarrow{s_{i-1}}2\omega_{i-2}-2\omega_{i-1}+\omega_{i}\xlongrightarrow{s_2s_3\ldots s_{i-2}}2\omega_{1}-2\omega_{2}+\omega_{i}
\xlongrightarrow{s_1}-2\omega_1+\omega_{i}\xlongrightarrow{s_{2}s_{3}\ldots s_{l-2}s_ls_{l-1}\ldots s_2}-2\omega_2+\omega_{i}
\xlongrightarrow{s_{3}s_{4}\ldots s_{l-2}s_ls_{l-1}\ldots s_3}-2\omega_3+\omega_{i}
\xlongrightarrow{\left(s_{i}s_{i+1}\ldots s_{l-2}s_ls_{l-1}\ldots s_{i}\right)\left(s_{i-1}s_{i}\ldots s_{l-2}s_ls_{l-1}\ldots s_{i-1}\right)\ldots\left(s_{4}s_5\ldots s_{l-2}s_ls_{l-1}\ldots s_4\right)}-\omega_{i}\xlongrightarrow{s_ls_{l-1}\ldots\left(s_{i+1}s_{i+2}\ldots s_{l-2}s_ls_{l-1}\ldots s_{i+1}\right)}-\omega_{i}$.
\end{flushleft}

Case 3: $i=l-1$.
\begin{flushleft}
$\omega_{l-1}\xlongrightarrow{s_{l-2}\ldots s_1}\omega_{l-1}\xlongrightarrow{s_{l-1}}\omega_{l-2}-\omega_{l-1}\xlongrightarrow{s_{l}}\omega_{l-2}-\omega_{l-1}\xlongrightarrow{s_{l-2}}\omega_{l-3}-\omega_{l-2}+\omega_{l}
\xlongrightarrow{s_{l-3}}\omega_{l-4}-\omega_{l-3}+\omega_{l}\xlongrightarrow{s_2s_3\ldots s_{l-4}}\omega_{1}-\omega_{2}+\omega_{l}\xlongrightarrow{s_1}-\omega_{1}+\omega_{l}\xlongrightarrow{s_{l-1}\ldots s_1}-\omega_{1}+\omega_{l}\xlongrightarrow{s_{l}} -\omega_{1}+\omega_{l-2}-\omega_{l}\xlongrightarrow{s_{l-2}}-\omega_{1}+\omega_{l-3}+\omega_{l-1}
\xlongrightarrow{s_{2}s_3\ldots s_{l-3}}-\omega_2+\omega_{l-1}
\xlongrightarrow{s_{3}s_4\ldots s_{l-2}s_ls_{l-1}\ldots s_3}-\omega_3+\omega_{l}
\xlongrightarrow{s_{4}s_5\ldots s_{l-2}s_{l}s_{l-1}\ldots s_4}-\omega_4+\omega_{l-1}\xlongrightarrow{\left(s_l\right)\left(s_{l-1}\right)\left(s_{l-2}s_ls_{l-1}s_{l-2}\right)\ldots \left(s_{5}s_{6}\ldots s_{l-2}s_{l}s_{l-1}\ldots s_5\right)}-\omega_{l-\overline{l-1}}$.
\end{flushleft}

Case 4:  $i=l$.
\begin{flushleft}
$\omega_{l}\xlongrightarrow{s_{l-1}\ldots s_{1}}\omega_{l}\xlongrightarrow{s_{l}}\omega_{l-2}-\omega_{l}\xlongrightarrow{s_{l-2}}\omega_{l-3}-\omega_{l-2}+\omega_{l-1}
\xlongrightarrow{s_2s_3\ldots s_{l-3}}\omega_{1}-\omega_{2}+\omega_{l-1}\xlongrightarrow{s_1}-\omega_{1}+\omega_{l-1}
\xlongrightarrow{s_{l-2}\ldots s_{1}}-\omega_{1}+\omega_{l-1}
\xlongrightarrow{s_{l-1}}-\omega_{1}+\omega_{l-2}-\omega_{l-1}
\xlongrightarrow{s_{l}}-\omega_{1}+\omega_{l-2}-\omega_{l-1}
\xlongrightarrow{s_{l-2}}-\omega_{1}+\omega_{l-3}-\omega_{l-2}+\omega_{l}\xlongrightarrow{s_{2}s_3\ldots s_{l-3}}-\omega_2+\omega_{l}
\xlongrightarrow{\left(s_l\right)\left(s_{l-1}\right)\left(s_{l-2}s_ls_{l-1}s_{l-2}\right)\ldots \left(s_{3}s_4\ldots s_{l-2}s_ls_{l-1}\ldots s_3\right)}=-\omega_{l-\overline{l}}$.
\end{flushleft}
This proposition becomes clear by the above calculations. \end{proof}
\begin{remark}\label{ycdr1}
Delete all $s_m$ in the above calculations $w_0(\omega_i)$ such that $\omega\xlongrightarrow{s_m}\omega$, and keep the rest in order. The product of what is left above each arrow, from the back to the front, is $w_i$. The purpose to define $w_i\left(i\in I\right)$ is to guarantee $v_{\sigma_k(\omega_i)}\neq v_{\sigma_{k+1}(\omega_i)}$.
\end{remark}



\begin{proposition}\label{v+=gv-}
Let $v_{1}^{+}$ and $v_{1}^{-}$ be highest and lowest weight vectors in the fundamental representation $V_{a}(\omega_i)$ of $\yo$. There exists $y\in U\big(\nyo\big)$ such that $v_{1}^{-}=y.v_{1}^{+}$.
\end{proposition}
\begin{proof}
By either Remark \ref{rofry} or Proposition \ref{dcofbdl}, $V_a(\omega_i)=L(\omega_i)\oplus M$ as a $\mathfrak{so}\left(2l,\C\right)$-module, and both $v_{1}^{+}$ and $v_{1}^{-}$ are in $L(\omega_i)$. 
It follows from Proposition \ref{rileq2} and the formula $s_r\left(\omega\right)=\omega-\frac{2\left(\omega,\alpha_r\right)}{\left(\alpha_r, \alpha_r\right)}\alpha_r$ that the corresponding weight vector $v_{\sigma_{j+1}(\omega_i)}$ of weight $\sigma_{j+1}(\omega_i)$ can be defined as
$$v_{\sigma_{j+1}(\omega_i)}=\left(x_{j',0}^{-}\right)^{r_{j'}}v_{\sigma_j\omega_i}.$$

\noindent Case 1: $i=1,\ldots, l-2$.
\begin{align*}
v^{-}_1 &= \Big(x_{i,0}^{-}\ldots x_{l-2,0}^{-}x_{l,0}^{-}\ldots x_{i,0}^{-}\Big)\Big(\left(x_{i-1,0}^{-}\right)^2x_{i,0}^{-}\ldots x_{l-2,0}^{-}x_{l,0}^{-}\ldots x_{i,0}^{-}\Big)\ldots \\
 &\qquad \Big(\left(x_{1,0}^{-}\right)^2\ldots \left(x_{i-1,0}^{-}\right)^2x_{i,0}^{-}\ldots x_{l-2,0}^{-}x_{l,0}^{-}\ldots x_{i,0}^{-}\Big)v^{+}_1.
\end{align*}
Case 2: $i=l-1$.
\begin{align*}
  v^{-}_1 &= x_{l-\overline{l-1},0}^{-}\left(x_{l-2,0}^{-}x_{l-\overline{l-2},0}^{-}\right)\left(x_{l-3,0}^{-}x_{l-2,0}^{-} x_{l-\overline{l-3},0}^{-}\right)\ldots \\
   &\qquad \left(x_{2,0}^{-}\ldots  x_{l-2,0}^{-}x_{l,0}^{-}\right)\left(x_{1,0}^{-}\ldots  x_{l-2,0}^{-}x_{l-1,0}^{-}\right)v^{+}_1.
\end{align*}
Case 3: $i=l$.
\begin{align*}
  v^{-}_1 &= x_{l-\overline{l},0}^{-}\left(x_{l-2,0}^{-}x_{l-\overline{l-1},0}^{-}\right)\left(x_{l-3,0}^{-}x_{l-2,0}^{-} x_{l-\overline{l-2},0}^{-}\right)\ldots \\
   &\qquad \left(x_{2,0}^{-}\ldots  x_{l-2,0}^{-}x_{l-1,0}^{-}\right)\left(x_{1,0}^{-}\ldots  x_{l-2,0}^{-}x_{l,0}^{-}\right)v^{+}_1.
\end{align*}
\end{proof}
\section{On a lower bound for the dimension of the local Weyl module $W(\pi)$}

In this section, we construct a highest weight module which is a tensor product of fundamental representations of $\yo$. By the maximality of the local Weyl modules, a lower bound for the local Weyl module $W(\pi)$ of $\yo$ can be obtained.

\begin{proposition}\label{mtoysl2l}
Let $L=V_{a_1}(\omega_{b_1})\otimes V_{a_2}(\omega_{b_2})\otimes\ldots\otimes V_{a_k}(\omega_{b_k})$, where $b_{i}\in I$.  If $\operatorname{Re}\left(a_1\right)\geq \operatorname{Re}\left(a_2\right)\geq \ldots \geq \operatorname{Re}\left(a_k\right)$, then $L$ is a highest weight representation of $\yo$.
\end{proposition}
\begin{proof}


Denote by $Y_i$ the subalgebra of $\yo$,  generated by $x_{i,r}^{\pm}$, and $h_{i,r}$, where $r\in \mathbb{Z}_{\geq 0}$. $Y_i\cong Y\left(\mathfrak{sl}_2\right)$. Let $v_m^{+}$ be the highest weight vectors of $V_{a_m}\left(\omega_{b_m}\right)$ ($1\leq m\leq k$), and let $v_1^{-}$ be the lowest weight vector of $V_{a_1}\left(\omega_{b_1}\right)$.

We prove this proposition by induction on $k$. Without loss of generality, we may assume that $k\geq 2$, $V_{a_2}(\omega_{b_2})\otimes V_{a_3}(\omega_{b_3})\otimes\ldots\otimes V_{a_k}(\omega_{b_k})$ is a highest weight representation of $\yo$, and $v^{+}=v_2^{+}\otimes \ldots \otimes v_k^{+}$ is a highest weight vector. To show that $L$ is a highest weight representation, it follows from Corollary \ref{v-w+gvtw} that it suffices to show that $$v^{-}_{1}\otimes v^{+}\in \yo\left(v^{+}_1\otimes v^{+}\right).$$ We divide the proof into the following steps.

Step 1: $\sigma_i^{-1}\left(\alpha_{i'}\right)\in \Delta^{+}$.

Proof: It is routine to check.


Step 2: $Y_{i'}\left(v_{\sigma_i\left(\omega_{b_1}\right)}\right)$ is a highest weight module of $Y_{i'}$.

Proof: Since the weight $\sigma_i\left(\omega_{b_1}\right)$ is on the Weyl group orbit of the highest weight and the representation $V_{a_1}(\omega_{b_1})$ is finite-dimensional, the weight space of weight $\sigma_i\left(\omega_{b_1}\right)$ is 1-dimensional. The elements $h_{j,s}$ form a commutative subalgebra, so $v_{\sigma_i\left(\omega_{b_1}\right)}$ is an eigenvector of $h_{i',r}$. Therefore we only have to show that $v_{\sigma_i\left(\omega_{b_1}\right)}$ is a maximal vector.
Suppose to the contrary that $x_{i',k}^{+}v_{\sigma_i\omega_{b_i}}\neq 0$.  Then $x_{i',k}^{+}v_{\sigma_i\omega_{b_i}}$ is a weight vector of weight $\sigma_i\left(\omega_{b_1}\right)+\alpha_{i'}$.
Thus $\omega_{b_1} + \sigma_i^{-1}(\alpha_{i'})$ is also a weight of $V_{a_1}(\omega_{b_1})$ and, by Step 1, $\omega_{b_1}$ precedes $\omega_{b_1} + \sigma_i^{-1}(\alpha_{i'})$: this contradicts the fact that $\omega_{b_1}$ is the highest weight of $V_{a_1}(\omega_{b_1})$.



Step 3: 
As a highest weight $Y_{i'}$-module, $Y_{i'}\left(v_{\sigma_i\left(\omega_{b_1}\right)}\right)$ has highest weight $\frac{P\left(u+1\right)}{P\left(u\right)}$ for some monic polynomial $P(u)$.

Proof: It follows from the representation theory of $\ysl$.

Step 4: $P\left(u\right)$ has of degree 1 or 2.

Proof: The degree of $P$ equals the eigenvalue of $h_{i',0}$ on $v_{\sigma_i\left(\omega_{b_1}\right)}$. Since $$h_{i',0} v_{\sigma_i\left(\omega_{b_1}\right)}=\Big(\sigma_{i}\left(\omega_{b_1}\right)\left(h_{i',0}\right)\Big)v_{\sigma_i\left(\omega_{b_1}\right)},$$ it follows from Proposition \ref{rileq2} that $\sigma_{i}\left(\omega_{b_1}\right)\left(h_{i',0}\right)=1$ or 2. Therefore the degree of $P\left(u\right)$ equals 1 or 2.


Step 5: If $\operatorname{Deg}\Big(P\left(u\right)\Big)=1$, then $Y_{i'}\left(v_{\sigma_i\left(\omega_{b_1}\right)}\right)$ is 2-dimensional and isomorphic to $W_1\left(a\right)$;  If $\operatorname{Deg}\Big(P\left(u\right)\Big)=2$, then $Y_{i'}\left(v_{\sigma_i\left(\omega_{b_1}\right)}\right)$ is either 3-dimensional or 4-dimensional, and is isomorphic to $W_2\left(a\right)$ or $W_1\left(b\right)\otimes W_1\left(a\right)$ $\big(\operatorname{Re}(b)>\operatorname{Re}(a)\big)$, respectively. Moreover, if $i\geq 1$, then $\operatorname{Re}\left(a\right)>\operatorname{Re}\left(a_1\right)$. The values of both $b$ and $a$ will be explicitly computed in the next section.

Proof: Let us assume for the moment that this is done (the proof will be given in the next section), and let us proceed to the next step.

Step 6:  $Y_{i'}\left(v_{\sigma_i\left(\omega_{b_1}\right)}\otimes v_2^{+}\otimes\ldots \otimes v_k^{+}\right)=Y_{i'}\left(v_{\sigma_i\left(\omega_{b_1}\right)}\right)\otimes Y_{i'}\left(v_2^{+}\right)\otimes\ldots\otimes Y_{i'}\left(v_k^{+}\right)$.

Proof: Let $W$ be one of the modules $W_1\left(a\right)$, $W_2\left(a\right)$ or $W_1\left(b\right)\otimes W_1\left(a\right)$. $Y_{i'}\left(v_m^{+}\right)$ is nontrivial if and only if $b_m=i'$; moreover, $Y_{i'}\left(v_m^{+}\right)\cong W_1\left(a_m\right)$.  Suppose in $\{b_1,\ldots,b_k\}$ that $b_{j_1}=\ldots=b_{j_{m}}=i'$ with $j_1<\ldots<j_{m}$; moreover, if $s\notin\{j_1,\ldots, j_m\}$, then $b_s\neq i'$.
Note in Step 5 that $Y_{i'}\left(v_{\sigma_i\left(\omega_{b_1}\right)}\right)\cong W$. Thus
\begin{align*}
   Y_{i'}&\left(v_{\sigma_i\left(\omega_{b_1}\right)}\right)\otimes Y_{i'}\left(v_2^{+}\right)\otimes\ldots\otimes Y_{i'}\left(v_k^{+}\right) \\
   &\cong \begin{cases}
 W\otimes W_1\left(a_{j_1}\right)\otimes \ldots\otimes W_{1}\left(a_{j_m}\right)\qquad \mathrm{if}\qquad b_1\neq i' \\
 W\otimes W_1\left(a_{j_2}\right)\otimes \ldots\otimes W_{1}\left(a_{j_m}\right) \qquad \mathrm{if}\qquad b_1=i'.
\end{cases}
\end{align*}

Since $\operatorname{Re}\left(a\right)\geq \operatorname{Re}\left(a_1\right)\geq \operatorname{Re}\left(a_{j_1}\right)\geq \ldots\geq \operatorname{Re}\left(a_{j_m}\right)$, it follows from either Corollary \ref{tpihw} or Corollary \ref{tpihw2} that $Y_{i'}\left(v_{\sigma_i\left(\omega_{b_1}\right)}\right)\otimes Y_{i'}\left(v_2^{+}\right)\otimes\ldots\otimes Y_{i'}\left(v_k^{+}\right)$ is a highest weight module with highest weight vector $v_{\sigma_i\left(\omega_{b_1}\right)}\otimes v_2^{+}\otimes\ldots \otimes v_k^{+}$. Therefore $$Y_{i'}\left(v_{\sigma_i\left(\omega_{b_1}\right)}\otimes v_2^{+}\otimes\ldots \otimes v_k^{+}\right)\supseteq Y_{i'}\left(v_{\sigma_i\left(\omega_{b_1}\right)}\right)\otimes Y_{i'}\left(v_2^{+}\right)\otimes\ldots\otimes Y_{i'}\left(v_k^{+}\right).$$

It easy to see by the coproduct of Yangians  and Proposition \ref{c1dgd2d} that
$$Y_{i'}\left(v_{\sigma_i\left(\omega_{b_1}\right)}\otimes v_2^{+}\otimes\ldots \otimes v_k^{+}\right)\subseteq Y_{i'}\left(v_{\sigma_i\left(\omega_{b_1}\right)}\right)\otimes Y_{i'}\left(v_2^{+}\right)\otimes\ldots\otimes Y_{i'}\left(v_k^{+}\right).$$ Thus the claim is true.

Step 7: $v_{\sigma_{i+1}(\omega_{b_1})}\otimes v^{+}\in Y_{i'}\left(v_{\sigma_i\left(\omega_{b_1}\right)}\otimes v^{+}\right)$.

Proof: $v_{\sigma_{i+1}(\omega_{b_1})}\otimes v^{+}\in Y_{i'}\left(v_{\sigma_{i}\omega_{b_1}}\right)\otimes Y_{i'}\left(v^{+}\right)=Y_{i'}\left(v_{\sigma_i\left(\omega_{b_1}\right)}\otimes v^{+}\right)$.

Step 8: $v_1^{-}\otimes v^{+}\in \yo\left(v_1^{+}\otimes v^{+}\right)$.

Proof: It follows from Step 7 immediately by induction on the subscript of $\sigma_i$.

It follows from Step 8 and Corollary \ref{v-w+gvtw} that $L=\yo\left(v_1^{+}\otimes v^{+}\right)$.
\end{proof}

\begin{remark}\label{yocap1}
In what follows, we record the values of root(s) of the associated polynomial of $Y_{i'}\left(v_{\sigma_i\left(\omega_{b_1}\right)}\right)$ in Step 5 of Proposition \ref{mtoysl2l},  which will be calculated explicitly in the next section.  In next paragraph,  $a_1\xlongrightarrow{x_{a,0}^{-}} \left(b_2,a_2\right)\xlongrightarrow{\left(x_{b,0}^{-}\right)^2} \left(b_3,a_3\right)\xlongrightarrow{\left(x_{c,0}^{-}\right)^2}$ means that $Y_{a}\left(v\right)=W_1\left(a_1\right)$, $Y_{b}\Big(x_{a,0}^{-}v\Big)\cong W$, $Y_{c}\Big(\left(x_{b,0}^{-}\right)^2x_{a,0}^{-}v\Big)\cong V$, and the lowest weight vector reached if there is no number after the arrow, where $W$ is at least 3 dimensional and is a quotient of  $W_1\left(b_2\right)\otimes W_1\left(a_2\right)$; $V$ is at least 3 dimensional and is a quotient of  $W_1\left(b_3\right)\otimes W_1\left(a_3\right)$. 
The second (last) parenthesis means the parenthesis in Proposition \ref{v+=gv-}, counting from the back to the front.

When $b_1=1,2,\ldots, l-2$,\\
$a_1\xlongrightarrow{x_{b_1,0}^{-}}a_1+\frac{1}{2}\xlongrightarrow{x_{b_1+1,0}^{-}}\ldots\xlongrightarrow{x_{l-2,0}^{-}}a_1+\frac{l-1-b_1}{2}
\xlongrightarrow{x_{l-1,0}^{-}} a_1+\frac{l-1-b_1}{2}\xlongrightarrow{x_{l,0}^{-}}\\ a_1+\frac{l-b_1}{2}
\xlongrightarrow{x_{l-2,0}^{-}} a_1+\frac{l-b_1+1}{2}\xlongrightarrow{x_{l-3,0}^{-}}\ldots\xlongrightarrow{x_{b_1+1,0}^{-}} a_1+l-b_1-1\xlongrightarrow{x_{b_1,0}^{-}}\\ \left(a_1+l-b_1-\frac{1}{2}, a_1+\frac{1}{2}\right)\xlongrightarrow{\left(x_{b_1-1,0}^{-}\right)^2} \left(a_1+l-b_1,a_1+1\right)\xlongrightarrow{\left(x_{b_1-2,0}^{-}\right)^2}\ldots \xlongrightarrow{\left(x_{2,0}^{-}\right)^2} \left(a_1+l-b_1-1+\frac{b_1-1}{2}=a_1+\frac{2l-b_1-3}{2}, a_1+\frac{b_1-1}{2}\right)\xlongrightarrow{\left(x_{1,0}^{-}\right)^2}$\\
(The second parenthesis) $\left(a_1+1\right)\xlongrightarrow{x_{b_1,0}^{-}}\left(a_1+1\right)+\frac{1}{2}\xlongrightarrow{x_{b_1+1,0}^{-}}\left(a_1+1\right)+1\xlongrightarrow{x_{b_1+2,0}^{-}}\ldots \xlongrightarrow{\left(x_{3,0}^{-}\right)^2}\Big(\left(a_1+1\right)+\frac{2l-b_1-4}{2}, \left(a_1+1\right)+\frac{b_1-2}{2}\Big)\xlongrightarrow{\left(x_{2,0}^{-}\right)^2}$


$\ldots$\\
(The last parenthesis)
$\left(a_1+b_1-1\right)\xlongrightarrow{x_{b_1,0}^{-}}\left(a_1+b_1-1\right)+\frac{1}{2}\xlongrightarrow{x_{b_1+1,0}^{-}}\\\left(a_1+b_1-1\right)+1\xlongrightarrow{x_{b_1+2,0}^{-}}\ldots \xlongrightarrow{x_{b_1+1,0}^{-}} \left(a_1+b_1-1\right)+l-b_1-1=a_1+l-2\xlongrightarrow{x_{b_1,0}^{-}}\text{the lowest weight vector reached}$.

When $b_1=l-1$,\\
$a_1\xlongrightarrow{x_{l-1,0}^{-}} a_1+\frac{1}{2}
\xlongrightarrow{x_{l-2,0}^{-}} a_1+1\xlongrightarrow{x_{l-3,0}^{-}}\ldots\xlongrightarrow{x_{2,0}^{-}}a_1+\frac{l-2}{2}\xlongrightarrow{x_{1,0}^{-}}$\\
(The second parenthesis)$\left(a_1+1\right)\xlongrightarrow{x_{l,0}^{-}}\left(a_1+1\right)+\frac{1}{2}\xlongrightarrow{x_{l-2,0}^{-}}\left(a_1+1\right)+1 \xlongrightarrow{x_{l-3,0}^{-}}\ldots \xlongrightarrow{x_{3,0}^{-}}\left(a_1+1\right)+\frac{l-3}{2}=a_1+\frac{l-1}{2}\xlongrightarrow{x_{2,0}^{-}}$\\
(The third parenthesis)$\left(a_1+2\right)\xlongrightarrow{x_{l-1,0}^{-}}\left(a_1+2\right)+\frac{1}{2}\xlongrightarrow{x_{l-2,0}^{-}}\left(a_1+2\right)+1 \xlongrightarrow{x_{l-3,0}^{-}}\ldots \xlongrightarrow{x_{4,0}^{-}}\left(a_1+2\right)+\frac{l-4}{2}=a_1+\frac{l}{2}\xlongrightarrow{x_{3,0}^{-}}$

$\ldots$\\
(The last two parentheses)$\left(a_1+l-3\right)\xlongrightarrow{x_{l-\overline{l-2},0}^{-}}\left(a_1+l-3\right)+\frac{1}{2}\xlongrightarrow{x_{l-2,0}^{-}}\\\left(a_1+l-3\right)+1 \xlongrightarrow{x_{l-\overline{l-1},0}^{-}}\text{the lowest weight vector reached}$.
\end{remark}

%

%

A broader condition for the cyclicity of the module $L$ will be given in Theorem \ref{mt2c4}. We first show the following lemma. Since the proof is similar to the proof of Lemma \ref{C3l317ss}, we omit the proof.
\begin{lemma}\label{ycdlsbij}
$V_{a_m}\left(\omega_{b_m}\right)\otimes V_{a_n}\left(\omega_{b_n}\right)$ is a highest weight representation if $a_n-a_m\notin S\left(b_m,b_n\right)$, where the set $S\left(b_m,b_n\right)$ is defined as follows:
\begin{enumerate}
  \item $S\left(b_m,b_n\right)=\left\{\frac{|b_m-b_n|}{2}+1+r, l+r-\frac{b_m+b_n}{2}|0\leq r<\text{min}\{b_m, b_n\}\right\}$, where\newline $1\leq b_m, b_n\leq l-2$;
  \item  $S\left(l-1,b_n\right)=S\left(l,b_n\right)=S\left(b_n,l-1\right)=S\left(b_n,l\right)=\\\qquad\qquad\qquad\qquad \left\{\frac{l-1-b_n}{2}+1+r|0\leq r<b_n, 1\leq b_n\leq l-2\right\}$;
  \item $S\left(l-1,l\right)=S\left(l,l-1\right)=\left\{2,4,\ldots,l-2+\overline{l}\right\}$;
   \item $S\left(l-1,l-1\right)=S\left(l,l\right)=\left\{1,3,\ldots,l-1-\overline{l}\right\}$.
\end{enumerate}
\end{lemma}

Note that $S(b_i, b_j)=S(b_j, b_i)$. Similar to the proof of Theorem \ref{3mt2icl} that
\begin{theorem}\label{mt2c4}
Let $L=V_{a_1}(\omega_{b_1})\otimes V_{a_2}(\omega_{b_2})\otimes\ldots\otimes V_{a_k}(\omega_{b_k})$, and $S(b_i, b_j)$ be defined as above.
\begin{enumerate}
  \item If $a_j-a_i\notin S(b_i, b_j)$ for $1\leq i<j\leq k$, then $L$ is a highest weight representation of $\yo$.
  \item If $a_j-a_i\notin S(b_i, b_j)$ for $1\leq i\neq j\leq k$, then $L$ is an irreducible representation of $\yo$.
\end{enumerate}
\end{theorem}
\begin{remark}
We note that in Theorem 7.2 \cite{ChPr8},  the authors gave a sufficient and necessary condition for the cyclicity and irreducibility of $V_{a_m}\left(\omega_{b_m}\right)\otimes V_{a_n}\left(\omega_{b_n}\right)$, where $b_m\leq b_n$. They proved that $V_{a_m}\left(\omega_{b_m}\right)\otimes V_{a_n}\left(\omega_{b_n}\right)$ is reducible if and only if $\pm(a_n-a_m)\notin T(b_m, b_n)$, where $T(b_m,b_n)$ is a finite set of positive rational numbers. We note that $T(b_m,b_n)=S(b_m,b_n)$ for most of the cases but when $1\leq b_m\leq b_n\leq l-2$ and $b_m+b_n\geq l+1$. At this case, $T(b_m,b_n)\subsetneqq S(b_m,b_n)$. That means the methodology developed in this thesis only provides a sufficient condition when $\g=\nyo$. The main reason is that we use  Lemma \ref{beautiful} in our proof and the lemma only provides a sufficient condition for a tensor product of fundamental representations of $\ysl$ to be cyclic. For example, $V_{2}(a-1)\otimes V_1(a)$ is irreducible, but it fails the prescribed condition of Lemma \ref{beautiful}.
\end{remark}
\section{Supplement Step 5 of Proposition \ref{mtoysl2l}}\label{mtoysl2ls}
In this subsection, we will complete Step 5 of Proposition  \ref{mtoysl2l}. The step 5 is shown in 3 cases: $b_1=1$, $b_1=l-1$ or $b_1=l$, and $2\leq b_1\leq l-2$.

We would like to review the following important fact (Corollary \ref{slw1a}): in $W_1\left(a\right)$,
\begin{center}
$h_{i,1}v_1=av_1$, $x_{i,1}^{-}v_1=ax_{i,0}^{-}v_1$ and $h_{i,1}x_{i,0}^{-}v_1=-ax_{i,0}^{-}v_1.\qquad\qquad (\ref{Cor3.3})$
\end{center}
\subsection{Case 1: $b_1=1$.}

In this case, the results are summarized in Proposition \ref{i=1yo2n}. Recall that in this case, $$v^{-}_1=x_{1,0}^{-}\ldots x_{l-2,0}^{-}x_{l,0}^{-}x_{l-1,0}^{-}\ldots x_{1,0}^{-}v^{+}_1.$$


\begin{lemma}\label{so1c1lf1}For $1\leq k\leq l-2$,
$Y_{k}\left(x_{k+1,0}^{-}x_{k,0}^{-}x_{k-1,0}^{-}\ldots x_{1,0}^{-}v^{+}_1\right)$ is a trivial representation of dimension 1.
\end{lemma}
\begin{proof}

It is easy to see that $\sigma_{k+1}^{-1}\left(\alpha_k\right)\in \Delta^{+}$. Similar to Step 2 of Proposition \ref{mtoysl2l} that $x_{k,0}^{+}x_{k+1,0}^{-}x_{k,0}^{-}x_{k-1,0}^{-}\ldots x_{1,0}^{-}v^{+}_1=0$. Thus $Y_{k}\left(x_{k+1,0}^{-}x_{k,0}^{-}x_{k-1,0}^{-}\ldots x_{1,0}^{-}v^{+}_1\right)$ is a highest weight representation. Let $P\left(u\right)$ be the associated polynomial.

By case 1 of Proposition \ref{rileq2},
for $1\leq k\leq l-4$, $$wt\left(x_{k+1,0}^{-}x_{k,0}^{-}x_{k-1,0}^{-}\ldots x_{1,0}^{-}v^{+}_1\right)=-\omega_{k+1}+\omega_{k+2};$$
for $k=l-3$, $$wt\left(x_{l-2,0}^{-}x_{l-3,0}^{-}x_{l-4,0}^{-}\ldots x_{1,0}^{-}v^{+}_1\right)=-\omega_{l-2}+\omega_{l-1}+\omega_{l};$$
for $k=l-2$, $$wt\left(x_{l-1,0}^{-}x_{l-2,0}^{-}x_{l-3,0}^{-}\ldots x_{1,0}^{-}v^{+}_1\right)=-\omega_{l-1}+\omega_{l}.$$
Hence $h_{k,0}x_{k+1,0}^{-}x_{k,0}^{-}x_{k-1,0}^{-}\ldots x_{1,0}^{-}v^{+}_1=0$, and then the degree of $P\left(u\right)$ is 0. Therefore $Y_{k}\left(x_{k+1,0}^{-}x_{k,0}^{-}x_{k-1,0}^{-}\ldots x_{1,0}^{-}v^{+}_1\right)$ is a trivial representation of dimension 1. In particular, $h_{k,1}x_{k+1,0}^{-}x_{k,0}^{-}x_{k-1,0}^{-}\ldots x_{1,0}^{-}v^{+}_1=0$.
\end{proof}

\begin{proposition}\label{i=1yo2n}\
\begin{enumerate}
  \item $Y_{k}\left(x_{k-1,0}^{-}x_{k-2,0}^{-}\ldots x_{1,0}^{-}v^{+}_1\right)\cong W_1\left(a_1+\frac{k-1}{2}\right)$ for $1\leq k\leq l-1$.
  \item $Y_{l}\left(x_{l-1,0}^{-}\ldots x_{2,0}^{-}x_{1,0}^{-}v^{+}_1\right)\cong W_1\left(a_1+\frac{l-2}{2}\right)$.
  \item $Y_{l-2}\left(x_{l,0}^{-}x_{l-1,0}^{-}\ldots x_{2,0}^{-}x_{1,0}^{-}v^{+}_1\right)\cong W_1\left(a_1+\frac{l-1}{2}\right)$.
  \item $Y_{k}\left(x_{k+1,0}^{-}\ldots x_{l-2,0}^{-}x_{l,0}^{-}\ldots x_{1,0}^{-}v^{+}_1\right)\cong W_1\left(a_1+\frac{2l-k-3}{2}\right)$ for $1\leq k\leq l-3$.
\end{enumerate}
\end{proposition}
\begin{proof}
(i). It follows from Proposition \ref{rileq2} that the weight of $x_{k-1,0}^{-}x_{k-2,0}^{-}\ldots x_{1,0}^{-}v^{+}_1$ is $-\omega_{k-1}+\omega_{k}+\delta_{l-1, k}\omega_{l}$. Thus $$h_{k,0}x_{k-1,0}^{-}x_{k-2,0}^{-}\ldots x_{1,0}^{-}v^{+}_1= x_{k-1,0}^{-}x_{k-2,0}^{-}\ldots x_{1,0}^{-}v^{+}_1.$$ Therefore the associated polynomial $P(u)$ to $Y_k(x_{k-1,0}^{-}x_{k-2,0}^{-}\ldots x_{1,0}^{-}v^{+}_1)$ has of degree 1. Suppose $P(u)=u-a$. Note that
$$\frac{P\left(u+1\right)}{P\left(u\right)}=  \frac{u-(a-1)}{u-a}= 1+u^{-1}+au^{-2}+a^2u^{-3}+\ldots.$$
Thus the eigenvalue of $x_{k-1,0}^{-}x_{k-2,0}^{-}\ldots x_{1,0}^{-}v^{+}_1$ under $h_{k,1}$ will tell the value of $a$.

We claim that $a=a_1+\frac{k-1}{2}$, and prove it by using induction on $k$. If $k=1$, then $h_{1,1}v^{+}_1=a_1v^{+}_1$. The claim is true. Suppose the claim is true for $k-1$.
By the induction hypothesis, we have $$h_{k-1,1}x_{k-2,0}^{-}\ldots x_{1,0}^{-}v^{+}_1=\left(a_1+\frac{k-2}{2}\right)x_{k-2,0}^{-}\ldots x_{1,0}^{-}v^{+}_1;$$ and then by $(\ref{Cor3.3})$ $$x_{k-1,1}^{-}x_{k-2,0}^{-}\ldots x_{1,0}^{-}v^{+}_1=\left(a_1+\frac{k-2}{2}\right)x_{k-1,0}^{-}x_{k-2,0}^{-}\ldots x_{1,0}^{-}v^{+}_1.$$
We are about to show that the claim is true for $k$.
\begin{align*}
   h_{k,1}&x_{k-1,0}^{-}x_{k-2,0}^{-}\ldots x_{1,0}^{-}v^{+}_1 \\
   &= [h_{k,1},x_{k-1,0}^{-}]x_{k-2,0}^{-}\ldots x_{1,0}^{-}v^{+}_1\\
 &= \left(x_{k-1,1}^{-}+\frac{1}{2}x_{k-1,0}^{-}+x_{k-1,0}^{-}h_{k,0}\right)x_{k-2,0}^{-}\ldots x_{1,0}^{-}v^{+}_1\\
 &= \left(a_1+\frac{k-2}{2}+\frac{1}{2}\right)x_{k-1,0}^{-}x_{k-2,0}^{-}\ldots x_{1,0}^{-}v^{+}_1\\
 &= \left(a_1+\frac{k-1}{2}\right)x_{k-1,0}^{-}x_{k-2,0}^{-}\ldots x_{1,0}^{-}v^{+}_1.
\end{align*}
Thus the claim is true for $k$. Therefore the claim is true by induction.\\
(ii). It follows from Proposition \ref{rileq2} that
the weight of $x_{l-1,0}^{-}\ldots x_{2,0}^{-}x_{1,0}^{-}v^{+}_1$ is $-\omega_{l-1}+\omega_l$. Thus $$h_{l,0}x_{l-1,0}^{-}\ldots x_{2,0}^{-}x_{1,0}^{-}v^{+}_1= x_{l-1,0}^{-}\ldots x_{2,0}^{-}x_{1,0}^{-}v^{+}_1.$$ Therefore the associated polynomial $P(u)$ to $Y_l(x_{l-1,0}^{-}\ldots x_{2,0}^{-}x_{1,0}^{-}v^{+}_1)$ has of degree 1. Suppose $P(u)=u-b$. Similar to Step 1,  the eigenvalue of $x_{l-1,0}^{-}\ldots x_{2,0}^{-}x_{1,0}^{-}v^{+}_1$ under $h_{l,1}$ will tell the value of $b$.
\begin{align*}
   h_{l,1}&x_{l-1,0}^{-}x_{l-2,0}^{-}\ldots x_{1,0}^{-}v^{+}_1 \\
   &= x_{l-1,0}^{-}[h_{l,1},x_{l-2,0}^{-}]x_{l-3,0}^{-}\ldots x_{1,0}^{-}v^{+}_1\\
 &= x_{l-1,0}^{-}\left(x_{l-2,1}^{-}+\frac{1}{2}x_{l-2,0}^{-}+x_{l-2,0}^{-}h_{l,0}\right)x_{l-3,0}^{-}\ldots x_{1,0}^{-}v^{+}_1\\
 &= \left(a_1+\frac{l-3}{2}+\frac{1}{2}\right)x_{l-1,0}^{-}x_{l-2,0}^{-}\ldots x_{1,0}^{-}v^{+}_1\\
 &= \left(a_1+\frac{l-2}{2}\right)x_{l-1,0}^{-}x_{l-2,0}^{-}\ldots x_{1,0}^{-}v^{+}_1.
\end{align*}

(iii). The proof of this item is similar to the proof of the item (i). We only provide the computation of the eigenvalue of $x_{l,0}^{-}x_{l-1,0}^{-}x_{l-2,0}^{-}\ldots x_{1,0}^{-}v^{+}_1$ under $h_{l-2,1}$.
\begin{align*}
   h_{l-2,1}&x_{l,0}^{-}x_{l-1,0}^{-}x_{l-2,0}^{-}\ldots x_{1,0}^{-}v^{+}_1 \\
   &=[h_{l-2,1}, x_{l,0}^{-}]x_{l-1,0}^{-}x_{l-2,0}^{-}\ldots x_{1,0}^{-}v^{+}_1 +x_{l,0}^{-}h_{l-2,1},x_{l-1,0}^{-}x_{l-2,0}^{-}\ldots x_{1,0}^{-}v^{+}_1 \\
 &=\left(x_{l,1}^{-}+\frac{1}{2}x_{l,0}^{-}+x_{l,0}^{-}h_{l-2,0}\right)x_{l-1,0}^{-}x_{l-2,0}^{-}x_{l-3,0}^{-}\ldots x_{1,0}^{-}v^{+}_1\\
 &= \left(a_1+\frac{l-2}{2}+\frac{1}{2}\right)x_{l,0}^{-}x_{l-1,0}^{-}x_{l-2,0}^{-}\ldots x_{1,0}^{-}v^{+}_1\\
 &= \left(a_1+\frac{l-1}{2}\right)x_{l,0}^{-}x_{l-1,0}^{-}x_{l-2,0}^{-}\ldots x_{1,0}^{-}v^{+}_1.
\end{align*}

(iv). The proof of this item is similar to the proof of the item (i). We only provide the computation of the eigenvalue of $x_{k+1,0}^{-}\ldots x_{l-2,0}^{-}x_{l,0}^{-}\ldots x_{1,0}^{-}v^{+}_1$ under $h_{k,1}$ for $1\leq k\leq l-3$. We would like introduce some notation to simplify the following formula.  We abbreviate $x_{k+1,0}^{-}\ldots x_{l-2,0}^{-}x_{l,0}^{-}\ldots x_{2,0}^{-}x_{1,0}^{-}v^{+}_1$ to $\overline{x_{k+1,0}^{-}\ldots x_{1,0}^{-}}v^{+}_1$. The overline means the expression contains terms $x_{l,0}^{-}$. We claim that$$h_{k,1}\overline{x_{k+1,0}^{-}\ldots x_{1,0}^{-}}v^{+}_1=\left(a_1+\frac{2l-k-3}{2}\right)\overline{x_{k+1,0}^{-}\ldots x_{1,0}^{-}}v^{+}_1.$$

We use induction downward on $k$.

When $k=l-3$.
\begin{align*}
   h_{l-3,1}&x_{l-2,0}^{-}x_{l,0}^{-}x_{l-1,0}^{-}\ldots x_{2,0}^{-}x_{1,0}^{-}v^{+}_1  \\
   &= [h_{l-3,1}x_{l-2,0}^{-}]x_{l,0}^{-}x_{l-1,0}^{-}\ldots x_{2,0}^{-}x_{1,0}^{-}v^{+}_1\\ &+x_{l-2,0}^{-}x_{l,0}^{-}x_{l-1,0}^{-}h_{l-3,1}x_{l-2,0}^{-}\ldots x_{2,0}^{-}x_{1,0}^{-}v^{+}_1\\
   &= \left(x_{l-2,1}^{-}+\frac{1}{2}x_{l-2,0}^{-}+x_{l-2,0}^{-}h_{l-3,0}\right)x_{l,0}^{-}x_{l-1,0}^{-}\ldots x_{2,0}^{-}x_{1,0}^{-}v^{+}_1\\
   &= \left(a_1+\frac{l-1}{2}+\frac{1}{2}\right)x_{l-2,0}^{-}x_{l,0}^{-}x_{l-1,0}^{-}\ldots x_{2,0}^{-}x_{1,0}^{-}v^{+}_1 \\
   &= \left(a_1+\frac{l}{2}\right)x_{l-2,0}^{-}x_{l,0}^{-}x_{l-1,0}^{-}\ldots x_{2,0}^{-}x_{1,0}^{-}v^{+}_1.
\end{align*}
Thus the claim is true for $k=l-3$. Suppose that the claim is true for $k+1$.
\begin{align*}
   h_{k,1}& x_{k+1,0}^{-}\ldots x_{l-2,0}^{-}x_{l,0}^{-}\ldots x_{2,0}^{-}x_{1,0}^{-}v^{+}_1 \\
   &= [h_{k,1}x_{k+1,0}^{-}]\overline{x_{k+2,0}^{-}\ldots x_{1,0}^{-}}v^{+}_1+\overline{x_{k+1,0}^{-}\ldots x_{k+2,0}^{-}}h_{k,1}x_{k+1,0}^{-}x_{k,0}^{-}\ldots x_{1,0}^{-}v^{+}_1\\
   &= \left(x_{k+1,1}^{-}+\frac{1}{2}x_{k+1,0}^{-}+x_{k+1,0}^{-}h_{k,0}\right)\overline{x_{k+2,0}^{-}\ldots x_{1,0}^{-}}v^{+}_1+0\\
   &= \left(a_1+\frac{2l-k-4}{2}+\frac{1}{2}\right)\overline{x_{k+1,0}^{-}\ldots x_{1,0}^{-}}v^{+}_1\\
      &= \left(a_1+\frac{2l-k-3}{2}\right)x_{k+1,0}^{-}\ldots x_{l-2,0}^{-}x_{l,0}^{-}\ldots x_{2,0}^{-}x_{1,0}^{-}v^{+}_1.
\end{align*}
Therefore the claim is true by induction.
\end{proof}
\subsection{Case 2: $b_1=l-1$ or $b_1=l$.}

Let $b_1=l-1$. Recall that in this case
\begin{align*}
  v^{-}_1 &= x_{l-\overline{l-1},0}^{-}\left(x_{l-2,0}^{-}x_{l-\overline{l-2},0}^{-}\right)\left(x_{l-3,0}^{-}x_{l-2,0}^{-} x_{l-\overline{l-3},0}^{-}\right)\ldots \\
   & \left(x_{2,0}^{-}\ldots  x_{l-2,0}^{-}x_{l,0}^{-}\right)\left(x_{1,0}^{-}\ldots  x_{l-2,0}^{-}x_{l-1,0}^{-}\right)v^{+}_1.
\end{align*}

In order to simplify the expressions of calculations, we introduce a notation: $X_{k,m}$, where $1\leq k\leq m\leq l$. When $k\leq m\leq l-1$, $X_{k,m}=x_{k,0}^{-}x_{k+1,0}^{-}\ldots x_{m,0}^{-}$ without break. When $m=l$, $ X_{k,l}=x_{k,0}^{-}x_{k+1,0}^{-}\ldots x_{l-2,0}^{-}x_{l,0}^{-}$.

\begin{proposition} $Y_{i'}\left(v_{\sigma_i\left(\omega_{l-1}\right)}\right)\cong W_1\left(a\right)$, where $\operatorname{Re}\left(a\right)\geq \operatorname{Re}\left(a_1\right)$.
\end{proposition}
\begin{proof}

Recall that $\overline{r}=0$ if $r$ is even; $\overline{r}=1$ if $r$ is odd. We claim
\begin{align}\label{so2ll-1}
 h_{k,1}&X_{k+1,l-\overline{\left(r+1\right)}}X_{r,l-\bar{r}}\ldots X_{2,l}X_{1,l-1}v^{+}_1\nonumber\\
 &=\left(a_1+r+\frac{l-k-1}{2}\right)X_{k+1,l-\overline{\left(r+1\right)}}X_{r,l-\bar{r}}\ldots X_{2,l}X_{1,l-1}v^{+}_1.\nonumber\\
  &=\left(a_1+r+\frac{s}{2}\right)X_{k+1,l-\overline{\left(r+1\right)}}X_{r,l-\bar{r}}\ldots X_{2,l}X_{1,l-1}v^{+}_1,
\end{align}
where $s=l-k-1$ is the number of $x_{i,0}^{-}$ in $X_{k+1,l-\overline{\left(r+1\right)}}$. If $k=l-\overline{\left(r+1\right)}$, then $s=0$. 

We will prove equation (\ref{so2ll-1}) by induction on $r$. For a fixed $r$, induction on $s$.   It is easy to check that the claim (\ref{so2ll-1}) is true for $\left(r,s\right)=\left(0,0\right)$. Suppose that the claim (\ref{so2ll-1}) are true for all pairs $\left(m,n\right)$ such that  $m<r$ and all possible $n$ values.

We first show the claim is true for the pair $\left(r,0\right)$. By induction hypothesis, we have
\begin{align*}
x_{l-2,1}^{-}&x_{l-\overline{r},0}^{-}X_{r-1,l-\overline{r-1}}\ldots X_{2,l}X_{1,l-1}v^{+}_1\\
&=\Big(a_1+(r-1)+\frac{1}{2}\Big)x_{l-2,0}^{-}x_{l-\overline{r},0}^{-}X_{r-1,l-\overline{r-1}}\ldots X_{2,l}X_{1,l-1}v^{+}_1;\\
x_{l-2,1}^{-}&x_{l-\overline{r-1},0}^{-}X_{r-2,l-\overline{r-2}}\ldots X_{2,l}X_{1,l-1}v^{+}_1\\
&=\Big(a_1+(r-2)+\frac{1}{2}\Big)x_{l-2,0}^{-}x_{l-\overline{r-1},0}^{-}X_{r-2,l-\overline{r-2}}\ldots X_{2,l}X_{1,l-1}v^{+}_1;\\
h_{l-\overline{\left(r+1\right)},1}&x_{l-\overline{r-1},0}^{-}X_{r-2,l-\overline{r-2}}\ldots X_{2,l}X_{1,l-1}v^{+}_1\\
&=-\Big(a_1+(r-2)\Big)x_{l-2,0}^{-}x_{l-\overline{r-1},0}^{-}X_{r-2,l-\overline{r-2}}\ldots X_{2,l}X_{1,l-1}v^{+}_1;
\end{align*}
Then
\begin{align*}
X_{r,l-3}&[h_{l-\overline{\left(r+1\right)},1},x_{l-2,0}^{-}]x_{l-\overline{r},0}^{-}X_{r-1,l-\overline{r-1}}\ldots X_{2,l}X_{1,l-1}v^{+}_1\\ &=X_{r,l-3}\big(x_{l-2,1}^{-}+\frac{1}{2}x_{l-2,0}^{-}+x_{l-2,0}^{-}h_{l-\overline{\left(r+1\right)},0}\big)x_{l-\overline{r},0}^{-}X_{r-1,l-\overline{r-1}}\ldots X_{1,l-1}v^{+}_1\\
&=\Big(a_1+(r-1)+\frac{1}{2}+\frac{1}{2}+0\Big)X_{r,l-\overline{r}}X_{r-1,l-\overline{r-1}}\ldots X_{2,l}X_{1,l-1}v^{+}_1.
\end{align*}
Similarly we have
\begin{align*}
X_{r,l-\overline{r}}&\left(X_{r-1,l-3} [h_{l-\overline{\left(r+1\right)},1},x_{l-2,0}^{-}]x_{l-\overline{r-1},0}^{-}\right)X_{r-2,l-\overline{r-2}}\ldots X_{2,l}X_{1,l-1}v^{+}_1\\
&=\Big(a_1+(r-2)+\frac{1}{2}+\frac{1}{2}-1\Big)X_{r,l-\overline{r}}X_{r-1,l-\overline{r-1}}\ldots X_{2,l}X_{1,l-1}v^{+}_1,
\end{align*}
and
\begin{align*}
X_{r,l-\overline{r}}&\left(X_{r-1,l-2}h_{l-\overline{\left(r+1\right)},1}x_{l-\overline{r-1},0}^{-}\right)X_{r-2,l-\overline{r-2}}\ldots X_{2,l}X_{1,l-1}v^{+}_1\\
&=-\Big(a_1+(r-2)\Big)X_{r,l-\overline{r}}X_{r-1,l-\overline{r-1}}\ldots X_{2,l}X_{1,l-1}v^{+}_1.
\end{align*}
Thus
\begin{align*}
h_{l-\overline{\left(r+1\right)},1}&X_{r,l-\overline{r}}X_{r-1,l-\overline{r-1}}X_{r-2,l-\overline{r-2}}\ldots X_{2,l}X_{1,l-1}v^{+}_1 \\
&= \left(X_{r,l-3} [h_{l-\overline{\left(r+1\right)},1},x_{l-2,0}^{-}]x_{l-\overline{r},0}^{-}\right)X_{r-1,l-\overline{r-1}}\ldots X_{2,l}X_{1,l-1}v^{+}_1\\
&+ X_{r,l-\overline{r}}h_{l-\overline{\left(r+1\right)},1}X_{r-1,l-\overline{r-1}}X_{r-2,l-\overline{r-2}}\ldots X_{2,l}X_{1,l-1}v^{+}_1\\
&= \left(X_{r,l-3} [h_{l-\overline{\left(r+1\right)},1},x_{l-2,0}^{-}]x_{l-\overline{r},0}^{-}\right)X_{r-1,l-\overline{r-1}}\ldots X_{2,l}X_{1,l-1}v^{+}_1\\
&+ X_{r,l-\overline{r}}\left(X_{r-1,l-3} [h_{l-\overline{\left(r+1\right)},1},x_{l-2,0}^{-}]x_{l-\overline{r-1},0}^{-}\right)X_{r-2,l-\overline{r-2}}\ldots X_{2,l}X_{1,l-1}v^{+}_1\\
&+ X_{r,l-\overline{r}}\left(X_{r-1,l-2} h_{l-\overline{\left(r-1\right)},1}x_{l-\overline{r+1},0}^{-}\right)X_{r+2,l-\overline{r+2}}\ldots X_{2,l}X_{1,l-1}v^{+}_1\\
   &= \left(a_1+\left(r-1\right)+\frac{1}{2}+\frac{1}{2}\right)X_{r,l-\bar{r}}\ldots X_{2,l}X_{1,l-1}v^{+}_1\\
   &+ \left(a_1+\left(r-2\right)+\frac{1}{2}+\frac{1}{2}-1\right)X_{r,l-\bar{r}}\ldots X_{2,l}X_{1,l-1}v^{+}_1\\
   &- \Big(a_1+\left(r-2\right)\Big)X_{r,l-\bar{r}}\ldots X_{2,l}X_{1,l-1}v^{+}_1\\
   &= \left(a_1+r\right)X_{r,l-\bar{r}}\ldots X_{2,l}X_{1,l-1}v^{+}_1.
\end{align*}
%
We now prove (\ref{so2ll-1}) for the pair $\left(r,s\right)$. By the induction hypothesis, we may assume that the claim (\ref{so2ll-1}) is true for $1,2,\ldots, s-1$. Similar to the above computations,
\begin{align*}
   h_{k,1}&X_{k+1,l-\overline{\left(r+1\right)}}X_{r,l-\bar{r}}\ldots X_{2,l}X_{1,l-1}v^{+}_1\\
   &= [h_{k,1},x_{k+1,0}^{-}]X_{k+2,l-\overline{\left(r+1\right)}}X_{r,l-\bar{r}}\ldots X_{2,l}X_{1,l-1}v^{+}_1\\
   &+X_{k+1,l-\overline{\left(r+1\right)}} h_{k,1}X_{r,l-\bar{r}}\ldots X_{2,l}X_{1,l-1}v^{+}_1\\
   &= [h_{k,1},x_{k+1,0}^{-}]X_{k+2,l-\overline{\left(r+1\right)}}X_{r,l-\bar{r}}\ldots X_{2,l}X_{1,l-1}v^{+}_1\\
   &+X_{k+1,l-\overline{\left(r+1\right)}}X_{r,k-2}[h_{k,1},x_{k-1,0}^{-}]X_{k,l-\bar{r}}\ldots X_{2,l}X_{1,l-1}v^{+}_1\\
   &+X_{k+1,l-\overline{\left(r+1\right)}}X_{r,k-1}h_{k,1}X_{k,l-\bar{r}}\ldots X_{2,l}X_{1,l-1}v^{+}_1\\
   &= \left(a_1+r+\frac{s-1}{2}+\frac{1}{2}\right)X_{k+1,l-\overline{\left(r+1\right)}}X_{r,l-\bar{r}}\ldots X_{2,l}X_{1,l-1}v^{+}_1\\
   &+\left(a_1+\left(r-1\right)+\frac{s+1}{2}+\frac{1}{2}-1\right)X_{k+1,l-\overline{\left(r+1\right)}}X_{r,l-\bar{r}}\ldots X_{2,l}X_{1,l-1}v^{+}_1\\
   &-\left(a_1+\left(r-1\right)+\frac{s}{2}\right)X_{k+1,l-\overline{\left(r+1\right)}}X_{r,l-\bar{r}}\ldots X_{2,l}X_{1,l-1}v^{+}_1\\
   &= \left(a_1+r+\frac{s}{2}\right)X_{k+1,l-\overline{\left(r+1\right)}}X_{r,l-\bar{r}}\ldots X_{2,l}X_{1,l-1}v^{+}_1.
\end{align*}


By induction we know that the claim (\ref{so2ll-1}) is true in general.

Comparing the coefficients, we have $\operatorname{Re}\left(a\right)\geq \operatorname{Re}\left(a_1\right)$.
\end{proof}
\begin{remark}
By symmetry of nodes $l-1$ and $l$, we can similarly prove a similar result for the case $b_1=l$.
\end{remark}
\subsection{Case 3: $2\leq b_1\leq l-2$.}

For the simplicity of subscript of $x_{b_1,0}^{-}$, we assume that $b_1=i$ in this subsection.

Recall that in this case
\begin{align*}
v^{-}_1 &= \left(x_{i,0}^{-}\ldots x_{l-2,0}^{-}x_{l,0}^{-}\ldots x_{i,0}^{-}\right)\left(\left(x_{i-1,0}^{-}\right)^2x_{i,0}^{-}\ldots x_{l-2,0}^{-}x_{l,0}^{-}\ldots x_{i,0}^{-}\right)\ldots \\
 & \left(\left(x_{1,0}^{-}\right)^2\ldots \left(x_{i-1,0}^{-}\right)^2x_{i,0}^{-}\ldots x_{l-2,0}^{-}x_{l,0}^{-}\ldots x_{i,0}^{-}\right)v^{+}_1.
\end{align*}

Denote, for $1\leq m<i$,
$$\left(x_{m,0}^{-}\right)^2\ldots\left(x_{i-1,0}^{-}\right)^2x_{i,0}^{-}\ldots x_{l-2,0}^{-}x_{l,0}^{-}\ldots x_{i+1,0}^{-}x_{i,0}^{-}v^{+}_1=\overline{\left(x_{m,0}^{-}\right)^2\ldots x_{i,0}^{-}}v^{+}_1.$$

\begin{proposition}\label{yoc3p1} Let $i\leq k\leq l-2$ and $1\leq m\leq i-2$.
\begin{enumerate}
  \item $Y_{k+1}\left(x_{k,0}^{-}\ldots x_{i+1,0}^{-}x_{i,0}^{-}v^{+}_1\right)\cong W_1\left(a_1+\frac{k-i}{2}\right)$.
  \item $Y_{l}\left(x_{l-1,0}^{-}\ldots x_{i+1,0}^{-}x_{i,0}^{-}v^{+}_1\right)\cong W_1\left(a_1+\frac{l-i-1}{2}\right)$.
  \item $Y_{k}\left(x_{k+1,0}^{-}\ldots x_{l-2,0}^{-}x_{l,0}^{-}x_{l-1,0}^{-}\ldots x_{i+1,0}^{-}x_{i,0}^{-}v^{+}_1\right)\cong W_1\left(a_1+\frac{2l-i-k-2}{2}\right)$.
  \item \begin{enumerate}
  \item If $i<l-2$, $Y_{i-1}\Big(\overline{x_{i,0}^{-}\ldots x_{i,0}^{-}}v^{+}_1\Big)\cong W_1\left(a_1+l-i-\frac{1}{2}\right)\otimes W_1\left(a_1+\frac{1}{2}\right)$.
  \item If $i=l-2$, $Y_{i-1}\Big(\overline{x_{i,0}^{-}\ldots x_{i,0}^{-}}v^{+}_1\Big)$ is isomorphic to either $W_2\left(a_1+\frac{1}{2}\right)$ or $W_1\left(a_1+\frac{3}{2}\right)\otimes W_1\left(a_1+\frac{1}{2}\right)$, respectively.
\end{enumerate}
\item \begin{enumerate}
  \item If $i<l-2$, $Y_{m}\Big(\overline{\left(x_{m+1,0}^{-}\right)^2\ldots x_{i,0}^{-}}v^{+}_1\Big)\cong W_1\left(a_1+\frac{2l-i-m-2}{2}\right)\otimes$\\ $ W_1\left(a_1+\frac{i-m}{2}\right)$.
  \item If $i=l-2$, $Y_{m}\Big(\overline{\left(x_{m+1,0}^{-}\right)^2\ldots x_{i,0}^{-}}v^{+}_1\Big)$ is isomorphic to either\\ $W_2\left(a_1+\frac{i-m}{2}\right)$ or $W_1\left(a_1+\frac{i-m+2}{2}\right)\otimes W_1\left(a_1+\frac{i-m}{2}\right)$.
\end{enumerate}
\item $Y_{i}\Big(\overline{\left(x_{1,0}^{-}\right)^2\ldots x_{i,0}^{-}}v^{+}_1\Big)\cong W_1\left(a_1+1\right)$.
\end{enumerate}
\end{proposition}
\begin{proof}
The items $\left(i\right), \left(ii\right)$ and $\left(iii\right)$ can be proved very similarly as the ones in  Proposition \ref{i=1yo2n}, so we omit the proof. The item $\left(iv\right)$ is proved in Lemma \ref{l-2is34d}. Lemmas \ref{l-2i34d2} and \ref{l-2i34d23} are devoted to the proof of the item $\left(v\right)$. The item $\left(vi\right)$ is showed in Lemma \ref{i=l-2v2}.
\end{proof}

\begin{lemma}\label{l-2is34d}\
\begin{enumerate}
  \item If $i<l-2$, $Y_{i-1}\left(\overline{x_{i,0}^{-}\ldots x_{i,0}^{-}}v^{+}_1\right)\cong W_1\left(a_1+l-i-\frac{1}{2}\right)\otimes W_1\left(a_1+\frac{1}{2}\right)$.
  \item If $i=l-2$, $Y_{i-1}\left(\overline{x_{i,0}^{-}\ldots x_{i,0}^{-}}v^{+}_1\right)$ is either 3-dimensional or 4-dimensional, which is isomorphic to either $W_2\left(a_1+\frac{1}{2}\right)$ or $W_1\left(a_1+\frac{3}{2}\right)\otimes W_1\left(a_1+\frac{1}{2}\right)$, respectively.
\end{enumerate}

\end{lemma}
\begin{proof}
It follows from Proposition \ref{rileq2} that $wt(\overline{x_{i,0}^{-}\ldots x_{i,0}^{-}}v^{+}_1)=2\omega_{i-1}-\omega_{i}$. Thus $$h_{i-1,0}\overline{x_{i,0}^{-}\ldots x_{i,0}^{-}}v^{+}_1= 2\overline{x_{i,0}^{-}\ldots x_{i,0}^{-}}v^{+}_1.$$ Thus the associated polynomial $P\left(u\right)$ to $Y_{i-1}\left(\overline{x_{i,0}^{-}\ldots x_{i,0}^{-}}v^{+}_1\right)$ has of degree 2. Suppose $P\left(u\right)=\left(u-a\right)\left(u-b\right)$ with $\operatorname{Re}\left(a\right)\leq \operatorname{Re}\left(b\right)$.
Thus we have
\begin{align*}
&\frac{P\left(u+1\right)}{P\left(u\right)}\\
&\quad= \frac{u-\left(a-1\right)}{u-a}\frac{u-\left(b-1\right)}{u-b} \\
&\quad= \left(1+u^{-1}+au^{-2}+a^2u^{-3}+\ldots\right)\left(1+u^{-1}+bu^{-2}+b^2u^{-3}+\ldots\right)\\
&\quad= 1+2u^{-1}+\left(a+b+1\right)u^{-2}+\left(a^2+b^2+a+b\right)u^{-3}+\ldots.
\end{align*}
The eigenvalues of $\overline{x_{i,0}^{-}\ldots x_{i,0}^{-}}v^{+}_1$ under $h_{i-1,1}$ and $h_{i-1,2}$ will tell the values of $a$ and $b$.

Note that the weight of $\overline{x_{i+1,0}^{-}\ldots x_{i,0}^{-}}v^{+}_1$ is $\omega_{i-1}+\omega_{i}-\omega_{i+1}$. Thus $$h_{i-1,0}\overline{x_{i+1,0}^{-}\ldots x_{i,0}^{-}}v^{+}_1=\overline{x_{i+1,0}^{-}\ldots x_{i,0}^{-}}v^{+}_1.$$
It follows from $(iii)$ of Proposition \ref{yoc3p1} that
$$h_{i,1}x_{i+1,0}^{-}\ldots x_{l-2,0}^{-}x_{l,0}^{-}\ldots x_{i,0}^{-}v^{+}_1=\left(a_1+l-i-1\right)x_{i+1,0}^{-}\ldots x_{l-2,0}^{-}x_{l,0}^{-}\ldots x_{i,0}^{-}v^{+}_1.$$
Denote $A=a_1+l-i-1$.
\begin{align*}
   h_{i-1,1}&\overline{x_{i,0}^{-}\ldots x_{i,0}^{-}}v^{+}_1 \\
   &= [h_{i-1,1},x_{i,0}^{-}]\overline{x_{i+1,0}^{-}\ldots x_{i,0}^{-}}v^{+}_1 +\overline{x_{i,0}^{-}\ldots x_{i+1,0}^{-}}[h_{i-1,1},x_{i,0}^{-}]v^{+}_1\\
    &= \left(x_{i,1}^{-}+\frac{1}{2}x_{i,0}^{-}+x_{i,0}^{-}h_{i-1,0}\right)\overline{x_{i+1,0}^{-}\ldots x_{i,0}^{-}}v^{+}_1\\
     &+\overline{x_{i,0}^{-}\ldots x_{i+1,0}^{-}}\left(x_{i,1}^{-}+\frac{1}{2}x_{i,0}^{-}+x_{i,0}^{-}h_{i-1,0}\right)v^{+}_1 \\
     &= \Bigg(\left(A+\frac{1}{2}\right)+\left(a_1+\frac{1}{2}\right)+1\Bigg)\overline{x_{i,0}^{-}\ldots x_{i,0}^{-}}v^{+}_1\\
      &= \left(2a_1+l-i+1\right)\overline{x_{i,0}^{-}\ldots x_{i,0}^{-}}v^{+}_1.
\end{align*}
\begin{align*}
   h_{i-1,2}&\overline{x_{i,0}^{-}\ldots x_{i,0}^{-}}v^{+}_1 \\
   &= [h_{i-1,2},x_{i,0}^{-}]\overline{x_{i+1,0}^{-}\ldots x_{i,0}^{-}}v^{+}_1+\overline{x_{i,0}^{-}\ldots x_{i+1,0}^{-}}[h_{i-1,2},x_{i,0}^{-}]v^{+}_1\\
    &= \left([h_{i-1,1}, x_{i,1}^{-}]+\frac{1}{2}\left(h_{i-1,1}x_{i,0}^{-}+x_{i,0}^{-}h_{i-1,1}\right)\right)\overline{x_{i+1,0}^{-}\ldots x_{i,0}^{-}}v^{+}_1 \\
     &+\overline{x_{i,0}^{-}\ldots x_{i+1,0}^{-}}\left([h_{i-1,1}, x_{i,1}^{-}]+\frac{1}{2}\left(h_{i-1,1}x_{i,0}^{-}+x_{i,0}^{-}h_{i-1,1}\right)\right)v^{+}_1 \\
     &= [h_{i-1,1}, x_{i,1}^{-}]\overline{x_{i+1,0}^{-}\ldots x_{i,0}^{-}}v^{+}_1+\frac{1}{2}h_{i-1,1}\overline{x_{i,0}^{-}\ldots x_{i,0}^{-}}v^{+}_1\\
      &+ \frac{1}{2}\overline{x_{i,0}^{-}\ldots x_{i+1,0}^{-}} h_{i-1,1}^{-}x_{i,0}^{-}v^{+}_1+\overline{x_{i,0}^{-}\ldots x_{i+1,0}^{-}}[h_{i-1,1},x_{i,1}^{-}]v^{+}_1\\
      &+\frac{1}{2}\overline{x_{i,0}^{-}\ldots x_{i+1,0}^{-}}h_{i-1,1}x_{i,0}^{-}v^{+}_1\\
       &= h_{i-1,1}x_{i,1}^{-}\overline{x_{i+1,0}^{-}\ldots x_{i,0}^{-}}v^{+}_1-x_{i,1}^{-}h_{i-1,1}\overline{x_{i+1,0}^{-}\ldots x_{i,0}^{-}}v^{+}_1\\
       &+\frac{1}{2}h_{i-1,1}\overline{x_{i,0}^{-}\ldots x_{i,0}^{-}}v^{+}_1 +\overline{x_{i,0}^{-}\ldots x_{i+1,0}^{-}}h_{i-1,1}x_{i,1}^{-}v^{+}_1\\
      &+\overline{x_{i,0}^{-}\ldots x_{i+1,0}^{-}}h_{i-1,1}x_{i,0}^{-}v^{+}_1\\
      &= Ah_{i-1,1}\overline{x_{i,0}^{-}\ldots x_{i,0}^{-}}v^{+}_1-x_{i,1}^{-}\overline{x_{i+1,0}^{-}\ldots x_{i+1,0}^{-}}h_{i-1,1}x_{i,0}^{-}v^{+}_1\\
     &+\frac{1}{2}h_{i-1,1}\overline{x_{i,0}^{-}\ldots x_{i,0}^{-}}v^{+}_1+\overline{x_{i,0}^{-}\ldots x_{i+1,0}^{-}}h_{i-1,1}x_{i,1}^{-}v^{+}_1\\
      &+\overline{x_{i,0}^{-}\ldots x_{i+1,0}^{-}}h_{i-1,1}x_{i,0}^{-}v^{+}_1\\
       &= A\left(A+a_1+2\right)\overline{x_{i,0}^{-}\ldots x_{i,0}^{-}}v^{+}_1\\
       &-\left(a_1+\frac{1}{2}\right)A\overline{x_{i,0}^{-}\ldots x_{i,0}^{-}}v^{+}_1+\frac{1}{2}\left(A+a_1+2\right)\overline{x_{i,0}^{-}\ldots x_{i,0}^{-}}v^{+}_1\\
      &+ a_1\left(a_1+\frac{1}{2}\right)\overline{x_{i,0}^{-}\ldots x_{i,0}^{-}}v^{+}_1+\left(a_1+\frac{1}{2}\right)\overline{x_{i,0}^{-}\ldots x_{i,0}^{-}}v^{+}_1\\
      &= \left(A+\frac{1}{2}\right)\left(A+a_1+2\right)\overline{x_{i,0}^{-}\ldots x_{i,0}^{-}}v^{+}_1\\
       &-\left(a_1+\frac{1}{2}\right)\left(A-a_1-1\right)\overline{x_{i,0}^{-}\ldots x_{i,0}^{-}}v^{+}_1\\
&= \left(2\left(a_1+\frac{l-i+1}{2}\right)^2+\frac{\left(l-i-2\right)\left(l-i\right)}{2}\right)\overline{x_{i,0}^{-}\ldots x_{i,0}^{-}}v^{+}_1.
\end{align*}
Thus $a+b+1=\left(2a_1+\left(l-i\right)+1\right)$ and $\left(a^2+b^2+a+b\right)=2\left(a_1+\frac{l-i+1}{2}\right)^2+\frac{\left(l-i-2\right)\left(l-i\right)}{2}$. Suppose $a=a_1+\frac{l-i+1}{2}-x$. Then $b=a_1+\frac{l-i-1}{2}+x$. Since $\operatorname{Re}\left(b\right)\geq \operatorname{Re}\left(a\right)$, $\operatorname{Re}\left(x\right)\geq \frac{1}{2}$. Substituting them to $\left(a^2+b^2+a+b\right)=2\left(a_1+\frac{l-i+1}{2}\right)^2+\frac{\left(l-i-2\right)\left(l-i\right)}{2}$, we have
$$x^2-x=\frac{\left(l-i\right)\left(l-i-2\right)}{4}.$$ 
Thus $x=\frac{1}{2}+\sqrt{\frac{\left(l-i\right)\left(l-i-2\right)+1}{4}}=\frac{l-i}{2}$ is a real number. Then $a=a_1+\frac{1}{2}$ and $b=a_1+l-i-\frac{1}{2}$.

Note that the dimension of $Y_{i-1}\left(\overline{x_{i,0}^{-}\ldots x_{i,0}^{-}}v^{+}_1\right)$ is at least 3. The dimension of the local Weyl module $W\left(P\right)$ of $\ysl$ is 4, which is isomorphic to\\ $W_1\left(a_1+l-i-\frac{1}{2}\right)\otimes W_1\left(a_1+\frac{1}{2}\right)$. If $i<l-2$, then $W\left(P\right)$ is irreducible, so is $Y_{i-1}\left(\overline{x_{i,0}^{-}\ldots x_{i,0}^{-}}v^{+}_1\right)$. If $i=l-2$, then $Y_{i-1}\left(\overline{x_{i,0}^{-}\ldots x_{i,0}^{-}}v^{+}_1\right)$ is either 3-dimensional and isomorphic to $W_2\left(a_1+\frac{3}{2}\right)$ or 4-dimensional and isomorphic to $W_1\left(a_1+\frac{3}{2}\right)\otimes W_1\left(a_1+\frac{1}{2}\right)$.
\end{proof}

It follows from either Corollary \ref{w2ihw} or Corollary \ref{w1bw1a} that\begin{corollary}\label{soi-1if}
$$x_{i-1,1}^{-}x_{i-1,0}^{-}\overline{x_{i,0}^{-}\ldots x_{i,0}^{-}}v^{+}_1=\left(a_1+\frac{l-i-1}{2}\right)\overline{\left(x_{i-1,0}^{-}\right)^2\ldots x_{i,0}^{-}}v^{+}_1.$$
$$\left(x_{i-1,1}^{-}x_{i-1,0}^{-}+x_{i-1,0}^{-}x_{i-1,1}^{-}\right)\overline{x_{i,0}^{-}\ldots x_{i,0}^{-}}v^{+}_1=\left(2a_1+{l-i}\right)\overline{\left(x_{i-1,0}^{-}\right)^2\ldots x_{i,0}^{-}}v^{+}_1.$$
\begin{align*}
\left(x_{i-1,1}^{-}\right)^2&\overline{x_{i,0}^{-}\ldots x_{i,0}^{-}}v^{+}_1\\
&=\Bigg(\left(a_1+\frac{l-i}{2}\right)^2-\frac{\left(l-i-2\right)\left(l-i\right)+1}{4}\Bigg)\overline{\left(x_{i-1,0}^{-}\right)^2\ldots x_{i,0}^{-}}v^{+}_1.
\end{align*}
\begin{align*}
x_{i-1,0}^{-}x_{i-1,2}^{-}&\overline{x_{i,0}^{-}\ldots x_{i,0}^{-}}v^{+}_1\\
&=\Bigg(\left(a_1+\frac{l-i+1}{2}\right)^2+\frac{\left(l-i-2\right)\left(l-i\right)}{4}\Bigg)\overline{\left(x_{i-1,0}^{-}\right)^2\ldots x_{i,0}^{-}}v^{+}_1.
\end{align*}
\end{corollary}



If $i=2$, jump to Lemma \ref{i=l-2v2}. If $i\geq 3$, the calculations continue.

\begin{lemma}\label{l-2i34d2}
\
\begin{enumerate}
  \item If $i<l-2$, $Y_{i-2}\Big(\overline{\left(x_{i-1,0}^{-}\right)^2\ldots x_{i,0}^{-}}v^{+}_1\Big)\cong W_1\left(a_1+l-i\right)\otimes W_1\left(a_1+1\right)$.
  \item If $i=l-2$, $Y_{i-2}\Big(\overline{\left(x_{i-1,0}^{-}\right)^2\ldots x_{i,0}^{-}}v^{+}_1\Big)$ is either 3-dimensional and isomorphic to $W_2\left(a_1+1\right)$ or 4-dimensional and isomorphic to $W_1\left(a_1+2\right)\otimes W_1\left(a_1+1\right)$.
\end{enumerate}
\end{lemma}
\begin{proof} Note that the weight of $\overline{x_{i,0}^{-}\ldots x_{i,0}^{-}}v_1^{+}$ is $2\omega_{i-1}-\omega_i$, and the weight of $\overline{x_{i-1,0}^{-}\ldots x_{i,0}^{-}}v_1^{+}$ is $\omega_{i-2}$.
It follows from Proposition \ref{rileq2} that the weight of vector $\overline{(x_{i-1,0}^{-})^2\ldots x_{i,0}^{-}}v^{+}_1$ is $2\omega_{i-2}-2\omega_{i-1}+\omega_{i}$. Thus $$h_{i-2,0}\overline{\left(x_{i-1,0}^{-}\right)^2\ldots x_{i,0}^{-}}v^{+}_1=2\overline{\left(x_{i-1,0}^{-}\right)^2\ldots x_{i,0}^{-}}v^{+}_1.$$
Thus the associated polynomial $P\left(u\right)$ has of degree 2. Suppose $P\left(u\right)=\left(u-a\right)\left(u-b\right)$ with $\operatorname{Re}\left(a\right)\leq \operatorname{Re}\left(b\right)$.
\begin{align*}
h_{i-2,1}&\overline{\left(x_{i-1,0}^{-}\right)^2\ldots x_{i,0}^{-}}v^{+}_1 \\
   &= [h_{i-2,1}, \left(x_{i-1,0}^{-}\right)^2]\overline{x_{i,0}^{-}\ldots x_{i,0}^{-}}v_1^{+}\\
   &= [h_{i-2,1}, x_{i-1,0}^{-}]\overline{x_{i-1,0}^{-}\ldots x_{i,0}^{-}}v_1^{+}+ x_{i-1,0}^{-}[h_{i-2,1}, x_{i-1,0}^{-}]\overline{x_{i,0}^{-}\ldots x_{i,0}^{-}}v_1^{+}\\
    &= \left(x_{i-1,1}^{-}+\frac{1}{2}x_{i-1,0}^{-}+x_{i-1,0}^{-}h_{i-2,0}\right)\overline{x_{i-1,0}^{-}\ldots x_{i,0}^{-}}v_1^{+}\\
    &+x_{i-1,0}^{-}\left(x_{i-1,1}^{-}+\frac{1}{2}x_{i-1,0}^{-}+x_{i-1,0}^{-}h_{i-2,0}\right)\overline{x_{i,0}^{-}\ldots x_{i,0}^{-}}v_1^{+}\\
   &= \left(x_{i-1,1}^{-}x_{i-1,0}^{-}+x_{i-1,0}^{-}x_{i-1,1}^{-}\right)\overline{x_{i,0}^{-}\ldots x_{i,0}^{-}}v_1^{+}+2\overline{\left(x_{i-1,0}^{-}\right)^2\ldots x_{i,0}^{-}}v^{+}_1\\
   &= \left(2a_1+\left(l-i\right)+2\right)\overline{\left(x_{i-1,0}^{-}\right)^2\ldots x_{i,0}^{-}}v^{+}_1.
\end{align*}
Denote $A=a_1+\frac{l-i}{2}$.
\begin{align*}
h_{i-2,2}&\overline{\left(x_{i-1,0}^{-}\right)^2\ldots x_{i,0}^{-}}v^{+}_1  \\
   &= [h_{i-2,2}, \left(x_{i-1,0}^{-}\right)^2]\overline{x_{i,0}^{-}\ldots x_{i,0}^{-}}v_1^{+} \\
   &= [h_{i-2,2}, x_{i-1,0}^{-}]\overline{x_{i-1,0}^{-}\ldots x_{i,0}^{-}}v_1^{+}+ x_{i-1,0}^{-}[h_{i-2,2}, x_{i-1,0}^{-}]\overline{x_{i,0}^{-}\ldots x_{i,0}^{-}}v_1^{+}\\
   &=\Big([h_{i-2,1}, x_{i-1,1}^{-}]+\frac{1}{2}\left(h_{i-2,1}x_{i-1,0}^{-}+x_{i-1,0}^{-}h_{i-2,1}^{-}\right)\Big)\overline{x_{i-1,0}^{-}\ldots x_{i,0}^{-}}v_1^{+}\\
   &+x_{i-1,0}^{-}\Big([h_{i-2,1}, x_{i-1,1}^{-}]+\frac{1}{2}\left(h_{i-2,1}x_{i-1,0}^{-}+x_{i-1,0}^{-}h_{i-2,1}^{-}\right)\Big)\overline{x_{i,0}^{-}\ldots x_{i,0}^{-}}v_1^{+}\\
   &=\left(h_{i-2,1}x_{i-1,1}^{-}-x_{i-1,1}^{-}h_{i-2,1}\right)\overline{x_{i-1,0}^{-}\ldots x_{i,0}^{-}}v_1^{+}\\
    &+\frac{1}{2}\left(h_{i-2,1}x_{i-1,0}^{-}+x_{i-1,0}^{-}h_{i-2,1}^{-}\right)\overline{x_{i-1,0}^{-}\ldots x_{i,0}^{-}}v_1^{+}\\
   &+x_{i-1,0}^{-}\left(h_{i-2,1}x_{i-1,1}^{-}-x_{i-1,1}^{-}h_{i-2,1}\right)\overline{x_{i,0}^{-}\ldots x_{i,0}^{-}}v_1^{+}\\
   &+\frac{1}{2}x_{i-1,0}^{-}\left(h_{i-2,1}x_{i-1,0}^{-}+x_{i-1,0}^{-}h_{i-2,1}^{-}\right)\overline{x_{i,0}^{-}\ldots x_{i,0}^{-}}v_1^{+}\\
   &=h_{i-2,1}x_{i-1,1}^{-}\overline{x_{i-1,0}^{-}\ldots x_{i,0}^{-}}v_1^{+}-x_{i-1,1}^{-}h_{i-2,1}\overline{x_{i-1,0}^{-}\ldots x_{i,0}^{-}}v_1^{+}\\
   &+\frac{1}{2}h_{i-2,1}\overline{\left(x_{i-1,0}^{-}\right)^2\ldots x_{i,0}^{-}}v^{+}_1+ x_{i-1,0}^{-}h_{i-2,1}x_{i-1,1}^{-}\overline{x_{i,0}^{-}\ldots x_{i,0}^{-}}v_1^{+}\\
   &+x_{i-1,0}^{-}h_{i-2,1}\overline{x_{i-1,0}^{-}\ldots x_{i,0}^{-}}v_1^{+}\\
   &=\left(a_1+\frac{l-i-1}{2}\right)h_{i-2,1}\left(x_{i-1,0}^{-}\right)^2\overline{x_{i,0}^{-}\ldots x_{i,0}^{-}}v^{+}_1\\
   &-x_{i-1,1}^{-}\left(x_{i-1,1}^{-}+\frac{1}{2}x_{i-1,0}^{-}+x_{i-1,0}^{-}h_{i-2,0}\right)\overline{x_{i,0}^{-}\ldots x_{i,0}^{-}}v^{+}_1\\
   &+\frac{1}{2}h_{i-2,1}\overline{\left(x_{i-1,0}^{-}\right)^2x_{i,0}^{-}\ldots x_{i,0}^{-}}v^{+}_1 \\
   &+x_{i-1,0}^{-}\left(x_{i-1,2}^{-}+\frac{1}{2}x_{i-1,1}^{-}+x_{i-1,1}^{-}h_{i-2,0}\right)\overline{x_{i,0}^{-}\ldots x_{i,0}^{-}}v^{+}_1\\
   &+x_{i-1,0}^{-}\left(x_{i-1,1}^{-}+\frac{1}{2}x_{i-1,0}^{-}+x_{i-1,0}^{-}h_{i-2,0}\right)\overline{x_{i,0}^{-}\ldots x_{i,0}^{-}}v^{+}_1\\
   &=A\left(2A+2\right)\overline{\left(x_{i-1,0}^{-}\right)^2\ldots x_{i,0}^{-}}v^{+}_1\\
   &-\left(A+\frac{1}{2}\right)\left(A-\frac{1}{2}\right)\overline{\left(x_{i-1,0}^{-}\right)^2\ldots x_{i,0}^{-}}v^{+}_1\\
   &+\frac{\left(l-i\right)\left(l-i-2\right)}{2}\overline{\left(x_{i-1,0}^{-}\right)^2\ldots x_{i,0}^{-}}v^{+}_1-\frac{1}{2}\left(A-\frac{1}{2}\right)\overline{\left(x_{i-1,0}^{-}\right)^2\ldots x_{i,0}^{-}}v^{+}_1\\
   &+\Bigg(\left(A+\frac{1}{2}\right)^2+\frac{1}{2}\left(A+\frac{1}{2}\right)\Bigg)\overline{\left(x_{i-1,0}^{-}\right)^2\ldots x_{i,0}^{-}}v^{+}_1\\
   &+\left(A+\frac{1}{2}\right)\overline{\left(x_{i-1,0}^{-}\right)^2\ldots x_{i,0}^{-}}v^{+}_1+\frac{1}{2}\overline{\left(x_{i-1,0}^{-}\right)^2\ldots x_{i,0}^{-}}v^{+}_1\\
   &= \Bigg(2\left(a_1+\frac{l-i+2}{2}\right)^2+\frac{\left(l-i-2\right)\left(l-i\right)}{2}\Bigg)\overline{\left(x_{i-1,0}^{-}\right)^2\ldots x_{i,0}^{-}}v^{+}_1.
\end{align*}
Similar to the last lemma,  $a=a_1+1$ and $b=a_1+l-i$. If $i<l-2$, then $Y_{i-1}\left(\overline{x_{i,0}^{-}\ldots x_{i,0}^{-}}v^{+}_1\right)\cong W_1\left(a_1+l-i\right)\otimes W_1\left(a_1+1\right)$. If $i=l-2$, then $Y_{i-1}\left(\overline{x_{i,0}^{-}\ldots x_{i,0}^{-}}v^{+}_1\right)$ is either 3-dimensional and isomorphic to $W_2\left(a_1+1\right)$ or 4-dimensional and isomorphic to $W_1\left(a_1+2\right)\otimes W_1\left(a_1+1\right)$.
\end{proof}
If $i>3$, the computations will keep going. Similar to the above lemma, we have
\begin{lemma}\label{l-2i34d23} Let $1\leq m\leq i-2$.

\begin{enumerate}
  \item If $i<l-2$, $Y_{m}\Big(\overline{\left(x_{m+1,0}^{-}\right)^2\ldots x_{i,0}^{-}}v^{+}_1\Big)\cong W_1\left(a_1+\frac{2l-i-m-2}{2}\right)\otimes W_1\left(a_1+\frac{i-m}{2}\right)$.
  \item If $i=l-2$, $Y_{m}\Big(\overline{\left(x_{m+1,0}^{-}\right)^2\ldots x_{i,0}^{-}}v^{+}_1\Big)$ is either 3-dimensional and isomorphic to $W_2\left(a_1+\frac{i-m}{2}\right)$ or 4-dimensional andis isomorphic to\\ $W_1\left(a_1+\frac{i-m+2}{2}\right)\otimes W_1\left(a_1+\frac{i-m}{2}\right)$.
\end{enumerate}
\end{lemma}
\begin{proof}
It follows from Proposition \ref{rileq2} that the weight of vector $\overline{(x_{m+1,0}^{-})^2\ldots x_{i,0}^{-}}v^{+}_1$ is $2\omega_{m}-2\omega_{m+1}+\omega_{i}$. Thus
$$h_{m,0}\overline{\left(x_{m+1,0}^{-}\right)^2\ldots x_{i,0}^{-}}v^{+}_1=2\overline{\left(x_{m+1,0}^{-}\right)^2\ldots x_{i,0}^{-}}v^{+}_1.$$
We are going to show the following claims are true by induction on $m$ downward:
$$h_{m,1}\overline{\left(x_{m+1,0}^{-}\right)^2\ldots x_{i,0}^{-}}v^{+}_1= \left(2a_1+l-m\right)\overline{\left(x_{m+1,0}^{-}\right)^2\ldots x_{i,0}^{-}}v^{+}_1;$$ and
\begin{align*}
   &h_{m,2}\overline{\left(x_{m+1,0}^{-}\right)^2\ldots x_{i,0}^{-}}v^{+}_1  \\
   &=\Big(2\left(a_1+\frac{l-m}{2}\right)^2+\frac{\left(l-i-2\right)\left(l-i\right)}{2}\Big)\overline{\left(x_{m+1,0}^{-}\right)^2\ldots x_{i,0}^{-}}v^{+}_1.
\end{align*}
The basis of induction is proved in the above lemma. Suppose that the claims are true for $k\geq m+1$. We next show that they are also true for $k=m$.
By the induction hypothesis, we may assume that $$h_{m+1,1}\overline{\left(x_{m+2,0}^{-}\right)^2\ldots x_{i,0}^{-}}v^{+}_1= \left(2a_1+l-m-1\right)\overline{\left(x_{m+2,0}^{-}\right)^2\ldots x_{i,0}^{-}}v^{+}_1;$$ and
\begin{align*}
   &h_{m+1,2}\overline{\left(x_{m+2,0}^{-}\right)^2\ldots x_{i,0}^{-}}v^{+}_1  \\
   &=\Big(2\left(a_1+\frac{l-m-1}{2}\right)^2+\frac{\left(l-i-2\right)\left(l-i\right)}{2}\Big)\overline{\left(x_{m+2,0}^{-}\right)^2\ldots x_{i,0}^{-}}v^{+}_1.
\end{align*}
Set $c=2a_1+l-m-1$.
\begin{align*}
h_{m,1}&\overline{\left(x_{m+1,0}^{-}\right)^2\ldots x_{i,0}^{-}}v^{+}_1 \\
&= h_{m,1}\left(x_{m+1,0}^{-}\right)^2\overline{\left(x_{m+2,0}^{-}\right)^2\ldots x_{i,0}^{-}}v^{+}_1\\
&= [h_{m,1},\left(x_{m+1,0}^{-}\right)^2]\overline{\left(x_{m+2,0}^{-}\right)^2\ldots x_{i,0}^{-}}v^{+}_1\\
&= [h_{m,1},x_{m+1,0}^{-}]x_{m+1,0}^{-}\overline{\left(x_{m+2,0}^{-}\right)^2\ldots x_{i,0}^{-}}v^{+}_1\\
&+ x_{m+1,0}^{-}[h_{m,1},x_{m+1,0}^{-}]\overline{\left(x_{m+2,0}^{-}\right)^2\ldots x_{i,0}^{-}}v^{+}_1\\
&=\left(x_{m+1,1}^{-}+\frac{1}{2}x_{m+1,0}^{-}+x_{m+1,0}^{-}h_{m,0}\right)x_{m+1,0}^{-}\overline{\left(x_{m+2,0}^{-}\right)^2\ldots x_{i,0}^{-}}v^{+}_1 \\
&+ x_{m+1,0}^{-}\left(x_{m+1,1}^{-}+\frac{1}{2}x_{m+1,0}^{-}+x_{m+1,0}^{-}h_{m,0}\right)\overline{\left(x_{m+2,0}^{-}\right)^2\ldots x_{i,0}^{-}}v^{+}_1 \\
&=\left(x_{m+1,1}^{-}x_{m+1,0}^{-}+x_{m+1,0}^{-}x_{m+1,1}^{-}\right)\overline{\left(x_{m+2,0}^{-}\right)^2\ldots x_{i,0}^{-}}v^{+}_1\\
&+2\overline{\left(x_{m+1,0}^{-}\right)^2\ldots x_{i,0}^{-}}v^{+}_1\\
&=\left(2a_1+l-m-2+2\right)\overline{\left(x_{m+1,0}^{-}\right)^2\ldots x_{i,0}^{-}}v^{+}_1\\
&=\Big(2a_1+l-m\Big)\overline{\left(x_{m+1,0}^{-}\right)^2\ldots x_{i,0}^{-}}v^{+}_1.
\end{align*}
Denote $\overline{\left(x_{m+2,0}^{-}\right)^2\ldots x_{i,0}^{-}}v^{+}_1$ by $v$.
\begin{align*}
h_{m-1,2}&\overline{\left(x_{m+1,0}^{-}\right)^2\ldots x_{i,0}^{-}}v^{+}_1 \\
&= h_{m-1,2}\left(x_{m+1,0}^{-}\right)^2v\\
&= [h_{m-1,2},\left(x_{m+1,0}^{-}\right)^2]v\\
&= [h_{m-1,2},x_{m+1,0}^{-}]x_{m+1,0}^{-}v\\
&+ x_{m+1,0}^{-}[h_{m-1,2},x_{m+1,0}^{-}]v\\
&=\Big([h_{m,1}, x_{m+1,1}^{-}]+\frac{1}{2}\left(h_{m,1}x_{m+1,0}^{-}+x_{m+1,0}^{-}h_{m,1}\right)\Big)x_{m+1,0}^{-}v \\
&+ x_{m+1,0}^{-}\Big([h_{m,1}, x_{m+1,1}^{-}]+\frac{1}{2}\left(h_{m,1}x_{m+1,0}^{-}+x_{m+1,0}^{-}h_{m,1}\right)\Big)v\\
&=[h_{m,1}, x_{m+1,1}^{-}]x_{m+1,0}^{-}v \\
&+ \frac{1}{2}h_{m,1}\left(x_{m+1,0}^{-}\right)^2v+x_{m+1,0}^{-}h_{m,1}x_{m+1,0}^{-}v\\
&+ x_{m+1,0}^{-}[h_{m,1}, x_{m+1,1}^{-}]v\\
&=h_{m,1}x_{m+1,1}^{-}x_{m+1,0}^{-}v-x_{m+1,1}^{-}h_{m,1}x_{m+1,0}^{-}v\\
&+ \frac{1}{2}h_{m,1}\left(x_{m+1,0}^{-}\right)^2v+x_{m+1,0}^{-}[h_{m,1},x_{m+1,0}^{-}]v\\
&+ x_{m+1,0}^{-}[h_{m,1}, x_{m+1,1}^{-}]v \\
&=\frac{1}{2}\left(c-2\right)h_{m,1}\left(x_{m+1,0}^{-}\right)^2v\\
&-x_{m+1,1}^{-}\left(x_{m+1,1}^{-}+\frac{1}{2}x_{m+1,0}^{-}+x_{m+1,0}^{-}h_{m,0}\right)v \\
&+ \frac{1}{2}h_{m,1}\left(x_{m+1,0}^{-}\right)^2v\\
&+x_{m+1,0}^{-}\left(x_{m+1,1}^{-}+\frac{1}{2}x_{m+1,0}^{-}+x_{m+1,0}^{-}h_{m,0}\right)v\\
&+ x_{m+1,0}^{-}\left(x_{m+1,2}^{-}+\frac{1}{2}x_{m+1,1}^{-}+x_{m+1,1}^{-}h_{m,0}\right)v \\
&=\frac{1}{2}\left(c-1\right)h_{m,1}\left(x_{m+1,0}^{-}\right)^2v\\
&-\left(x_{m+1,1}^{-}\right)^2v-\frac{1}{2}x_{m+1,1}^{-}x_{m+1,0}^{-}v \\
&+\frac{3}{2}x_{m+1,0}^{-}x_{m+1,1}^{-}v+\frac{1}{2}\left(x_{m+1,0}^{-}\right)^2v\\
&+ x_{m+1,0}^{-}x_{m+1,2}^{-}v \\
&=\frac{1}{2}\left(c-1\right)\left(c+1\right)\overline{\left(x_{m+1,0}^{-}\right)^2\ldots x_{i,0}^{-}}v^{+}_1\\
&-\Big(\frac{1}{4}c^2-\frac{c}{2}-\frac{\left(l-i-2\right)\left(l-i\right)}{4}\Big)\overline{\left(x_{m+1,0}^{-}\right)^2\ldots x_{i,0}^{-}}v^{+}_1\\
&-\frac{1}{4}\left(c-2\right)\overline{\left(x_{m+1,0}^{-}\right)^2\ldots x_{i,0}^{-}}v^{+}_1 \\
&+\frac{3}{4}c\left(x_{m+1,0}^{-}\right)^2v+\frac{1}{2}\overline{\left(x_{m+1,0}^{-}\right)^2\ldots x_{i,0}^{-}}v^{+}_1\\
&+\frac{1}{2}\Big(\frac{1}{2}c^2+\frac{\left(l-i-2\right)\left(l-i\right)}{2}\Big)\overline{\left(x_{m+1,0}^{-}\right)^2\ldots x_{i,0}^{-}}v^{+}_1 \\
&=\Big(\frac{1}{2}\left(c+1\right)^2+\frac{\left(l-i-2\right)\left(l-i\right)}{2}\Big)\overline{\left(x_{m+1,0}^{-}\right)^2\ldots x_{i,0}^{-}}v^{+}_1\\
&=\Bigg(2\left(a_1+\frac{l-m}{2}\right)^2+\frac{\left(l-i-2\right)\left(l-i\right)}{2}\Bigg)\overline{\left(x_{m+1,0}^{-}\right)^2\ldots x_{i,0}^{-}}v^{+}_1.
\end{align*}
After a similar discussion as in Lemma \ref{l-2is34d}, we have $a=a_1+\frac{i-m}{2}$ and $b=a_1+\frac{2l-i-m-2}{2}$. If $i<l-2$, $Y_{m}\Big(\overline{\left(x_{m+1,0}^{-}\right)^2\ldots x_{i,0}^{-}}v^{+}_1\Big)\cong W_1\left(a_1+\frac{2l-i-m-2}{2}\right)\otimes W_1\left(a_1+\frac{i-m}{2}\right)$. If $i=l-2$, $Y_{m}\Big(\overline{\left(x_{m+1,0}^{-}\right)^2\ldots x_{i,0}^{-}}v^{+}_1\Big)$ is either 3 or 4-dimensional and isomorphic to $W_2\left(a_1+\frac{i-m}{2}\right)$ or $W_1\left(a_1+\frac{i-m+2}{2}\right)\otimes W_1\left(a_1+\frac{i-m}{2}\right)$, respectively.
\end{proof}


\begin{lemma}\label{i=l-2v2}
$Y_{i}\left(\overline{\left(x_{1,0}^{-}\right)^2\ldots x_{i,0}^{-}}v^{+}_1\right)\cong W_1\left(a_1+1\right)$.
\end{lemma}
\begin{proof}
It follows from Proposition \ref{rileq2} that the weight of vector $\overline{(x_{1,0}^{-})^2\ldots x_{i,0}^{-}}v^{+}_1$ is $-2\omega_{1}+\omega_{i}$. Thus $$h_{i,0}\overline{\left(x_{1,0}^{-}\right)^2\ldots x_{i,0}^{-}}v^{+}_1=\overline{\left(x_{1,0}^{-}\right)^2\ldots x_{i,0}^{-}}v^{+}_1.$$ It tells us that the associated polynomial has of degree 1. Thus we have $$Y_{i}\left(\overline{\left(x_{1,0}^{-}\right)^2\ldots x_{i,0}^{-}}v^{+}_1\right)\cong W_1\left(a\right).$$

The eigenvalue of $h_{i,1}$ on $\overline{\left(x_{1,0}^{-}\right)^2\ldots x_{i,0}^{-}}v^{+}_1$ tells us the value of $a$.
Note that $wt(x_{i-1,0}^{-}\overline{x_{i,0}^{-}\ldots x_{i,0}^{-}}v^{+}_1)=\omega_{i-1}$.
\begin{align*}
   h_{i,1}&\overline{\left(x_{1,0}^{-}\right)^2\ldots x_{i,0}^{-}}v^{+}_1 \\
   &=\left(x_{1,0}^{-}\right)^2\ldots\left(x_{i-2,0}^{-}\right)^2[h_{i,1}, x_{i-1,0}^{-}]\overline{x_{i-1,0}^{-}\ldots x_{i,0}^{-}}v^{+}_1\\
   &+\left(x_{1,0}^{-}\right)^2\ldots \left(x_{i-2,0}^{-}\right)^2x_{i-1,0}^{-}[h_{i,1}, x_{i-1,0}^{-}]\overline{x_{i,0}^{-}\ldots x_{i,0}^{-}}v^{+}_1 \\
   &+\left(x_{1,0}^{-}\right)^2\ldots \left(x_{i-1,0}^{-}\right)^2 h_{i,1}\overline{x_{i,0}^{-}\ldots x_{i,0}^{-}}v^{+}_1\\
   &=  \left(x_{1,0}^{-}\right)^2\ldots\left(x_{i-2,0}^{-}\right)^2x_{i-1,0}^{-}\left(x_{i-1,1}^{-}+\frac{1}{2}x_{i-1,0}^{-}+x_{i-1,0}^{-}h_{i,0}\right)\overline{x_{i,0}^{-}\ldots x_{i,0}^{-}}v^{+}_1\\
   &+ \left(x_{1,0}^{-}\right)^2\ldots\left(x_{i-2,0}^{-}\right)^2\left(x_{i-1,1}^{-}+\frac{1}{2}x_{i-1,0}^{-}+x_{i-1,0}^{-}h_{i,0}\right)x_{i-1,0}^{-}\overline{x_{i,0}^{-}\ldots x_{i,0}^{-}}v^{+}_1\\
   &- \left(a_1+l-i-1\right)\overline{\left(x_{1,0}^{-}\right)^2\ldots x_{i,0}^{-}}v^{+}_1\\
   &= \Big(2a_1+l-i-\left(a_1+l-i-1\right)\Big)\overline{\left(x_{1,0}^{-}\right)^2\ldots x_{i,0}^{-}}v^{+}_1\\
   &= \left(a_1+1\right)\overline{\left(x_{1,0}^{-}\right)^2\ldots x_{i,0}^{-}}v^{+}_1.
\end{align*}
\end{proof}
\begin{lemma}
Denote $\mathbf{s}_1=s_1s_2\ldots s_{l-2}s_ls_{l-1}\ldots s_i$. Then $\mathbf{s}_1^{-1}\left(\alpha_j\right)\in \Delta^{+}$, where $j=2,3,\ldots, l$, and $\alpha_j$ are the simple roots of $\nyo$.
\end{lemma}
\begin{proof}
$\mathbf{s}_1^{-1}=s_i\ldots s_{l-1}s_ls_{l-2}\ldots s_1$. We claim that \begin{equation*}
\mathbf{s}_1^{-1}\left(\mu_j-\mu_{j+1}\right)=\begin{cases}
u_j-u_{j+1}\qquad \mathrm{if}\qquad i< j\neq l-1 \\
\mu_{j-1}-u_{j+\delta_{ij}} \qquad \mathrm{if}\qquad 2\leq j\leq i.
\end{cases}
\end{equation*}
and $$\mathbf{s}_1^{-1}\left(\mu_{l-1}\pm \mu_{l}\right)=\mu_{l-1}\mp\mu_{l}.$$
The proof of the claim is trivial.
\end{proof}
\begin{proposition}\label{v2aga1}Let $v_2=\overline{\left(x_{1,0}^{-}\right)^2\ldots x_{i,0}^{-}}v^{+}_1$.
\begin{enumerate}
  \item $Y_{i+1}\left(x_{i,0}^{-}v_2\right)$, $Y_{i+2}\left(x_{i+1,0}^{-}x_{i,0}^{-}v_2\right)$, $\ldots$ , $Y_{i}\left(\overline{x_{i+1,0}^{-}\ldots x_{i,0}^{-}}v_2\right)$ are 2-dimensional, and are isomorphic to $W_1\left(a\right)$ with $\operatorname{Re}\left(a\right)\geq \operatorname{Re}\left(a_1\right)$.
  \item $Y_{i-1}\left(\overline{x_{i,0}^{-}\ldots x_{i,0}^{-}}v_2\right)$, $Y_{i-2}\Big(\overline{\left(x_{i-1,0}\right)^2\ldots x_{i,0}^{-}}v_2\Big)$, $\ldots$, $Y_{2}\Big(\overline{(x_{3,0}^{-})^2\ldots x_{i,0}^{-}}v_2\Big)$ are either 3-dimensional or 4-dimensional and are isomorphic to either $W_2\left(a\right)$ or $W_1\left(b\right)\otimes W_1\left(a\right)$, $\operatorname{Re}\left(b\right)\geq \operatorname{Re}\left(a\right)$, respectively.
  \item $Y_{i}\Big(\overline{\left(x_{2,0}^{-}\right)^2\ldots x_{i,0}^{-}}v_2\Big)\cong W_1\left(a+2\right)$.
\end{enumerate}
The values of $a$ and $b$ can be recovered from the corresponding cases of Proposition \ref{yoc3p1} by replacing $a_1$ by $a_1+1$.
\end{proposition}
\begin{proof}
Removing the first node in the Dynkin diamgram of the Lie algebra of type $D_l$, then we get a simple Lie algebra which is isomorphic to the Lie algebra of type $D_{l-1}$. Denote $I'=\{2,3,\ldots, l\}$.  Let $Y^{\left(1\right)}$ be the Yangian generated by all $x_{j,r}^{\pm}$ and $h_{j,r}$ for $j\in I'$ and $r\in \mathbb{Z}_{\geq 0}$. $Y^{\left(1\right)}\cong Y\left(D_{l-1}\right)$.

From Proposition \ref{rileq2}, the weight of $v_2$ is $\mathbf{s}_1(\omega_i)$, and then $x_{j,0}^{+}v_2$ has weight $\mathbf{s}_1(\omega_i)+\alpha_j$, where $j\geq 2$. If $x_{j,0}^{+}v_2\neq 0$, $\mathbf{s}_1(\omega_i)+\alpha_j$ is a weight of $V_{a_1}(\omega_i)$, so is $\omega_i+\mathbf{s}_1^{-1}\left(\alpha_j\right)$. By the above Lemma, $\omega_i+\mathbf{s}_1^{-1}\left(\alpha_j\right)$ precedes $\omega_i$, which contradicts to the fact that $\omega_i$ is the highest weight. Thus $x_{j,0}^{+}v_2=0$. Therefore $v_2$ is a maximal vector of $Y^{\left(1\right)}$. $v_2$ has weight $-2\omega_1+\omega_i$ and then $h_{j,0}v_2=\delta_{ij}v_2$. therefore
$Y^{\left(1\right)}\left(v_2\right)$ is a highest weight representation with highest weight $P_j=1$ if $j\neq i$ and $P_i=\left(u-a\right)$. It follows from the above lemma that $a=a_1+1$. The rest of the proof of this Proposition is similarly to the proof of Proposition \ref{yoc3p1}, just replacing $a_1$ by $a_1+1$.

Note that if $i=2$, only part $\left(i\right)$ is true and necessary.
\end{proof}
For what follows, we assume that $i\geq 3$.
Let $3\leq m\leq i$. Similarly as the proposition above, we define $Y^{\left(m\right)}$  to be the Yangian generated by all $x_{j,r}^{\pm}$ and $h_{j,r}$ for all $j>m-1$. It is easy to see that $l-m+1\geq l-i+1\geq 2+1=3$.

If $l-m+1\geq 4$, then $Y^{\left(m\right)}$ is isomorphic to Yangian of $D_{l-m+1}$. By  induction on $m$, we have the following results. Again, when $m=i$, only the part $\left(i\right)$ of the below proposition is true and necessary.

\begin{proposition}\label{vmaga1}Let $3\leq m\leq i$ and $v_m=\overline{\left(x_{m-1,0}^{-}\right)^2\ldots x_{i,0}^{-}}v_{m-1}$, where $v_2$ is defined as above.
\begin{enumerate}
  \item $Y_{i+1}\left(x_{i,0}^{-}v_m\right)$, $\ldots$, $Y_{i}\left(\overline{x_{i+1,0}^{-}\ldots  x_{i,0}^{-}}v_m\right)$ are 2-dimensional, which are isomorphic to $W_1\left(a\right)$ with $\operatorname{Re}\left(a\right)\geq \operatorname{Re}\left(a_1\right)$.
  \item $Y_{i-1}\left(\overline{x_{i,0}^{-}\ldots x_{i,0}^{-}}v_m\right)$, $Y_{i-2}\Big(\overline{\left(x_{i-1,0}^{-}\right)^2\ldots x_{i,0}^{-}}v_m\Big)\ldots$, $Y_{m}\Big(\overline{\left(x_{m+1,0}^{-}\right)^2\ldots x_{i,0}^{-}}v_m\Big)$ are either 3-dimensional or 4-dimensional, which is isomorphic to either $W_2\left(a\right)$ or $W_1\left(b\right)\otimes W_1\left(a\right)$, $b\geq a$, respectively.
  \item $Y_{i}\Big(\overline{\left(x_{m,0}^{-}\right)^2\ldots x_{i,0}^{-}}v_{m}\Big)\cong W_1\left(a_1+m\right)$.
\end{enumerate}
\end{proposition}

\begin{remark}
When $i<l-2$, $v^{-}_1=\overline{x_{i,0}^{-}x_{i+1,0}^{-}\ldots  x_{i,0}^{-}}v^{+}_i$.
\end{remark}
If $l-m+1=3$, then $m=i=l-2$. In this case, the above proposition is still true for $v_3,\ldots, v_{l-3}$. By the part $\left(iii\right)$ of the above proposition,
$$Y_{l-2}\Big(v_{l-2}\Big)=Y_{l-2}\Big(\overline{\left(x_{l-3,0}^{-}\right)^2\ldots x_{l-2,0}^{-}}v_{l-3}\Big)\cong W_1\left(a_1+l-3\right).$$ Explicit calculations show that
\begin{lemma}\begin{enumerate}\
               \item $Y_{l-1}\left(x_{l-2,0}^{-}v_{l-2}\right)\cong W_1\left(a_1+l-3+\frac{1}{2}\right)$.
               \item $Y_{l}\left(x_{l-1,0}^{-}x_{l-2,0}^{-}v_{l-2}\right)\cong W_1\left(a_1+l-3+\frac{1}{2}\right)$.
               \item $Y_{l-2}\left(x_{l,0}^{-}x_{l-1,0}^{-}x_{l-2,0}^{-}v_{l-2}\right)\cong W_1\left(a_1+l-3+1\right)$.
             \end{enumerate}
\end{lemma}
\section{On the local Weyl modules of $\yo$}\ \newline
\begin{theorem}[Corollary B, \cite{FoLi}]
As modules of $\nyo[t]$, $$\operatorname{Dim}\Big(W(\lambda)\Big)=\prod_{i}\Big(\operatorname{Dim}\big(W(\omega_i)\big)\Big)^{m_i}.$$ 
\end{theorem}

It follows from Propositions \ref{mtoysl2l} and \ref{vtv'hwv} that

\begin{theorem}\label{wmiatpso}
Let $\pi=\Big(\pi_1\left(u\right),\ldots, \pi_{l}\left(u\right)\Big)$, where $\pi_i\left(u\right)=\prod\limits_{j=1}^{m_i}\left(u-a_{i,j}\right)$. Let $S=\{a_{1,1},\ldots, a_{1,m_1},\ldots, a_{l,1}\ldots, a_{l,m_l}\}$ be the multiset of roots of these polynomials. Let $a_1=a_{m,n}$ be one of the numbers in $S$ with the maximal real part, and let $b_1=m$. Similarly, let $a_r=a_{s,t}\left(r\geq 2\right)$ be one of the numbers in $S\setminus\{a_1, \ldots, a_{r-1}\}$ ($r\geq 2$) with the maximal real part, and $b_r=s$. Let $k=m_1+\ldots+m_l$. Then $L=V_{a_1}(\omega_{b_1})\otimes V_{a_2}(\omega_{b_2})\otimes\ldots\otimes V_{a_k}(\omega_{b_k})$ is a highest weight representation, and its associated polynomial is $\pi$.
\end{theorem}

By comparing the upper bound of dimension of $W(\pi)$ and dimension of $L$ above, we can determine explicitly the local Weyl module $W(\pi)$ and its dimension.
\begin{theorem}
The local Weyl module $W(\pi)$ is isomorphic to $L$ as in Theorem \ref{wmiatpso}. The dimension of $W(\pi)$ can be recovered from Theorem \ref{dofrdl}, Remark \ref{rofry} and Proposition \ref{dcofbdl}.
\end{theorem}

\chapter{Local Weyl modules of $\ysp$}
In this chapter, the local Weyl modules of $\ysp$ are studied. The structure of the local Weyl modules is determined, and the dimensions of the local Weyl modules are obtained.  In the process of characterizing the local Weyl modules, a sufficient condition for a tensor product of fundamental representations of Yangians to be a highest weight representation is obtained, which shall lead to an irreducibility criterion for the tensor product.

Let $\pi=\big(\pi_1(u),\pi_2(u),\ldots, \pi_l(u)\big)$ be a generic $l$-tuple of monic polynomials in $u$, and $\pi_i\left(u\right)=\prod\limits_{j=1}^{m_i}\left(u-a_{i,j}\right)$. Let $k=m_1+m_2+\ldots+m_l$, $S=\{a_{i,j}|i=1,\ldots,l; j=1,\ldots,m_i\}$, and $\lambda=\sum\limits_{i=1}^{l}m_i\omega_i$.
Let $a_{m,n}$ be one of the numbers in $S$ with the maximal real part. Then define $a_1=a_{m,n}$ and $b_1=m$. Inductively, let $a_{s,t}$ be one of the numbers in $S-\{a_1,\ldots, a_{r-1}\}$ with the maximal real part. Then define $a_r=a_{s,t}$ and $b_r=s$.  We prove that the ordered tensor product $L=V_{a_1}\left(\omega_{b_1}\right)\otimes V_{a_2}\left(\omega_{b_2}\right)\otimes \ldots \otimes V_{a_k}\left(\omega_{b_k}\right)$ is a highest weight representation. A standard argument shows that the associated l-tuple of polynomials of $L$ is $\pi$.
Since $L$ is a quotient of $W(\pi)$, a lower bound on the dimension of $W(\pi)$ is obtained.

Let $W(\lambda)$ be the Weyl module associated to $\lambda$ of $\nysp[t]$. In \cite{Na}, the author proved $\operatorname{Dim}\Big(W(\lambda)\Big)=\prod\limits_{i\in I} \Big(\operatorname{Dim}\big(W(\omega_i)\big)\Big)^{m_i}.$ It follows from Corollary \ref{dkrvawocsp} that $\operatorname{Dim}(W(\lambda))=\operatorname{Dim}(L)$. Since $\operatorname{Dim}\big(W(\lambda)\big)\geq \operatorname{Dim}(W(\pi))\geq \operatorname{Dim}(L)$,  $$W\left(\pi\right)\cong L.$$
The dimension of $W(\pi)$ can be recovered from Theorem \ref{fmsimlieC} and Remark \ref{rofry}.

Similar to the proof that $L$ is a highest weight representation (Proposition \ref{mtoyspl}), we can obtain a sufficient condition for a tensor product of fundamental representations of the form $V_{a_1}(\omega_{b_1})\otimes V_{a_2}(\omega_{b_2})\otimes\ldots\otimes V_{a_k}(\omega_{b_k})$ to be a highest weight representation. If $a_j-a_i\notin S(b_i, b_j)$ for $1\leq i<j\leq k$, then $L$ is a highest weight representation, where $S(b_i, b_j)$ is a finite set of positive rational numbers. By Proposition \ref{VoWWoVhi} and Lemma \ref{dualfrc1}, an irreducible criterion for a tensor product of fundamental representations of $\ysp$ is obtained: if $a_j-a_i\notin S(b_i, b_j)$ for $1\leq i\neq j\leq k$, then $L$ is irreducible.


\section{From the highest weight vector to the lowest one in $V_{a}(\omega_i)$}

Let $s_i$ be the fundamental reflections of the Weyl group of $\mathfrak{sp}\left(2l,\C\right)$ for $i=1,\ldots, l$. Let $\{\mu_1,\mu_2,\ldots,\mu_{l}\}$ be the coordinate functions on the Cartan subalgebra of $\nysp$. For $1\leq i\leq l-1$, $s_i\left(\mu_i\right)=\mu_{i+1}$, $s_i\left(\mu_{i+1}\right)=\mu_{i}$ and $s_i\left(\mu_j\right)=\mu_j$ for $j\neq i, i+1$.
For $i=l$, $s_l\left(\mu_{l}\right)=-\mu_{l}$ and $s_l\left(\mu_j\right)=\mu_j$ for $j\neq l$.
The fundamental weights of $\nysp$ are given by $\omega_i=\mu_1+\ldots+\mu_i$ for $1\leq i\leq l$. It follows that $\mu_1=\omega_1$, $\mu_2=-\omega_1+\omega_2$, $\ldots$, $\mu_{l-1}=-\omega_{l-2}+\omega_{l-1}$, and $\mu_{l}=-\omega_{l-1}+\omega_{l}$. Thus $\alpha_1=2\omega_1-\omega_2$, $\alpha_i=-\omega_{i-1}+2\omega_{i}-\omega_{i+1}$ for $2\leq i\leq l-1$, and
$\alpha_{l}=-2\omega_{l-1}+2\omega_{l}$.

Summarize all information above, we have
\begin{proposition}
$s_i\left(\omega_j\right)=\omega_j$ for $i\neq j$. $s_1\left(\omega_1\right)=-\omega_1+\omega_2$, $s_2\left(\omega_2\right)=\omega_1-\omega_2+\omega_3,\ldots,$ $s_{l-2}\left(\omega_{l-2}\right)=\omega_{l-3}-\omega_{l-2}+\omega_{l-1}$, $s_{l-1}\left(\omega_{l-1}\right)=\omega_{l-2}-\omega_{l-1}+\omega_{l}$ and $s_{l}\left(\omega_{l}\right)=2\omega_{l-1}-\omega_{l}$.
\end{proposition}
Let $w_0=-1$ be the longest element of the Weyl group $\W$ of $\nysp$. One reduced expression of the longest element $w_0$  is given by
\begin{align*}
   w_0= \left(s_l\right)&\left(s_{l-1}s_{l}s_{l-1}\right)\left(s_{l-2}s_{l-1}s_ls_{l-1}s_{l-2}\right)\ldots \left(s_{2}\ldots s_{l-2}s_{l-1}s_ls_{l-1}s_{l-2}\ldots s_{2}\right) \\
   & \left(s_{1}\ldots s_{l-2}s_{l-1}s_ls_{l-1}s_{l-2}\ldots s_{1}\right).
\end{align*}

According to the reduced expression of $w_0$, define
\begin{align*}
   w_i=&\left(s_i\ldots s_{l-1}s_ls_{l-1}\ldots s_{i}\right)\ldots\left(s_{2}\ldots s_{i-1}s_i\ldots s_{l-1}s_ls_{l-1}\ldots s_{i}\right)\\
   &\left(s_{1}\ldots s_{i-1}s_i\ldots s_{l-1}s_ls_{l-1}\ldots s_{i}\right).
\end{align*}

Let $\sigma_k\in \W$ be the product of the last $k$ terms in $w_i$ and keep the same order as in $w_i$. For every suitable value $k$, there exists a $k'\in I$ such that $\sigma_{k+1}=s_{k'}\sigma_{k}$. 
Let $v_{\sigma_k(\omega_i)}$ be a vector in the weight space of weight $\sigma_k(\omega_i)$. Since the weight space of weight $\omega_i$ is 1-dimensional, the weight space of weight $\sigma_k(\omega_i)$ is 1-dimensional, and then $v_{\sigma_k(\omega_i)}$ is unique, up to a scalar.

\begin{proposition}\label{siok}
Let $\sigma_k(\omega_i)=r_{k'}\omega_{k'}+\sum\limits_{k'\neq j} r_j\omega_j$. Then $r_{i'}\in\{1,2\}$.
\end{proposition}
\begin{proof}
The proof of this proposition is similar to Proposition \ref{rileq2}. We only provides the detail of the computation of $w_0(\omega_i)$.

Case 1: $i=1$.
\begin{flushleft}
$\omega_1\xlongrightarrow{s_1}-\omega_1+\omega_2\xlongrightarrow{s_2}-\omega_2+\omega_3\xlongrightarrow{s_3}-\omega_3+\omega_4\xlongrightarrow{s_4} \ldots\xlongrightarrow{s_{l-3}}-\omega_{l-3}+\omega_{l-2}\xlongrightarrow{s_{l-2}}-\omega_{l-2}+\omega_{l-1}\xlongrightarrow{s_{l-1}}-\omega_{l-1} +\omega_{l}\xlongrightarrow{s_{l}}\omega_{l-1}-\omega_{l}\xlongrightarrow{s_{l-1}}\omega_{l-2}-\omega_{l-1}\xlongrightarrow{s_2\ldots s_{l-3}s_{l-2}}\omega_{1}-\omega_{2}\xlongrightarrow{s_{1}}-\omega_{1}\xlongrightarrow{\left(s_l\right)\left(s_{l-1}s_{l}s_{l-1}\right)\ldots \left(s_{2}\ldots s_{l-2}s_{l-1}s_ls_{l-1}s_{l-2}\ldots s_{2}\right)}-\omega_{1}$.
\end{flushleft}

Case 2: $1<i\leq l-1$.
\begin{flushleft}
$\omega_i\xlongrightarrow{s_{i-1}\ldots s_1}\omega_i\xlongrightarrow{s_i}\omega_{i-1}-\omega_{i}+\omega_{i+1}\xlongrightarrow{s_{l-3}\ldots s_{i+2}s_{i+1}}\omega_{i-1}-\omega_{l-3}+\omega_{l-2}\xlongrightarrow{s_{l-2}}\omega_{i-1}-\omega_{l-2}+\omega_{l-1}\xlongrightarrow{s_{l-1}} \omega_{i-1}-\omega_{l-1}+\omega_{l}\xlongrightarrow{s_{l}}\omega_{i-1}+\omega_{l-1}-\omega_{l}\xlongrightarrow{s_{l-1}}\omega_{i-1}+
\omega_{l-2}-\omega_{l-1}\xlongrightarrow{s_{i}\ldots s_{l-3}s_{l-2}}2\omega_{i-1}-\omega_{i}\xlongrightarrow{s_{i-1}}2\omega_{i-2}-2\omega_{i-1}+\omega_{i}\xlongrightarrow{s_2\ldots s_{i-3}s_{i-2}}2\omega_{1}-2\omega_{2}+\omega_{i}\xlongrightarrow{s_1}-2\omega_1+\omega_{i}\xlongrightarrow{s_{2}\ldots s_{l-1}s_ls_{l-1}\ldots s_2}-2\omega_2+\omega_{i}\xlongrightarrow{s_{3}\ldots s_{l-1}s_ls_{l-1}\ldots s_3}-2\omega_3+\omega_{i}\xlongrightarrow{\left(s_{i}\ldots s_ls_{l-1}\ldots s_{i}\right)\ldots\left(s_{4}\ldots s_{l-1}s_ls_{l-1}\ldots s_4\right)}-\omega_{i}\xlongrightarrow{s_l\left(s_{l-1} s_ls_{l-1}\right)\ldots\left(s_{i+1}\ldots s_{l-1}s_ls_{l-1}\ldots s_{i+1}\right)}-\omega_{i}$.
\end{flushleft}

Case 3:  $i=l$.
\begin{flushleft}
$\omega_{l}\xlongrightarrow{s_{l-1}\ldots s_1}\omega_{l}\xlongrightarrow{s_{l}}2\omega_{l-1}-\omega_{l}\xlongrightarrow{s_{l-1}}2\omega_{l-2}-2\omega_{l-1}+\omega_{l}
\xlongrightarrow{s_2\ldots s_{l-3}s_{l-2}}2\omega_{1}-2\omega_{2}+\omega_{l}\xlongrightarrow{s_1}-2\omega_{1}+\omega_{l}\xlongrightarrow{s_{2}\ldots s_{l-1}s_{l}s_{l-1}\ldots s_2}-2\omega_{2}+\omega_{l}\xlongrightarrow {\left(s_{l-2}s_{l-1}s_{l}s_{l-1}s_{l-2}\right)\ldots\left(s_{3}\ldots s_{l-1}s_{l}s_{l-1}\ldots s_3\right)}-2\omega_{l-2}+\omega_{l}\xlongrightarrow{s_{l-1}}-2\omega_{l-2}+\omega_{l}\xlongrightarrow{s_{l}}-2\omega_{l-2}+2\omega_{l-1}-\omega_{l} \xlongrightarrow{s_{l-1}}-2\omega_{l-1}+\omega_{l}\xlongrightarrow{s_{l}}-\omega_{l}.$
\end{flushleft}
This proposition immediately follows from above calculations. \end{proof}
\begin{proposition}\label{v+=gv-2}
Let $v_{1}^{+}$ and $v_{1}^{-}$ be highest and lowest weight vectors in the fundamental representation $V_{a}(\omega_i)$ of $\nysp$. There exists $y\in \nysp$ such that $$v_{1}^{-}=y.v_{1}^{+}.$$
\end{proposition}
\begin{proof}
Recall that $V_a(\omega_i)\cong L(\omega_i)$ as a $\nysp$-module, and both $v_{1}^{+}$ and $v_{1}^{-}$ are in $L(\omega_i)$. It follows from Proposition \ref{siok} and formula $s_i\left(\omega\right)=\omega-\frac{2\left(\omega,\alpha_i\right)}{\left(\alpha_i, \alpha_i\right)}\alpha_i$ that  

\begin{align*}
   v^{-}_1&= \Big(x_{i,0}^{-}\ldots x_{l-1,0}^{-}x_{l,0}^{-}x_{l-1,0}^{-}\ldots x_{i,0}^{-}\Big) \left(\left(x_{i-1,0}^{-}\right)^2x_{i,0}^{-}\ldots x_{l,0}^{-}\ldots x_{i,0}^{-}\right)\ldots\\
   &  \left(\left(x_{2,0}^{-}\right)^2\ldots \left(x_{i-1,0}^{-}\right)^2x_{i,0}^{-}\ldots x_{l-1,0}^{-}x_{l,0}^{-}x_{l-1,0}^{-}\ldots x_{i,0}^{-}\right)\\
   &\left(\left(x_{1,0}^{-}\right)^2\ldots \left(x_{i-1,0}^{-}\right)^2x_{i,0}^{-}\ldots x_{l-1,0}^{-}x_{l,0}^{-}x_{l-1,0}^{-}\ldots x_{i,0}^{-}\right)v^{+}_1.
\end{align*}
\end{proof}
\section{On a lower bound for the dimension of the local Weyl module $W(\pi)$}
In this section, we construct a highest weight module which is a tensor product of fundamental representations of $\ysp$. By the maximality of the local Weyl modules, a lower bound for the local Weyl module $W(\pi)$ of $\ysp$ can be obtained.

Recall that $Y_l=span\{x_{l,r}^{\pm}, h_{l,r}|r\in\Z_{\geq 0}\}\cong \ysl$. However, $x_{l,r}^{\pm}, h_{l,r}$ do not satisfy the defining relations of $\ysl$. Therefore, we need to rescale these generators if we want to use the previous chapter results. Let $$\tilde{x}_{l,r}^{\pm}=\frac{\sqrt{2}}{2^{r+1}}x_{l,r}^{\pm},\qquad\tilde{h}_{l,r}=\frac{1}{2^{r+1}}h_{l,r}.$$ Then $span\{\tilde{x}_{l,r}^{\pm}, \tilde{h}_{l,r}|r\in\Z_{\geq 0}\}\cong \ysl$ and the new generators satisfy the defining relations of $\ysl$. $Y_i=span\{x_{i,r}^{\pm}, h_{i,r}|r\in\Z_{\geq 0}\}\cong \ysl$($1\leq i\leq l-1$) and $x_{i,r}^{\pm}, h_{i,r}$ satisfy the defining relations automatically.

In the defining relations of $\ysp$, $d_1=\ldots=d_{l-1}=1$ and $d_{l}=2$.
\begin{proposition}\label{mtoyspl}
Let $L=V_{a_1}(\omega_{b_1})\otimes V_{a_2}(\omega_{b_2})\otimes\ldots\otimes V_{a_k}(\omega_{b_k})$ be a $\ysp$-module, where $\operatorname{Re}\left(a_1\right)\geq \operatorname{Re}\left(a_2\right)\geq \ldots \geq \operatorname{Re}\left(a_k\right)$ and $b_{i}\in\{1,\ldots, l\}$.  Then $L$ is a highest weight representation.
\end{proposition}
\begin{proof}
Let $v_m^{+}$ be the highest weight vectors of $V_{a_m}\left(\omega_{b_m}\right)$. Let $v_1^{-}$ be the lowest weight vector of $V_{a_1}\left(\omega_{b_1}\right)$.

We prove this proposition by induction on $k$.
Without loss of generality, we may assume that $k\geq 2$ and $V_{a_2}(\omega_{b_2})\otimes V_{a_3}(\omega_{b_3})\otimes\ldots\otimes V_{a_k}(\omega_{b_k})$ is a highest weight representation of $\ysp$, and its highest weight vector is $v^{+}=v_2^{+}\otimes \ldots\otimes v_k^{+}$. To show that $L$ is a highest weight representation, it follows from Corollary \ref{v-w+gvtw} that it suffices to show that $$v^{-}_{1}\otimes v^{+}\in \ysp\left(v^{+}_1\otimes v^{+}\right).$$ We divide the proof into the following steps.

Step 1: $\sigma_i^{-1}\left(\alpha_{i'}\right)\in \Delta^{+}$.

Proof: It is easy to check by the definition of $\sigma_i$ and the way we choose $i'$.




Step 2: $Y_{i'}\left(v_{\sigma_i\left(\omega_{b_1}\right)}\right)$ is a highest weight module of $Y_{i'}$.

Proof: Since the weight $\sigma_i\left(\omega_{b_1}\right)$ is on the Weyl group orbit of the highest weight and the representation $V_{a_1}(\omega_{b_1})$ is finite-dimensional, the weight space of weight $\sigma_i\left(\omega_{b_1}\right)$ is 1-dimensional. The elements $h_{j,s}$ form a commutative subalgebra, so $v_{\sigma_i\left(\omega_{b_1}\right)}$ is an eigenvector of $h_{i',r}$. Therefore we only have to show that $v_{\sigma_i\left(\omega_{b_1}\right)}$ is a maximal vector. Suppose to the contrary that $x_{i',k}^{+}v_{\sigma_i\left(\omega_{b_1}\right)}\neq 0$.  Then $x_{i',k}^{+}v_{\sigma_i\left(\omega_{b_1}\right)}$ is a weight vector of weight $\sigma_i(\omega_{b_1})+\alpha_{i'}$, so $\omega_{b_1}+\sigma_i^{-1}\left(\alpha_{i'}\right)$ is a weight. Because $\sigma_i^{-1}\left(\alpha_{i'}\right)\in \Delta^{+}$,  $\omega_{b_1}$ is a weight preceding the weight $\omega_{b_1}+\sigma_i^{-1}\left(\alpha_{i'}\right)$, which contradicts the maximality of $\omega_{b_1}$ in the representation $L\left(\omega_{b_1}\right)$.



Step 3: Let $P\left(u\right)$ be the associated polynomial of  $Y_{i'}\left(v_{\sigma_i\left(\omega_{b_1}\right)}\right)$. Then as a highest weight $Y_{i'}$-module $Y_{i'}\left(v_{\sigma_i\left(\omega_{b_1}\right)}\right)$ has highest weight $\frac{P\left(u+1\right)}{P\left(u\right)}$.

Proof: It follows from the representation theory of $\ysl$.

Step 4: $P\left(u\right)$ has of degree 1 or 2.

Proof: The degree of $P\left(u\right)$ equals the eigenvalue of $v_{\sigma_i\left(\omega_{b_1}\right)}$ under $h_{i',0}$. Note that $$h_{i',0} v_{\sigma_i\left(\omega_{b_1}\right)}=\Big(\sigma_{i}\left(\omega_{b_1}\right)\left(h_{i',0}\right)\Big)v_{\sigma_i\left(\omega_{b_1}\right)}.$$
It follows from Proposition \ref{siok} that the degree of $P\left(u\right)$ equals 1 or 2.


Step 5: If $\operatorname{Deg}\Big(P\left(u\right)\Big)=1$, then for $i'\neq l$, $Y_{i'}\left(v_{\sigma_i\left(\omega_{b_1}\right)}\right)$ is 2-dimensional and isomorphic to $W_1\left(a\right)$; for $i'=l$, $Y_{i'}\left(v_{\sigma_i\left(\omega_{b_1}\right)}\right)$ is 2-dimensional and isomorphic to $W_1\left(\frac{a}{2}\right)$;  If $\operatorname{Deg}\Big(P\left(u\right)\Big)=2$, then $Y_{i'}\left(v_{\sigma_i\left(\omega_{b_1}\right)}\right)$ is either 3-dimensional or 4-dimensional and isomorphic to $W_2\left(a\right)$ or $W_1\left(b\right)\otimes W_1\left(a\right)\left(\operatorname{Re}\left(b\right)>\operatorname{Re}\left(a\right)\right)$, respectively. Moreover $\operatorname{Re}\left(a\right)\geq \operatorname{Re}\left(a_1\right)$. The values of both $b$ and $a$ will be explicitly computed in the next section.

Proof: Let us assume for the moment that this is done (the proof will be given in the next section), and let us proceed to the next step.

Step 6:  $Y_{i'}\left(v_{\sigma_i\left(\omega_{b_1}\right)}\otimes v_2^{+}\otimes\ldots \otimes v_k^{+}\right)=Y_{i'}\left(v_{\sigma_i\left(\omega_{b_1}\right)}\right)\otimes Y_{i'}\left(v_2^{+}\right)\otimes\ldots\otimes Y_{i'}\left(v_k^{+}\right)$.

Proof: For $1\leq i\leq l-1$, let $W$ be one of the modules $W_1\left(a\right)$, $W_2\left(a\right)$ or $W_1\left(b\right)\otimes W_1\left(a\right)$; for $i=l$, let $W$ be the module $W_1\left(\frac{a}{2}\right)$. $Y_{i'}\left(v_m^{+}\right)$ is nontrivial if and only if $b_m=i'$. 
Suppose that in $\{b_1,\ldots,b_k\}$, $b_{j_1}=\ldots=b_{j_{m}}=i'$ with $j_1<\ldots<j_{m}$; moreover, if $s\notin\{j_1,\ldots, j_m\}$, then $b_s\neq i'$.
Note in Step 5 that $Y_{i'}\left(v_{\sigma_i\left(\omega_{b_1}\right)}\right)\cong W$. If $i'\neq l$,
\begin{align*}
   Y_{i'}\left(v_{\sigma_i\left(\omega_{b_1}\right)}\right)\otimes Y_{i'}\left(v_2^{+}\right)&\otimes\ldots\otimes Y_{i'}\left(v_k^{+}\right) \\
   &\cong \begin{cases}
 W\otimes W_1\left(a_{j_1}\right)\otimes \ldots\otimes W_{1}\left(a_{j_m}\right)\qquad \mathrm{if}\qquad b_1\neq i' \\
 W\otimes W_1\left(a_{j_2}\right)\otimes \ldots\otimes W_{1}\left(a_{j_m}\right) \qquad \mathrm{if}\qquad b_1=i';
\end{cases}
\end{align*}
if $i'=l$,
\begin{align*}
    Y_{l}\left(v_{\sigma_i\left(\omega_{b_1}\right)}\right)\otimes Y_{l}\left(v_2^{+}\right)&\otimes\ldots\otimes Y_{l}\left(v_k^{+}\right)\\
   &\cong \begin{cases}
 W\otimes W_1\left(\frac{a_{j_1}}{2}\right)\otimes \ldots\otimes W_{1}\left(\frac{a_{j_m}}{2}\right)\qquad \mathrm{if}\qquad b_1\neq l \\
 W\otimes W_1\left(\frac{a_{j_2}}{2}\right)\otimes \ldots\otimes W_{1}\left(\frac{a_{j_m}}{2}\right) \qquad \mathrm{if}\qquad b_1=l.
\end{cases}
\end{align*}

Since $\operatorname{Re}\left(a\right)\geq \operatorname{Re}\left(a_1\right)\geq \operatorname{Re}\left(a_{j_1}\right)\geq \ldots\geq \operatorname{Re}\left(a_{j_m}\right)$ by Step 5, it follows from either Corollary \ref{tpihw} or Corollary \ref{tpihw2} that $Y_{i'}\left(v_{\sigma_i\left(\omega_{b_1}\right)}\right)\otimes Y_{i'}\left(v_2^{+}\right)\otimes\ldots\otimes Y_{i'}\left(v_k^{+}\right)$ is a highest weight module of $Y_{i'}$ with highest weight vector $v_{\sigma_i\left(\omega_{b_1}\right)}\otimes v_2^{+}\otimes\ldots \otimes v_k^{+}$. Thus $$Y_{i'}\left(v_{\sigma_i\left(\omega_{b_1}\right)}\otimes v_2^{+}\otimes\ldots \otimes v_k^{+}\right)\supseteq Y_{i'}\left(v_{\sigma_i\left(\omega_{b_1}\right)}\right)\otimes Y_{i'}\left(v_2^{+}\right)\otimes\ldots\otimes Y_{i'}\left(v_k^{+}\right).$$ By the coproduct of Yangians and Proposition \ref{c1dgd2d}, it is obvious that $$Y_{i'}\left(v_{\sigma_i\left(\omega_{b_1}\right)}\otimes v_2^{+}\otimes\ldots \otimes v_k^{+}\right)\subseteq Y_{i'}\left(v_{\sigma_i\left(\omega_{b_1}\right)}\right)\otimes Y_{i'}\left(v_2^{+}\right)\otimes\ldots\otimes Y_{i'}\left(v_k^{+}\right).$$ Therefore the claim is true.

Step 7: $v_{\sigma_{i+1}(\omega_{b_1})}\otimes v^{+}\in Y_{i'}\left(v_{\sigma_i\left(\omega_{b_1}\right)}\otimes v^{+}\right)$.

Proof: $v_{\sigma_{i+1}(\omega_{b_1})}\otimes v^{+}\in Y_{i'}\left(v_{\sigma_{i}\omega_{b_1}}\right)\otimes Y_{i'}\left(v^{+}\right)=Y_{i'}\left(v_{\sigma_i\left(\omega_{b_1}\right)}\otimes v^{+}\right)$ by Step 6.

Step 8: $v_1^{-}\otimes v^{+}\in \ysp \left(v_1^{+}\otimes v^{+}\right)$.

Proof: It follows immediately from Step 7 by induction on the subscript $i$ of $v_{\sigma_{i}(\omega_{b_1})}$.

It follows from Step 8 and Corollary \ref{v-w+gvtw} that  $L=\ysp \left(v_1^{+}\otimes v^{+}\right)$.
\end{proof}


\begin{remark}
The following is the record of the root (s) of the associated polynomials of $Y_{i'}\left(v_{\sigma_i\left(\omega_{b_1}\right)}\right)$ for the possible $i$ values. We refer to Remark \ref{yocap1} for the notation used below.

When $1\leq b_1\leq l-1$,

\noindent $a_1\xlongrightarrow{x_{b_1,0}^{-}}a_1+\frac{1}{2}\xlongrightarrow{x_{b_1+1,0}^{-}}a_1+1\xlongrightarrow{x_{b_1+2,0}^{-}}\ldots \xlongrightarrow{x_{l-1,0}^{-}}\frac{1}{2}\left(a_1+\frac{l-b_1+1}{2}\right) \xlongrightarrow{x_{l,0}^{-}}a_1+\frac{l-b_1+3}{2} \xlongrightarrow{x_{l-1,0}^{-}}a_1+\frac{l-b_1+4}{2}\xlongrightarrow{x_{l-2,0}^{-}}\ldots \xlongrightarrow{x_{b_1+1,0}^{-}}a_1+\frac{l-b_1+3}{2}+\frac{l-b_1-1}{2}=a_1+l-b_1+1 \xlongrightarrow{x_{b_1,0}^{-}}\left(a_1+l-b_1+\frac{3}{2},a_1+\frac{1}{2}\right)\xlongrightarrow{\left(x_{b_1-1,0}^{-}\right)^2}\left(a_1+l-b_1+2, a_1+1\right)
\xlongrightarrow{\left(x_{b_1-2,0}^{-}\right)^2}\ldots \xlongrightarrow{\left(x_{2,0}^{-}\right)^2}\left(a_1+l-b_1+1+\frac{b_1-1}{2}=a_1+l-\frac{b_1}{2}+\frac{1}{2}, a_1+\frac{b_1-1}{2}\right)\xlongrightarrow{\left(x_{1,0}^{-}\right)^2}$

\noindent(The second parenthesis) $\left(a_1+1\right)\xlongrightarrow{x_{b_1,0}^{-}}\left(a_1+1\right)+\frac{1}{2}\xlongrightarrow{x_{b_1+1,0}^{-}}\ldots \xlongrightarrow{\left(x_{3,0}^{-}\right)^2}$\newline $\left(a_1+l-\frac{b_1}{2}+1, a_1+\frac{b_1}{2}\right)\xlongrightarrow{\left(x_{2,0}^{-}\right)^2}$


$\ldots$

\noindent(The last parenthesis) $\left(a_1+b_1-1\right)\xlongrightarrow{x_{b_1,0}^{-}}\left(a_1+b_1-1\right)+\frac{1}{2}\xlongrightarrow{x_{b_1+1,0}^{-}}\ldots \xlongrightarrow{x_{b_1+1,0}^{-}}a_1+l\xlongrightarrow{x_{b_1,0}^{-}}\text{the lowest weight vector reached}$.

When $b_1=l$,

\noindent$\frac{a_1}{2}\xlongrightarrow{x_{l,0}^{-}}\left(a_1+1,a_1\right)\xlongrightarrow{\left(x_{l-1,0}^{-}\right)^2} \left(a_1+\frac{3}{2},a_1+\frac{1}{2}\right) \xlongrightarrow{\left(x_{l-2,0}^{-}\right)^2}\ldots \xlongrightarrow{\left(x_{2,0}^{-}\right)^2}\\ \left(a_1+\frac{l}{2}, a_1+\frac{l-2}{2}\right)\xlongrightarrow{\left(x_{1,0}^{-}\right)^2}$

\noindent(The second parenthesis) $\frac{a_1+1}{2}\xlongrightarrow{x_{l,0}^{-}}\Big(\left(a_1+1\right)+1, a_1+1\Big)\xlongrightarrow{\left(x_{l-1,0}^{-}\right)^2}\\ \left(a_1+\frac{5}{2},a_1+\frac{3}{2}\right) \xlongrightarrow{\left(x_{l-2,0}^{-}\right)^2}\ldots \xlongrightarrow{\left(x_{3,0}^{-}\right)^2}\left(a_1+\frac{l+1}{2}, a_1+\frac{l-1}{2}\right)\xlongrightarrow{\left(x_{2,0}^{-}\right)^2}$

$\ldots$

\noindent(The last two parentheses) $\frac{a_1+l-2}{2}\xlongrightarrow{x_{l,0}^{-}}\left(a_1+l-1, a_1+l-2\right)\xlongrightarrow{\left(x_{l-1,0}^{-}\right)^2}\\ \frac{a_1+l-1 }{2}\xlongrightarrow{x_{l,0}^{-}}\text{the lowest weight vector reached}$.
\end{remark}

A precise condition for the cyclicity of the tensor product $L$ can be obtained in the next Theorem. We first show the following lemma. Since the proof is similar to the proof of Lemma \ref{C3l317ss}, we omit the proof.
\begin{lemma}
$V_{a_m}\left(\omega_{b_m}\right)\otimes V_{a_n}\left(\omega_{b_n}\right)$ is a highest weight representation if $a_n-a_m\notin S\left(b_m,b_n\right)$, where the set $S\left(b_m,b_n\right)$ is defined as follows:
\begin{enumerate}
  \item $S\left(b_m,b_n\right)=\left\{\frac{|b_m-b_n|}{2}+1+r, l+2+r-\frac{b_m+b_n}{2}|0\leq r<\text{min}\{b_m, b_n\}\right\}$,\newline where $1\leq b_m, b_n\leq l-1$;
  \item  $S\left(l,b_n\right)=\left\{\frac{l-b_n+1}{2}+1+r, \frac{l-b_n-1}{2}+1+r|0\leq r<b_n, 1\leq b_n\leq l-1\right\}$;
  \item $S\left(b_m,l\right)=\left\{\frac{l-b_m+1}{2}+2+r|0\leq r< b_m, 1\leq b_m\leq l-1\right\};$
   \item $S\left(l,l\right)=\left\{2,3,\ldots,l+1\right\}$.
\end{enumerate}
\end{lemma}

Similar to the proof of Theorem \ref{3mt2icl}, we can prove the following Theorem.
\begin{theorem}\label{yspmt2}
Let $L=V_{a_1}(\omega_{b_1})\otimes V_{a_2}(\omega_{b_2})\otimes\ldots\otimes V_{a_k}(\omega_{b_k})$, and $S(b_i, b_j)$ be defined as the above lemma.
\begin{enumerate}
  \item If $a_j-a_i\notin S(b_i, b_j)$ for $1\leq i<j\leq k$, then $L$ is a highest weight representation of $\ysp$.
  \item If $a_j-a_i\notin S(b_i, b_j)$ for $1\leq i\neq j\leq k$, then $L$ is an irreducible representation of $\ysp$.
\end{enumerate}
\end{theorem}
\begin{remark}\label{c5yspris1}
It can be seen from Theorem \ref{yspmt2} that the irreducibility of $V_{a_{i}}\left(\omega_{b_{i}}\right)\otimes V_{a_{j}}\left(\omega_{b_{j}}\right)$ is closely related to the difference of $a_j-a_i$. However, it is not known what is the precise necessary and sufficient condition for the irreducibility of the tensor product $V_{a_{i}}\left(\omega_{b_{i}}\right)\otimes V_{a_{j}}\left(\omega_{b_{j}}\right)$. Our result on the cyclicity condition for $L$ is an analogue of a special case $m_1=m_2=1$ of the results in \cite{Ch3}. We now give a detail interpretation. Note that there is a different labeling on the nodes of the Dynkin Diagram and we transfer it to our labeling. The following come from case (iii) of Corollary 6.2 \cite{Ch3}.
\begin{align*}
\mathcal{S}(l,l)&=q^6\{q^0, q^2, \ldots, q^{2l-2}\}=\{q^6, q^8, \ldots, q^{2l+4}\}.\\
\mathcal{S}(l-i+1,l)&=q^4\{q^{i-1}, q^{i+1}, \ldots, q^{2l-1-i}\}=\{q^{i+3}, q^{i+5}, \ldots, q^{2l+3-i}\}.\\
\mathcal{S}(l,l-i+1)&=q^4\{q^{i-1}, q^{i+1}, \ldots, q^{2l+1-i}\}=\{q^{i+3}, q^{i+5}, \ldots, q^{2l+5-i}\}.\\
\mathcal{S}(l-i_1+1,l-i_2+1)&=\mathcal{S}(l-i_2+1,l-i_1+1)\left(i_1\leq i_2\right)\\
&=q^2\{q^{i_2-i_1}, q^{i_2+i_1}, \ldots, q^{2l-i_1-i_2}, q^{2l-i_2+i_1}\}\\
&=\{q^{i_2-i_1+2}, q^{i_2+i_1+2}, \ldots, q^{2l-i_1-i_2+2}, q^{2l-i_2+i_1+2}\}.\\
\mathcal{S}(l-i_1+1,l-i_2+1)&=\mathcal{S}(l-i_2+1,l-i_1+1)\left(i_1\geq i_2\right)\\
&=q^2\{q^{i_1-i_2}, q^{i_1+i_2}, \ldots, q^{2l-i_1-i_2}, q^{2l-i_1+i_2}\}\\
&=\{q^{i_1-i_2+2}, q^{i_1+i_2+2}, \ldots, q^{2l-i_1-i_2+2}, q^{2l-i_1+i_2+2}\}.
\end{align*}
For the reader's convenience we give a parallel result in the Yangian case by making the sets $S(b_i,b_j)$ more explicit.
\begin{align*}
S\left(l,l\right)&=\left\{2,3,\ldots,l+1\right\}\\
S\left(l-i+1,l\right)&=\left\{\frac{i}{2}+2,\ldots, l+2-\frac{i}{2}\right\}.\\
S\left(l,l-i+1\right)&=\left\{\frac{i}{2},\frac{i}{2}+1,\ldots, l+1-\frac{i}{2}\right\}.\\
S(l-i_1+1&,\ l-i_2+1)=S(l-i_2+1,l-i_1+1)\left(i_1\leq i_2\right)\\
&=\{\frac{i_2-i_1}{2}+1, \frac{i_2+i_1}{2}+1, \ldots,l-\frac{i_2+i_1}{2}+1, l-\frac{i_2-i_1}{2}+1 \}.\\
S(l-i_1+1&,\ l-i_2+1)=S(l-i_2+1,l-i_1+1)\left(i_1\geq i_2\right)\\
&=\{\frac{i_1-i_2}{2}+1, \frac{i_1+i_2}{2}+1, \ldots,l-\frac{i_1+i_2}{2}+1, l-\frac{i_1-i_2}{2}+1\}.
\end{align*}
Note that the numbers in the sets $S(l-i_1+1,l-i_2+1)$ are the exponents of $q$ in $\mathcal{S}(l-i_1+1,l-i_2+1)$ up to a factor of 2 and a constant.
\end{remark}
\section{Supplement Step 5 of Proposition \ref{mtoyspl}}\label{mtoyspls}


We prove Step 5 by three cases: $2\leq b_1\leq l-1$, $b_1=1$ and $b_1=l$. For the simplicity of the notation, let $b_1=i$.

\noindent\textbf{Case 1:  $2\leq i\leq l-1$.}
\begin{proposition}\label{spc1il-1p1} Let $i\leq k\leq l-1$ and $1\leq m\leq i-2$.
\begin{enumerate}
  \item $Y_{k}\Big( {\left(x_{k-1,0}^{-}x_{k-2,0}^{-}\ldots x_{i,0}^{-}\right)}v^{+}_1\Big)\cong W_{1}\left(a_1+\frac{k-i}{2}\right)$ as a $Y_{k}$-module for.
  \item $Y_{l}\Big( {x_{l-1,0}^{-}x_{l-2,0}^{-}\ldots x_{i,0}^{-}}v^{+}_1\Big)\cong W_1\Big(\frac{1}{2}\left(a_1+\frac{l-i+1}{2}\right)\Big)$ as a $Y_{l}$-module.
  \item $Y_{k}\left(\overline{x_{k+1,0}^{-}\ldots x_{i,0}^{-}}v^{+}_1\right)\cong W_1\left(a_1+\frac{2l-i-k+2}{2}\right)$ as a $Y_{k}$-module.
  \item
  \begin{enumerate}
    \item $Y_{i-1}\Big(\overline{\left(x_{i,0}^{-}\right)\ldots x_{i,0}^{-}}v^{+}_1\Big)\cong W_1\left(a_1+l-i+\frac{3}{2}\right)\otimes W_1\left(a_1+\frac{1}{2}\right)$.
    \item $Y_m\Big(\overline{\left(x_{m+1,0}^{-}\right)^2\ldots x_{i,0}^{-}}v^{+}_1\Big)\cong W_1\left(a_1+\frac{2l-i-m+2}{2}\right)\otimes W_1\left(a_1+\frac{i-m}{2}\right)$.
  \end{enumerate}
 \item $Y_i\Big(\overline{\left(x_{1,0}^{-}\right)^2\ldots x_{i,0}^{-}}v^{+}_1\Big)\cong W_1\left(a_1+1\right)$.
\end{enumerate}
\end{proposition}
\begin{proof}
The first item can be proved similarly as item (i) of Proposition \ref{yoc3p1}, so we omit the proof. Lemma \ref{ylviw2a0} is devoted to prove the second item, and Lemmas \ref{ylviw2a} and \ref{yspc1il4} is for the third item. The fourth is proved in Lemmas \ref{yspc1i5}, \ref{yspc1i6} and \ref{yspc1i7}. The item (v) will be showed in Lemma \ref{spv2}.
\end{proof}

\begin{lemma}\label{ylviw2a0}
$Y_{l}\Big( {x_{l-1,0}^{-}x_{l-2,0}^{-}\ldots x_{i,0}^{-}}v^{+}_1\Big)\cong W_1\Big(\frac{1}{2}\left(a_1+\frac{l-i+1}{2}\right)\Big)$.
\end{lemma}
\begin{proof}


  By item (i) of this proposition, $$h_{l-1,1} {x_{l-2,0}^{-}\ldots x_{i,0}^{-}}v^{+}_1=\left(a_1+\frac{l-i-1}{2}\right) {x_{l-2,0}^{-}\ldots x_{i,0}^{-}}v^{+}_1,$$ and by the (\ref{Cor3.3}) then $$x_{l-1,1}^{-} {x_{l-2,0}^{-}\ldots x_{i,0}^{-}}v^{+}_1=\left(a_1+\frac{l-i-1}{2}\right) {x_{l-1,0}^{-}\ldots x_{i,0}^{-}}v^{+}_1.$$
It follows from Proposition \ref{siok} that $wt(x_{l-1,0}^{-}x_{l-2,0}^{-}\ldots x_{i,0}^{-}v^{+}_1)=\omega_{i-1}-\omega_{l-1}+\omega_{l}.$
Thus
\begin{align*}
\tilde{h}_{l,0}x_{l-1,0}^{-}x_{l-2,0}^{-}\ldots x_{i,0}^{-}v^{+}_1 &=x_{l-1,0}^{-}x_{l-2,0}^{-}\ldots x_{i,0}^{-}v^{+}_1.
\end{align*}
Therefore the degree of associated polynomial equals 1. Let $P\left(u\right)=u-a$ be the associated polynomial. The eigenvalue of $x_{l-1,0}^{-}x_{l-2,0}^{-}\ldots x_{i,0}^{-}v^{+}_1$ under $h_{l,1}$ will tell the value of $a$.
\begin{align*}
  \tilde{h}_{l,1}&\left(x_{l-1,0}^{-}x_{l-2,0}^{-}\ldots x_{i,0}^{-}\right)v^{+}_1\\
 &= \frac{1}{4}h_{l,1}\left(x_{l-1,0}^{-}x_{l-2,0}^{-}\ldots x_{i,0}^{-}\right)v^{+}_1\\
 &=  \frac{1}{4}\left(2x_{l-1,1}^{-}+h_{l,0}x_{l-1,0}^{-}+x_{l-1,0}^{-}h_{l,0}\right)x_{l-2,0}^{-}\ldots x_{i,0}^{-}v^{+}_1 \\
 &=  \frac{1}{4}\left(2x_{l-1,1}^{-}+h_{l,0}x_{l-1,0}^{-}\right)x_{l-2,0}^{-}\ldots x_{i,0}^{-}v^{+}_1 \\
 &=  \frac{1}{2}\left(a_1+\frac{l-i-1}{2}+1\right)x_{l-1,0}^{-}x_{l-2,0}^{-}\ldots x_{i,0}^{-}v^{+}_1\\
  &=  \frac{1}{2}\left(a_1+\frac{l-i+1}{2}\right)x_{l-1,0}^{-}x_{l-2,0}^{-}\ldots x_{i,0}^{-}v^{+}_1.
\end{align*}
Note that $\tilde{x}_{l,0}^{-}=\frac{\sqrt{2}}{2}x_{l,0}^{-}$.
\begin{align*}
    x_{l,1}^{-}x_{l-1,0}^{-}\ldots x_{i,0}^{-}v^{+}_1
   &= 2\sqrt{2}\tilde{x}_{l,1}^{-}x_{l-1,0}^{-}\ldots x_{i,0}^{-}v^{+}_1\\
   &= \sqrt{2}\left(a_1+\frac{l-i+1}{2}\right)\tilde{x}_{l,0}^{-}x_{l-1,0}^{-}\ldots x_{i,0}^{-}v^{+}_1\\
   &= \left(a_1+\frac{l-i+1}{2}\right)x_{l,0}^{-}x_{l-1,0}^{-}\ldots x_{i,0}^{-}v^{+}_1.
\end{align*}
\end{proof}

\begin{lemma}\label{ylviw2a}
$Y_{l-1}\left(x_{l,0}^{-}x_{l-1,0}^{-}\ldots x_{i,0}^{-}v^{+}_1\right)$ is a $Y_{l-1}$-modules, which is isomorphic to the module $W_{1}\left(a_1+\frac{l-i+3}{2}\right)$.
\end{lemma}

\begin{proof} It follows from Proposition \ref{siok} that $wt\left(x_{l,0}^{-}x_{l-1,0}^{-}\ldots x_{i,0}^{-}\right)v^{+}_1=\omega_{i-1}+\omega_{l-1}-\omega_{l}.$ Then $$h_{l-1,0}\left(x_{l,0}^{-}x_{l-1,0}^{-}\ldots x_{i,0}^{-}\right)v^{+}_1=\left(x_{l,0}^{-}x_{l-1,0}^{-}\ldots x_{i,0}^{-}\right)v^{+}_1.$$
Therefore the degree of associated polynomial $P\left(u\right)$ equals 1. Let $P\left(u\right)=u-a$ be the associated polynomial. By the theory of the local Weyl modules of $\ysl$ that the eigenvalue of $x_{l,0}^{-}x_{l-1,0}^{-}\ldots x_{i,0}^{-}v^{+}_1$ under $h_{l-1,1}$ will tell the value of $a$.
\begin{align*}
   h_{l-1,1}\big(x_{l,0}^{-}x_{l-1,0}^{-}&\ldots x_{i,0}^{-}\big)v^{+}_1 \\
   &= [h_{l-1,1}x_{l,0}^{-}]x_{l-1,0}^{-}\ldots x_{i,0}^{-}v^{+}_1+x_{l,0}^{-}h_{l-1,1}x_{l-1,0}^{-}\ldots x_{i,0}^{-}v^{+}_1 \\
   &= \left(2x_{l,1}^{-}+2x_{l,0}^{-}+2x_{l,0}^{-}h_{l-1,0}\right)x_{l-1,0}^{-}\ldots x_{i,0}^{-}v^{+}_1\\
   &-\left(a_1+\frac{l-i-1}{2}\right)x_{l,0}^{-}x_{l-1,0}^{-}\ldots x_{i,0}^{-}v^{+}_1\\
   &= 2x_{l,1}^{-}x_{l-1,0}^{-}\ldots x_{i,0}^{-}v^{+}_1-\left(a_1+\frac{l-i-1}{2}\right)x_{l,0}^{-}x_{l-1,0}^{-}\ldots x_{i,0}^{-}v^{+}_1\\
   &= 2\left(a_1+\frac{l-i+1}{2}\right)x_{l,0}^{-}x_{l-1,0}^{-}\ldots x_{i,0}^{-}v^{+}_1\\ &-\left(a_1+\frac{l-i-1}{2}\right)x_{l,0}^{-}x_{l-1,0}^{-}\ldots x_{i,0}^{-}v^{+}_1\\
   &=\left(a_1+\frac{l-i+3}{2}\right)x_{l,0}^{-}x_{l-1,0}^{-}\ldots x_{i,0}^{-}v^{+}_1.
\end{align*}
Thus $Y_{l-1}\left(x_{l,0}^{-}x_{l-1,0}^{-}\ldots x_{i+1,0}^{-}x_{i,0}^{-}v^{+}_1\right)\cong W_1\left(a_1+\frac{l-i+3}{2}\right)$.
\end{proof}

The first item of next lemma does not show up at the case of $\yo$, we provide a proof.
\begin{lemma}\label{yspc1il4}
Suppose $i\leq k\leq l-2$. $Y_{k}\left(x_{k+1,0}^{-}\ldots x_{l-1,0}^{-}x_{l,0}^{-}x_{l-1,0}^{-}\ldots x_{i,0}^{-}v^{+}_1\right)\cong W_1\left(a_1+\frac{2l-i-k+2}{2}\right)$ as a $Y_{k}$-module.

\end{lemma}
\begin{proof} We first consider $k=l-2$.
It follows from Proposition \ref{siok} that $$wt\left(x_{l-1,0}^{-}x_{l,0}^{-}x_{l-1,0}^{-}\ldots x_{i,0}^{-}\right)v^{+}_1=\omega_{i-1}+\omega_{l-2}-\omega_{l-1}.$$ Then
$$h_{l-2,0}x_{l-1,0}^{-}x_{l,0}^{-}x_{l-1,0}^{-}\ldots x_{i,0}^{-}v^{+}_1=x_{l-1,0}^{-}x_{l,0}^{-}x_{l-1,0}^{-}\ldots x_{i,0}^{-}v^{+}_1.$$
Therefore the degree of associated polynomial $P\left(u\right)$ equals 1. Let $P\left(u\right)=u-a$ be the associated polynomial. By the theory of the local Weyl modules of $\ysl$ that the eigenvalue of $x_{l-1,0}^{-}x_{l,0}^{-}x_{l-1,0}^{-}\ldots x_{i,0}^{-}v^{+}_1$ under $h_{l-2,1}$ will tell the value of $a$.

It follows from Lemma \ref{so1c1lf1} that $$h_{l-2,1}x_{l-1,0}^{-}\ldots x_{i,0}^{-}v^{+}_1=0.$$
\begin{align*}
   h_{l-2,1}&x_{l-1,0}^{-}x_{l,0}^{-}x_{l-1,0}^{-}\ldots x_{i,0}^{-}v^{+}_1 \\
   &=[h_{l-2,1},x_{l-1,0}^{-}]x_{l,0}^{-}x_{l-1,0}^{-}\ldots x_{i,0}^{-}v^{+}_1+x_{l-1,0}^{-}x_{l,0}^{-}h_{l-2,1}x_{l-1,0}^{-}x_{l-2,0}^{-}\ldots x_{i,0}^{-}v^{+}_1   \\
   &= \left(x_{l-1,1}^{-}+\frac{1}{2}x_{l-1,0}^{-}+x_{l-1,0}^{-}h_{l-2,0}\right)x_{l,0}^{-}x_{l-1,0}^{-}\ldots x_{i,0}^{-}v^{+}_1+0\\
   &= \Bigg(\left(a_1+\frac{l-i+3}{2}\right)+\frac{1}{2}\Bigg)x_{l-1,0}^{-}x_{l,0}^{-}x_{l-1,0}^{-}\ldots x_{i,0}^{-}v^{+}_1\\
   &= \left(a_1+\frac{l-i+4}{2}\right)x_{l-1,0}^{-}x_{l,0}^{-}x_{l-1,0}^{-}\ldots x_{i,0}^{-}v^{+}_1.
\end{align*}

For $i\leq k\leq l-3$, we omit the proof since it is similar to the proof of (iii) of Proposition \ref{yoc3p1}.
\end{proof}

We introduce some notations to simplify the expression of $v_{\sigma_j(\omega_i)}$. Denote $x_{k,0}^{-}x_{k+1,0}^{-}\ldots x_{l-1,0}^{-}x_{l,0}^{-}x_{l-1,0}^{-}\ldots x_{i+1,0}^{-}x_{i,0}^{-}$ by $\overline{x_{k,0}^{-}\ldots x_{i,0}^{-}}$ for $i\leq k<l$. The overline means that the expression $\overline{x_{k,0}^{-}\ldots x_{i,0}^{-}}$ contains the term $x_{l,0}^{-}$ and is a piece of $x_{i,0}^{-}x_{i+1,0}^{-}\ldots x_{l-1,0}^{-}x_{l,0}^{-}x_{l-1,0}^{-}\ldots x_{i+1,0}^{-}x_{i,0}^{-}$ without break. Similarly, we denote $\left(x_{k,0}^{-}\right)^2\left(x_{k+1,0}^{-}\right)^2\ldots \left(x_{i-1,0}^{-}\right)^2x_{i,0}^{-}$ $x_{i+1,0}^{-}\ldots x_{l-1,0}^{-}x_{l,0}^{-}x_{l-1,0}^{-}\ldots x_{i+1,0}^{-}x_{i,0}^{-}$ for $k<i$ by $\overline{\left(x_{k,0}^{-}\right)^2\ldots x_{i,0}^{-}}$.

\begin{lemma}\label{yspc1i5}
$Y_{i-1}\left(\overline{x_{i,0}^{-}\ldots x_{i,0}^{-}}v^{+}_1\right)\cong W_1\left(a_1+l-i+\frac{3}{2}\right)\otimes W_1\left(a_1+\frac{1}{2}\right)$.
\end{lemma}
\begin{proof}
It follows from Proposition \ref{siok} that $wt\left(x_{i,0}^{-}\ldots x_{l-1,0}^{-}x_{l,0}^{-}x_{l-1,0}^{-}\ldots x_{i,0}^{-}\right)v^{+}_1=2\omega_{i-1}-\omega_{i}.$ Then $$h_{i-1,0}\overline{x_{i,0}^{-}\ldots x_{i,0}^{-}}v^{+}_1= 2\overline{x_{i,0}^{-}\ldots x_{i,0}^{-}}v^{+}_1.$$ Thus the associated polynomial $P\left(u\right)$ to $Y_{i-1}\left(\overline{x_{i,0}^{-}\ldots x_{i,0}^{-}}v^{+}_1\right)$ has of degree 2. Suppose $P\left(u\right)=\left(u-a\right)\left(u-b\right)$ with $\operatorname{Re}\left(a\right)\leq \operatorname{Re}\left(b\right)$.
The eigenvalues of $\overline{x_{i,0}^{-}\ldots x_{i,0}^{-}}v^{+}_1$ under $h_{i-1,1}$ and $h_{i-1,2}$ will tell the values of $a$ and $b$.
\begin{align*}
   h_{i-1,1}&\overline{x_{i,0}^{-}\ldots x_{i,0}^{-}}v^{+}_1\\
   &= [h_{i-1,1},x_{i,0}^{-}]\overline{x_{i+1,0}^{-}\ldots x_{i,0}^{-}}v^{+}_1 +\overline{x_{i,0}^{-}\ldots x_{i+1,0}^{-}}[h_{i-1,1},x_{i,0}^{-}]v^{+}_1\\
    &= \left(x_{i,1}^{-}+\frac{1}{2}x_{i,0}^{-}+x_{i,0}^{-}h_{i-1,0}\right)\overline{x_{i+1,0}^{-}\ldots x_{i,0}^{-}}v^{+}_1\\
     &+\overline{x_{i,0}^{-}\ldots x_{i+1,0}^{-}}\left(x_{i,1}^{-}+\frac{1}{2}x_{i,0}^{-}+x_{i,0}^{-}h_{i-1,0}\right)v^{+}_1 \\
     &= \Bigg(\left(a_1+\frac{2l-i-i+2}{2}+\frac{1}{2}+1\right)+\left(a_1+\frac{1}{2}\right)\Bigg)\overline{x_{i,0}^{-}\ldots x_{i,0}^{-}}v^{+}_1\\
      &= \Big(2a_1+\left(l-i\right)+3\Big)\overline{x_{i,0}^{-}\ldots x_{i,0}^{-}}v^{+}_1.
\end{align*}
\begin{align*}
    h_{i-1,2}&\overline{x_{i,0}^{-}\ldots x_{i,0}^{-}}v^{+}_1 \\
   &= [h_{i-1,2},x_{i,0}^{-}]\overline{x_{i+1,0}^{-}\ldots x_{i,0}^{-}}v^{+}_1+\overline{x_{i,0}^{-}\ldots x_{i+1,0}^{-}}[h_{i-1,2},x_{i,0}^{-}]v^{+}_1\\
    &= \left([h_{i-1,1}, x_{i,1}^{-}]+\frac{1}{2}\left(h_{i-1,1}x_{i,0}^{-}+x_{i,0}^{-}h_{i-1,1}\right)\right)\overline{x_{i+1,0}^{-}\ldots x_{i,0}^{-}}v^{+}_1 \\
     &+\overline{x_{i,0}^{-}\ldots x_{i+1,0}^{-}}\left([h_{i-1,1}, x_{i,1}^{-}]+\frac{1}{2}\left(h_{i-1,1}x_{i,0}^{-}+x_{i,0}^{-}h_{i-1,1}\right)\right)v^{+}_1 \\
     &= [h_{i-1,1}, x_{i,1}^{-}]\overline{x_{i+1,0}^{-}\ldots x_{i,0}^{-}}v^{+}_1+\frac{1}{2}h_{i-1,1}\overline{x_{i,0}^{-}\ldots x_{i,0}^{-}}v^{+}_1\\
      &+ \overline{x_{i,0}^{-}\ldots x_{i+1,0}^{-}}h_{i-1,1}x_{i,0}^{-}v^{+}_1+\overline{x_{i,0}^{-}\ldots x_{i+1,0}^{-}}h_{i-1,1}x_{i,1}^{-}v^{+}_1\\
       &= h_{i-1,1}x_{i,1}^{-}\overline{x_{i+1,0}^{-}\ldots x_{i,0}^{-}}v^{+}_1-x_{i,1}^{-}h_{i-1,1}\overline{x_{i+1,0}^{-}\ldots x_{i,0}^{-}}v^{+}_1+\frac{1}{2}h_{i-1,1}\overline{x_{i,0}^{-}\ldots x_{i,0}^{-}}v^{+}_1\\
      &+ \overline{x_{i,0}^{-}\ldots x_{i+1,0}^{-}}h_{i-1,1}x_{i,0}^{-}v^{+}_1+\overline{x_{i,0}^{-}\ldots x_{i+1,0}^{-}}h_{i-1,1}x_{i,1}^{-}v^{+}_1\\
     &= \left(a_1+l-i+1\right)h_{i-1,1}\overline{x_{i,0}^{-}\ldots x_{i,0}^{-}}v^{+}_1-\overline{x_{i,1}^{-}\ldots x_{i+1,0}^{-}} h_{i-1,1}x_{i,0}^{-}v^{+}_1\\
      &+ \frac{1}{2}h_{i-1,1}\overline{x_{i,0}^{-}\ldots x_{i,0}^{-}}v^{+}_1+\overline{x_{i,0}^{-}\ldots x_{i+1,0}^{-}}h_{i-1,1}x_{i,0}^{-}v^{+}_1+\overline{x_{i,0}^{-}\ldots x_{i+1,0}^{-}}h_{i-1,1}x_{i,1}^{-}v^{+}_1\\
       &= 2\left(a_1+l-i+1\right)\left(a_1+\frac{l-i+3}{2}\right)\overline{x_{i,0}^{-}\ldots x_{i,0}^{-}}v^{+}_1\\ &-\left(a_1+\frac{1}{2}\right)\left(a_1+l-i+1\right)\overline{x_{i,0}^{-}\ldots x_{i,0}^{-}}v^{+}_1+\left(a_1+\frac{l-i+3}{2}\right)\overline{x_{i,0}^{-}\ldots x_{i,0}^{-}}v^{+}_1\\
       &+\left(a_1+\frac{1}{2}\right)\overline{x_{i,0}^{-}\ldots x_{i,0}^{-}}v^{+}_1+ a_1\left(a_1+\frac{1}{2}\right)\overline{x_{i,0}^{-}\ldots x_{i,0}^{-}}v^{+}_1\\
      &= \Bigg(2\left(a_1+l-i+\frac{3}{2}\right)\left(a_1+\frac{l-i+3}{2}\right)\Bigg)\overline{x_{i,0}^{-}\ldots x_{i,0}^{-}}v^{+}_1\\
       &- \Bigg(\left(a_1+\frac{1}{2}\right)\left(a_1+l-i+1-a_1-1\right)\Bigg)\overline{x_{i,0}^{-}\ldots x_{i,0}^{-}}v^{+}_1\\
      &= \Bigg(2\left(a_1+\frac{l-i+3}{2}\right)^2+\frac{\left(l-i+2\right)\left(l-i\right)}{2}\Bigg)\overline{x_{i,0}^{-}\ldots x_{i,0}^{-}}v^{+}_1.
\end{align*}
We have $a+b+1=\Big(2a_1+\left(l-i\right)+3\Big)$ and $\left(a^2+b^2+a+b\right)=2\left(a_1+\frac{l-i+3}{2}\right)^2+\frac{\left(l-i+2\right)\left(l-i\right)}{2}.$
We claim that $\operatorname{Re}(a)\geq \operatorname{Re}(a_1)$. Suppose $a=a_1+\frac{l-i+3}{2}-x$, and then $b=a_1+\frac{l-i+3}{2}+x-1$. Thus $\operatorname{Re}\left(x\right)>0$ \big($\operatorname{Re}\left(x\right)\neq 0$ because we require $\operatorname{Re}\left(b\right)\geq \operatorname{Re}\left(a\right)$\big). Substituting them to $\left(a^2+b^2+a+b\right)=2\left(a_1+\frac{l-i+3}{2}\right)^2+\frac{\left(l-i+2\right)\left(l-i\right)}{2}$, we have
$$x^2-x=\frac{\left(l-i\right)\left(l-i+2\right)}{4}.$$
Thus $x=\frac{1}{2}+\sqrt{\frac{\left(l-i\right)\left(l-i+2\right)+1}{4}}=\frac{l-i+2}{2}$. Then $a=a_1+\frac{1}{2}$ and $b=a_1+l-i+\frac{3}{2}$. By the representation theory of the local Weyl modules of $\ysl$, the local Weyl modules of $\ysl $ associated to $P\left(u\right)$ is $W_1\left(a_1+l-i+\frac{3}{2}\right)\otimes W_1\left(a_1+\frac{1}{2}\right)$, which is irreducible. Thus $$Y_{i-1}\left(\overline{x_{i,0}^{-}\ldots x_{i,0}^{-}}v^{+}_1\right)\cong W_1\left(a_1+l-i+\frac{3}{2}\right)\otimes W_1\left(a_1+\frac{1}{2}\right).$$ 
\end{proof}
It follows from Corollary \ref{w1bw1a} that\begin{corollary}
$$x_{i-1,1}^{-}x_{i-1,0}^{-}\overline{x_{i,0}^{-}\ldots x_{i,0}^{-}}v^{+}_1=\left(a_1+\frac{l-i+1}{2}\right)\overline{\left(x_{i-1,0}^{-}\right)^2\ldots x_{i,0}^{-}}v^{+}_1.$$
$$x_{i-1,0}^{-}x_{i-1,1}^{-}\overline{x_{i,0}^{-}\ldots x_{i,0}^{-}}v^{+}_1=\left(a_1+\frac{l-i+3}{2}\right)\overline{\left(x_{i-1,0}^{-}\right)^2\ldots x_{i,0}^{-}}v^{+}_1.$$
$$\left(x_{i-1,1}^{-}x_{i-1,0}^{-}+x_{i-1,0}^{-}x_{i-1,1}^{-}\right)\overline{x_{i,0}^{-}\ldots x_{i,0}^{-}}v^{+}_1=\left(2a_1+{l-i+2}\right)\overline{\left(x_{i-1,0}^{-}\right)^2\ldots x_{i,0}^{-}}v^{+}_1.$$
\begin{align*}
\left(x_{i-1,1}^{-}\right)^2&\overline{x_{i,0}^{-}\ldots x_{i,0}^{-}}v^{+}_1\\
&=\Bigg(\left(a_1+\frac{l-i+2}{2}\right)^2-\frac{\left(l-i+2\right)\left(l-i\right)+1}{4}\Bigg)\overline{\left(x_{i-1,0}^{-}\right)^2\ldots x_{i,0}^{-}}v^{+}_1.
\end{align*}
\begin{align*}
x_{i-1,0}^{-}& x_{i-1,2}^{-}\overline{x_{i,0}^{-}\ldots x_{i,0}^{-}}v^{+}_1\\
&=\Bigg(\left(a_1+\frac{l-i+3}{2}\right)^2+\frac{\left(l-i+2\right)\left(l-i\right)}{4}\Bigg)\overline{\left(x_{i-1,0}^{-}\right)^2\ldots x_{i,0}^{-}}v^{+}_1.
\end{align*}
\end{corollary}


For $i=2$, jump to Lemma \ref{spv2}. Now we suppose $i\geq 3$.

\begin{lemma}\label{yspc1i6}
$Y_{i-2}\Big(\overline{\left(x_{i-1,0}^{-}\right)^2\ldots x_{i,0}^{-}}v^{+}_1\Big)\cong W_1\left(a_1+l-i+2\right)\otimes W_1\left(a_1+1\right)$.
\end{lemma}
\begin{proof} 
It follows from Proposition \ref{siok} that
$$h_{i-2,0}\overline{\left(x_{i-1,0}^{-}\right)^2\ldots x_{i,0}^{-}}v^{+}_1=2\overline{\left(x_{i-1,0}^{-}\right)^2\ldots x_{i,0}^{-}}v^{+}_1.$$
Thus the associated polynomial $P\left(u\right)$ has of degree 2. Let $P\left(u\right)=\left(u-a\right)\left(u-b\right)$ with $\operatorname{Re}\left(a\right)\leq \operatorname{Re}\left(b\right)$.
The eigenvalues of both $\overline{\left(x_{i-1,0}^{-}\right)^2\ldots x_{i,0}^{-}}v^{+}_1$ under $h_{i-2,1}$ and $\overline{\left(x_{i-1,0}^{-}\right)^2\ldots x_{i,0}^{-}}v^{+}_1$ under $h_{i-2,2}$ will tell the values of $a$ and $b$.
\begin{align*}
h_{i-2,1}&\overline{\left(x_{i-1,0}^{-}\right)^2\ldots x_{i,0}^{-}}v^{+}_1 \\
   &= [h_{i-2,1}, \left(x_{i-1,0}^{-}\right)^2]\overline{x_{i,0}^{-}\ldots x_{i,0}^{-}}v^{+}_1\\
   &= [h_{i-2,1}, x_{i-1,0}^{-}]\overline{x_{i-1,0}^{-}\ldots x_{i,0}^{-}}v^{+}_1+ x_{i-1,0}^{-}[h_{i-2,1}, x_{i-1,0}^{-}]\overline{x_{i,0}^{-}\ldots x_{i,0}^{-}}v^{+}_1\\
   &=\left(x_{i-1,1}^{-}+\frac{1}{2}x_{i-1,0}^{-}+x_{i-1,0}^{-}h_{i-2,0}\right)\overline{x_{i-1,0}^{-}\ldots x_{i,0}^{-}}v^{+}_1\\
   &+x_{i-1,0}^{-}\left(x_{i-1,1}^{-}+\frac{1}{2}x_{i-1,0}^{-}+x_{i-1,0}^{-}h_{i-2,0}\right)\overline{x_{i,0}^{-}\ldots x_{i,0}^{-}}v^{+}_1\\
   &= \left(x_{i-1,1}^{-}x_{i-1,0}^{-}+x_{i-1,0}^{-}x_{i-1,1}^{-}\right)\overline{x_{i,0}^{-}\ldots x_{i,0}^{-}}v^{+}_1+2\overline{\left(x_{i-1,0}^{-}\right)^2\ldots x_{i,0}^{-}}v^{+}_1\\
   &= \left(2a_1+\left(l-i\right)+4\right)\overline{\left(x_{i-1,0}^{-}\right)^2\ldots x_{i,0}^{-}}v^{+}_1.
\end{align*}
\begin{align*}
   h_{i-2,2}&\overline{\left(x_{i-1,0}^{-}\right)^2\ldots x_{i,0}^{-}}v^{+}_1 \\
   &= [h_{i-2,2}, \left(x_{i-1,0}^{-}\right)^2]\overline{x_{i,0}^{-}\ldots x_{i,0}^{-}}v^{+}_1 \\
   &= [h_{i-2,2}, x_{i-1,0}^{-}]\overline{x_{i-1,0}^{-}\ldots x_{i,0}^{-}}v^{+}_1+ x_{i-1,0}^{-}[h_{i-2,2}, x_{i-1,0}^{-}]\overline{x_{i,0}^{-}\ldots x_{i,0}^{-}}v^{+}_1\\
   &=\left([h_{i-2,1}, x_{i-1,1}^{-}]+\frac{1}{2}\left(h_{i-2,1}x_{i-1,0}^{-}+x_{i-1,0}^{-}h_{i-2,1}\right)\right)\overline{x_{i-1,0}^{-}\ldots x_{i,0}^{-}}v^{+}_1\\
   &+x_{i-1,0}^{-}\left([h_{i-2,1}, x_{i-1,1}^{-}]+\frac{1}{2}\left(h_{i-2,1}x_{i-1,0}^{-}+x_{i-1,0}^{-}h_{i-2,1}\right)\right)\overline{x_{i,0}^{-}\ldots x_{i,0}^{-}}v^{+}_1\\
   &=h_{i-2,1}x_{i-1,1}^{-}\overline{x_{i-1,0}^{-}\ldots x_{i,0}^{-}}v^{+}_1-x_{i-1,1}^{-}h_{i-2,1}\overline{x_{i-1,0}^{-}\ldots x_{i,0}^{-}}v^{+}_1\\
   &+\frac{1}{2}h_{i-2,1}x_{i-1,0}^{-}\overline{x_{i-1,0}^{-}\ldots x_{i,0}^{-}}v^{+}_1+x_{i-1,0}^{-}h_{i-2,1}\overline{x_{i-1,0}^{-}\ldots x_{i,0}^{-}}v^{+}_1\\
   &+x_{i-1,0}^{-}h_{i-2,1}x_{i-1,1}^{-}\overline{x_{i,0}^{-}\ldots x_{i,0}^{-}}v^{+}_1\\
   &=h_{i-2,1}x_{i-1,1}^{-}x_{i-1,0}^{-}\overline{x_{i,0}^{-}\ldots x_{i,0}^{-}}v^{+}_1-x_{i-1,1}^{-}[h_{i-2,1},x_{i-1,0}^{-}]\overline{x_{i,0}^{-}\ldots x_{i,0}^{-}}v^{+}_1\\
   &+\frac{1}{2}h_{i-2,1}\overline{\left(x_{i-1,0}^{-}\right)^2\ldots x_{i,0}^{-}}v^{+}_1\\
   &+x_{i-1,0}^{-}[h_{i-2,1},x_{i-1,1}^{-}]\overline{x_{i,0}^{-}\ldots x_{i,0}^{-}}v^{+}_1+x_{i-1,0}^{-}[h_{i-2,1},x_{i-1,0}^{-}]\overline{x_{i,0}^{-}\ldots x_{i,0}^{-}}v^{+}_1\\
   &=\left(a_1+\frac{l-i+1}{2}\right)h_{i-2,1}\overline{\left(x_{i-1,0}^{-}\right)^2\ldots x_{i,0}^{-}}v^{+}_1\\
   &-x_{i-1,1}^{-}\left(x_{i-1,1}^{-}+\frac{1}{2}x_{i-1,0}^{-}+x_{i-1,0}^{-}h_{i-2,0}\right)\overline{x_{i,0}^{-}\ldots x_{i,0}^{-}}v^{+}_1\\
   &+\frac{1}{2}h_{i-2,1}\overline{\left(x_{i-1,0}^{-}\right)^2\ldots x_{i,0}^{-}}v^{+}_1 \\ &+x_{i-1,0}^{-}\left(x_{i-1,2}^{-}+\frac{1}{2}x_{i-1,1}^{-}+x_{i-1,1}^{-}h_{i-2,0}\right)\overline{x_{i,0}^{-}\ldots x_{i,0}^{-}}v^{+}_1\\
   &+x_{i-1,0}^{-}\left(x_{i-1,1}^{-}+\frac{1}{2}x_{i-1,0}^{-}+x_{i-1,0}^{-}h_{i-2,0}\right)\overline{x_{i,0}^{-}\ldots x_{i,0}^{-}}v^{+}_1\\
   &=\left(a_1+\frac{l-i+2}{2}\right)\left(2a_1+\left(l-i\right)+4\right)\overline{\left(x_{i-1,0}^{-}\right)^2\ldots x_{i,0}^{-}}v^{+}_1\\
   &-\Big(\left(a_1+\frac{l-i+2}{2}\right)^2-\frac{\left(l-i\right)\left(l-i+2\right)+1}{4}\Big)\overline{\left(x_{i-1,0}^{-}\right)^2\ldots x_{i,0}^{-}}v^{+}_1\\
   &-\frac{1}{2}\left(a_1+\frac{l-i+1}{2}\right)\overline{\left(x_{i-1,0}^{-}\right)^2\ldots x_{i,0}^{-}}v^{+}_1\\
   &+\Big(\left(a_1+\frac{l-i+3}{2}\right)^2+\frac{\left(l-i\right)\left(l-i+2\right)}{4}\Big)\overline{\left(x_{i-1,0}^{-}\right)^2\ldots x_{i,0}^{-}}v^{+}_1\\
   &+\frac{3}{2}\left(a_1+\frac{l-i+3}{2}\right)\overline{\left(x_{i-1,0}^{-}\right)^2\ldots x_{i,0}^{-}}v^{+}_1+\frac{1}{2}\overline{\left(x_{i-1,0}^{-}\right)^2\ldots x_{i,0}^{-}}v^{+}_1\\
   &= \Big(2\left(a_1+\frac{l-i+4}{2}\right)^2+\frac{\left(l-i+2\right)\left(l-i\right)}{2}\Big)\overline{\left(x_{i-1,0}^{-}\right)^2\ldots x_{i,0}^{-}}v^{+}_1.
\end{align*}

Similar to the above lemma, we have $a=a_1+1$ and $b=a_1+l-i+2$. Then
$$Y_{i-2}\Big(\overline{\left(x_{i-1,0}^{-}\right)^2\ldots x_{i,0}^{-}v^{+}_1}\Big)\cong W_1\left(a_1+l-i+2\right)\otimes W_1\left(a_1+1\right).$$
\end{proof}

Similar to Lemma \ref{l-2i34d23}, using induction on $m$ downward, we have

\begin{lemma}\label{yspc1i7}
 Let $1\leq m\leq i-3$.
$Y_m\Big(\overline{\left(x_{m+1,0}^{-}\right)^2\ldots x_{i,0}^{-}}v^{+}_1\Big)$ is 4-dimensional, which is isomorphic to $W_1\left(a_1+\frac{2l-i-m+2}{2}\right)\otimes W_1\left(a_1+\frac{i-m}{2}\right)$.
\end{lemma}

\begin{lemma}\label{spv2}
$Y_i\Big(\overline{\left(x_{1,0}^{-}\right)^2\ldots x_{i,0}^{-}}v^{+}_1\Big)\cong W_1\left(a_1+1\right)$.
\end{lemma}

The proof is similar to Lemma \ref{i=l-2v2}.

\begin{lemma}\label{sps1v2p}
Denote $\mathbf{s}_1=s_1\ldots s_{l-1}s_ls_{l-1}\ldots s_i$. Then $\mathbf{s}_1^{-1}\left(\alpha_j\right)\in \Delta^{+}$, where $j=2,3,\ldots, l$, and $\alpha_j$ are the simple roots of $\nysp$.
\end{lemma}
\begin{proof}
$\mathbf{s}_1^{-1}=s_i\ldots s_{l-1}s_ls_{l-1}\ldots s_1$. We claim that \begin{equation*}
\mathbf{s}_1^{-1}\left(\mu_j-\mu_{j+1}\right)=\begin{cases}
\mu_j-\mu_{j+1}\qquad \mathrm{if}\qquad i< j\leq l-1 \\
\mu_{j-1}-\mu_{j+\delta_{ij}} \qquad \mathrm{if}\qquad 2\leq j\leq i.
\end{cases}
\end{equation*}
and $$\mathbf{s}_1^{-1}\left(2\mu_{l}\right)=2\mu_{l}.$$
The proof of the claim is trivial.
\end{proof}

Let $v_2=\left(x_{1,0}^{-}\right)^2\ldots \left(x_{i-1,0}^{-}\right)^2x_{i,0}^{-}x_{i+1,0}^{-}\ldots x_{l-1,0}^{-}x_{l,0}^{-}x_{l-1,0}^{-}\ldots x_{i+1,0}^{-}x_{i,0}^{-}v^{+}_1$.

\begin{proposition}\label{spc1il-1p2} Let $i\leq k\leq l-1$ and $2\leq m\leq i-2$.
\begin{enumerate}
  \item $Y_{k}\Big( {\left(x_{k-1,0}^{-}x_{k-2,0}^{-}\ldots x_{i,0}^{-}\right)}v_2\Big)\cong W_{1}\left(a_1+1+\frac{k-i}{2}\right)$ as a $Y_{k}$-module.
  \item $Y_{l}\Big( {x_{l-1,0}^{-}x_{l-2,0}^{-}\ldots x_{i,0}^{-}}v_2\Big)\cong W_1\Big(\frac{1}{2}\left(a_1+1+\frac{l-i+1}{2}\right)\Big)$ as a $Y_{l}$-module.
  \item $Y_{k}\left(\overline{x_{k+1,0}^{-}\ldots x_{i,0}^{-}}v_2\right)\cong W_1\left(a_1+1+\frac{2l-i-k+2}{2}\right)$ as a $Y_{k}$-module.
  \item
\begin{enumerate}
    \item $Y_{i-1}\Big(\overline{\left(x_{i,0}^{-}\right)\ldots x_{i,0}^{-}}v_2\Big)\cong W_1\left(a_1+1+l-i+\frac{3}{2}\right)\otimes W_1\left(a_1+1+\frac{1}{2}\right)$.
    \item $Y_m\Big(\overline{\left(x_{m+1,0}^{-}\right)^2\ldots x_{i,0}^{-}}v_2\Big)\cong W_1\left(a_1+1+\frac{2l-i-m+2}{2}\right)\otimes W_1\left(a_1+1+\frac{i-m}{2}\right)$.
  \end{enumerate}
\item $Y_i\Big(\overline{\left(x_{2,0}^{-}\right)^2\ldots x_{i,0}^{-}}v_2\Big)\cong W_1\left(a_1+2\right)$.
\end{enumerate}
\end{proposition}
\begin{proof}
Removing the first node in the Dynkin diagram of the Lie algebra of type $C_l$, we get the Lie algebra which is isomorphic to the simple Lie algebra of type $C_{l-1}$. Denote $I'=\{2,3,\ldots, l\}$.  Let $Y^{\left(1\right)}$ be the Yangian generated by all $x_{i,k}^{\pm}$ and $h_{i,k}$ for $i\in I'$ and $k\in \mathbb{Z}_{\geq 0}$. $Y^{\left(1\right)}\cong Y\big(\mathfrak{sp}\left(2\left(l-1\right),\C\right)\big)$.

From Proposition \ref{siok}, the weight of $v_2$ is $\mathbf{s}_1(\omega_i)$, and then $x_{j,0}^{+}v_2$ has weight $\mathbf{s}_1(\omega_i)+\alpha_j$, where $j\geq 2$. If $x_{j,0}^{+}v_2\neq 0$, $\mathbf{s}_1(\omega_i)+\alpha_j$ is a weight of $V_{a_1}(\omega_i)$, so is $\omega_i+\mathbf{s}_1^{-1}\left(\alpha_j\right)$. By the above Lemma, $\omega_i$ precedes $\omega_i+\mathbf{s}_1^{-1}\left(\alpha_j\right)$, which contradicts to the fact that $\omega_i$ is the highest weight. Thus $x_{j,0}^{+}v_2=0$. Therefore $v_2$ is a maximal vector of $Y^{\left(1\right)}$, and
$Y^{\left(1\right)}\left(v_2\right)$ is a highest weight representation. Since $v_2$ has weight $-2\omega_1+\omega_i$, the degree of the associated polynomial $P_j$ equals 0 if $j\neq i$; and $\operatorname{Deg}\Big(P_i\left(u\right)\Big)=1$. Therefore $P_j\left(u\right)=1$ if $j\neq i$ and $P_i\left(u\right)=\left(u-a\right)$. It follows from Lemma \ref{spv2} that $a=a_1+1$. The rest of the proof of this proposition is similar to the proofs of Proposition \ref{spc1il-1p1}, just replacing $a_1$ by $a_1+1$.
\end{proof}

Note that if $i=2$, only $\left(i\right), \left(ii\right)$ and $\left(iii\right)$ of the above proposition is true and necessary. Now we assume that $i\geq 3$.

Define inductively $v_{m+1}=\overline{\left(x_{m,0}^{-}\right)^2\ldots x_{i,0}^{-}}v_{m}$ for $2\leq m\leq i-1$, and let $Y^{\left(m\right)}$ be the Yangian generated by all $x_{r,k}^{\pm}$ and $h_{r,k}$ for $r>m$ and $k\in \mathbb{Z}_{\geq 0}$. Since $m\leq i-1\leq l-2$, $Y^{\left(m\right)}$ is a Yangian of type $C_{l-m}$. Then we have no problem to transfer all proof of above proposition to the following.
\begin{proposition} Let $i\leq k\leq l-1$ and $m+1\leq n\leq i-2$.
\begin{enumerate}
  \item $Y_{k}\Big( {\left(x_{k-1,0}^{-}x_{k-2,0}^{-}\ldots x_{i,0}^{-}\right)}v_{m+1}\Big)\cong W_{1}\left(a_1+m+\frac{k-i}{2}\right)$ as a $Y_{k}$-module.
  \item $Y_{l}\Big( {x_{l-1,0}^{-}x_{l-2,0}^{-}\ldots x_{i,0}^{-}}v_{m+1}\Big)\cong W_1\Big(\frac{1}{2}\left(a_1+m+\frac{l-i+1}{2}\right)\Big)$ as a     $Y_{l}$-module.
  \item $Y_{k}\left(\overline{x_{k+1,0}^{-}\ldots x_{i,0}^{-}}v_{m+1}\right)\cong W_1\left(a_1+m+\frac{2l-i-k+2}{2}\right)$ as a $Y_{k}$-module.
  \item
\begin{enumerate}
    \item $Y_{i-1}\Big(\overline{\left(x_{i,0}^{-}\right)\ldots x_{i,0}^{-}}v_{m+1}\Big)\cong W_1\left(a_1+m+l-i+\frac{3}{2}\right)\otimes W_1\left(a_1+m+\frac{1}{2}\right)$.
    \item $Y_n\Big(\overline{\left(x_{n+1,0}^{-}\right)^2\ldots x_{i,0}^{-}}v_2\Big)\cong W_1\left(a_1+1+\frac{2l-i-m+2}{2}\right)\otimes W_1\left(a_1+m+\frac{i-n}{2}\right)$.
  \end{enumerate}
\item $Y_i\Big(\overline{\left(x_{m+1,0}^{-}\right)^2\ldots x_{i,0}^{-}}v_{m+1}\Big)\cong W_1\left(a_1+m+1\right)$.
\end{enumerate}
\end{proposition}
Note that $v_{1}^{-}=\overline{x_{i,0}^{-}\ldots x_{i,0}^{-}}v_{i}^{+}$. When $m=i-1$, only $\left(i\right), \left(ii\right)$ and $\left(iii\right)$ of the above proposition is true and necessary.

This completes the proof of Step 5 of Proposition 5.4 for the case $2\leq i\leq l-1$.

\noindent \textbf{Case 2: $i=1$.}

Due to the previous case $2\leq i\leq l-1$, it is no hard to prove the following proposition. Because the similarity of the proof to the one of Proposition \ref{spc1il-1p1}, we omit the proof here.

 \begin{proposition} Let $1\leq k\leq l-1$.
\begin{enumerate}
  \item $Y_{k}\left(x_{k-1,0}^{-}\ldots x_{2,0}^{-}x_{1,0}^{-}v^{+}_1\right)\cong W_{1}\left(a_1+\frac{k-1}{2}\right)$ as a $Y_{k}$-module.
  \item $Y_{l}\left(x_{l-1}^{-}\ldots x_{2,0}^{-}x_{1,0}^{-}v^{+}_1\right)\cong W_{1}\left(\frac{1}{2}\left(a_1+\frac{l}{2}\right)\right)$ as a $Y_{l}$-module.
  \item $Y_{k}\left(x_{k+1,0}^{-}\ldots x_{l-1,0}^{-}x_{l,0}^{-}\ldots x_{2,0}^{-}x_{1,0}^{-}v^{+}_1\right)\cong W_{1}\left(a_1+\frac{2l-k+1}{2}\right)$,as a $Y_{k}$-modules.
\end{enumerate}
 \end{proposition}

\noindent\textbf{ Case 3: $i=l$.}
 \begin{proposition}\label{yspc3p1}\
\begin{enumerate}
  \item $Y_l\left(v_{1}^{+}\right)=W_1\left(\frac{a_1}{2}\right)$.
  \item $Y_{l-1}\left(x_{l,0}^{-}v^{+}_1\right)\cong W_2\left(a_1\right)$.
  \item For $2\leq k\leq l-1$, $Y_{k-1}\Big(\left(x_{k,0}^{-}\right)^2\ldots \left(x_{l-2,0}^{-}\right)^2\left(x_{l-1,0}^{-}\right)^2x_{l,0}^{-}v^{+}_1\Big)$ is isomorphic to either $W_2\left(a_1+\frac{l-k}{2}\right)$ or $W_1\left(a_1+\frac{l-k}{2}+1\right)\otimes W_1\left(a_1+\frac{l-k}{2}\right)$.
  \item $Y_l\Big(\left(x_{1,0}^{-}\right)^2\ldots \left(x_{l-2,0}^{-}\right)^2\left(x_{l-1,0}^{-}\right)^2x_{l,0}^{-}v^{+}_1\Big)$ is isomorphic to $ W_1\left(\frac{a_1+1}{2}\right)$.
\end{enumerate}
 \end{proposition}
\begin{proof}
The first item will be proved in Lemma \ref{spc3c1}. Lemma \ref{spi=l2} is devoted to prove the second item. The third item will be showed in Lemmas \ref{spl-2i34d} and \ref{spi=l4}. Lemma \ref{spi=l3} is reserved to prove the fourth item.
\end{proof}
\begin{corollary}\label{spc3c1}
$Y_l\left(v_{1}^{+}\right)=W_1\left(\frac{a_1}{2}\right)$.
\end{corollary}
\begin{proof} It follows from (ii) of Theorem \ref{cfdihwr} that
\begin{align*}
  h_l\left(u\right)v_1^{+} &= \frac{P\left(u+d_l\right)}{P\left(u\right)}\\
  &= \frac{u-\left(a_1-2\right)}{u-a_1}v_1^{+} \\
   &= \left(1+2u^{-1}+2a_1u^{-2}+2a_1^2u^{-2}+\ldots\right)v_1^{+}.
\end{align*}
Since $\tilde{h}_{lr}=\frac{1}{2^{r+1}}h_{lr}$, $\tilde{h}_{lr}v_1^{+}=\frac{1}{2^{r+1}}h_{lr}v_1^{+}=\frac{1}{2^{r+1}}2a_1^{r}v_1^{+}=\frac{a_1^r}{2^{r}}v_1^{+}=\left(\frac{a_1}{2}\right)^rv_1^{+}$. Thus $$\tilde{h}_{1}\left(u\right)v_1^{+}=\left(1+u^{-1}+\frac{a_1}{2}u^{-2}+\left(\frac{a_1}{2}\right)^2u^{-3}+\ldots\right)v_1^{+}=\frac{u-\left(\frac{a_1}{2}-1\right)}{u-\frac{a_1}{2}}.$$ The associated polynomial for $Y_l\left(v_1^{+}\right)$ is $P\left(u\right)=u-\frac{a_1}{2}$. By the theory of the local Weyl modules of $\ysl$, $Y_l\left(v_{1}^{+}\right)\cong W_1\left(\frac{a_1}{2}\right)$ as desired.
\end{proof}

\begin{lemma}\label{spi=l2}
$Y_{l-1}\left(x_{l,0}^{-}v^{+}_1\right)\cong W_2\left(a_1\right)$ as $\ysl$-module.
\end{lemma}
\begin{proof}
It follows from Proposition \ref{siok} that $wt(x_{l,0}^{-}v^{+}_1)=2\omega_{l-1}-\omega_{l}$. Then $$h_{l-1,0} x_{l,0}^{-}v^{+}_1=2x_{l,0}^{-}v^{+}_1.$$
Thus the associated polynomial $P\left(u\right)$ has of degree 2. Let $P\left(u\right)=\left(u-a\right)\left(u-b\right)$ with $\operatorname{Re}\left(a\right)\leq \operatorname{Re}\left(b\right)$.
The eigenvalues of both $x_{l,0}^{-}v^{+}_1$ under $h_{i-2,1}$ and under $h_{i-2,2}$ will tell the values of $a$ and $b$.
\begin{align*}
  h_{l-1,1} x_{l,0}^{-}v^{+}_1 &= [h_{l-1,1} x_{l,0}^{-}]v^{+}_1 \\
   &= \left(2x_{l,1}^{-}+2x_{l,0}^{-}+2x_{l,0}^{-}h_{l-1,0}\right)v^{+}_1 \\
   &= \left(2a_1+2\right)x_{l,0}^{-}v^{+}_1\\
   &= 2\left(a_1+1\right)x_{l,0}^{-}v^{+}_1.
\end{align*}
\begin{align*}
  h_{l-1,2} x_{l,0}^{-}v^{+}_1 &= [h_{l-1,2} x_{l,0}^{-}]v^{+}_1 \\
   &= \Big([h_{l-1,1}, x_{l,1}^{-}]+h_{l-1,1}x_{l,0}^{-}+x_{l,0}^{-}h_{l-1,1}\Big)v^{+}_1 \\
   &= \left(h_{l-1,1}x_{l,1}^{-}+h_{l-1,1}x_{l,0}^{-}\right)v^{+}_1 \\
   &= h_{l-1,1}\left(x_{l,1}^{-}+x_{l,0}^{-}\right)v^{+}_1\\
   &= 2\left(a_1+1\right)^2x_{l,0}^{-}v^{+}_1.
\end{align*}
It follows from Lemma \ref{g2puisd3} that the associated polynomial is $\big(u-a_1\big)\big(u-(a_1+1)\big)$.

We claim that $Y_{l-1}\left(x_{l,0}^{-}v^{+}_1\right)\cong W_2\left(a_1\right)$.  It is clear from (i) of Corollary \ref{w1bw1a} that we only have to show that $x_{l-1,1}^{-}\left(x_{l,0}^{-}v^{+}_1\right)$ is a scalar multiple of $x_{l-1,0}^{-}\left(x_{l,0}^{-}v^{+}_1\right)$.
It follows from the defining relations of Yangians that
\begin{equation}\label{yl-1isw2a}
[x_{l-1,1}^{-}, x_{l,0}^{-}]v^{+}_1-[x_{l-1,0}^{-}, x_{l,1}^{-}]v^{+}_1=x_{l-1,0}^{-}x_{l,0}^{-}v^{+}_1+x_{l,0}^{-}x_{l-1,0}^{-}v^{+}_1.
\end{equation}

Note that $x_{l-1,1}^{-}v^{+}_1=0$ and $x_{l,1}^{-}v^{+}_1=a_1x_{l,0}^{-}v^{+}_1$. Equation (\ref{yl-1isw2a}) implies $$x_{l-1,1}^{-}x_{l,0}^{-}v^{+}_1=\left(a_1+1\right)x_{l-1,0}^{-}x_{l,0}^{-}v^{+}_1.$$
By $\left(i\right)$ of Corollary \ref{w1bw1a}, $Y_{l-1}\left(x_{l,0}^{-}v^{+}_1\right)\ncong W_{1}\left(a_1+1\right)\otimes W_{1}\left(a_1\right)$. By the theory of the local Weyl modules of $\ysl$, $Y_{l-1}\left(x_{l,0}^{-}v^{+}_1\right)\cong W_2\left(a_1\right)$.
\end{proof}



\begin{lemma}\label{spl-2i34d}
$Y_{l-2}\Big(\left(x_{l-1,0}^{-}\right)^2x_{l,0}^{-}v^{+}_1\Big)$ is isomorphic to either $W_2\left(a_1+\frac{1}{2}\right)$ or\\ $W_1\left(a_1+\frac{3}{2}\right)\otimes W_1\left(a_1+\frac{1}{2}\right)$.
\end{lemma}
\begin{proof}

We will need $wt\left(x_{l,0}^{-}v^{+}_1\right)=2\omega_{l-1}-\omega_{l}$, $wt\left(x_{l-1,0}^{-}x_{l,0}^{-}v^{+}_1\right)=\omega_{l-2}$, and $wt\Big(\left(x_{l-1,0}^{-}\right)^2x_{l,0}^{-}v^{+}_1\Big)=2\omega_{l-2}-2\omega_{l-1}+\omega_{l}$.

By Proposition \ref{siok} that $wt\Big(\left(x_{l-1,0}^{-}\right)^2x_{l,0}^{-}v^{+}_1\Big)=2\omega_{l-2}-2\omega_{l-1}+\omega_{l}$, and then $$h_{l-2,0}\left(x_{l-1,0}^{-}\right)^2x_{l,0}^{-}v^{+}_1=2\left(x_{l-1,0}^{-}\right)^2x_{l,0}^{-}v^{+}_1.$$
Thus the associated polynomial $P\left(u\right)$ has of degree 2. Let $P\left(u\right)=\left(u-a\right)\left(u-b\right)$ with $\operatorname{Re}\left(a\right)\leq \operatorname{Re}\left(b\right)$.
The eigenvalues of $\left(x_{l-1,0}^{-}\right)^2x_{l,0}^{-}v^{+}_1$ under $h_{l-2,1}$ and under $h_{l-2,2}$ will tell the values of $a$ and $b$.
\begin{align*}
   h_{l-2,1}\left(x_{l-1,0}^{-}\right)^2&x_{l,0}^{-}v^{+}_1\\
   &=\left[h_{l-2,1},x_{l-1,0}^{-}\right]x_{l-1,0}^{-}x_{l,0}^{-}v^{+}_1+x_{l-1,0}^{-}\left[h_{l-2,1},x_{l-1,0}^{-}\right]x_{l,0}^{-}v^{+}_1\\ &+\left(x_{l-1,0}^{-}\right)^2 h_{l-2,1}x_{l,0}^{-}v^{+}_1 \\
   &=\left(x_{l-1,1}^{-}+\frac{1}{2}x_{l-1,0}^{-}+x_{l-1,0}^{-}h_{l-2,0}\right)x_{l-1,0}^{-}x_{l,0}^{-}v^{+}_1\\
   &+x_{l-1,0}^{-}\left(x_{l-1,1}^{-}+\frac{1}{2}x_{l-1,0}^{-}+x_{l-1,0}^{-}h_{l-2,0}\right)x_{l,0}^{-}v^{+}_1\\
   &= \left(x_{l-1,1}^{-}x_{l-1,0}^{-}+x_{l-1,0}^{-}x_{l,1}^{-}\right)x_{l,0}^{-}v^{+}_1+2\left(x_{l-1,0}^{-}\right)^2x_{l,0}^{-}v^{+}_1\\
   &= 2\left(a_1+\frac{3}{2}\right)\left(x_{l-1,0}^{-}\right)^2x_{l,0}^{-}v^{+}_1.
\end{align*}
\begin{align*}
   h_{l-2,2}&\left(x_{l-1,0}^{-}\right)^2x_{l,0}^{-}v^{+}_1\\
   &= \left[h_{l-2,2},x_{l-1,0}^{-}\right]x_{l-1,0}^{-}x_{l,0}^{-}v^{+}_1+x_{l-1,0}^{-}\left[h_{l-2,2},x_{l-1,0}^{-}\right]x_{l,0}^{-}v^{+}_1\\
   &=\Big(\left[h_{l-2,1}, x_{l-1,1}^{-}\right]+\frac{1}{2}\left(h_{l-2,1}x_{l-1,0}^{-}+x_{l-1,0}^{-}h_{l-2,1}^{-}\right)\Big)x_{l-1,0}^{-}x_{l,0}^{-}v^{+}_1\\
   &+x_{l-1,0}^{-}\Big(\left[h_{l-2,1}, x_{l-1,1}^{-}\right]+\frac{1}{2}\left(h_{l-2,1}x_{l-1,0}^{-}+x_{l-1,0}^{-}h_{l-2,1}^{-}\right)\Big)x_{l,0}^{-}v^{+}_1\\
   &=h_{l-2,1}x_{l-1,1}^{-}x_{l-1,0}^{-}x_{l,0}^{-}v^{+}_1-x_{l-1,1}^{-}h_{l-2,1}x_{l-1,0}^{-}x_{l,0}^{-}v^{+}_1+\frac{1}{2}h_{l-2,1} \left(x_{l-1,0}^{-}\right)^2x_{l,0}^{-}v^{+}_1\\
   &+x_{l-1,0}^{-}h_{l-2,1}x_{l-1,0}^{-}x_{l,0}^{-}v^{+}_1
   +x_{l-1,0}^{-}h_{l-2,1}x_{l-1,1}^{-}x_{l,0}^{-}v^{+}_1\\
   &=\left(a_1+\frac{1}{2}\right)h_{l-2,1}\left(x_{l-1,0}^{-}\right)^2x_{l,0}^{-}v^{+}_1-x_{l-1,1}^{-} \left[h_{l-2,1},x_{l-1,0}^{-}\right]x_{l,0}^{-}v^{+}_1\\
   &+x_{l-1,0}^{-}\left[h_{l-2,1},x_{l-1,0}^{-}\right]x_{l,0}^{-}v^{+}_1
   +\left(a_1+1\right)x_{l-1,0}^{-}\left[h_{l-2,1},x_{l-1,0}^{-}\right]x_{l,0}^{-}v^{+}_1\\
   &=h_{l-2,1}\left(a_1+\frac{1}{2}\right)\left(x_{l-1,0}^{-}\right)^2x_{l,0}^{-}v^{+}_1-x_{l-1,1}^{-}\left[h_{l-2,1},x_{l-1,0}^{-}\right]x_{l,0}^{-}v^{+}_1\\
   &+\left(a_1+2\right)x_{l-1,0}^{-}\left[h_{l-2,1},x_{l-1,0}^{-}\right]x_{l,0}^{-}v^{+}_1\\
   &=2\left(a_1+\frac{3}{2}\right)\left(a_1+\frac{1}{2}\right)\left(x_{l-1,0}^{-}\right)^2x_{l,0}^{-}v^{+}_1-x_{l-1,1}^{-}\left(x_{l-1,1}^{-}+\frac{1}{2}x_{l-1,0}^{-}\right)x_{l,0}^{-}v^{+}_1\\
   &+\left(a_1+2\right)x_{l-1,0}^{-}\left(x_{l-1,1}^{-}+\frac{1}{2}x_{l-1,0}^{-}\right)x_{l,0}^{-}v^{+}_1\\
   &=\Bigg(2\left(a_1+\frac{3}{2}\right)\left(a_1+\frac{1}{2}\right)-a_1\left(a_1+\frac{3}{2}\right)+\left(a_1+2\right) \left(a_1+\frac{3}{2}\right)\Bigg)\\
   &\qquad\quad\cdot\left(x_{l-1,0}^{-}\right)^2x_{l,0}^{-}v^{+}_1\\
   &= 2\left(a_1+\frac{3}{2}\right)^2\left(x_{l-1,0}^{-}\right)^2x_{l,0}^{-}v^{+}_1.\\
\end{align*}
It follows from Lemma \ref{g2puisd3} that  the associated polynomial is
$$\Bigg(u-\left(a_1+\frac{1}{2}\right)\Bigg)\Bigg(u-\left(a_1+\frac{3}{2}\right)\Bigg).$$ By the theory of the local Weyl modules of $\ysl$, $Y_{l-2}\Big(\left(x_{l-1,0}^{-}\right)^2x_{l,0}^{-}v^{+}_1\Big)$ is isomorphic to either $W_2\left(a_1+\frac{1}{2}\right)$ or $W_1\left(a_1+\frac{3}{2}\right)\otimes W_1\left(a_1+\frac{1}{2}\right)$.
\end{proof}

\begin{lemma}\label{spi=l4} Let $2\leq k\leq l-1$.
$Y_{k-1}\Big(\left(x_{k,0}^{-}\right)^2\ldots \left(x_{l-2,0}^{-}\right)^2\left(x_{l-1,0}^{-}\right)^2x_{l,0}^{-}v^{+}_1\Big)$ is isomorphic to either $W_2\left(a_1+\frac{l-k}{2}\right)$ or $W_1\left(a_1+\frac{l-k}{2}+1\right)\otimes W_1\left(a_1+\frac{l-k}{2}\right)$.
\end{lemma}
\begin{proof}
We prove this proposition by induction on $k$ downward. The basis of induction is proved in the above lemma. Suppose that claims are true for all $m\geq k+1$. Therefore, we may assume that $Y_{k}\left(\left(x_{k+1,0}^{-}\right)^2\ldots \left(x_{l-2,0}^{-}\right)^2\left(x_{l-1,0}^{-}\right)^2x_{l,0}^{-}v^{+}_1\right)$ is isomorphic to either $W_2\left(a_1+\frac{l-k-1}{2}\right)$ or $W_1\left(a_1+\frac{l-k-1}{2}+1\right)\otimes W_1\left(a_1+\frac{l-k-1}{2}\right)$.

We next show they are also true for $k$.
For the simplicity of the expression, define $$v=\left(x_{k+1,0}^{-}\right)^2\ldots \left(x_{l-2,0}^{-}\right)^2\left(x_{l-1,0}^{-}\right)^2x_{l,0}^{-}v^{+}_1.$$
It follows from Proposition \ref{siok} that $ wt\left(\left(x_{k,0}^{-}\right)^2v\right)=2\omega_{k-1}-2\omega_{k}+\omega_l$. Then
$$h_{k-1,0}\left(x_{k,0}^{-}\right)^2v=2\left(x_{k,0}^{-}\right)^2v.$$
We are going to show the following claims are true:
$$h_{k-1,1}\left(x_{k,0}^{-}\right)^2v=2\left(a_1+\frac{l-k+1}{2}\right)\left(x_{k,0}^{-}\right)^2v;$$and
$$h_{k-1,2}\left(x_{k,0}^{-}\right)^2v=2\left(a_1+\frac{l-k+1}{2}\right)^2\left(x_{k,0}^{-}\right)^2v.
$$
Note that $wt\left(v\right)=2\omega_k-2\omega_{k+1}+\omega_l$ and $wt\left(x_{k,0}^{-}v\right)=\omega_{k-1}-\omega_{k+1}+\omega_l$.
\begin{align*}
 h_{k-1,1}\left(x_{k,0}^{-}\right)^2v
&= [h_{k-1,1},\left(x_{k,0}^{-}\right)^2]v\\
&= [h_{k-1,1},x_{k,0}^{-}]x_{k,0}^{-}v+x_{k,0}^{-}[h_{k-1,1},x_{k,0}^{-}]v\\
&=\left(x_{k,1}^{-}+\frac{1}{2}x_{k,0}^{-}+x_{k,0}^{-}h_{k-1,0}^{-}\right)x_{k,0}^{-}v \\
&+ x_{k,0}^{-}\left(x_{k,1}^{-}+\frac{1}{2}x_{k,0}^{-}+x_{k,0}^{-}h_{k-1,0}^{-}\right)v \\
&=\left(x_{k,1}^{-}x_{k,0}^{-}+x_{k,0}^{-}x_{k,1}^{-}\right)v+2\left(x_{k,0}^{-}\right)^2v\\
&=\left(2a_1+l-k+2\right)\left(x_{k,0}^{-}\right)^2v\\
&=2\left(a_1+\frac{l-k+2}{2}\right)\left(x_{k,0}^{-}\right)^2v.
\end{align*}
Set $c=\left(2a_1+l-k\right)$.
\begin{align*}
 h_{k-1,2}&\left(x_{k,0}^{-}\right)^2v\\
&= [h_{k-1,2},\left(x_{k,0}^{-}\right)^2]v\\
&= [h_{k-1,2},x_{k,0}^{-}]x_{k,0}^{-}v+x_{k,0}^{-}[h_{k-1,2},x_{k,0}^{-}]v\\
&=\left([h_{k-1,1}, x_{k,1}^{-}]+\frac{1}{2}\left(h_{k-1,1}x_{k,0}^{-}+x_{k,0}^{-}h_{k-1,1}^{-}\right)\right)x_{k,0}^{-}v \\
&+ x_{k,0}^{-}\left([h_{k-1,1}, x_{k,1}^{-}]+\frac{1}{2}\left(h_{k-1,1}x_{k,0}^{-}+x_{k,0}^{-}h_{k-1,1}^{-}\right)\right)v\\
&=[h_{k-1,1}, x_{k,1}^{-}]x_{k,0}^{-}v+\frac{1}{2}h_{k-1,1}\left(x_{k,0}^{-}\right)^2v+x_{k,0}^{-}h_{k-1,1}^{-}x_{k,0}^{-}v+
x_{k,0}^{-}h_{k-1,1} x_{k,1}^{-}v\\
&=h_{k-1,1}x_{k,1}^{-}x_{k,0}^{-}v-x_{k,1}^{-}h_{k-1,1}x_{k,0}^{-}v
+ \frac{1}{2}h_{k-1,1}\left(x_{k,0}^{-}\right)^2v\\
&+ x_{k,0}^{-}[h_{k-1,1}^{-},x_{k,0}^{-}]v
+ x_{k,0}^{-}[h_{k-1,1}, x_{k,1}^{-}]v \\
&=\frac{1}{2}\left(c-1\right)h_{k-1,1}\left(x_{k,0}^{-}\right)^2v-x_{k,1}^{-}[h_{k-1,1},x_{k,0}^{-}]v
+ \frac{1}{2}h_{k-1,1}\left(x_{k,0}^{-}\right)^2v\\
&+ x_{k,0}^{-}[h_{k-1,1}^{-},x_{k,0}^{-}]v
+ x_{k,0}^{-}[h_{k-1,1}, x_{k,1}^{-}]v \\
&=\frac{c}{2}h_{k-1,1}\left(x_{k,0}^{-}\right)^2v-x_{k,1}^{-}\left(x_{k,1}^{-}+\frac{1}{2}x_{k,0}^{-}+x_{k,0}^{-}h_{k-1,0}^{-}\right)v \\
&+x_{k,0}^{-}\left(x_{k,1}^{-}+\frac{1}{2}x_{k,0}^{-}+x_{k,0}^{-}h_{k-1,0}^{-}\right)v+ x_{k,0}^{-}\left(x_{k,2}^{-}+\frac{1}{2}x_{k,1}^{-}+x_{k,1}^{-}h_{k-1,0}\right)v \\
&=\frac{c}{2}h_{k-1,1}\left(x_{k,0}^{-}\right)^2v-x_{k,1}^{-}x_{k,1}^{-}v-\frac{1}{2}x_{k,1}^{-}x_{k,0}^{-}v \\
&+\frac{3}{2}x_{k,0}^{-}x_{k,1}^{-}v+\frac{1}{2}\left(x_{k,0}^{-}\right)^2v+x_{k,0}^{-}x_{k,2}^{-}v\\
&=\frac{c}{2}\left(c+2\right)\left(x_{k,0}^{-}\right)^2v-\Bigg(\left(\frac{c}{2}+\frac{1}{2}\right)\left(\frac{c}{2}-\frac{1}{2}\right)\Bigg) \left(x_{k,0}^{-}\right)^2v \\
&-\frac{c-1}{4}\left(x_{k,0}^{-}\right)^2v+\frac{3}{2}\frac{c+1}{2}\left(x_{k,0}^{-}\right)^2v+\frac{1}{2}\left(x_{k,0}^{-}\right)^2v +\left(\frac{c+1}{2}\right)^2\left(x_{k,0}^{-}\right)^2v \\
&=\left(\frac{c^2}{2}+c-\frac{c^2}{4}+\frac{1}{4}-\frac{c}{4}+\frac{1}{4}+\frac{3c}{4}+\frac{3}{4}+\frac{1}{2}+\frac{c^2}{4}+ \frac{c}{2}+\frac{1}{4}\right)\left(x_{k,0}^{-}\right)^2v\\
&= \frac{1}{2}\left(c+2\right)^2\left(x_{k,0}^{-}\right)^2v\\
&= 2\left(a_1+\frac{l-k+2}{2}\right)^2\left(x_{k,0}^{-}\right)^2\left(x_{k+1,0}^{-}\right)^2\ldots \left(x_{l-2,0}^{-}\right)^2\left(x_{l-1,0}^{-}\right)^2x_{l,0}^{-}v^{+}_1.
\end{align*}
Thus the associated polynomial is $\Big(u-\left(a_1+\frac{l-k}{2}\right)\Big)\Big(u-\left(a_1+\frac{l-k+2}{2}\right)\Big)$. Then we know $Y_{k-1}\Big(\left(x_{k,0}^{-}\right)^2\ldots \left(x_{l-1,0}^{-}\right)^2x_{l,0}^{-}v^{+}_1\Big)$ is isomorphic to either $W_2\left(a_1+\frac{l-k}{2}\right)$ or $W_1\left(a_1+\frac{l-k}{2}+1\right)\otimes W_1\left(a_1+\frac{l-k}{2}\right)$ by the theory of the local Weyl modules of $\ysl$.

By induction, the lemma is proved.
\end{proof}

\begin{lemma}\label{spi=l3}
$Y_l\Big(\left(x_{1,0}^{-}\right)^2\ldots \left(x_{l-2,0}^{-}\right)^2\left(x_{l-1,0}^{-}\right)^2x_{l,0}^{-}v^{+}_1\Big)\cong W_1\left(\frac{a_1+1}{2}\right)$.
\end{lemma}

\begin{proof}
By Proposition \ref{siok}, $wt\Big(\left(x_{1,0}^{-}\right)^2\ldots \left(x_{l-2,0}^{-}\right)^2\left(x_{l-1,0}^{-}\right)^2x_{l,0}^{-}v^{+}_1\Big)=-2\omega_1+\omega_{l}$, and then
 $$\tilde{h}_{l,0}\left(x_{1,0}^{-}\right)^2\ldots \left(x_{l-2,0}^{-}\right)^2\left(x_{l-1,0}^{-}\right)^2x_{l,0}^{-}v^{+}_1=\left(x_{1,0}^{-}\right)^2\ldots \left(x_{l-2,0}^{-}\right)^2\left(x_{l-1,0}^{-}\right)^2x_{l,0}^{-}v^{+}_1,$$
which tells that the associated polynomial has of degree 1. Thus we have $$Y_l\Big(\left(x_{1,0}^{-}\right)^2\ldots \left(x_{l-2,0}^{-}\right)^2\left(x_{l-1,0}^{-}\right)^2x_{l,0}^{-}v^{+}_1\Big)\cong W_1\left(a\right).$$ Then the eigenvalue of $\left(x_{1,0}^{-}\right)^2\ldots \left(x_{l-2,0}^{-}\right)^2\left(x_{l-1,0}^{-}\right)^2x_{l,0}^{-}v^{+}_1$ under $\tilde{h}_{l,1}$ will tell us the value of $a$. 
\begin{align*}
    4\tilde{h}_{l,1}\left(x_{1,0}^{-}\right)^2&\ldots \left(x_{l-2,0}^{-}\right)^2\left(x_{l-1,0}^{-}\right)^2x_{l,0}^{-}v^{+}_1 \\
   &= h_{l,1}\left(x_{1,0}^{-}\right)^2\ldots \left(x_{l-2,0}^{-}\right)^2\left(x_{l-1,0}^{-}\right)^2x_{l,0}^{-}v^{+}_1 \\
   &=\left(x_{1,0}^{-}\right)^2\ldots \left(x_{l-2,0}^{-}\right)^2[h_{l,1}, x_{l-1,0}^{-}]x_{l-1,0}^{-}x_{l,0}^{-}v^{+}_1\\
   &+\left(x_{1,0}^{-}\right)^2\ldots  \left(x_{l-2,0}^{-}\right)^2x_{l-1,0}^{-}[h_{l,1}, x_{l-1,0}^{-}]x_{l,0}^{-}v^{+}_1 \\
   &+\left(x_{1,0}^{-}\right)^2\ldots  \left(x_{l-2,0}^{-}\right)^2\left(x_{l-1,0}^{-}\right)^2 h_{l,1}x_{l,0}^{-}v^{+}_1\\
   &=\left(x_{1,0}^{-}\right)^2\ldots  \left(2x_{l-1,1}^{-}+2x_{l-1,0}^{-}+2x_{l-1,0}^{-}h_{l,0}^{-}\right)x_{l-1,0}^{-}x_{l,0}^{-}v^{+}_1\\
   &+\left(x_{1,0}^{-}\right)^2\ldots  x_{l-1,0}^{-}\left(2x_{l-1,1}^{-}+2x_{l-1,0}^{-}+2x_{l-1,0}^{-}h_{l,0}^{-}\right)x_{l,0}^{-}v^{+}_1\\
   &-2a_1\left(x_{1,0}^{-}\right)^2\ldots  \left(x_{l-1,0}^{-}\right)^2 x_{l,0}^{-}v^{+}_1\\
   &= \left(x_{1,0}^{-}\right)^2\ldots  \left(2\left(x_{l-1,1}^{-}x_{l-1,0}^{-}+x_{l-1,0}^{-}x_{l-1,1}^{-}\right)\right)x_{l,0}^{-}v^{+}_1\\
   &+\left(4-2\cdot 2-2a_1\right)\left(x_{1,0}^{-}\right)^2\ldots  \left(x_{l-1,0}^{-}\right)^2 x_{l,0}^{-}v^{+}_1\\
   &= \left(2\left(2a_1+1\right)-2a_1\right)\left(x_{1,0}^{-}\right)^2\ldots  \left(x_{l-1,0}^{-}\right)^2 x_{l,0}^{-}v^{+}_1\\
   &= 2\left(a_1+1\right)\left(x_{1,0}^{-}\right)^2\ldots \left(x_{l-2,0}^{-}\right)^2\left(x_{l-1,0}^{-}\right)^2 x_{l,0}^{-}v^{+}_1.
\end{align*}
Therefore $$\tilde{h}_{l,1}\left(x_{1,0}^{-}\right)^2\ldots \left(x_{l-1,0}^{-}\right)^2x_{l,0}^{-}v^{+}_1=\frac{1}{2}\left(a_1+1\right)\left(x_{1,0}^{-}\right)^2\ldots \left(x_{l-1,0}^{-}\right)^2 x_{l,0}^{-}v^{+}_1.$$
So the claim is proved.\end{proof}


\begin{proposition}\label{spv2aga12}Let $v_{2}=\left(x_{1,0}^{-}\right)^2\ldots \left(x_{l-2,0}^{-}\right)^2\left(x_{l-1,0}^{-}\right)^2x_{l,0}^{-}v^{+}_1$.
\begin{enumerate}
      \item $Y_{l-1}\left(x_{l,0}^{-}v_{2}\right)$, $Y_{l-2}\left(\left(x_{l-1,0}^{-}\right)^2x_{l,0}^{-}v_{2}\right)$, $Y_{l-3}\left(\left(x_{l-2,0}^{-}\right)^2\left(x_{l-1,0}^{-}\right)^2x_{l,0}^{-}v_{2}\right)$, $\ldots$,\newline $Y_{2}\left(\left(x_{3,0}^{-}\right)^2\ldots \left(x_{l-2,0}^{-}\right)^2\left(x_{l-1,0}^{-}\right)^2x_{l,0}^{-}v_{2}\right)$ are isomorphic to $W_2\left(a\right)$ or $W_1\left(a+1\right)\otimes W_1\left(a\right)$ with $\operatorname{Re}\left(a\right)\geq \operatorname{Re}\left(a_1\right)$.
  \item $Y_{l}\Big(\left(x_{2,0}^{-}\right)^2\left(x_{3,0}^{-}\right)^2\ldots \left(x_{l-1,0}^{-}\right)^2x_{l,0}^{-}v_{2}\Big)\cong W_1\left(\frac{a_1+2}{2}\right)$.
\end{enumerate}
\end{proposition}
\begin{proof}
The main idea of the proof is to use $Y^{(1)}$ as in the proof for Proposition \ref{spc1il-1p2}. We omit the proof.
\end{proof}

Similarly, we have

\begin{proposition}\label{spvmaga13} Let $2\leq m\leq l-2$. Let $v_{m+1}=\left(x_{m,0}^{-}\right)^2\ldots \left(x_{l-1,0}^{-}\right)^2x_{l,0}^{-}v_{m}$.
\begin{enumerate}
      \item $Y_{l-1}\Big(x_{l,0}^{-}v_{m+1} \Big)$,$Y_{l-2}\Big(\left(x_{l-1,0}^{-}\right)^2x_{l,0}^{-}v_{m+1}\Big)$, $\ldots$, \newline $Y_{m}\Big(\left(x_{m+1,0}^{-}\right)^2\ldots \left(x_{l-1,0}^{-}\right)^2x_{l,0}^{-}v_{m+1}\Big)$ are isomorphic to $W_2\left(a\right)$ or $W_1\left(a+1\right)\otimes W_1\left(a\right)$ with $\operatorname{Re}\left(a\right)\geq \operatorname{Re}\left(a_1\right)$.
  \item $Y_{l}\Big(\left(x_{m+1,0}^{-}\right)^2\ldots \left(x_{l-1,0}^{-}\right)^2x_{l,0}^{-}v_{m+1}\Big)\cong W_1\left(\frac{a_1+m+1}{2}\right)$.
\end{enumerate}
\end{proposition}
\begin{remark} Suppose $m=l-2$.
\begin{enumerate}
  \item $Y^{\left(l-2\right)}$ is isomorphic to the Yangian of simple Lie algebra of type $C_2$. The condition of Proposition \ref{spc1il-1p2} is still true.
  \item $v^{-}_{1}=x_{l,0}^{-}\left(x_{l-1,0}^{-}\right)^2x_{l,0}^{-}v_{l-1}$. In Proposition \ref{spvmaga13}, we only have to compute $Y_{l-1}\Big(x_{l,0}^{-}v_{l-1}\Big)$ and $Y_{l}\Big(\left(x_{l-1,0}^{-}\right)^2\Big)x_{l,0}^{-}v_{l-1}$.
\end{enumerate}

\end{remark}
\section{On the local Weyl modules of $\ysp$}
Let $\lambda=\sum\limits_{i\in I} m_i\omega_i$. In \cite{Na}, the dimension of the local Weyl module $W(\lambda)$ of the current algebra $\nysp[t]$ is given.
\begin{proposition}[Corollary 9.5, \cite{Na}]\label{dwmocsp}
Let $\lambda=\sum\limits_{i\in I} m_i\omega_i$. Then
$$\operatorname{Dim}\Big(W(\lambda)\Big)=\prod\limits_{i\in I} \bigg(\operatorname{Dim}\Big(W(\omega_i)\Big)\bigg)^{m_i}.$$
\end{proposition}

The next theorem follows from Propositions \ref{mtoyspl}, \ref{vtv'hwv}.
\begin{theorem}\label{wmiatpsp}
Let $\pi=\Big(\pi_1\left(u\right),\ldots, \pi_{l}\left(u\right)\Big)$, where $\pi_i\left(u\right)=\prod\limits_{j=1}^{m_i}\left(u-a_{i,j}\right)$. Let $S=\{a_{1,1},\ldots, a_{1,m_1},\ldots, a_{l,1}\ldots, a_{l,m_l}\}$ be the multiset of roots of these polynomials. Let $a_1=a_{m,n}$ be one of the numbers in $S$ with the maximal real part, and let $b_1=m$. Similarly, let $a_r=a_{s,t}\left(r\geq 2\right)$ be one of the numbers in $S\setminus\{a_1, \ldots, a_{r-1}\}$ ($r\geq 2$) with the maximal real part, and $b_r=s$. Let $k=m_1+\ldots+m_l$. Then $L=V_{a_1}(\omega_{b_1})\otimes V_{a_2}(\omega_{b_2})\otimes\ldots\otimes V_{a_k}(\omega_{b_k})$ is a highest weight representation of $\ysp$, and its associated polynomial is $\pi$.
\end{theorem}

The structure of the local Weyl modules is obtained, so are the dimensions.
\begin{theorem}The local Weyl module $W(\pi)$ of $\ysp$ associated to $\pi$ is isomorphic to $L=V_{a_1}(\omega_{b_1})\otimes V_{a_2}(\omega_{b_2})\otimes\ldots\otimes V_{a_k}(\omega_{b_k})$.
\end{theorem}
\begin{proof}
On the one hand, by Theorem \ref{ubodowm} $\operatorname{Dim}\big(W(\pi)\big)\leq \operatorname{Dim}\big(W(\lambda)\big)$; on the other hand, $L$ is a quotient of $W(\pi)$, and then $\operatorname{Dim}\Big(W\left(\pi\right)\Big)\geq \operatorname{Dim}\left(L\right)$. By Corollary \ref{dkrvawocsp}, we have $\operatorname{Dim}\big(W(\lambda)\big)=\operatorname{Dim}\left(L\right)$, and then this implies that $\operatorname{Dim}\big(W(\pi)\big)=\operatorname{Dim}\left(L\right)$. Therefore
$W\left(\pi\right)\cong L$.

The dimension of the local Weyl module $W(\pi)$ can be recovered by Theorem \ref{fmsimlieC} and Remark \ref{rofry}.\end{proof}

\chapter{Local Weyl modules of $\yso$}
In this chapter, the local Weyl modules of $\yso$ are studied. The structure of the local Weyl modules is determined, and the dimensions of the local Weyl modules are obtained.  In the process of characterizing the local Weyl modules, a sufficient condition for a tensor product of fundamental representations of Yangians to be a highest weight representation is obtained, which shall lead to an irreducibility criterion for the tensor product.

Let $\pi=\big(\pi_1(u),\pi_2(u),\ldots, \pi_l(u)\big)$ be a generic $l$-tuple of monic polynomials in $u$, and $\pi_i\left(u\right)=\prod\limits_{j=1}^{m_i}\left(u-a_{i,j}\right)$. Let $k=m_1+m_2+\ldots+m_l$, $S=\{a_{i,j}|i=1,\ldots,l; j=1,\ldots,m_i\}$, and $\lambda=\sum\limits_{i=1}^{l}m_i\omega_i$.
Let $a_{m,n}$ be one of the numbers in $S$ with the maximal real part. Then define $a_1=a_{m,n}$ and $b_1=m$. Inductively, let $a_{s,t}$ be one of the numbers in $S-\{a_1,\ldots, a_{r-1}\}$ with the maximal real part. Then define $a_r=a_{s,t}$ and $b_r=s$. We construct a tensor product $L=V_{a_1}\left(\omega_{b_1}\right)\otimes V_{a_2}\left(\omega_{b_2}\right)\otimes \ldots \otimes V_{a_k}\left(\omega_{b_k}\right),$ where $V_{a_i}(\omega_{b_i})$ are fundamental representations of $\yso$. We prove that $L$ is a highest weight representation. A standard discussion shows that $L$ is a highest weight representation associated to $\pi$.
The dimension of $V_{a_i}\left(\omega_{b_i}\right)$ is known, so is the one of $L$. Then a lower bound on the dimension of the local Weyl module $W(\pi)$ is obtained.

Let $W(\lambda)$ be the Weyl module associated to $\lambda$ of the current algebra $\nyso[t]$. In \cite{Na}, the author proved $\operatorname{Dim}\Big(W(\lambda)\Big)=\prod\limits_{i\in I} \Big(\operatorname{Dim}\big(W(\omega_i)\big)\Big)^{m_i}.$ It follows from Corollary \ref{dkrvawocsp} that $\operatorname{Dim}(W(\lambda))=\operatorname{Dim}(L)$. Since $\operatorname{Dim}\big(W(\lambda)\big)\geq \operatorname{Dim}(W(\pi))\geq \operatorname{Dim}(L)$,  $$W\left(\pi\right)\cong V_{a_1}(\omega_{b_1})\otimes V_{a_2}(\omega_{b_2})\otimes\ldots\otimes V_{a_k}(\omega_{b_k}).$$
The dimension of $W(\pi)$ can be recovered from Theorem \ref{fmotbc1} and Proposition \ref{dcofbdl}.

Similar to the proof that $L$ is a highest weight representation (Proposition \ref{mtoysob}), we can obtain a sufficient condition for a tensor product of fundamental representations of the form $V_{a_1}(\omega_{b_1})\otimes V_{a_2}(\omega_{b_2})\otimes\ldots\otimes V_{a_k}(\omega_{b_k})$ to be a highest weight representation. If $a_j-a_i\notin S(b_i, b_j)$ for $1\leq i<j\leq k$, then $L$ is a highest weight representation, where $S(b_i, b_j)$ is a finite set of positive rational numbers. By Proposition \ref{VoWWoVhi} and Lemma \ref{dualfrc1}, an irreducible criterion for a tensor product of fundamental representations of $\yso$ is obtained: if $a_j-a_i\notin S(b_i, b_j)$ for $1\leq i\neq j\leq k$, then $L$ is irreducible.

\section{From the highest weight vector to the lowest one of $V_a(\omega_i)$}
In this section, we would like to find a path from the highest weight vector to the lowest one of a fundamental representation $V_a(\omega_i)$ of $\ysp$.
Let $s_i$ be the fundamental reflections of the Weyl group of $\mathfrak{so}\left(2l+1,\C\right)$ for $i=1,\ldots, l$. Let $\{\mu_1,\mu_2,\ldots,\mu_{l}\}$ be the coordinate functions on the Cartan subalgebra of $\nyso$. For $1\leq i\leq l-1$, $s_i\left(\mu_i\right)=\mu_{i+1}$, $s_i\left(\mu_{i+1}\right)=\mu_{i}$ and $s_i\left(\mu_j\right)=\mu_j$ for $j\neq i, i+1$.
For $i=l$, $s_l\left(\mu_{l}\right)=-\mu_{l}$ and $s_l\left(\mu_j\right)=\mu_j$ for $j\neq l$.
The fundamental weights of $\mathfrak{so}\left(2l+1,\C\right)$ are given by $\omega_i=\mu_1+\ldots+\mu_i$ for $1\leq i\leq l-1$ and $\omega_{l}=\frac{1}{2} \left(\mu_1+\ldots+\mu_{l-1}+\mu_l\right)$. It follows that $\mu_1=\omega_1$, $\mu_2=-\omega_1+\omega_2$, $\ldots$, $\mu_{l-1}=-\omega_{l-2}+\omega_{l-1}$, and $\mu_{l}=-\omega_{l-1}+2\omega_{l}$.
$\alpha_1=2\omega_1-\omega_2$, $\alpha_i=-\omega_{i-1}+2\omega_{i}-\omega_{i+1}$($2\leq i\leq l-2$), $\alpha_{l-1}=-\omega_{l-2}+2\omega_{l-1}-2\omega_{l}$, and $\alpha_{l}=-\omega_{l-1}+2\omega_{l}$.
\begin{proposition}
$s_i\left(\omega_j\right)=\omega_j$ for $i\neq j$. $s_1\left(\omega_1\right)=-\omega_1+\omega_2$, $s_2\left(\omega_2\right)=\omega_1-\omega_2+\omega_3,\ldots,$ $s_{l-3}\left(\omega_{l-3}\right)=\omega_{l-4}-\omega_{l-3}+\omega_{l-2}$, $s_{l-2}\left(\omega_{l-2}\right)=\omega_{l-3}-\omega_{l-2}+\omega_{l-1}$, $s_{l-1}\left(\omega_{l-1}\right)=\omega_{l-2}-\omega_{l-1}+2\omega_{l}$ and $s_{l}\left(\omega_{l}\right)=\omega_{l-1}-\omega_{l}$.
\end{proposition}

Let $w_0=-1$ be the longest element of the Weyl group $\W$ of $\nyso$. One reduced expression of the longest element $w_0$  is given by:
\begin{align*}
   w_0&=  \left(s_l\right)\left(s_{l-1}s_{l}s_{l-1}\right)\left(s_{l-2}s_{l-1}s_ls_{l-1}s_{l-2}\right)\left(s_{l-3}s_{l-2}s_{l-1}s_ls_{l-1}s_{l-2}s_{l-3}\right)\ldots\\
   & \left(s_{2}\ldots s_{l-2}s_{l-1}s_ls_{l-1}s_{l-2}\ldots s_{2}\right)\left(s_{1}\ldots s_{l-2}s_{l-1}s_ls_{l-1}s_{l-2}\ldots s_{1}\right).
\end{align*}
According to the reduce expression of $w_0$, define
\begin{align*}
   w_i&=\left(s_i\ldots s_{l-1}s_ls_{l-1}\ldots s_{i}\right)\ldots\left(s_{2}\ldots s_{i-1}s_i\ldots s_{l-1}s_ls_{l-1}\ldots s_{i}\right)\\
   &\left(s_{1}\ldots s_{i-1}s_i\ldots s_{l-1}s_ls_{l-1}\ldots s_{i}\right).
\end{align*}
Let $\sigma_k\in \W$ be the product of the last $k$ terms in $w_i$ and keep the same order as in $w_i$. For every suitable value $j$, there exists a $k'\in I$ such that $\sigma_{k+1}=s_{k'}\sigma_{k}$. 
Let $v_{\sigma_k(\omega_i)}$ be a vector in the weight space of weight $\sigma_k(\omega_i)$. Since the weight space of weight $\omega_i$ is 1-dimensional, the weight space of weight $\sigma_k(\omega_i)$ is 1-dimensional, and then $v_{\sigma_k(\omega_i)}$ is unique, up to a scalar.

\begin{proposition}\label{rileq2b}
Let $\sigma_k(\omega_i)=r_{k'}\omega_{k'}+\sum\limits_{k'\neq j} r_j\omega_j$. Then $r_{i'}\in\{1,2\}$.
\end{proposition}
\begin{proof}
The proof of this proposition is similar to Proposition \ref{rileq2}. We only provides the detail of the computation of $w_0(\omega_i)$.

Case 1: $i=1$.
\begin{flushleft}
$\omega_1\xlongrightarrow{s_1}-\omega_1+\omega_2\xlongrightarrow{s_2}-\omega_2+\omega_3\xlongrightarrow{s_3}-\omega_3+\omega_4\xlongrightarrow{s_4} \ldots\xlongrightarrow{s_{l-3}}-\omega_{l-3}+\omega_{l-2}\xlongrightarrow{s_{l-2}}-\omega_{l-2}+\omega_{l-1}\xlongrightarrow{s_{l-1}}-\omega_{l-1}+2\omega_{l} \xlongrightarrow{s_{l}}\omega_{l-1}-2\omega_{l}\xlongrightarrow{s_{l-1}}\omega_{l-2}-\omega_{l-1}\xlongrightarrow{s_2\ldots s_{l-3}s_{l-2}}\omega_{1}-\omega_{2}\xlongrightarrow{s_{1}}-\omega_{1}\xlongrightarrow{s_{l}\left(s_{l-1}s_ls_{l-1}\right)\ldots\left(s_{2}\ldots s_ls_{l-1}\ldots s_2\right)}-\omega_{1}$.
\end{flushleft}

Case 2: $1<i\leq l-1$.
\begin{flushleft}
$\omega_i\xlongrightarrow{s_{i-1}\ldots s_1}\omega_i\xlongrightarrow{s_i}\omega_{i-1}-\omega_{i}+\omega_{i+1}\xlongrightarrow{s_{l-3}\ldots s_{i+1}}\omega_{i-1}-\omega_{l-3}+\omega_{l-2}\xlongrightarrow{s_{l-2}}\omega_{i-1}-\omega_{l-2}+\omega_{l-1}\xlongrightarrow{s_{l-1}} \omega_{i-1}-\omega_{l-1}+2\omega_{l}\xlongrightarrow{s_{l}}\omega_{i-1}+\omega_{l-1}-2\omega_{l}\xlongrightarrow{s_{l-1}}\omega_{i-1}+
\omega_{l-2}-\omega_{l-1}\xlongrightarrow{s_{i}\ldots s_{l-3}s_{l-2}}2\omega_{i-1}-\omega_{i}\xlongrightarrow{s_{i-1}}2\omega_{i-2}-2\omega_{i-1}+\omega_{i}\xlongrightarrow{s_2\ldots s_{i-2}}2\omega_{1}-2\omega_{2}+\omega_{i}\xlongrightarrow{s_1}-2\omega_1+\omega_{i}\xlongrightarrow{s_{2}\ldots s_{l-1}s_ls_{l-1}\ldots s_2}-2\omega_2+\omega_{i}\xlongrightarrow{s_{3}\ldots s_{l-1}s_ls_{l-1}\ldots s_3}-2\omega_3+\omega_{i}\xlongrightarrow{\left(s_{i}\ldots s_ls_{l-1}\ldots s_{i}\right)\ldots\left(s_{2}\ldots s_ls_{l-1}\ldots s_2\right)}-\omega_{i}\xlongrightarrow{s_{l}\left(s_{l-1}s_ls_{l-1}\right)\ldots\left(s_{i+1}\ldots s_{l-1}s_ls_{l-1}\ldots s_{i+1}\right)}-\omega_{i}$.
\end{flushleft}

Case 3:  $i=l$.
\begin{flushleft}
$\omega_{l}\xlongrightarrow{s_{l-1}\ldots s_1}\omega_{l}\xlongrightarrow{s_{l}}\omega_{l-1}-\omega_{l}\xlongrightarrow{s_{l-1}}\omega_{l-2}-\omega_{l-1}+\omega_{l}
\xlongrightarrow{s_2\ldots s_{l-3}s_{l-2}}\omega_{1}-\omega_{2}+\omega_{l}\xlongrightarrow{s_1}-\omega_{1}+\omega_{l}\xlongrightarrow{s_{2}\ldots s_{l-1}s_{l}s_{l-1}\ldots s_2}-\omega_{2}+\omega_{l}\xlongrightarrow {\left(s_{l-2}s_{l-1}s_{l}s_{l-1}s_{l-2}\right)\ldots\left(s_{3}\ldots s_{l-1}s_{l}s_{l-1}\ldots s_3\right)}-\omega_{l-2}+\omega_{l}\xlongrightarrow{s_{l-1}}-\omega_{l-2}+\omega_{l}\xlongrightarrow{s_{l}}-\omega_{l-2}+\omega_{l-1}-\omega_{l} \xlongrightarrow{s_{l-1}}-\omega_{l-1}+\omega_{l}\xlongrightarrow{s_{l}}-\omega_{l}.$
\end{flushleft}
\end{proof}

\begin{proposition}\label{v+=gv-b}
Let $v_{1}^{+}$ and $v_{1}^{-}$ be highest and lowest weight vectors in the fundamental representation $V_{a_1}(\omega_i)$ of $Y\left(\mathfrak{so}\left(2l,\C\right)\right)$. There exists $y\in \nyso$ such that $$v_{1}^{-}=y.v_{1}^{+}.$$
\end{proposition}
\begin{proof}
It follows from Proposition \ref{dcofbdl} that $V_a(\omega_i)=L(\omega_i)\oplus M$ as a $\nyso$-module, and both $v_{1}^{+}$ and $v_{1}^{-}$ are in $L(\omega_i)$. Note that $s_{j'}\Big(\sigma_j(\omega_i)\Big)=\sigma_j(\omega_i)-\frac{2\Big(\sigma_j(\omega_i),\alpha_{j'}\Big)}{\Big(\alpha_{j'},\alpha_{j'}\Big)}\alpha_{j'}.$  
By Proposition \ref{rileq2b},
$s_{j'}\Big(\sigma_j(\omega_i)\Big)=\sigma_j(\omega_i)-r_{j'}\alpha_{j'}.$



Case 1: $1\leq i\leq l-1$.
\begin{align*}
  v^{-}_1 &= \left(x_{i,0}^{-}\ldots x_{l-1,0}^{-}\left(x_{l,0}^{-}\right)^2x_{l-1,0}^{-}\ldots x_{i,0}^{-}\right)\\
  &\Big(\left(x_{i-1,0}^{-}\right)^2x_{i,0}^{-}\ldots x_{l-1,0}^{-}\left(x_{l,0}^{-}\right)^2x_{l-1,0}^{-}\ldots x_{i,0}^{-}\Big) \\
  &\qquad\qquad\qquad\vdots\\
   &\Big(\left(x_{2,0}^{-}\right)^2\ldots \left(x_{i-1,0}^{-}\right)^2x_{i,0}^{-}\ldots x_{l-1,0}^{-}\left(x_{l,0}^{-}\right)^2x_{l-1,0}^{-}\ldots x_{i,0}^{-}\Big)\\
   &\Big(\left(x_{1,0}^{-}\right)^2\ldots \left(x_{i-1,0}^{-}\right)^2x_{i,0}^{-}\ldots x_{l-1,0}^{-}\left(x_{l,0}^{-}\right)^2x_{l-1,0}^{-}\ldots x_{i,0}^{-}\Big)v^{+}_1.
\end{align*}

Case 2: $i=l$.
$$v^{-}_1=
x_{l,0}^{-}\left(x_{l-1,0}^{-}x_{l,0}^{-}\right)\ldots\left(x_{2,0}^{-}x_{3,0}^{-}\ldots x_{l-1,0}^{-} x_{l,0}^{-}\right)\left(x_{1,0}^{-}x_{2,0}^{-}\ldots  x_{l-1,0}^{-}x_{l,0}^{-}\right)v^{+}_1.$$
\end{proof}
\section{On a lower bound for the local Weyl modules of $\yso$}

In this section, we would like to obtain a lower bound for the local Weyl module $W(\pi)$ of $\yso$. To achieve this, we construct a highest weight representation whose associated polynomial is $\pi$. It follows from the universal property of the local Weyl modules of the Yangian that a lower bound can be obtained.

\begin{proposition}\label{mtoysob}
Let $L=V_{a_1}(\omega_{b_1})\otimes V_{a_2}(\omega_{b_2})\otimes\ldots\otimes V_{a_k}(\omega_{b_k})$, where $b_{i}\in\{1,\ldots, l\}$. If $\operatorname{Re}\left(a_1\right)\geq \operatorname{Re}\left(a_2\right)\geq \ldots \geq \operatorname{Re}\left(a_k\right)$, then $L$ is a highest weight representation.
\end{proposition}
\begin{proof}


Let $v_m^{+}$ be the highest weight vector of $V_{a_m}\left(\omega_{b_m}\right)$. Let $v_1^{-}$ be the lowest weight vector of $V_{a_1}\left(\omega_{b_1}\right)$.

We prove this proposition by induction on $k$.
Without loss of generality, we may assume that $k\geq 2$ and assume that $V_{a_2}(\omega_{b_2})\otimes V_{a_3}(\omega_{b_3})\otimes\ldots\otimes V_{a_k}(\omega_{b_k})$ is a highest weight representation of $\yso$. 
Thus the highest weight vector of $V_{a_2}(\omega_{b_2})\otimes V_{a_3}(\omega_{b_3})\otimes\ldots\otimes V_{a_k}(\omega_{b_k})$ is $v^{+}=v_2^{+}\otimes \ldots\otimes v_k^{+}$. To show that $L$ is a highest weight representation, by Corollary \ref{v-w+gvtw}, it suffices to show that $$v^{-}_{1}\otimes v^{+}\in \yso\left(v^{+}_1\otimes v^{+}\right).$$ We divide the proof into the following steps.

Step 1: $\sigma_i^{-1}\left(\alpha_{i'}\right)\in \Delta^{+}$.

Proof: It is easy to check by the definition of $\sigma_i$ and the way we choose $i'$.




Step 2: $Y_{i'}\left(v_{\sigma_i\left(\omega_{b_1}\right)}\right)$ is a highest weight module of $Y_{i'}$.

Proof: Since the weight $\sigma_i\left(\omega_{b_1}\right)$ is on the Weyl group orbit of the highest weight and the representation $V_{a_1}(\omega_{b_1})$ is finite-dimensional, the weight space of weight $\sigma_i\left(\omega_{b_1}\right)$ is 1-dimensional. The elements $h_{j,s}$ form a commutative subalgebra, so $v_{\sigma_i\left(\omega_{b_1}\right)}$ is an eigenvector of $h_{i',r}$. Therefore we only have to show that $v_{\sigma_i\left(\omega_{b_1}\right)}$ is a maximal vector. Suppose to the contrary that $x_{i',k}^{+}v_{\sigma_i\left(\omega_{b_1}\right)}\neq 0$.  Then $x_{i',k}^{+}v_{\sigma_i\left(\omega_{b_1}\right)}$ is a weight vector of weight $\sigma_i(\omega_{b_1})+\alpha_{i'}$, so $\omega_{b_1}+\sigma_i^{-1}\left(\alpha_{i'}\right)$ is a weight. Because $\sigma_i^{-1}\left(\alpha_{i'}\right)\in \Delta^{+}$,  $\omega_{b_1}$ is a weight preceding the weight $\omega_{b_1}+\sigma_i^{-1}\left(\alpha_{i'}\right)$, which contradicts the maximality of $\omega_{b_1}$ in the representation $L\left(\omega_{b_1}\right)$.



Step 3: Let $P\left(u\right)$ be the associated polynomial of $Y_{i'}\left(v_{\sigma_i\left(\omega_{b_1}\right)}\right)$. Then as a highest weight $Y_{i'}$-module, $Y_{i'}\left(v_{\sigma_i\left(\omega_{b_1}\right)}\right)$ has highest weight $\frac{P\left(u+1\right)}{P\left(u\right)}$.

Proof: It follows from the representation theory of $\ysl$.

Step 4: $P\left(u\right)$ has of degree 1 or 2.

Proof: The degree of $P\left(u\right)$ equals the eigenvalue of $v_{\sigma_i\left(\omega_{b_1}\right)}$ under $\tilde{h}_{i',0}$. Note that $$\tilde{h}_{i',0} v_{\sigma_i\left(\omega_{b_1}\right)}=\Bigg(\sigma_{i}\left(\omega_{b_1}\right)\left(\tilde{h}_{i',0}\right)\Bigg)v_{\sigma_i\left(\omega_{b_1}\right)}.$$
It follows from Proposition \ref{rileq2b} that the degree of $P\left(u\right)$ equals 1 or 2.


Step 5: If $\operatorname{Deg}\Big(P\left(u\right)\Big)=1$, $Y_{i'}\left(v_{\sigma_i\left(\omega_{b_1}\right)}\right)$ is 2-dimensional and isomorphic to $W_1\left(a\right)$. If $\operatorname{Deg}\Big(P\left(u\right)\Big)=2$, then $Y_{i'}\left(v_{\sigma_i\left(\omega_{b_1}\right)}\right)$ is either 3-dimensional or 4-dimensional and isomorphic to $W_2\left(a\right)$ or $W_1\left(b\right)\otimes W_1\left(a\right)\left(\operatorname{Re}\left(b\right)>\operatorname{Re}\left(a\right)\right)$, respectively. Moreover, for any $i\in I$, $\operatorname{Re}\left(a\right)\geq \operatorname{Re}\left(\frac{a_1}{d_i}\right)$. The values of both $b$ and $a$ will be explicitly computed in the next section.

Proof: We suppose that this step is true for this moment. This step will be showed in Section \ref{mtoysols}.

Step 6:  $Y_{i'}\left(v_{\sigma_i\left(\omega_{b_1}\right)}\otimes v_2^{+}\otimes\ldots \otimes v_k^{+}\right)=Y_{i'}\left(v_{\sigma_i\left(\omega_{b_1}\right)}\right)\otimes Y_{i'}\left(v_2^{+}\right)\otimes\ldots\otimes Y_{i'}\left(v_k^{+}\right)$.

Proof: For $i\neq l$, let $W$ be one of the modules $W_1\left(a\right)$, $W_2\left(a\right)$ or $W_1\left(b\right)\otimes W_1\left(a\right)$; for $i=l$, let $W$ be one of the modules $W_1\left(a\right)$, $W_2\left(a\right)$ or $W_1\left(b\right)\otimes W_1\left(a\right)$. $Y_{i'}\left(v_m^{+}\right)$ is nontrivial if and only if $b_m=i'$.    Suppose in $\{b_1,\ldots,b_k\}$ that $b_{j_1}=\ldots=b_{j_{m}}=b_j=i'$ with $j_1<\ldots<j_{m}$; moreover, if $s\notin\{j_1,\ldots, j_m\}$, then $b_s\neq i'$.
Note in Step 5 that $Y_{i'}\left(v_{\sigma_i\left(\omega_{b_1}\right)}\right)\cong W$. Thus if $i'=l$,
\begin{align*}
   Y_{i'}\left(v_{\sigma_i\left(\omega_{b_1}\right)}\right)&\otimes Y_{i'}\left(v_2^{+}\right)\otimes\ldots\otimes Y_{i'}\left(v_k^{+}\right) \\
   &= \begin{cases}
 W\otimes W_1\left(a_{j_1}\right)\otimes \ldots\otimes W_{1}\left(a_{j_m}\right)\text{\ if $b_1\neq i'$} \\
 W\otimes W_1\left(a_{j_2}\right)\otimes \ldots\otimes W_{1}\left(a_{j_m}\right)\text{\ if $b_1= i'$};
\end{cases}
\end{align*}
if $i'\neq l$,
\begin{align*}
  Y_{i'}\left(v_{\sigma_i\left(\omega_{b_1}\right)}\right) &\otimes Y_{i'}\left(v_2^{+}\right)\otimes\ldots\otimes Y_{i'}\left(v_k^{+}\right) \\
   &=\begin{cases}
 W\otimes W_1\left(\frac{a_{j_1}}{2}\right)\otimes \ldots\otimes W_{1}\left(\frac{a_{j_m}}{2}\right)\text{\ if $b_1\neq i'$} \\
 W\otimes W_1\left(\frac{a_{j_2}}{2}\right)\otimes \ldots\otimes W_{1}\left(\frac{a_{j_m}}{2}\right)\text{\ if $b_1= i'$}.
\end{cases}
\end{align*}

Since $\operatorname{Re}\left(a\right)\geq \operatorname{Re}\left(a_1\right)\geq \operatorname{Re}\left(a_{j_1}\right)\geq \ldots\geq \operatorname{Re}\left(a_{j_m}\right)$, it follows from either Corollary \ref{tpihw} or Corollary \ref{tpihw2} that $Y_{i'}\left(v_{\sigma_i\left(\omega_{b_1}\right)}\right)\otimes Y_{i'}\left(v_2^{+}\right)\otimes\ldots\otimes Y_{i'}\left(v_k^{+}\right)$ is a highest weight $Y_{i'}$-module with highest weight vector $v_{\sigma_i\left(\omega_{b_1}\right)}\otimes v_2^{+}\otimes\ldots \otimes v_k^{+}$. Thus
$$Y_{i'}\left(v_{\sigma_i\left(\omega_{b_1}\right)}\otimes v_2^{+}\otimes\ldots \otimes v_k^{+}\right)\supseteq Y_{i'}\left(v_{\sigma_i\left(\omega_{b_1}\right)}\right)\otimes Y_{i'}\left(v_2^{+}\right)\otimes\ldots\otimes Y_{i'}\left(v_k^{+}\right).$$

By the coproduct of the Yangians and Proposition \ref{c1dgd2d} that $$Y_{i'}\left(v_{\sigma_i\left(\omega_{b_1}\right)}\otimes v_2^{+}\otimes\ldots \otimes v_k^{+}\right)\subseteq Y_{i'}\left(v_{\sigma_i\left(\omega_{b_1}\right)}\right)\otimes Y_{i'}\left(v_2^{+}\right)\otimes\ldots\otimes Y_{i'}\left(v_k^{+}\right).$$
Therefore the claim is true.

Step 7: $v_{\sigma_{i+1}(\omega_{b_1})}\otimes v^{+}\in Y_{i'}\left(v_{\sigma_i\left(\omega_{b_1}\right)}\otimes v^{+}\right)$.

Proof: $v_{\sigma_{i+1}(\omega_{b_1})}\otimes v^{+}\in Y_{i'}\left(v_{\sigma_{i}\omega_{b_1}}\right)\otimes Y_{i'}\left(v^{+}\right)=Y_{i'}\left(v_{\sigma_i\left(\omega_{b_1}\right)}\otimes v^{+}\right)$ by Step 6.

Step 8: $v_1^{-}\otimes v^{+}\in \yso \left(v_1^{+}\otimes v^{+}\right)$.

Proof: It follows from Step 7 immediately by induction.

It follows from Step 8 and Corollary \ref{v-w+gvtw} that $L=\yso \left(v_1^{+}\otimes v^{+}\right)$.
\end{proof}
The proof above leads us to an irreducibility criterion for the tensor product of fundamental representations of the Yangian.

\begin{remark}
The following is the record of the root (s) of the associated polynomials of $Y_{i'}\left(v_{\sigma_i\left(\omega_{b_1}\right)}\right)$ for the possible $i$ values. We refer to Remark \ref{yocap1} for the notation used below.

When $1\leq b_1\leq l-1$,

\noindent$\frac{a_1}{2}\xlongrightarrow{x_{b_1,0}^{-}}\frac{a_1}{2}+\frac{1}{2}\xlongrightarrow{x_{b_1+1,0}^{-}}\frac{a_1}{2}+1 \xlongrightarrow{x_{b_1+2,0}^{-}}\ldots \xlongrightarrow{x_{l-2,0}^{-}}\frac{a_1}{2}+\frac{l-b_1-1}{2}\xlongrightarrow{x_{l-1,0}^{-}}$\newline $\left(a_1+l-b_1, a_1+l-b_1-1\right) \xlongrightarrow{\left(x_{l,0}^{-}\right)^2}\frac{a_1}{2}+\frac{l-b_1}{2} \xlongrightarrow{x_{l-1,0}^{-}}\frac{a_1}{2}+\frac{l-b_1+1}{2}\xlongrightarrow{x_{l-2,0}^{-}}\ldots \xlongrightarrow{x_{b_1+1,0}^{-}}\frac{a_1}{2}+\frac{l-b_1}{2}+\frac{l-b_1-1}{2}=\frac{a_1}{2}+l-b_1-\frac{1}{2} \xlongrightarrow{x_{b_1,0}^{-}}\left(\frac{a_1}{2}+l-b_1, \frac{a_1}{2}+\frac{1}{2}\right) \xlongrightarrow{\left(x_{b_1-1,0}^{-}\right)^2}\left(\frac{a_1}{2}+l-b_1+\frac{1}{2}, \frac{a_1}{2}+1\right)
\xlongrightarrow{\left(x_{b_1-2,0}^{-}\right)^2}\ldots \xlongrightarrow{\left(x_{2,0}^{-}\right)^2}\\\left(\frac{a_1}{2}+l-b_1+\frac{b_1-2 }{2}=\frac{a_1}{2}+l-\frac{b_1}{2}-1, \frac{a_1}{2}+\frac{b_1-1 }{2}\right)\xlongrightarrow{\left(x_{1,0}^{-}\right)^2}$

\noindent(The second parenthesis) $\left(\frac{a_1}{2}+1\right)\xlongrightarrow{x_{b_1,0}^{-}}\left(\frac{a_1}{2}+1\right)+\frac{1}{2}\xlongrightarrow{x_{b_1+1,0}^{-}}\ldots \xlongrightarrow{\left(x_{3,0}^{-}\right)^2}$\newline$\left(\frac{a_1}{2}+l-\frac{b_1-1}{2}, \frac{a_1}{2}+\frac{b_1}{2}\right)\xlongrightarrow{\left(x_{2,0}^{-}\right)^2}$

$\ldots$

\noindent(The last parenthesis) $\left(\frac{a_1}{2}+b_1-1\right)\xlongrightarrow{x_{b_1,0}^{-}}\left(\frac{a_1}{2}+b_1-1\right)+\frac{1}{2}\xlongrightarrow{x_{b_1+1,0}^{-}}\ldots \xlongrightarrow{x_{b_1+1,0}^{-}}\frac{a_1}{2}+b_1-1+l-b_1-\frac{1}{2}=\frac{a_1}{2}+l-\frac{3}{2} \xlongrightarrow{x_{b_1,0}^{-}}\text{the lowest weight vector reached}$.

When $b_1=l$,

\noindent$a_1\xlongrightarrow{x_{l,0}^{-}}\frac{a_1}{2}+\frac{1}{2}\xlongrightarrow{x_{l-1,0}^{-}}\ldots\xlongrightarrow{x_{3,0}^{-}} \frac{a_1}{2}+\frac{l-2}{2}\xlongrightarrow{x_{2,0}^{-}} \frac{a_1}{2}+\frac{l-1}{2}\xlongrightarrow{x_{1,0}^{-}}$

\noindent(The second parenthesis) $a_1+2\xlongrightarrow{x_{l,0}^{-}}\frac{a_1+2}{2}+\frac{1}{2}\xlongrightarrow{x_{l-1,0}^{-}}\ldots\xlongrightarrow{x_{3,0}^{-}} \frac{a_1+2}{2}+\frac{l-2}{2}=\frac{a_1}{2}+\frac{l}{2}\xlongrightarrow{x_{2,0}^{-}}$

$\ldots$

\noindent(The last two parentheses) $a_1+2\left(l-2\right)\xlongrightarrow{x_{l,0}^{-}}\frac{a_1+2\left(l-2\right)}{2}+\frac{1}{2}\xlongrightarrow{x_{l-1,0}^{-}}a_1+2\left(l-1\right)\xlongrightarrow{x_{l,0}^{-}}\text{the lowest weight vector reached}$.
\end{remark}
A precise condition for the cyclicity of the tensor product $L$ can be obtained in the next Theorem. We first show the following lemma. Since the proof is similar to the proof of Lemma \ref{C3l317ss}, we omit the proof.
\begin{lemma}
$V_{a_m}\left(\omega_{b_m}\right)\otimes V_{a_n}\left(\omega_{b_n}\right)$ is a highest weight representation if $a_n-a_m\notin S\left(b_m,b_n\right)$, where the set $S\left(b_m,b_n\right)$ is defined as follows:
\begin{enumerate}
  \item $S\left(b_m,b_n\right)=\left\{|b_m-b_n|+2+2r, 2l-(b_m+b_n)+1+2r|0\leq r<\text{min}\{b_m, b_n\}\right\}$, where $1\leq b_m, b_n\leq l-1$;
  \item  $S\left(l,b_n\right)=\left\{l-b_n+2+2r|0\leq r<b_n, 1\leq b_n\leq l-1\right\}$;
  \item $S\left(b_m,l\right)=\left\{l-b_m+1+r, l-b_m+r|0\leq r< b_m, 1\leq b_m\leq l-1\right\}$;
   \item $S\left(l,l\right)=\left\{1,3,\ldots,2l-1\right\}$.
\end{enumerate}
\end{lemma}

Similar to the proof of Theorem \ref{3mt2icl}, we can prove the following Theorem.
\begin{theorem}
Let $L=V_{a_1}(\omega_{b_1})\otimes V_{a_2}(\omega_{b_2})\otimes\ldots\otimes V_{a_k}(\omega_{b_k})$, and $S(b_i, b_j)$ be defined as the above lemma.
\begin{enumerate}
  \item If $a_j-a_i\notin S(b_i, b_j)$ for $1\leq i<j\leq k$, then $L$ is a highest weight representation of $\yso$.
  \item If $a_j-a_i\notin S(b_i, b_j)$ for $1\leq i\neq j\leq k$, then $L$ is an irreducible representation of $\yso$.
\end{enumerate}
\end{theorem}
\begin{remark}\label{c5ysoris1}
This remark is parallel to Remark \ref{c5yspris1}. It is not known what is the precise necessary and sufficient condition for the irreducibility of the tensor product $V_{a_{i}}\left(\omega_{b_{i}}\right)\otimes V_{a_{j}}\left(\omega_{b_{j}}\right)$. Our result on the cyclicity condition for $L$ is an analogue of a special case $m_1=m_2=1$ of the results in \cite{Ch3}. We now give a detail interpretation. Note that there is a different labeling on the nodes of the Dynkin Diagram and we transfer it to our labeling. The following come from case (ii) of Corollary 6.2 \cite{Ch3}. Our result on the cyclicity condition for $L$ is an analogue of a special case $m_1=m_2=1$ of the results in \cite{Ch3}. We now give a detail interpretation. Note that there is a different labeling on the nodes of the Dynkin Diagram and we transfer it to our labeling. The following come from case (ii) of Corollary 6.2 \cite{Ch3}.
\begin{align*}
\mathcal{S}(l,l)&=q^2\{q^{0}, q^4, \ldots, q^{4l-4}\}=\{q^2, q^6, \ldots, q^{4l-2}\}.\\
\mathcal{S}(l-i+1,l)&=q^4\{q^{2i-3}, q^{2i-1}, \ldots, q^{4l-2i-1}\}\\
&=\{q^{2i+1}, q^{2i+3}, \ldots, q^{4l-2i+3}\}.\\
\mathcal{S}(l,l-i+1)&=q^4\{q^{2i-3}, q^{2i-1}, \ldots, q^{4l-2i-3}\}\\
&=\{q^{2i+1}, q^{2i+3}, \ldots, q^{4l-2i+1}\}.\\
\mathcal{S}(l-i_1+1,l-i_2+1)&=\mathcal{S}(l-i_2+1,l-i_1+1)\left(i_1\leq i_2\right)\\
&=q^6\{q^{2i_2-2i_1}, q^{2i_2+2i_1-6}, \ldots, q^{4l-2i_1-2i_2}, q^{4l-2i_2+2i_1-6}\}\\
&=\{q^{2i_2-2i_1+6}, q^{2i_2+2i_1}, \ldots, q^{4l-2i_1-2i_2+6}, q^{4l-2i_2+2i_1}\}.\\
\mathcal{S}(l-i_1+1,l-i_2+1)&=\mathcal{S}(l-i_2+1,l-i_1+1)\left(i_1\geq i_2\right)\\
&=q^6\{q^{2i_1-2i_2}, q^{2i_1+2i_2-6}, \ldots, q^{4l-2i_1-2i_2}, q^{4l-2i_1+2i_2-6}\}\\
&=\{q^{2i_1-2i_2+6}, q^{2i_1+2i_2}, \ldots, q^{4l-2i_1-2i_2+6}, q^{4l-2i_1+2i_2}\}.
\end{align*}
For the reader's convenience we give a parallel result in the Yangian case by making the sets $S(b_i,b_j)$ more explicit.
\begin{align*}
S\left(l,l\right)&=\left\{1,3,\ldots,2l-1\right\}\\
S\left(l-i+1,l\right)&=\left\{i-1, i,\ldots, l\right\}.\\
S\left(l,l-i+1\right)&=\left\{i+1,i+3,\ldots, 2l-i+1\right\}.\\
S(l-i_1+1&,\ l-i_2+1)=S(l-i_2+1,l-i_1+1)\left(i_1\leq i_2\right)\\
&=\{{i_2-i_1}+2, i_2+i_1-1, \ldots, 2l-i_2-i_1+2, 2l-{i_2+i_1}-1 \}.\\
S(l-i_1+1&,\ l-i_2+1)=S(l-i_2+1,l-i_1+1)\left(i_1\geq i_2\right)\\
&=\{{i_1-i_2}+2, i_1+i_2-1, \ldots, 2l-i_1-i_2+2, 2l-{i_1+i_2}-1 \}.\\
\end{align*}
Note that the numbers in the sets $S(l-i_1+1,l-i_2+1)$ are the exponents of $q$ in $\mathcal{S}(l-i_1+1,l-i_2+1)$ up to a factor of 2 and a constant.
\end{remark}
\section{Supplement Step 5 of Proposition \ref{mtoysob}}\label{mtoysols}

We are going to prove the step 5 by dividing the proof into 3 cases: $b_1=1$, $b_1=l$, and $2\leq b_1\leq l-1$. Before we prove the step case by case, the following notices are important.

In the definition of Yangians of type $B_l$, $d_1=\ldots=d_{l-1}=2$ and $d_l=1$. Recall that the algebra spanned by all monomials in the generators  $x_{i,r}^{\pm}, h_{i,r}$  is isomorphic to $\ysl$. However, $x_{i,r}^{+}, h_{i,r}$ does not satisfy the defining relations of $\ysl$ for $i=1,2,\ldots, l-1$. In order to use all known results in Chapters 4 and 5, we need to re-scale the generators. Let $\tilde{x}_{i,r}^{\pm}=\frac{\sqrt{2}}{2^{r+1}}x_{i,r}^{\pm},\tilde{h}_{i,r}=\frac{1}{2^{r+1}}h_{i,r}$. Then the algebra spanned by all monomials in the generators  $\tilde{x}_{i,r}^{\pm}, \tilde{h}_{i,r}$  is isomorphic to $\ysl$ and the generators satisfy the defining relations of $\ysl$.
 Moreover, we need to calculate ``new" defining relations in terms of the new generators. For the thesis's purpose, we only list the formulas we are going to use.

The following proposition just follows from the defining relations of Yangians.
\begin{proposition}\label{newrelationsbl1}
Suppose that $1\leq i, j\leq l-1$.
\begin{equation*}\label{}
[\tilde{h}_{i,r},\tilde{h}_{j,s}]=0, \qquad [\tilde{h}_{i,0}\ \tilde{x}_{j,s}^{-}]={}-
a_{ij}\tilde{x}_{j,s}^{-}, 
\end{equation*}
\begin{equation*}\label{}
[\tilde{h}_{i,1}, \tilde{x}_{i\pm 1,0}^{-}]=\tilde{x}_{i\pm 1,1}^{-}
+\frac{1}{2}\tilde{x}_{i\pm 1,0}^{-}+\tilde{x}_{i\pm 1,0}^{-}\tilde{h}_{i,0}\left(i\neq l-1\right),
\end{equation*}
\begin{equation*}\label{}
[\tilde{h}_{l-1,1}, \tilde{x}_{l-2,0}^{-}]=\tilde{x}_{l-2,1}^{-}
+\frac{1}{2}\tilde{x}_{l-2,0}^{-}+\tilde{x}_{l-2,0}^{-}\tilde{h}_{l-1,0},
\end{equation*}
\begin{equation*}\label{}
 [h_{l,0},\tilde{x}_{l-1,0}^{-}]=
2\tilde{x}_{l-1,0}^{-}, \qquad [h_{l,0},\tilde{x}_{l-1,1}^{-}]=
2\tilde{x}_{l-1,1}^{-},
\end{equation*}
\begin{equation*}\label{}
[h_{l,1}, \tilde{x}_{l-1,0}^{-}]=4\tilde{x}_{l-1,1}^{-}+
2\tilde{x}_{l-1,0}^{-}+2\tilde{x}_{l-1,0}^{-}h_{l,0},
\end{equation*}
\begin{equation*}\label{}
[h_{l,2}, \tilde{x}_{l-1,0}^{-}]=2[h_{l,1}, \tilde{x}_{l-1,1}^{-}]+
h_{l,1}\tilde{x}_{l-1,0}^{-}+\tilde{x}_{l-1,0}^{-}h_{l,1},
\end{equation*}
\begin{equation*}\label{}
 [\tilde{h}_{l-1,0},x_{l,0}^{-}]=
x_{l,0}^{-}, \qquad [\tilde{h}_{l-1,0},x_{l,1}^{-}]=
x_{l,1}^{-},
\end{equation*}
\begin{equation*}\label{}
[\tilde{h}_{l-1,1}, x_{l,0}^{-}]=\frac{1}{2}\left(x_{l,1}^{-}+
x_{l,0}^{-}+2x_{l,0}^{-}\tilde{h}_{l-1,0}\right).
\end{equation*}
\end{proposition}
For the simplicity of some discussions in the future, we also denote $x_{l,0}^{\pm}$ and $h_{l,0}$ by $\tilde{x}_{l,0}^{\pm}$ and $\tilde{h}_{l,0}$, respectively.

\subsection{Case 1: $ b_1=1$}

This is the easiest case. We claim that
\begin{proposition}\
\begin{enumerate}
  \item For $1\leq k\leq l-1$, $Y_k\left(\tilde{x}_{k-1,0}^{-}\tilde{x}_{k-2,0}^{-}\ldots \tilde{x}_{1,0}^{-}v^{+}_1\right)\cong W_1\left(\frac{a_1}{2}+\frac{k-1}{2}\right)$.
  \item $Y_l\left(\tilde{x}_{l-1,0}^{-}\tilde{x}_{l-2,0}^{-}\ldots \tilde{x}_{1,0}^{-}v^{+}_1\right)$ is isomorphic to either $W_2\left(a_1+l-2\right)$ or\\ $W_1\left(a_1+l-1\right)\otimes W_1\left(a_1+l-2\right).$
  \item For $1\leq k\leq l-1$, $Y_{k}\left(\tilde{x}_{k+1,0}^{-}\tilde{x}_{k+2,0}^{-}\ldots \tilde{x}_{l-1,0}^{-}\left(x_{l,0}^{-}\right)^2\tilde{x}_{l-1,0}^{-}\ldots \tilde{x}_{1,0}^{-}v^{+}_1\right)$ is isomorphic to $W_1\left(\frac{a_1}{2}+\frac{2l-k-2}{2}\right)$.
\end{enumerate}
\end{proposition}
\begin{proof}
The proof of (i) is similar to the one of Proposition \ref{i=1yo2n} by replacing $a_1$ by $\frac{a_1}{2}$. We omit the proof.
The proofs of (ii) and (iii) are in Lemmas \ref{blyll}, \ref{bllyl-1} and \ref{bllym}.
\end{proof}


\begin{lemma}\label{blyll}
$Y_l\left(\tilde{x}_{l-1,0}^{-}\tilde{x}_{l-2,0}^{-}\ldots \tilde{x}_{1,0}^{-}v^{+}_1\right)$ is isomorphic to either $W_2\left(a_1+l-2\right)$ or $W_1\left(a_1+l-1\right)\otimes W_1\left(a_1+l-2\right)$.
\end{lemma}
\begin{proof} By Proposition \ref{rileq2b}, $wt(\tilde{x}_{l-1,0}^{-}\tilde{x}_{l-2,0}^{-}\ldots \tilde{x}_{1,0}^{-}v^{+}_1)=-\omega_{l-1}+2\omega_{l}$, and then
\begin{align*}
  h_{l,0}\tilde{x}_{l-1,0}^{-}\tilde{x}_{l-2,0}^{-}\ldots \tilde{x}_{1,0}^{-}v^{+}_1 &= 2\tilde{x}_{l-1,0}^{-}\tilde{x}_{l-2,0}^{-}\ldots \tilde{x}_{1,0}^{-}v^{+}_1.
\end{align*}
Thus the associated polynomial $P\left(u\right)$ of $Y_l\left(\tilde{x}_{l-1,0}^{-}\tilde{x}_{l-2,0}^{-}\ldots \tilde{x}_{1,0}^{-}v^{+}_1\right)$ has of degree 2. Suppose $P\left(u\right)=\left(u-a\right)\left(u-b\right)$ with $\operatorname{Re}\left(a\right)\leq \operatorname{Re}\left(b\right)$.
The eigenvalues of both $\tilde{x}_{l-1,0}^{-}\tilde{x}_{l-2,0}^{-}\ldots \tilde{x}_{1,0}^{-}v^{+}_1$ under $h_{l,1}$ and under $h_{l,2}$ will tell the values of $a$ and $b$.
\begin{align*}
h_{l,1}\tilde{x}_{l-1,0}^{-}&\tilde{x}_{l-2,0}^{-}\ldots \tilde{x}_{1,0}^{-}v^{+}_1 \\
  &= [h_{l,1},\tilde{x}_{l-1,0}^{-}]\tilde{x}_{l-2,0}^{-}\ldots \tilde{x}_{1,0}^{-}v^{+}_1 \\
   &= \left(4\tilde{x}_{l-1,1}^{-}+2\tilde{x}_{l-1,0}^{-}+2\tilde{x}_{l-1,0}^{-}h_{l,0}\right) \tilde{x}_{l-2,0}^{-}\ldots \tilde{x}_{1,0}^{-}v^{+}_1\\
   &= \left(4\tilde{x}_{l-1,1}^{-}+2\tilde{x}_{l-1,0}^{-}\right) \tilde{x}_{l-2,0}^{-}\ldots \tilde{x}_{1,0}^{-}v^{+}_1\\
      &= \left(2\left(a_1+l-2\right)+2\right)\tilde{x}_{l-1,0}^{-} \tilde{x}_{l-2,0}^{-}\ldots \tilde{x}_{1,0}^{-}v^{+}_1\\
      &= 2\left(a_1+l-1\right)\tilde{x}_{l-1,0}^{-} \tilde{x}_{l-2,0}^{-}\ldots \tilde{x}_{1,0}^{-}v^{+}_1.
\end{align*}
\begin{align*}
h_{l,2}\tilde{x}_{l-1,0}^{-}&\tilde{x}_{l-2,0}^{-}\ldots \tilde{x}_{1,0}^{-}v^{+}_1 \\
  &= [h_{l,2},\tilde{x}_{l-1,0}^{-}]\tilde{x}_{l-2,0}^{-}\ldots \tilde{x}_{1,0}^{-}v^{+}_1 \\
   &= \Big(2[h_{l,1},\tilde{x}_{l-1,1}^{-}]+\left(h_{l,1}\tilde{x}_{l-1,0}^{-}+\tilde{x}_{l-1,0}^{-}h_{l,1}\right) \Big) \tilde{x}_{l-2,0}^{-}\ldots \tilde{x}_{1,0}^{-}v^{+}_1\\
   &= \big(2h_{l,1}\tilde{x}_{l-1,1}^{-}+h_{l,1}\tilde{x}_{l-1,0}^{-}\big) \tilde{x}_{l-2,0}^{-}\ldots \tilde{x}_{1,0}^{-}v^{+}_1\\
   &= h_{l,1}\left(2\tilde{x}_{l-1,1}^{-}+\tilde{x}_{l-1,0}^{-}\right) \tilde{x}_{l-2,0}^{-}\ldots \tilde{x}_{1,0}^{-}v^{+}_1\\
      &= h_{l,1}\left(a_1+l-1\right)\tilde{x}_{l-1,0}^{-} \tilde{x}_{l-2,0}^{-}\ldots \tilde{x}_{1,0}^{-}v^{+}_1\\
      &= 2\left(a_1+l-1\right)^2\tilde{x}_{l-1,0}^{-} \tilde{x}_{l-2,0}^{-}\ldots \tilde{x}_{1,0}^{-}v^{+}_1.
\end{align*}
By Lemma \ref{g2puisd3}, the associated polynomial of $Y_l\left(\tilde{x}_{l-1,0}^{-}\ldots \tilde{x}_{1,0}^{-}v^{+}_1\right)$ as $\ysl$-module is $P\left(u\right)=\Big(u-\left(a_1+l-2\right)\Big)\Big(u-\left(a_1+l-1\right)\Big)$. Then it follows from the theory of the local Weyl modules of $\ysl$ that $Y_l\left(\tilde{x}_{l-1,0}^{-}\tilde{x}_{l-2,0}^{-}\ldots \tilde{x}_{1,0}^{-}v^{+}_1\right)$ is isomorphic to either $W_2\left(a_1+l-2\right)$ or $W_1\left(a_1+{l-1}\right)\otimes W_1\left(a_1+{l-2}\right)$.
\end{proof}

\begin{lemma}\label{bllyl-1}
$Y_{l-1}\Big(\left(x_{l,0}^{-}\right)^2\tilde{x}_{l-1,0}^{-}\ldots \tilde{x}_{1,0}^{-}v^{+}_1\Big)\cong W_1\left(\frac{a_1}{2}+\frac{l-1}{2}\right)$. 
\end{lemma}
\begin{proof}By Proposition \ref{rileq2b}, $wt\left(x_{l,0}^{-}\right)^2\tilde{x}_{l-1,0}^{-}\ldots \tilde{x}_{1,0}^{-}v^{+}_1=\omega_{l-1}-2\omega_{l}$, and then
$$\tilde{h}_{l-1,0}\left(x_{l,0}^{-}\right)^2\tilde{x}_{l-1,0}^{-}\ldots \tilde{x}_{1,0}^{-}v^{+}_1=\left(x_{l,0}^{-}\right)^2\tilde{x}_{l-1,0}^{-}\ldots \tilde{x}_{1,0}^{-}v^{+}_1.$$
Thus the associated polynomial $P\left(u\right)$ has of degree 1. Say $P(u)=u-a$. 
The eigenvalue of $\left(x_{l,0}^{-}\right)^2\tilde{x}_{l-1,0}^{-}\tilde{x}_{l-2,0}^{-}\ldots \tilde{x}_{1,0}^{-}v^{+}_1$ under $\tilde{h}_{l-1,1}$ will tell the value of $a$.
\begin{align*}
\tilde{h}_{l-1,1}&\left(x_{l,0}^{-}\right)^2\tilde{x}_{l-1,0}^{-}\tilde{x}_{l-2,0}^{-}\ldots \tilde{x}_{1,0}^{-}v^{+}_1 \\
  &= [\tilde{h}_{l-1,1},x_{l,0}^{-}]x_{l,0}^{-}\tilde{x}_{l-1,0}^{-}\ldots \tilde{x}_{1,0}^{-}v^{+}_1+x_{l,0}^{-}[\tilde{h}_{l-1,1},x_{l,0}^{-}]\tilde{x}_{l-1,0}^{-}\ldots \tilde{x}_{1,0}^{-}v^{+}_1\\
  &+\left(x_{l,0}^{-}\right)^2 \tilde{h}_{l-1,1}\tilde{x}_{l-1,0}^{-}\ldots \tilde{x}_{1,0}^{-}v^{+}_1 \\
   &= \frac{1}{2}\left(x_{l,1}^{-}+x_{l,0}^{-}+2x_{l,0}^{-}\tilde{h}_{l-1,0}\right)x_{l,0}^{-}\tilde{x}_{l-1,0}^{-}\ldots \tilde{x}_{1,0}^{-}v^{+}_1\\
   &+ \frac{1}{2}x_{l,0}^{-}\left(x_{l,1}^{-}+x_{l,0}^{-}+2x_{l,0}^{-}\tilde{h}_{l-1,0}\right)\tilde{x}_{l-1,0}^{-}\ldots \tilde{x}_{1,0}^{-}v^{+}_1\\
   &- \left(\frac{a_1}{2}+\frac{l-2}{2}\right)\left(x_{l,0}^{-}\right)^2\tilde{x}_{l-1,0}^{-}\tilde{x}_{l-2,0}^{-}\ldots \tilde{x}_{1,0}^{-}v^{+}_1\\
    &= \frac{1}{2}\left(x_{l,1}^{-}x_{l,0}^{-}+x_{l,0}^{-}x_{l,1}^{-}\right)\tilde{x}_{l-1,0}^{-}\tilde{x}_{l-2,0}^{-}\ldots \tilde{x}_{1,0}^{-}v^{+}_1\\
    &- \left(\frac{a_1}{2}+\frac{l-2}{2}\right)\left(x_{l,0}^{-}\right)^2\tilde{x}_{l-1,0}^{-}\tilde{x}_{l-2,0}^{-}\ldots \tilde{x}_{1,0}^{-}v^{+}_1\\
     &= \frac{1}{2}\left(2a_1+2l-3\right)\left(x_{l,0}^{-}\right)^2\tilde{x}_{l-1,0}^{-}\tilde{x}_{l-2,0}^{-}\ldots \tilde{x}_{1,0}^{-}v^{+}_1\\
    &- \left(\frac{a_1}{2}+\frac{l-2}{2}\right)\left(x_{l,0}^{-}\right)^2\tilde{x}_{l-1,0}^{-}\tilde{x}_{l-2,0}^{-}\ldots \tilde{x}_{1,0}^{-}v^{+}_1\\
    &= \frac{1}{2} \left(a_1+l-1\right)\left(x_{l,0}^{-}\right)^2\tilde{x}_{l-1,0}^{-}\tilde{x}_{l-2,0}^{-}\ldots \tilde{x}_{1,0}^{-}v^{+}_1.
\end{align*}
\end{proof}
\begin{lemma}\label{bllym} Let $2\leq m\leq l-1$.
$Y_{m-1}\Big(\tilde{x}_{m,0}^{-}\ldots \tilde{x}_{l-1,0}^{-}\left(x_{l,0}^{-}\right)^2\tilde{x}_{l-1,0}^{-}\ldots \tilde{x}_{1,0}^{-}v^{+}_1\Big)$ is isomorphic to $W_1\left(\frac{a_1}{2}+\frac{2l-m-1}{2}\right)$.
\end{lemma}

\begin{proof}
By Proposition \ref{rileq2b}, $wt(\tilde{x}_{m,0}^{-}\ldots \tilde{x}_{l-1,0}^{-}\left(x_{l,0}^{-}\right)^2\tilde{x}_{l-1,0}^{-}\ldots \tilde{x}_{1,0}^{-}v^{+}_1)=\omega_{m-1}-\omega_m$, and then $$\tilde{h}_{m-1,0}\tilde{x}_{m,0}^{-}\ldots \tilde{x}_{l-1,0}^{-}\left(x_{l,0}^{-}\right)^2\tilde{x}_{l-1,0}^{-}\ldots \tilde{x}_{1,0}^{-}v^{+}_1=\tilde{x}_{m,0}^{-}\ldots \left(x_{l,0}^{-}\right)^2\ldots \tilde{x}_{1,0}^{-}v^{+}_1.$$
Thus we know that $$Y_{m-1}\Big(\tilde{x}_{m,0}^{-}\ldots \tilde{x}_{l-1,0}^{-}\left(x_{l,0}^{-}\right)^2\tilde{x}_{l-1,0}^{-}\ldots \tilde{x}_{1,0}^{-}v^{+}_1\Big)\cong W_1\left(a\right).$$
Thus the associated polynomial $P\left(u\right)$ to $$Y_{m-1}\Big(\tilde{x}_{m,0}^{-}\ldots \tilde{x}_{l-1,0}^{-}\left(x_{l,0}^{-}\right)^2\tilde{x}_{l-1,0}^{-}\tilde{x}_{l-2,0}^{-}\ldots \tilde{x}_{1,0}^{-}v^{+}_1\Big)$$ has of degree 1. Suppose $P\left(u\right)=u-a$.

We claim that $a=\frac{a_1}{2}+\frac{2l-m-1}{2}$.
We use induction on $m$ downward.

When $m=l-1$, this is the basis of induction.
\begin{align*}
\tilde{h}_{l-2,1}\tilde{x}_{l-1,0}^{-}&\left(x_{l,0}^{-}\right)^2\tilde{x}_{l-1,0}^{-}\ldots \tilde{x}_{1,0}^{-}v^{+}_1\\
&=[\tilde{h}_{l-2,1},\tilde{x}_{l-1,0}^{-}]\left(x_{l,0}^{-}\right)^2\tilde{x}_{l-1,0}^{-}\tilde{x}_{l-2,0}^{-}\ldots \tilde{x}_{1,0}^{-}v^{+}_1\\
&+\tilde{x}_{l-1,0}^{-}\left(x_{l,0}^{-}\right)^2\tilde{h}_{l-2,1}\tilde{x}_{l-1,0}^{-}\tilde{x}_{l-2,0}^{-}\ldots x_{1,0}^{-}v^{+}_1\\
&=\left(\tilde{x}_{l-1,1}^{-}+\frac{1}{2}\tilde{x}_{l-1,0}^{-}+\tilde{x}_{l-1,0}^{-}\tilde{h}_{l-2,0}\right)\left(x_{l,0}^{-}\right)^2\tilde{x}_{l-1,0}^{-}\tilde{x}_{l-2,0}^{-}\ldots \tilde{x}_{1,0}^{-}v^{+}_1+0\\
&=\left(\frac{a_1}{2}+\frac{l-1}{2}+\frac{1}{2}\right)\tilde{x}_{l-1,0}^{-}\left(x_{l,0}^{-}\right)^2\tilde{x}_{l-1,0}^{-}\tilde{x}_{l-2,0}^{-}\ldots \tilde{x}_{1,0}^{-}v^{+}_1\\
&=\left(\frac{a_1}{2}+\frac{l}{2}\right)\tilde{x}_{l-1,0}^{-}\left(x_{l,0}^{-}\right)^2\tilde{x}_{l-1,0}^{-}\tilde{x}_{l-2,0}^{-}\ldots \tilde{x}_{1,0}^{-}v^{+}_1.
\end{align*}
Supposed that claims are true for all value $k$ such that $k>m$. We next show the case when  $k=m$.
\begin{align*}
\tilde{h}_{m-1,1}&\tilde{x}_{m,0}^{-}\ldots \tilde{x}_{l-1,0}^{-}\left(x_{l,0}^{-}\right)^2\tilde{x}_{l-1,0}^{-}\ldots \tilde{x}_{1,0}^{-}v^{+}_1  \\
   &= [\tilde{h}_{m-1,1}\tilde{x}_{m,0}^{-}]\tilde{x}_{m+1,0}^{-}\ldots \tilde{x}_{l-1,0}^{-}\left(x_{l,0}^{-}\right)^2\tilde{x}_{l-1,0}^{-}\ldots \tilde{x}_{1,0}^{-}v^{+}_1\\
   &+ \tilde{x}_{m,0}^{-}\ldots \tilde{x}_{l-1,0}^{-}\left(x_{l,0}^{-}\right)^2\tilde{x}_{l-1,0}^{-}\ldots\tilde{x}_{m+1,0}^{-}\tilde{h}_{m-1,1}\tilde{x}_{m,0}^{-}\tilde{x}_{m-1,0}^{-}\ldots \tilde{x}_{1,0}^{-}v^{+}_1\\
   &=\left(\tilde{x}_{m,1}^{-}+\frac{1}{2}\tilde{x}_{m,0}^{-}+\tilde{x}_{m,0}^{-}\tilde{h}_{m-1,0}\right)\tilde{x}_{m+1,0}^{-}\ldots \tilde{x}_{l-1,0}^{-}\left(x_{l,0}^{-}\right)^2\tilde{x}_{l-1,0}^{-}\ldots \tilde{x}_{1,0}^{-}v^{+}_1\\
   &= \left(\frac{a_1}{2}+\frac{2l-m-2}{2}+\frac{1}{2}\right)\tilde{x}_{m,0}^{-}\ldots \tilde{x}_{l-1,0}^{-}\left(x_{l,0}^{-}\right)^2\tilde{x}_{l-1,0}^{-}\ldots \tilde{x}_{1,0}^{-}v^{+}_1 \\
   &= \left(\frac{a_1}{2}+\frac{2l-m-1}{2}\right)\tilde{x}_{m,0}^{-}\ldots \tilde{x}_{l-1,0}^{-}\left(x_{l,0}^{-}\right)^2\tilde{x}_{l-1,0}^{-}\ldots \tilde{x}_{1,0}^{-}v^{+}_1.
\end{align*}


By induction, the claim is true.
\end{proof}

\subsection{Case 2: $b_1=l$.}

\begin{proposition}\label{cosoblp1}\
\begin{enumerate}
  \item $Y_l\left(v_1^{+}\right)\cong W_1\left(a_1\right)$.
  \item $Y_k\left(\tilde{x}_{k+1,0}^{-}\ldots \tilde{x}_{l-1,0}^{-}x_{l,0}^{-}v^{+}_1\right)\cong W_1\left(\frac{a_1}{2}+\frac{l-k}{2}\right)$ for $1\leq k\leq l-1$.
  \item $Y_l\left(\tilde{x}_{1,0}^{-}\tilde{x}_{2,0}^{-}\ldots \tilde{x}_{l-1,0}^{-}x_{l,0}^{-}v^{+}_1\right)\cong W_1\left(a_1+2\right).$
\end{enumerate}
\end{proposition}
\begin{proof}
The proofs of the first and second items are similar to the one in Proposition \ref{i=1yo2n}, so we omit the proof. The third item is proved by Lemma \ref{solv2+}.
\end{proof}

\begin{lemma}\label{solv2+}
$Y_l\left(\tilde{x}_{1,0}^{-}\tilde{x}_{2,0}^{-}\ldots \tilde{x}_{l-1,0}^{-}x_{l,0}^{-}v^{+}_1\right)\cong W_1\left(a_1+2\right).$
\end{lemma}
\begin{proof}
By Proposition \ref{rileq2b}, $wt(\tilde{x}_{1,0}^{-}\tilde{x}_{2,0}^{-}\ldots \tilde{x}_{l-1,0}^{-}x_{l,0}^{-}v^{+}_1)=-\omega_{1}+\omega_{l}$, and then  $$h_{l,0}\tilde{x}_{1,0}^{-}\tilde{x}_{2,0}^{-}\ldots \tilde{x}_{l-1,0}^{-}x_{l,0}^{-}v^{+}_1=\tilde{x}_{1,0}^{-}\tilde{x}_{2,0}^{-}\ldots \tilde{x}_{l-1,0}^{-}x_{l,0}^{-}v^{+}_1.$$
Thus $Y_l\left(\tilde{x}_{1,0}^{-}\tilde{x}_{2,0}^{-}\ldots \tilde{x}_{l-1,0}^{-}x_{l,0}^{-}v^{+}_1\right)\cong W_1\left(a\right).$ The eigenvalue of $\tilde{x}_{1,0}^{-}\tilde{x}_{2,0}^{-}\ldots \tilde{x}_{l-1,0}^{-}x_{l,0}^{-}v^{+}_1$ under $h_{l,1}$ will tell the value of $a$.
\begin{align*}
   & h_{l,1}\tilde{x}_{1,0}^{-}\tilde{x}_{2,0}^{-}\ldots \tilde{x}_{l-1,0}^{-}x_{l,0}^{-}v^{+}_1 \\
   &= \tilde{x}_{1,0}^{-}\tilde{x}_{2,0}^{-}\ldots \tilde{x}_{l-2,0}^{-}[h_{l,1}\tilde{x}_{l-1,0}^{-}]x_{l,0}^{-}v^{+}_1+\tilde{x}_{1,0}^{-}\tilde{x}_{2,0}^{-}\ldots \tilde{x}_{l-2,0}^{-}x_{l-1,0}^{-}h_{l,1}\tilde{x}_{l,0}^{-}v^{+}_1\\
   &= \tilde{x}_{1,0}^{-}\tilde{x}_{2,0}^{-}\ldots \tilde{x}_{l-2,0}^{-}\left(4\tilde{x}_{l-1,1}^{-}+
2\tilde{x}_{l-1,0}^{-}+2\tilde{x}_{l-1,0}^{-}h_{l,0}\right)x_{l,0}^{-}v^{+}_1\\
   &-a_1\tilde{x}_{1,0}^{-}\tilde{x}_{2,0}^{-}\ldots \tilde{x}_{l-1,0}^{-}x_{l,0}^{-}v^{+}_1\\
   &= \left(4\left(\frac{a_1}{2}+\frac{1}{2}\right)-a_1\right)\tilde{x}_{1,0}^{-}\tilde{x}_{2,0}^{-}\ldots \tilde{x}_{l-1,0}^{-}x_{l,0}^{-}v^{+}_1\\
   &= \left(a_1+2\right)\tilde{x}_{1,0}^{-}\tilde{x}_{2,0}^{-}\ldots \tilde{x}_{l-1,0}^{-}x_{l,0}^{-}v^{+}_1.
\end{align*}\end{proof}
Let $v_2=\tilde{x}_{1,0}^{-}\tilde{x}_{2,0}^{-}\ldots \tilde{x}_{l-1,0}^{-}x_{l,0}^{-}v^{+}_1$, and $Y^{\left(1\right)}=span\{x_{i,r}^{\pm}, h_{i,r}|i>1\}$. It follows from the defining relations of Yangians that $Y^{\left(1\right)}\cong Y\big(\mathfrak{so}(2l-1,\C)\big)$.

\begin{proposition}\
\begin{enumerate}
  \item $Y_k\left(\tilde{x}_{k+1,0}^{-}\ldots \tilde{x}_{l-1,0}^{-}x_{l,0}^{-}v_2\right)\cong W_1\left(\frac{a_1+2}{2}+\frac{l-k}{2}\right)$ for $2\leq k\leq l-1$.
  \item $Y_l\left(\tilde{x}_{2,0}^{-}\ldots \tilde{x}_{l-1,0}^{-}x_{l,0}^{-}v_2\right)\cong W_1\left(a_1+4\right)$.
\end{enumerate}
\end{proposition}
\begin{proof}
Denote $\mathbf{s}_1=s_1s_2\ldots s_{l-2}s_{l-1}s_l$. It is routine to check $\mathbf{s}_1^{-1}\left(\alpha_j\right)\in \Delta^{+}$ for $j=2,3,\ldots, l$. Similar to Step 2 of Proposition \ref{mtoysob}, we have $x_{j,0}^{+}v_2=0$. $\tilde{h}_{j,r}v_2$ is a scalar multiple of $v_2$. Thus $Y^{\left(1\right)}\left(v_2\right)$ is a highest weight representation. Since the weight of $v_2$ is $-\omega_1+\omega_l$, the degree of the associated polynomial $P_j$ equals 0 if $j\neq l$; and$\operatorname{Deg}\Big(P_l\left(u\right)\Big)=1$. Therefore $P_j\left(u\right)=1$ if $j\neq l$ and $P_l\left(u\right)=\left(u-a\right)$. It follows from Lemma \ref{solv2+} that $a=a_1+2$. The rest of the proof of this proposition is similar to the proofs of Proposition \ref{cosoblp1}, just replacing $a_1$ by $a_1+2$.
\end{proof}

Similarly we define $v_{m+1}=\tilde{x}_{m,0}^{-}\tilde{x}_{m+1,0}^{-}\ldots \tilde{x}_{l-1,0}^{-}x_{l,0}^{-}v_m$ for $m\leq l-2$. Let $Y^{\left(m\right)}=span\{x_{i,r}^{\pm}, h_{i,r}|i>m\}$ with $2\leq m\leq l-2$. Due to the defining relations of Yangians, we know $Y^{\left(m\right)}\cong Y\big(\mathfrak{so}\left(2\left(l-m\right)+1,\C\right)\big)$. 
Similarly, we have
\begin{proposition} Let $m+1\leq k\leq l$.
\begin{enumerate}
  \item $Y_k\left(x_{k+1,0}^{-}\ldots \tilde{x}_{l-1,0}^{-}x_{l,0}^{-}v_{m+1}\right)\cong W_1\left(\frac{a_1+2m}{2}+\frac{l-k}{2}\right)$.
  \item $Y_l\left(\tilde{x}_{m+1,0}^{-}\ldots \tilde{x}_{l-1,0}^{-}x_{l,0}^{-}v_{m+1}\right)\cong W_1\Big(a_1+2\left(m+1\right)\Big)$.
\end{enumerate}
\end{proposition}

Note that $v_1^{-}=x_{l,0}^{-}\tilde{x}_{l-1,0}^{-}x_{l,0}^{-}v_{l-1}.$ The computations will finish when $m=l-2$. At the same time, $Y^{(m)}\cong Y\big(\mathfrak{so}\left(5,\C\right)\big)$.

\subsection{Case 3: $2\leq b_1\leq l-1$.}

For the simplicity of the expression in this case, denote $b_1$ by $i$. We will denote
$\tilde{x}_{k+1,0}^{-}\ldots \tilde{x}_{l-1,0}^{-}\left(x_{l,0}^{-}\right)^2\tilde{x}_{l-1,0}^{-}\ldots \tilde{x}_{i,0}^{-}$ by $\overline{\tilde{x}_{k+1,0}^{-}\ldots \tilde{x}_{i,0}^{-}}$, and \\$\left(\tilde{x}_{m+1,0}^{-}\right)^2\ldots \left(\tilde{x}_{i-1,0}^{-}\right)^2\tilde{x}_{i,0}^{-}\ldots \tilde{x}_{l-1,0}^{-}\left(x_{l,0}^{-}\right)^2\tilde{x}_{l-1,0}^{-}\ldots \tilde{x}_{i,0}^{-}$ by $\overline{\left(\tilde{x}_{m+1,0}^{-}\right)^2\ldots \tilde{x}_{i,0}^{-}}.$
We summarize the main calculations in the following proposition.
\begin{proposition}\label{c3pb} Let $i\leq k\leq l-1$ and $1\leq m\leq i-2$.
\begin{enumerate}
  \item $Y_k\left(\tilde{x}_{k-1,0}^{-}\tilde{x}_{k-2,0}^{-}\ldots \tilde{x}_{i,0}^{-}v^{+}_1\right)\cong W_1\left(\frac{a_1}{2}+\frac{k-i}{2}\right)$.
  \item  $Y_l\left(\tilde{x}_{l-1,0}^{-}\tilde{x}_{l-2,0}^{-}\ldots \tilde{x}_{i,0}^{-}v^{+}_1\right)$ is isomorphic to either $W_2\left(a_1+l-i-1\right)$ or\\ $W_1\left(a_1+l-i\right)\otimes W_1\left(a_1+l-i-1\right)$.
  \item $Y_{k}\left(\overline{\tilde{x}_{k+1,0}^{-}\ldots \tilde{x}_{i,0}^{-}}v^{+}_1\right)\cong W_1\left(\frac{a_1}{2}+\frac{2l-k-i-1}{2}\right)$.
  \item
\begin{enumerate}
    \item $Y_{i-1}\Big(\overline{x_{i,0}^{-}\ldots x_{i,0}^{-}}v^{+}_1\Big)\cong W_1\left(\frac{a_1}{2}+l-i\right)\otimes W_1\left(\frac{a_1}{2}+\frac{1}{2}\right)$.
    \item $Y_{m}\Big(\overline{\left(\tilde{x}_{m+1,0}^{-}\right)^2\ldots \tilde{x}_{i,0}^{-}}v^{+}_1\Big)\cong W_1\left(\frac{a_1}{2}+\frac{2l-i-m-1}{2}\right)\otimes W_1\left(\frac{a_1}{2}+\frac{i-m}{2}\right)$.
  \end{enumerate}
  \item  $Y_{i}\left(\overline{\left(\tilde{x}_{1,0}^{-}\right)^2\ldots \tilde{x}_{i,0}^{-}}v^{+}_1\right)\cong W_1\left(\frac{a_1}{2}+1\right)$.
\end{enumerate}
\end{proposition}


\begin{proof}

The proofs of parts (i), (ii) and (iii) follow from similar calculations in the Case 1 of this section. We omit the proofs.




The Claim (iv) is shown in 3 steps.

Step 1: $Y_{i-1}\left(\overline{\tilde{x}_{i,0}^{-}\ldots \tilde{x}_{i,0}^{-}}v^{+}_1\right)\cong W_1\left(\frac{a_1}{2}+l-i\right)\otimes W_1\left(\frac{a_1}{2}+\frac{1}{2}\right)$.

Proof: The proof is provided in Lemma \ref{ipfs4so}.

Step 2: $Y_{i-2}\Big(\overline{\left(\tilde{x}_{i-1,0}^{-}\right)^2\ldots \tilde{x}_{i,0}^{-}}v^{+}_1\Big)\cong W_1\left(\frac{a_1}{2}+l-i+\frac{1}{2}\right)\otimes W_1\left(\frac{a_1}{2}+1\right)$.

Proof: The proof is similar to the proof of Lemma \ref{l-2i34d2}, so we omit the proof. 

Step 3: $Y_{m}\Big(\overline{\left(\tilde{x}_{m+1,0}^{-}\right)^2\ldots \tilde{x}_{i,0}^{-}}v^{+}_1\Big)\cong W_1\left(\frac{a_1}{2}+\frac{2l-i-m-1}{2}\right)\otimes W_1\left(\frac{a_1}{2}+\frac{i-m}{2}\right)$.

Proof: The proof of this step is similar to case $\left(i\right)$ of Lemma \ref{l-2i34d23}, by induction on $m$ downward. We omit the proof.

The proof of the claim $\left(v\right)$ is provided in Lemma \ref{c3l3b}.
\end{proof}

\begin{lemma}\label{ipfs4so}
$Y_{i-1}\left(\overline{\tilde{x}_{i,0}^{-}\ldots\tilde{x}_{i,0}^{-}}v^{+}_1\right)\cong W_1\left(\frac{a_1}{2}+l-i\right)\otimes W_1\left(\frac{a_1}{2}+\frac{1}{2}\right)$.
\end{lemma}
\begin{proof}

By Proposition \ref{rileq2b}, $wt(\overline{\tilde{x}_{i,0}^{-}\ldots\tilde{x}_{i,0}^{-}}v^{+}_1)=2\omega_{i-1}-\omega_{i}$, and then $$\tilde{h}_{i-1,0}\overline{\tilde{x}_{i,0}^{-}\ldots\tilde{x}_{i,0}^{-}}v^{+}_1=2\overline{\tilde{x}_{i,0}^{-}\ldots\tilde{x}_{i,0}^{-}}v^{+}_1,$$ which tells the associated polynomial $P\left(u\right)$ of $Y_i\left(v_2\right)$ has of degree two. Say $$P\left(u\right)=\left(u-a\right)\left(u-b\right).$$ The values of both $a$ and $b$  will follow from the calculations $$\tilde{h}_{i-1,1}\overline{\tilde{x}_{i,0}^{-}\ldots\tilde{x}_{i,0}^{-}}v^{+}_1=\left(a+b+1\right)\overline{\tilde{x}_{i,0}^{-}\ldots\tilde{x}_{i,0}^{-}}v^{+}_1,$$ and $$\tilde{h}_{i-1,2}\overline{\tilde{x}_{i,0}^{-}\ldots\tilde{x}_{i,0}^{-}}v^{+}_1=\left(a^2+b^2+a+b\right)\overline{\tilde{x}_{i,0}^{-}\ldots\tilde{x}_{i,0}^{-}}v^{+}_1.$$
\begin{align*}
   &\tilde{h}_{i-1,1}\overline{\tilde{x}_{i,0}^{-}\ldots\tilde{x}_{i,0}^{-}}v^{+}_1\\
   &= [\tilde{h}_{i-1,1},\tilde{x}_{i,0}^{-}]\overline{\tilde{x}_{i+1,0}^{-}\ldots\tilde{x}_{i,0}^{-}}v^{+}_1+\overline{\tilde{x}_{i,0}^{-}\ldots\tilde{x}_{i+1,0}^{-}}[\tilde{h}_{i-1,1},\tilde{x}_{i,0}^{-}]v^{+}_1\\
   &= \left(\tilde{x}_{i,1}^{-}+\frac{1}{2}\tilde{x}_{i,0}^{-}+\tilde{x}_{i,0}^{-}\tilde{h}_{i-1,0}\right)\overline{\tilde{x}_{i+1,0}^{-}\ldots\tilde{x}_{i,0}^{-}}v^{+}_1\\
   &+\overline{\tilde{x}_{i,0}^{-}\ldots\tilde{x}_{i+1,0}^{-}}\left(\tilde{x}_{i,1}^{-}+\frac{1}{2}\tilde{x}_{i,0}^{-}+\tilde{x}_{i,0}^{-}\tilde{h}_{i-1,0}\right)v^{+}_1\\
   &=\left(\frac{a_1}{2}+\frac{2l-2i-1}{2}+\frac{1}{2}+1\right)\overline{\tilde{x}_{i,0}^{-}\ldots\tilde{x}_{i,0}^{-}}v^{+}_1\\
   &+\left(\frac{a_1}{2}+\frac{1}{2}\right)\overline{\tilde{x}_{i,0}^{-}\ldots\tilde{x}_{i,0}^{-}}v^{+}_1\\
   &= \left(a_1+\frac{2l-2i+3}{2}\right)\overline{\tilde{x}_{i,0}^{-}\ldots\tilde{x}_{i,0}^{-}}v^{+}_1\\
   &= 2\left(\frac{a_1}{2}+\frac{2l-2i+3}{4}\right)\overline{\tilde{x}_{i,0}^{-}\ldots\tilde{x}_{i,0}^{-}}v^{+}_1.
\end{align*}
Let $c=\frac{a_1}{2}+\frac{2l-2i-1}{2}$. Note that $2\left(\frac{a_1}{2}+\frac{2l-2i+3}{4}\right)=2c-\frac{2l-2i-5}{2}$.
\begin{align*}
   &\tilde{h}_{i-1,2}\overline{\tilde{x}_{i,0}^{-}\ldots\tilde{x}_{i,0}^{-}}v^{+}_1\\
   &= [\tilde{h}_{i-1,2},\tilde{x}_{i,0}^{-}]\overline{\tilde{x}_{i+1,0}^{-}\ldots\tilde{x}_{i,0}^{-}}v^{+}_1+\overline{\tilde{x}_{i,0}^{-} \ldots\tilde{x}_{i+1,0}^{-}}[\tilde{h}_{i-1,2},\tilde{x}_{i,0}^{-}]v^{+}_1\\
   &= \left([\tilde{h}_{i-1,1},\tilde{x}_{i,1}^{-}]+\frac{1}{2}\left(\tilde{h}_{i-1,1}\tilde{x}_{i,0}^{-}+\tilde{x}_{i,0}^{-}\tilde{h}_{i-1,1}\right)\right) \overline{\tilde{x}_{i+1,0}^{-}\ldots\tilde{x}_{i,0}^{-}}v^{+}_1\\
   &+\overline{\tilde{x}_{i,0}^{-}\ldots\tilde{x}_{i+1,0}^{-}} \left([\tilde{h}_{i-1,1},\tilde{x}_{i,1}^{-}]+\frac{1}{2}\left(\tilde{h}_{i-1,1}\tilde{x}_{i,0}^{-}+\tilde{x}_{i,0}^{-}\tilde{h}_{i-1,1}\right)\right)v^{+}_1\\
   &=[\tilde{h}_{i-1,1},\tilde{x}_{i,1}^{-}]\overline{\tilde{x}_{i+1,0}^{-}\ldots\tilde{x}_{i,0}^{-}}v^{+}_1
   +\frac{1}{2}\tilde{h}_{i-1,1}\tilde{x}_{i,0}^{-}\overline{\tilde{x}_{i+1,0}^{-}\ldots\tilde{x}_{i,0}^{-}}v^{+}_1\\
   &+\overline{\tilde{x}_{i,0}^{-}\ldots\tilde{x}_{i+1,0}^{-}}\tilde{h}_{i-1,1}\tilde{x}_{i,0}^{-}v^{+}_1
   +\overline{\tilde{x}_{i,0}^{-}\ldots\tilde{x}_{i+1,0}^{-}}[\tilde{h}_{i-1,1},\tilde{x}_{i,1}^{-}]v^{+}_1\\
   &=\tilde{h}_{i-1,1}\tilde{x}_{i,1}^{-}\overline{\tilde{x}_{i+1,0}^{-}\ldots\tilde{x}_{i,0}^{-}}v^{+}_1
   -\tilde{x}_{i,1}^{-}\tilde{h}_{i-1,1}\overline{\tilde{x}_{i+1,0}^{-}\ldots\tilde{x}_{i,0}^{-}}v^{+}_1\\
   &+\frac{1}{2}\tilde{h}_{i-1,1}\tilde{x}_{i,0}^{-}\overline{\tilde{x}_{i+1,0}^{-}\ldots\tilde{x}_{i,0}^{-}}v^{+}_1
   +\overline{\tilde{x}_{i,0}^{-}\ldots\tilde{x}_{i+1,0}^{-}}\tilde{h}_{i-1,1}\tilde{x}_{i,1}^{-}v^{+}_1\\
   &+\overline{\tilde{x}_{i,0}^{-}\ldots\tilde{x}_{i+1,0}^{-}}\tilde{h}_{i-1,1}\tilde{x}_{i,0}^{-}v^{+}_1\\
   &=c\cdot\tilde{h}_{i-1,1}\tilde{x}_{i,0}^{-}\overline{\tilde{x}_{i+1,0}^{-}\ldots\tilde{x}_{i,0}^{-}}v^{+}_1
   -\tilde{x}_{i,1}^{-}\overline{\tilde{x}_{i+1,0}^{-}\ldots\tilde{x}_{i+1,0}^{-}}\tilde{h}_{i-1,1}\tilde{x}_{i,0}^{-}v^{+}_1\\
   &+\frac{1}{2}\tilde{h}_{i-1,1}\tilde{x}_{i,0}^{-}\overline{\tilde{x}_{i+1,0}^{-}\ldots\tilde{x}_{i,0}^{-}}v^{+}_1
   +\overline{\tilde{x}_{i,0}^{-}\ldots\tilde{x}_{i+1,0}^{-}}\tilde{h}_{i-1,1}\tilde{x}_{i,1}^{-}v^{+}_1\\
   &+\overline{\tilde{x}_{i,0}^{-}\ldots\tilde{x}_{i+1,0}^{-}}\tilde{h}_{i-1,1}\tilde{x}_{i,0}^{-}v^{+}_1\\
   &=\left(c+\frac{1}{2}\right)\left(2c-\frac{2l-2i-5}{2}\right)\overline{\tilde{x}_{i,0}^{-}\ldots\tilde{x}_{i,0}^{-}}v^{+}_1-
   c\left(\frac{a_1}{2}+\frac{1}{2}\right)\overline{\tilde{x}_{i,0}^{-}\ldots\tilde{x}_{i,0}^{-}}v^{+}_1\\
   &+\left(\frac{a_1}{2}+\frac{1}{2}\right)\frac{a_1}{2}\overline{\tilde{x}_{i,0}^{-}\ldots\tilde{x}_{i,0}^{-}}v^{+}_1+
   \left(\frac{a_1}{2}+\frac{1}{2}\right)\overline{\tilde{x}_{i,0}^{-}\ldots\tilde{x}_{i,0}^{-}}v^{+}_1\\
   &= \left(\frac{a_1}{2}+\frac{2l-2i-1}{2}+\frac{1}{2}\right)\left(a_1+\frac{2l-2i+3}{2}\right)\overline{\tilde{x}_{i,0}^{-}\ldots\tilde{x}_{i,0}^{-}}v^{+}_1\qquad\qquad\qquad\qquad\\
       &- \left(\frac{a_1}{2}+\frac{1}{2}\right)\left(\frac{a_1}{2}+\frac{2l-2i-1}{2}\right)\overline{\tilde{x}_{i,0}^{-} \ldots\tilde{x}_{i,0}^{-}}v^{+}_1\\
       &+ \frac{a_1}{2}\left(\frac{a_1}{2}+\frac{1}{2}\right)\overline{\tilde{x}_{i,0}^{-}\ldots\tilde{x}_{i,0}^{-}}v^{+}_1+ \left(\frac{a_1}{2}+\frac{1}{2}\right)\overline{\tilde{x}_{i,0}^{-}\ldots\tilde{x}_{i,0}^{-}}v^{+}_1\\
       &=  \left(2\left(\frac{a_1}{2}+\frac{2l-2i+3}{4}\right)^2+\frac{\left(2l-2i-3\right)\left(2l-2i+1\right)}{8}\right)\overline{\tilde{x}_{i,0}^{-}\ldots\tilde{x}_{i,0}^{-}}v^{+}_1.
\end{align*}
Similar to Lemma \ref{l-2is34d}, $a=\frac{a_1}{2}+\frac{1}{2}$ and $b=\frac{a_1}{2}+l-i$. By the local Weyl modules theory of $\ysl$, we have
$$Y_{i-1}\left(\overline{\tilde{x}_{i,0}^{-}\ldots\tilde{x}_{i,0}^{-}}v^{+}_1\right)\cong W_1\left(\frac{a_1}{2}+l-i\right)\otimes W_1\left(\frac{a_1}{2}+\frac{1}{2}\right).$$
\end{proof}
\begin{lemma}\label{c3l3b}
Let $v_2=\overline{\left(\tilde{x}_{1,0}^{-}\right)^2\ldots \tilde{x}_{i,0}^{-}}v^{+}_1$. $Y_{i}\left(v_2\right)\cong W_1\left(\frac{a_1+2}{2}\right).$
\end{lemma}
\begin{proof}
By Proposition \ref{rileq2b}, $wt(v_2)=-2\omega_1+\omega_i$, and then $\tilde{h}_{i,0}v_2=v_2$, which tells us that the associated polynomial $P\left(u\right)$ of $Y_i\left(v_2\right)$ has of degree one. Say $P\left(u\right)=u-a$. The value of $a$ will follow from the calculation $\tilde{h}_{i,1}v_2=av_2$.

It follows from Proposition \ref{rileq2b} that $wt(\tilde{x}_{i,0}^{-}\ldots \tilde{x}_{l-1,0}^{-}\left(x_{l,0}^{-}\right)^2\tilde{x}_{l-1,0}^{-}\ldots \tilde{x}_{i,0}^{-}v^{+}_1)=2\omega_{i-1}-\omega_i$. $wt(\tilde{x}_{i-1,0}^{-}\overline{\tilde{x}_{i,0}^{-}\ldots \tilde{x}_{i,0}^{-}}v^{+}_1)=2\omega_{i-1}-\omega_i-\left(-\omega_{i-2}+2\omega_{i-1}-\omega_i\right)=\omega_{i-2}$.
\begin{align*}
  \tilde{h}_{i,1} &v_2 \\
   &=\tilde{h}_{i,1}\left(\tilde{x}_{1,0}^{-}\right)^2\ldots \left(\tilde{x}_{i-1,0}^{-}\right)^2\tilde{x}_{i,0}^{-}\ldots \tilde{x}_{l-1,0}^{-}\left(x_{l,0}^{-}\right)^2\tilde{x}_{l-1,0}^{-}\ldots \tilde{x}_{i,0}^{-}v^{+}_1\\
   &= \left(\tilde{x}_{1,0}^{-}\right)^2\ldots \left(\tilde{x}_{i-2,0}^{-}\right)^2[\tilde{h}_{i,1},\tilde{x}_{i-1,0}^{-}]\tilde{x}_{i-1,0}^{-}\tilde{x}_{i,0}^{-}\ldots \tilde{x}_{l-1,0}^{-}\left(x_{l,0}^{-}\right)^2\tilde{x}_{l-1,0}^{-}\ldots \tilde{x}_{i,0}^{-}v^{+}_1\\
   &+ \left(\tilde{x}_{1,0}^{-}\right)^2\ldots \left(\tilde{x}_{i-2,0}^{-}\right)^2\tilde{x}_{i-1,0}^{-}[\tilde{h}_{i,1},\tilde{x}_{i-1,0}^{-}]\tilde{x}_{i,0}^{-}\ldots \tilde{x}_{l-1,0}^{-}\left(x_{l,0}^{-}\right)^2\tilde{x}_{l-1,0}^{-}\ldots \tilde{x}_{i,0}^{-}v^{+}_1\\
   &+\left(\tilde{x}_{1,0}^{-}\right)^2\ldots \left(\tilde{x}_{i-1,0}^{-}\right)^2 \tilde{h}_{i,1}\tilde{x}_{i,0}^{-}\ldots \tilde{x}_{l-1,0}^{-}\left(x_{l,0}^{-}\right)^2\tilde{x}_{l-1,0}^{-}\ldots \tilde{x}_{i,0}^{-}v^{+}_1\\
   &= \left(\tilde{x}_{1,0}^{-}\right)^2\ldots \left(\tilde{x}_{i-2,0}^{-}\right)^2\left(\tilde{x}_{i-1,1}^{-}+\frac{1}{2}\tilde{x}_{i-1,0}^{-}+\tilde{x}_{i-1,0}^{-}\tilde{h}_{i,0}\right)\tilde{x}_{i-1,0}^{-}\overline{\tilde{x}_{i,0}^{-}\ldots \tilde{x}_{i,0}^{-}}v^{+}_1\\
   &+ \left(\tilde{x}_{1,0}^{-}\right)^2\ldots \left(\tilde{x}_{i-2,0}^{-}\right)^2\tilde{x}_{i-1,0}^{-}\left(\tilde{x}_{i-1,1}^{-}+\frac{1}{2}\tilde{x}_{i-1,0}^{-}+\tilde{x}_{i-1,0}^{-}\tilde{h}_{i,0}\right)\overline{\tilde{x}_{i,0}^{-}\ldots \tilde{x}_{i,0}^{-}}v^{+}_1\\
   &+\left(\tilde{x}_{1,0}^{-}\right)^2\ldots \left(\tilde{x}_{i-1,0}^{-}\right)^2 \tilde{h}_{i,1}\tilde{x}_{i,0}^{-}\ldots \tilde{x}_{l-1,0}^{-}\left(x_{l,0}^{-}\right)^2\tilde{x}_{l-1,0}^{-}\ldots \tilde{x}_{i,0}^{-}v^{+}_1\\
   &= \Bigg({a_1}+\frac{2l-2i+1}{2}+\frac{1}{2}+0+\frac{1}{2}-1-\left(\frac{a_1}{2}+\frac{2l-2i-1}{2}\right)\Bigg)v_2\\
   &= \left(\frac{a_1}{2}+1\right)v_2.
\end{align*}
Therefore $a=\frac{a_1+2}{2}$.
\end{proof}

\begin{proposition} Let $i\leq k\leq l-1$ and $2\leq m\leq i-2$.
\begin{enumerate}
  \item $Y_k\left(\tilde{x}_{k-1,0}^{-}\tilde{x}_{k-2,0}^{-}\ldots \tilde{x}_{i,0}^{-}v_2\right)\cong W_1\left(\frac{a_1+2}{2}+\frac{k-i}{2}\right)$.
  \item  $Y_l\left(\tilde{x}_{l-1,0}^{-}\tilde{x}_{l-2,0}^{-}\ldots \tilde{x}_{i,0}^{-}v_2\right)$ is isomorphic to either $W_2\left(\left(a_1+2\right)+l-i-1\right)$ or $W_1\left(\left(a_1+2\right)+l-i\right)\otimes W_1\left(\left(a_1+2\right)+l-i-1\right)$.
  \item $Y_{k}\left(\overline{\tilde{x}_{k+1,0}^{-}\ldots \tilde{x}_{i,0}^{-}}v_2\right)\cong W_1\left(\frac{a_1+2}{2}+\frac{2l-k-i-1}{2}\right)$.
  \item
\begin{enumerate}
    \item $Y_{i-1}\Big(\overline{\left(x_{i,0}^{-}\right)\ldots x_{i,0}^{-}}v_2\Big)\cong W_1\left(\frac{a_1+2}{2}+l-i\right)\otimes W_1\left(\frac{a_1+2}{2}+\frac{1}{2}\right)$.
    \item $Y_{m}\Big(\overline{\left(\tilde{x}_{m+1,0}^{-}\right)^2\ldots \tilde{x}_{i,0}^{-}}v_2\Big)\cong W_1\left(\frac{a_1+2}{2}+\frac{2l-i-m-1}{2}\right)\otimes W_1\left(\frac{a_1+2}{2}+\frac{i-m}{2}\right)$.
  \end{enumerate}
  \item $Y_{i}\Big(\overline{\left(\tilde{x}_{2,0}^{-}\right)^2\ldots \tilde{x}_{i,0}^{-}}v_2\Big)\cong W_1\left(\frac{a_1}{2}+2\right)$.
\end{enumerate}
\end{proposition}

\begin{proof}
Removed the first node in the Dynkin diagram of the Lie algebra of type $B_l$,  we get a simple Lie algebra which is isomorphic to Lie algebra of type $B_{l-1}$. Denote $\{2,3,\ldots, l\}$ by $I'$.  Let $Y^{\left(1\right)}$ be the Yangian generated by all $x_{j,r}^{\pm}$ and $h_{j,r}$ for $j\in I'$ and $r\in \mathbb{Z}_{\geq 0}$. $Y^{1}\cong \yso$.

From Proposition \ref{rileq2b}, the weight of $v_2$ is $\mathbf{s}_1(\omega_i)$. Similar to Step 2 of Proposition \ref{mtoysob}, we have $x_{j,0}^{+}v_2=0$. Therefore $v_2$ is a maximal vector of $Y^{\left(1\right)}$. Note that $h_{j,r}v_2$ is a scalar multiple of $v_2$. $v_2$ has weight $-2\omega_1+\omega_i$ and then $h_{j,0}v_2=\delta_{ij}v_2$. Therefore
$Y^{\left(1\right)}\left(v_2\right)$ is a highest weight representation with highest weight $P_j=1$ if $j\neq i$ and $P_i=\left(u-\frac{a}{2}\right)$. It follows from the Lemma \ref{c3l3b} that $a=a_1+2$. The rest of the proof of this proposition is similar to the proofs of Proposition \ref{c3pb}, just replacing $a_1$ by $a_1+2$.
\end{proof}

Let $v_{m+1}=\overline{\left(\tilde{x}_{m,0}^{-}\right)^2\ldots\tilde{x}_{i,0}^{-}}v_{m}$, where $2\leq m\leq i-1$.
Define $Y^{\left(m\right)}$ be the Yangian spanned by $\{h_{j,r}, x_{j,r}^{\pm}\}$ for $j>m$. Note that when $i=l-1$ and $m=i-1$, $Y^{\left(m\right)}=Y^{\left(l-2\right)}\cong Y\big(\mathfrak{so}(5,\C)\big)$. Thus $Y^{\left(m\right)}$ is isomorphic to the Yangian of type $B_2$.

Similarly to the above proposition, we have
\begin{proposition} Let $i\leq k\leq l-1$ and $m+1\leq n\leq i-2$.
\begin{enumerate}
  \item $Y_k\left(\tilde{x}_{k-1,0}^{-}\tilde{x}_{k-2,0}^{-}\ldots \tilde{x}_{i,0}^{-}v_{m+1}\right)\cong W_1\left(\frac{a_1+2m}{2}+\frac{k-i}{2}\right)$.
  \item  $Y_l\left(\tilde{x}_{l-1,0}^{-}\tilde{x}_{l-2,0}^{-}\ldots \tilde{x}_{i,0}^{-}v_{m+1}\right)$ is isomorphic to either $W_2\Big(\left(a_1+2m\right)+l-i-1\Big)$ or $W_1\Big(\left(a_1+2m\right)+l-i\Big)\otimes W_1\Big(\left(a_1+2m\right)+l-i-1\Big)$.
  \item $Y_{k}\left(\overline{\tilde{x}_{k+1,0}^{-}\ldots \tilde{x}_{i,0}^{-}}v_{m+1}\right)\cong W_1\left(\frac{a_1+2m}{2}+\frac{2l-k-i-1}{2}\right)$.
  \item
\begin{enumerate}
    \item $Y_{i-1}\Big(\overline{\left(x_{i,0}^{-}\right)\ldots x_{i,0}^{-}}v_{m+1}\Big)\cong W_1\left(\frac{a_1+2m}{2}+l-i\right)\otimes W_1\left(\frac{a_1+2m}{2}+\frac{1}{2}\right)$.
    \item $Y_{n}\Big(\overline{\left(\tilde{x}_{n+1,0}^{-}\right)^2\ldots \tilde{x}_{i,0}^{-}}v_{m+1}\Big)\cong W_1\left(\frac{a_1+2m}{2}+\frac{2l-i-n-1}{2}\right)\otimes W_1\left(\frac{a_1+2m}{2}+\frac{i-n}{2}\right)$.
  \end{enumerate}

  \item $Y_{i}\Big(\overline{\left(\tilde{x}_{m+1,0}^{-}\right)^2\ldots \tilde{x}_{i,0}^{-}}v_{m+1}\Big)\cong W_1\left(\frac{a_1+2\left(m+1\right)}{2}\right)$.
\end{enumerate}
\end{proposition}
\begin{remark}\
\begin{enumerate}
  \item If $m=i-2$, the computations part $(b)$ of $(iv)$ is not necessary.
  \item If $m=i-1$, the computations stop at step $\left(iii\right)$ when $k=i$.
\end{enumerate}

\end{remark}

\section{On the local Weyl modules of $\yso$}
Let $\lambda=\sum\limits_{i\in I} m_i\omega_i$. In \cite{Na}, the dimension of the local Weyl module $W(\lambda)$ is given.
\begin{proposition}[Corollary 9.5, \cite{Na}]\label{dwmocsp}
Let $\lambda=\sum\limits_{i\in I} m_i\omega_i$. Then
$$\operatorname{Dim}\Big(W(\lambda)\Big)=\prod\limits_{i\in I} \Big(\operatorname{Dim}\big(W(\omega_i)\big)\Big)^{m_i}.$$
\end{proposition}
The next theorem follows from Propositions \ref{mtoysob} and \ref{vtv'hwv}.
\begin{theorem}\label{wmiatpsob}
Let $\pi=\Big(\pi_1\left(u\right),\ldots, \pi_{l}\left(u\right)\Big)$, where $\pi_i\left(u\right)=\prod\limits_{j=1}^{m_i}\left(u-a_{i,j}\right)$. Let $S=\{a_{1,1},\ldots, a_{1,m_1},\ldots, a_{l,1}\ldots, a_{l,m_l}\}$ be the multiset of roots of these polynomials. Let $a_1=a_{m,n}$ be one of the numbers in $S$ with the maximal real part, and let $b_1=m$. Similarly, let $a_r=a_{s,t}\left(r\geq 2\right)$ be one of the numbers in $S\setminus\{a_1, \ldots, a_{r-1}\}$ ($r\geq 2$) with the maximal real part, and $b_r=s$. Let $k=m_1+\ldots+m_l$. Then $L=V_{a_1}(\omega_{b_1})\otimes V_{a_2}(\omega_{b_2})\otimes\ldots\otimes V_{a_k}(\omega_{b_k})$ is a highest weight representation of $\yso$, and its associated polynomial is $\pi$.
\end{theorem}

\begin{theorem}The local Weyl module $W(\pi)$ of $\yso$ associated to $\pi$ is isomorphic to the ordered tensor product $L$ as in the above theorem.
\end{theorem}
\begin{proof}
On the one hand, by Theorem \ref{ubodowm}, $\operatorname{Dim}\big(W(\pi)\big)\leq \operatorname{Dim}\big(W(\lambda)\big)$; on the other hand, $L$ is a quotient of $W(\pi)$, and then $\operatorname{Dim}\big(W(\pi)\big)\geq \operatorname{Dim}\left(L\right)$. By Corollary \ref{dkrvawocsp},  we have $\operatorname{Dim}\big(W(\lambda)\big)=\operatorname{Dim}\left(L\right)$. Therefore
\begin{center}
$W\left(\pi\right)\cong L$.
\end{center}
\end{proof}
\chapter{The local Weyl modules of $\ygg$ when $\g$ is of type $G_2$}
In this chapter, $\g$ denotes the exceptional simple Lie algebra of type $G_2$. the local Weyl modules of $\ygg$ are studied. The structure of the local Weyl modules is determined, and the dimensions of the local Weyl modules are obtained. In the process of characterizing the local Weyl modules, a sufficient condition for the tensor product of fundamental representations of $\yg$ to be a highest weight representation is obtained, which shall lead to an irreducibility criterion for the tensor product.


Let $\pi=\big(\pi_1(u),\pi_2(u)\big)$ be a pair of monic polynomials in $u$, and $\pi_i\left(u\right)=\prod\limits_{j=1}^{m_i}\left(u-a_{i,j}\right)$ for $i\in I=\{1,2\}$.
 Let $\lambda=m_1\omega_1+m_2\omega_2$, $k=m_1+m_2$ and $S=\{a_{ij}|i=1,2; j=1,2,\ldots, m_i\}$. Let $a_{m,n}$ be one of the numbers in $S$ with the maximal real part. Then define $a_1=a_{m,n}$ and $b_1=m$. Inductively, let $a_{s,t}$ be one of the numbers in $S-\{a_1,\ldots, a_{i-1}\}$ with the maximal real part. Then define $a_i=a_{s,t}$ and $b_i=s$. We prove that $L=V_{a_1}(\omega_{b_1})\otimes V_{a_2}(\omega_{b_2})\otimes\ldots\otimes V_{a_k}(\omega_{b_k})$ is a highest weight representation.  A standard argument shows that the associated l-tuple of polynomials of $L$ is $\pi$.
Since $L$ is a quotient of $W(\pi)$, a lower bound on the dimension of $W(\pi)$ is obtained.

In \cite{Na}, the author proved $\operatorname{Dim}\Big(W(\lambda)\Big)=\prod\limits_{i\in I} \Big(\operatorname{Dim}\big(W(\omega_i)\big)\Big)^{m_i}.$
We can show that $\operatorname{Dim}\Big(W(\lambda)\Big)=\operatorname{Dim}(L)$. The dimension of $W(\lambda)$ is known, then an upper bound for the dimension of the local Weyl module $W(\pi)$ is obtained. Comparing this dimension and the dimension of $L$, the structure of the local Weyl module $W(\pi)$ of $\ygg$ is obtained, namely, $$W\left(\pi\right)\cong L.$$

Similar to the proof that $L$ is a highest weight representation (Proposition \ref{g2Lihw}), we can obtain a sufficient condition for a tensor product $\tilde{L}$ of fundamental representations to be a highest weight representation. If $a_j-a_i\notin S(b_i, b_j)$ for $1\leq i<j\leq k$, then $\tilde{L}$ is a highest weight representation, where $S(b_i, b_j)$ is a finite set of positive rational numbers. By Proposition \ref{VoWWoVhi} and Lemma \ref{dualfrc1}, an irreducible criterion for a tensor product of fundamental representations of $\ygg$ is obtained: if $a_j-a_i\notin S(b_i, b_j)$ for $1\leq i\neq j\leq k$, then $\tilde{L}$ is irreducible.

\section{Information on the simple Lie algebra of type $G_2$.}

Here is some useful information on the simple Lie algebras of type $G_2$.

Indecomposable Cartan matrix of $G_2$:
                                             $\begin{pmatrix}
                                               2 & -1 \\
                                               -3 & 2
                                             \end{pmatrix}.$

In our notation, $\alpha_1$ is the positive long root, and $\alpha_2$ is the positive short root. It is the opposite of the paper \cite{ChMo3}.

Root system: $\Phi=\{\alpha_1,\alpha_2, \alpha_1+\alpha_2,\alpha_1+2\alpha_2,\alpha_1+3\alpha_2,2\alpha_1+3\alpha_2,-\alpha_1,-\alpha_2,\\ -\alpha_1-\alpha_2,-\alpha_1-2\alpha_2,-\alpha_1-3\alpha_2,-2\alpha_1-3\alpha_2, \}$.

Positive roots: $\{\alpha_1,\alpha_2, \alpha_1+\alpha_2,\alpha_1+2\alpha_2,\alpha_1+3\alpha_2,2\alpha_1+3\alpha_2\}$.

Simple roots: $\{\alpha_1,\alpha_2\}$.

Longest root: $2\alpha_1+3\alpha_2$.

Weyl group: $\W=\langle s_1,s_2\rangle$, where
$$s_1\left(\alpha_1\right)=-\alpha_1,\quad s_2\left(\alpha_1\right)=\alpha_1+3\alpha_2$$
$$s_1\left(\alpha_2\right)=\alpha_1+\alpha_2, \quad s_2\left(\alpha_2\right) = -\alpha_2.$$\
One reduced expression of the longest element in Weyl group:
$w_0=s_1s_2s_1s_2s_1s_2$.

Fundamental weights: $\{\omega_1=2\alpha_1+3\alpha_2, \omega_2=\alpha_1+2\alpha_2\}$.

Fundamental representations: $\operatorname{Dim}\Big(L\left(\omega_1\right)\Big)=14$ and $\operatorname{Dim}\Big(L\left(\omega_2\right)\Big)=7$.
$$s_1\left(\omega_1\right)=-\omega_1+3\omega_2,\quad s_1\left(\omega_2\right)=\omega_2,\quad s_2\left(\omega_1\right)=\omega_1,\quad s_2\left(\omega_2\right)=\omega_1-\omega_2.$$
\begin{proposition}\label{g2fwfr} Denote $s_i(\mu_1) = \mu_2$ by $\mu_1 \stackrel{s_i}{\longrightarrow} \mu_2$. We have
\begin{flushleft}
$\omega_1\xlongrightarrow{s_2}\omega_1\xlongrightarrow{s_1}-\omega_1+3\omega_2\xlongrightarrow{s_2}2\omega_1-3\omega_2
\xlongrightarrow{s_1}-2\omega_1+3\omega_2\xlongrightarrow{s_2}\omega_1-3\omega_2\xlongrightarrow{s_1}-\omega_1$.
\end{flushleft}

\begin{flushleft}
$\omega_2\xlongrightarrow{s_2}\omega_1-\omega_2\xlongrightarrow{s_1}-\omega_1+2\omega_2\xlongrightarrow{s_2}\omega_1-2\omega_2
\xlongrightarrow{s_1}-\omega_1+\omega_2\xlongrightarrow{s_2}-\omega_2\xlongrightarrow{s_1}-\omega_2$.
\end{flushleft}
\end{proposition}

Define $w_1=s_1s_2s_1s_2s_1$ and $w_2=s_2s_1s_2s_1s_2$. According to the expression of $w_i$ for a fixed $i\in I$, we define $\sigma_k\left(0\leq k\leq 5\right)$ to be the product of the last $k$ simple reflections $s_j$ in $w_i$ and keep the same orders as in $w_i$.  There exists $k'\in\{1,2\}$ such that $\sigma_{k+1}=s_{k'}\sigma_k$.

\section{A lower bound for the dimension of the local Weyl module $W(\pi)$}
In this case, let $D=
                       \begin{pmatrix}
                         3 & 0 \\
                         0 & 1 \\
                       \end{pmatrix}
                     .$
Then $DA$ is symmetric. We next define $Y_1$ and $Y_2$.
Let $Y_1=span\{x_{1,r}^{\pm}, h_{1,r}|r\in\Z_{\geq 0}\}$. $Y_1\cong \ysl$, but $x_{1,r}^{\pm}, h_{1,r}$ do not satisfy the defining relations of $\ysl$. Therefore we need to re-scale the generators. Let $\tilde{x}_{1,r}^{\pm}=\frac{\sqrt{3}}{3^{r+1}}x_{1,r}^{\pm},\tilde{h}_{1,r}=\frac{1}{3^{r+1}}h_{1,r}$. Then $\tilde{x}_{1,r}^{\pm}, \tilde{h}_{1,s}$ satisfy the defining relations of $\ysl$.
$Y_2=span\{x_{2,r}^{\pm}, h_{2,r}|r\in\Z_{\geq 0}\}\cong \ysl$, and $x_{2,r}^{\pm}, h_{2,r}$ satisfy the defining relations of $\ysl$. Denote $h_{2,r}$ by $\tilde{h}_{2,r}$ for the simplicity of the following expressions.

\begin{proposition}\label{g2Lihw}
Let $L=V_{a_1}(\omega_{b_1})\otimes V_{a_2}(\omega_{b_2})\otimes\ldots\otimes V_{a_k}(\omega_{b_k})$. If $\operatorname{Re}\left(a_1\right)\geq \operatorname{Re}\left(a_2\right)\geq\ldots\geq \operatorname{Re}\left(a_k\right)$ and $b_i\in\{1,2\}$, then $L$ is a highest weight representation.
\end{proposition}
\begin{proof}


Let $v_m^{+}$ be a highest weight vector of $V_{a_m}\left(\omega_{b_m}\right)$($1\leq m\leq k$), and let $v_1^{-}$ be a lowest weight vector of $V_{a_1}\left(\omega_{b_1}\right)$.

We prove this proposition by induction on $k$. Without loss of generality, we may assume that $k\geq 2$, and $V_{a_2}(\omega_{b_2})\otimes V_{a_3}(\omega_{b_3})\otimes\ldots\otimes V_{a_k}(\omega_{b_k})$ is a highest weight representation of $\ygg$ generated by the highest weight vector $v^{+}=v_2^{+}\otimes \ldots \otimes v_k^{+}$. To show that $L$ is a highest weight representation, it follows from Corollary \ref{v-w+gvtw} that it suffice to show that $$v^{-}_{1}\otimes v^{+}\in \ygg\left(v^{+}_1\otimes v^{+}\right).$$ We divide the proof into the following steps.

Step 1: $\sigma_i^{-1}\alpha_{i'}\in \Delta^{+}$.

Proof: It is routine to check.


Step 2: $v_{\sigma_i\left(\omega_{b_1}\right)}$ is a highest weight vector of the $Y_{i'}$-module $Y_{i'}\left(v_{\sigma_i\left(\omega_{b_1}\right)}\right)$.

Proof: Since the weight $\sigma_i\left(\omega_{b_1}\right)$ is on the Weyl group orbit of the highest weight and the representation $V_{a_1}(\omega_{b_1})$ is finite-dimensional, the weight space of weight $\sigma_i\left(\omega_{b_1}\right)$ is 1-dimensional. The elements $h_{j,s}$ form a commutative subalgebra, so $v_{\sigma_i\left(\omega_{b_1}\right)}$ is an eigenvector of $h_{i',r}$. Therefore we only have to show that $v_{\sigma_i\left(\omega_{b_1}\right)}$ is a maximal vector. Suppose to the contrary that $x_{i',k}^{+}v_{\sigma_i\left(\omega_{b_1}\right)}\neq 0$.  Then $x_{i',k}^{+}v_{\sigma_i\left(\omega_{b_1}\right)}$ is a weight vector of weight $\sigma_i(\omega_{b_1})+\alpha_{i'}$, so $\omega_{b_1}+\sigma_i^{-1}\left(\alpha_{i'}\right)$ is a weight. Because $\sigma_i^{-1}\left(\alpha_{i'}\right)\in \Delta^{+}$,  $\omega_{b_1}$ is a weight preceding the weight $\omega_{b_1}+\sigma_i^{-1}\left(\alpha_{i'}\right)$, which contradicts the maximality of $\omega_{b_1}$ in the representation $L\left(\omega_{b_1}\right)$.



Step 3: Let $P\left(u\right)$ be the associated polynomial of $Y_{i'}\left(v_{\sigma_i\left(\omega_{b_1}\right)}\right)$. Then as a highest weight $Y_{i'}$-module, $Y_{i'}\left(v_{\sigma_i\left(\omega_{b_1}\right)}\right)$ has highest weight $\frac{P\left(u+1\right)}{P\left(u\right)}$.

Proof: This follows from the representation theory of $\ysl$.

Step 4: $P\left(u\right)$ has of degree 1, 2, or 3.

Proof: The degree of $P\left(u\right)$ equals to the eigenvalue of $\tilde{h}_{i',0}$ on $v_{\sigma_i\left(\omega_{b_1}\right)}$. Since $$\tilde{h}_{i',0} v_{\sigma_i\left(\omega_{b_1}\right)}=\Bigg(\sigma_{i}\left(\omega_{b_1}\right)\left(\tilde{h}_{i',0}\right)\Bigg)v_{\sigma_i\left(\omega_{b_1}\right)},$$ it follows from Proposition \ref{g2fwfr} that $\sigma_{i}\left(\omega_{b_1}\right)\left(\tilde{h}_{i',0}\right)=1, 2$ or 3. Therefore the degree of $P\left(u\right)$ equals to 1, 2, or 3.


Step 5: If $\operatorname{Deg}\big(P\left(u\right)\big)=1$, then $Y_{i'}\left(v_{\sigma_i\left(\omega_{b_1}\right)}\right)$ is 2-dimensional and is isomorphic to $W_1\left(\frac{a}{d_{i'}}\right)$;  If $\operatorname{Deg}\Big(P\left(u\right)\Big)=2$, then $Y_{i'}\left(v_{\sigma_i\left(\omega_{b_1}\right)}\right)$ is a quotient of $W_1\left(\frac{a+1}{d_{i'}}\right)\otimes W_1\left(\frac{a}{d_{i'}}\right)$; If $\operatorname{Deg}\Big(P\left(u\right)\Big)=3$, then $i'=2$ and $Y_{i'}\left(v_{\sigma_i\left(\omega_{b_1}\right)}\right)$ is a quotient of $W_1\left(a+2\right)\otimes W_1\left(a+1\right)\otimes W_1\left(a\right)$. Moreover, if $i\geq 1$, then $\operatorname{Re}\left(a\right)\geq \operatorname{Re}\left(a_1\right)-\frac{1}{2}$. The values of $a$ will be explicitly computed in the next section.

Proof: Let us assume for the moment that this is done (the proof will be given in the next section). We proceed to the next step.

Step 6:  $Y_{i'}\left(v_{\sigma_i\left(\omega_{b_1}\right)}\otimes v_2^{+}\otimes\ldots \otimes v_k^{+}\right)=Y_{i'}\left(v_{\sigma_i\left(\omega_{b_1}\right)}\right)\otimes Y_{i'}\left(v_2^{+}\right)\otimes\ldots\otimes Y_{i'}\left(v_k^{+}\right)$.

Proof: For $i'=1$, let $W$ be one of the modules of $W_1\left(\frac{a}{3}\right)$ or $W_1\left(\frac{a+1}{3}\right)\otimes W_1\left(\frac{a}{3}\right)$; for $i'=2$, let $W$ be one of the modules of $W_1\left(a\right)$, $W_1\left(a+1\right)\otimes W_1\left(a\right)$, or $W_1\left(a+2\right)\otimes W_1\left(a+1\right)\otimes W_1\left(a\right)$. $Y_{i'}\left(v_m^{+}\right)$ is nontrivial if and only if $b_m=i'$; moreover, $Y_{2}\left(v_m^{+}\right)\cong W_1\left(a_m\right)$ and $Y_{1}\left(v_m^{+}\right)\cong W_1\left(\frac{a_m}{3}\right)$.  Suppose in $\{b_1,\ldots,b_k\}$ that $b_{j_1}=\ldots=b_{j_{m}}=i'$ with $j_1<\ldots<j_{m}$; moreover, if $s\notin\{j_1,\ldots, j_m\}$, then $b_s\neq i'$.
Thus for $i'=1$, $Y_{i'}\left(v_{\sigma_i\left(\omega_{b_1}\right)}\right)\otimes Y_{i'}\left(v_2^{+}\right)\otimes\ldots\otimes Y_{i'}\left(v_k^{+}\right)$ is a quotient of
\begin{equation*}
\begin{cases}
 W\otimes W_1\left(\frac{a_{j_1}}{3}\right)\otimes \ldots\otimes W_{1}\left(\frac{a_{j_m}}{3}\right)\qquad \text{if}\qquad b_1=2 \\
 W\otimes W_1\left(\frac{a_{j_2}}{3}\right)\otimes \ldots\otimes W_{1}\left(\frac{a_{j_m}}{3}\right) \qquad \text{if}\qquad b_1=1;
\end{cases}
\end{equation*}
for $i'=2$, $Y_{i'}\left(v_{\sigma_i\left(\omega_{b_1}\right)}\right)\otimes Y_{i'}\left(v_2^{+}\right)\otimes\ldots\otimes Y_{i'}\left(v_k^{+}\right)$ is a quotient of
\begin{equation*}
\begin{cases}
 W\otimes W_1\left(a_{j_1}\right)\otimes \ldots\otimes W_{1}\left(a_{j_m}\right)\qquad \text{if}\qquad b_1=1 \\
 W\otimes W_1\left(a_{j_2}\right)\otimes \ldots\otimes W_{1}\left(a_{j_m}\right) \qquad \text{if}\qquad b_1=2.
\end{cases}
\end{equation*}

Denote $a_{j_0}=\operatorname{Re}\left(a\right)+\frac{1}{2}$.
Since $\operatorname{Re}\left(a\right)+\frac{1}{2}\geq \operatorname{Re}\left(a_1\right)\geq \operatorname{Re}\left(a_{j_1}\right)\geq \ldots\geq \operatorname{Re}\left(a_{j_m}\right)$, $a_{j_k}-a_{j_l}\neq 1$ for $0\leq l<k\leq m$. Then it follows from Corollary \ref{ctpihw} that $W\otimes W_1\otimes \ldots\otimes W_1$ is a highest weight representation, so is $Y_{i'}\left(v_{\sigma_i\left(\omega_{b_1}\right)}\right)\otimes Y_{i'}\left(v_2^{+}\right)\otimes\ldots\otimes Y_{i'}\left(v_k^{+}\right)$. The highest weight vector of $Y_{i'}\left(v_{\sigma_i\left(\omega_{b_1}\right)}\right)\otimes Y_{i'}\left(v_2^{+}\right)\otimes\ldots\otimes Y_{i'}\left(v_k^{+}\right)$ is $v_{\sigma_i\left(\omega_{b_1}\right)}\otimes v_2^{+}\otimes\ldots \otimes v_k^{+}$. Therefore $$Y_{i'}\left(v_{\sigma_i\left(\omega_{b_1}\right)}\otimes v_2^{+}\otimes\ldots \otimes v_k^{+}\right)\supseteq Y_{i'}\left(v_{\sigma_i\left(\omega_{b_1}\right)}\right)\otimes Y_{i'}\left(v_2^{+}\right)\otimes\ldots\otimes Y_{i'}\left(v_k^{+}\right).$$ By the coproduct of Yangians  and Proposition \ref{c1dgd2d} that
$$Y_{i'}\left(v_{\sigma_i\left(\omega_{b_1}\right)}\otimes v_2^{+}\otimes\ldots \otimes v_k^{+}\right)\subseteq Y_{i'}\left(v_{\sigma_i\left(\omega_{b_1}\right)}\right)\otimes Y_{i'}\left(v_2^{+}\right)\otimes\ldots\otimes Y_{i'}\left(v_k^{+}\right).$$ Thus the claim is true.

Step 7: $v_{\sigma_{i+1}\left(\omega_{b_1}\right)}\otimes v^{+}\in Y_{i'}\left(v_{\sigma_i\left(\omega_{b_1}\right)}\otimes v^{+}\right)$.

Proof: $v_{\sigma_{i+1}\left(\omega_{b_1}\right)}\otimes v^{+}\in Y_{i'}\left(v_{\sigma_{i}\omega_{b_1}}\right)\otimes Y_{i'}\left(v^{+}\right)=Y_{i'}\left(v_{\sigma_i\left(\omega_{b_1}\right)}\otimes v^{+}\right)$.

Step 8: $v_1^{-}\otimes v^{+}\in \ygg\left(v_1^{+}\otimes v^{+}\right)$.

Proof: It follows from Step 7 immediately by induction on the subscripts of $\sigma_i$.

It follows from Step 8 and Corollary \ref{v-w+gvtw} that  $L=\ygg\left(v_1^{+}\otimes v^{+}\right)$.
\end{proof}
\begin{remark}
The following is the record of the root (s) of the associated polynomials of $Y_{i'}\left(v_{\sigma_i\left(\omega_{b_1}\right)}\right)$ as in Proposition \ref{g2Lihw} for the possible $i$ values. We refer to Remark \ref{yocap1} for the notation used below.

When $b_1=1$,

$\frac{a_1}{3}\xlongrightarrow{x_{1,0}^{-}}\left(a_1+\frac{3}{2},a_1+\frac{1}{2},a_1-\frac{1}{2}\right)\xlongrightarrow{x_{2,0}^{-}}\left(\frac{a_1+2}{3}, \frac{a_1+1}{3}\right)\xlongrightarrow{x_{1,0}^{-}}\\\left(a_1+\frac{7}{2},a_1+\frac{5}{2},a_1+\frac{3}{2}\right) \xlongrightarrow{x_{2,0}^{-}}\frac{a_1}{3}+1\xlongrightarrow{x_{1,0}^{-}}$.

When $b_1=2$,

$a_1\xlongrightarrow{x_{2,0}^{-}}\frac{a_1+\frac{3}{2}}{3}\xlongrightarrow{x_{1,0}^{-}}\left(a_1+3,a_1+2\right) \xlongrightarrow{x_{2,0}^{-}}\frac{a_1+\frac{7}{2}}{3}\xlongrightarrow{x_{1,0}^{-}}a_1+5\xlongrightarrow{x_{2,0}^{-}}$.
\end{remark}

\begin{lemma}
$V_{a_m}\left(\omega_{b_m}\right)\otimes V_{a_n}\left(\omega_{b_n}\right)$ is a highest weight representation if $a_n-a_m\notin S\left(b_m,b_n\right)$, where the set $S\left(b_m,b_n\right)$ is defined as follows:
$S(1,1)=\{3,4,5,6\}$, $S(1,2)=\{\frac{1}{2},\frac{3}{2},\frac{5}{2},\frac{7}{2},\frac{9}{2}\}$, $S(2,1)=\{\frac{9}{2},\frac{13}{2}\}$ and $S(2,2)=\{1,3,4,6\}$.
\end{lemma}

Similar to the proof of Theorem \ref{3mt2icl}, we can prove the following Theorem.
\begin{theorem}
Let $L=V_{a_1}(\omega_{b_1})\otimes V_{a_2}(\omega_{b_2})\otimes\ldots\otimes V_{a_k}(\omega_{b_k})$, and $S(b_i, b_j)$ be defined as in the above lemma.
\begin{enumerate}
  \item If $a_j-a_i\notin S(b_i, b_j)$ for $1\leq i<j\leq k$, then $L$ is a highest weight representation of $\ygg$.
  \item If $a_j-a_i\notin S(b_i, b_j)$ for $1\leq i\neq j\leq k$, then $L$ is an irreducible representation of $\ygg$.
\end{enumerate}
\end{theorem}
\begin{remark}
In Section 6.2 \cite{Ch3}, Chari gave the set $S(i_1, i_2)$, $i_1\leq i_2$ of values of $a_1^{-1}a_2$ for which the tensor product $V(i_1, a_1)\otimes V(i_2, a_2)$ may fail irreducibility. Note that we interchange the nodes of the Dynkin diagram of $G_2$ in this paper in the following discussion.
In the paper, V. Chari found that
\begin{align*}
\mathcal{S}(1,1)&=\{q^6, q^8, q^{10}, q^{12}\}.\\
\mathcal{S}(2,1)&=\{q^7, q^{11}\}.\\
\mathcal{S}(1,2)&=\{q^3, q^7\}.\\
\mathcal{S}(2,2)&=\{q^2, q^6, q^{8}, q^{12}\}.
\end{align*}
\end{remark}
\section{Supplement proof of the step 5 of Proposition \ref{g2Lihw}}\label{g2spos5}
In this section we complete the proof of step 5 of Proposition \ref{g2Lihw}. We divide the proof into two cases: $b_1=1$ and $b_1=2$.
The following two formulas will be used many times in this section, which are from the defining relations of Yangians.
\begin{enumerate}
  \item $[h_{2,1},x_{1,0}^{-}]=\Big(3x_{1,1}^{-}+\frac{3}{2}\left(3x_{1,0}^{-}+2x_{1,0}^{-}h_{2,0}\right)\Big)$.
  \item $[\tilde{h}_{1,1}, x_{2,0}^{-}]=\left(\frac{1}{3}x_{2,1}^{-}+\frac{1}{2}x_{2,0}^{-}+x_{2,0}^{-}\tilde{h}_{1,0}\right)$.
\end{enumerate}

\subsection{Case 1: $b_1=1$.}

In this case, $$v_1^{-}=x_{1,0}^{-}\left(x_{2,0}^{-}\right)^3\left(x_{1,0}^{-}\right)^2\left(x_{2,0}^{-}\right)^3x_{1,0}^{-}v_1^{+}.$$
We summarize all the computations into the following proposition.
\begin{proposition}As modules of $\ysl$, the following are true.
\begin{enumerate}
  \item $Y_1\left(v_{1}^{+}\right)\cong W_1\left(\frac{a_1}{3}\right).$
  \item $Y_2\left(x_{1,0}^{-}v_1^{+}\right)$ is a quotient of $W_1\left(a_1+\frac{3}{2}\right)\otimes W_1\left(a_1+\frac{1}{2}\right)\otimes W_1\left(a_1-\frac{1}{2}\right)$.
  \item $Y_1\Big(\left(x_{2,0}^{-}\right)^3x_{1,0}^{-}v_1^{+}\Big)\cong W_1\left(\frac{a_1+2}{3}\right)\otimes W_1\left(\frac{a_1+1}{3}\right)$.
  \item $Y_2\left(\left(x_{1,0}^{-}\right)^2\left(x_{2,0}^{-}\right)^3x_{1,0}^{-}v_1^{+}\right)$ is a quotient of $W_1\left(a_1+\frac{7}{2}\right)\otimes W_1\left(a_1+\frac{5}{2}\right)\otimes W_1\left(a_1+\frac{3}{2}\right)$.
  \item $Y_1\left(\left(x_{2,0}^{-}\right)^3\left(x_{1,0}^{-}\right)^2\left(x_{2,0}^{-}\right)^3x_{1,0}^{-}v_1^{+}\right)\cong W_1\left(\frac{a_1}{3}+1\right)$.
\end{enumerate}
\end{proposition}
\begin{proof}
We omit the proof of the item (i) since it is similar to the proof of Lemma \ref{spc3c1}. The second item is proved in Lemma \ref{c7c1l1}. Lemma \ref{c7c1l2} is devoted to prove the third item. The fourth is proved in Lemma \ref{c7c1l3}. The item (v) will be showed in Lemma \ref{c7c1l4}.
\end{proof}

\begin{remark}\label{c7c1r1}In $Y_1\left(v_{1}^{+}\right)$,
\begin{enumerate}
  \item $x_{1,k}^{-}v_{1}^{+}=a_1^kx_{1,0}^{-} v_{1}^{+}$;
  \item $h_{1,k}x_{1,0}^{-}v_{1}^{+}=-3a_1^k x_{1,0}^{-}v_{1}^{+}$.
\end{enumerate}
\end{remark}
\begin{lemma}\label{c7c1l1}
$Y_2\left(x_{1,0}^{-}v_1^{+}\right)$ is a quotient of $W_1\left(a_1+\frac{3}{2}\right)\otimes W_1\left(a_1+\frac{1}{2}\right)\otimes W_1\left(a_1-\frac{1}{2}\right)$.
\end{lemma}
\begin{proof}
It follows from Proposition \ref{g2fwfr} that $wt(x_{1,0}^{-}v_1^{+})=-\omega_1+3\omega_2$, then $$h_{2,0}x_{1,0}^{-}v_1^{+}=3x_{1,0}^{-}v_1^{+}.$$
\begin{align*}
h_{2,1}x_{1,0}^{-}v_1^{+} &= [h_{2,1},x_{1,0}^{-}]v_1^{+} \\
  &= \left(3x_{1,1}^{-}+\frac{3}{2}\left(3x_{1,0}^{-}+2x_{1,0}^{-}h_{2,0}\right)\right)v_1^{+}\\
  &= \left(3x_{1,1}^{-}+\frac{9}{2}x_{1,0}^{-}\right)v_1^{+}\\
  &= 3\left(a_1+\frac{3}{2}\right)x_{1,0}^{-}v_1^{+}.
\end{align*}
\begin{align*}
h_{2,2}x_{1,0}^{-}v_1^{+} &= [h_{2,2},x_{1,0}^{-}]v_1^{+} \\
  &= \left([h_{2,1},x_{1,1}^{-}]+\frac{3}{2}\left(h_{2,1}x_{1,0}^{-}+x_{1,0}^{-}h_{2,1}\right)\right)v_1^{+}\\
    &= \left(h_{2,1}x_{1,1}^{-}+\frac{3}{2}h_{2,1}x_{1,0}^{-}\right)v_1^{+}\\
&=h_{2,1}\left(3x_{1,1}^{-}+\frac{3}{2}x_{1,0}^{-}\right)v_1^{+}\\
&= \left(a_1+\frac{3}{2}\right)h_{2,1}x_{1,0}^{-}v_1^{+}\\
&=3\left(a_1+\frac{3}{2}\right)^2x_{1,0}^{-}v_1^{+}.
\end{align*}
In general, by induction on $k$, we have $$h_{2,k}x_{1,0}^{-}v_1^{+}=3\left(a_1+\frac{3}{2}\right)^kx_{1,0}^{-}v_1^{+}.$$ It is not hard to see that the associated polynomial is $$P\left(u\right)=\Bigg(u-\left(a_1+\frac{3}{2}\right)\Bigg)\Bigg(u-\left(a_1+\frac{1}{2}\right)\Bigg)\Bigg(u-\left(a_1-\frac{1}{2}\right)\Bigg).$$ By the representation theory of the local Weyl modules of $\ysl$, $Y_2\left(x_{1,0}^{-}v_1^{+}\right)$ is a quotient of $W_1\left(a+\frac{3}{2}\right)\otimes W_1\left(a+\frac{1}{2}\right)\otimes W_1\left(a-\frac{1}{2}\right)$.
\end{proof}
\begin{corollary}In $W_1\left(a+\frac{3}{2}\right)\otimes W_1\left(a+\frac{1}{2}\right)\otimes W_1\left(a-\frac{1}{2}\right)$,
\begin{enumerate}
  \item $x_{2,1}^{-}\left(x_{2,0}^{-}\right)^2x_{1,0}^{-}v_1^{+}=\left(a_1-\frac{1}{2}\right)\left(x_{2,0}^{-}\right)^3x_{1,0}^{-}v_1^{+}$.
  \item $x_{2,0}^{-}x_{2,1}^{-}x_{2,0}^{-}x_{1,0}^{-}v_1^{+}=\left(a_1+\frac{1}{2}\right)\left(x_{2,0}^{-}\right)^3x_{1,0}^{-}v_1^{+}$.
  \item $\left(x_{2,0}^{-}\right)^2 x_{2,1}^{-}x_{1,0}^{-}v_1^{+}=\left(a_1+\frac{3}{2}\right)\left(x_{2,0}^{-}\right)^3x_{1,0}^{-}v_1^{+}$.
  \item $x_{2,2}^{-}\left(x_{2,0}^{-}\right)^2x_{1,0}^{-}v_1^{+}=\left(a_1-\frac{1}{2}\right)^2\left(x_{2,0}^{-}\right)^3x_{1,0}^{-}v_1^{+}$.
  \item $x_{2,0}^{-}x_{2,2}^{-}x_{2,0}^{-}x_{1,0}^{-}v_1^{+}=\left(a_1+\frac{1}{2}\right)^2\left(x_{2,0}^{-}\right)^3x_{1,0}^{-}v_1^{+}$.
  \item $\left(x_{2,0}^{-}\right)^2 x_{2,2}^{-}x_{1,0}^{-}v_1^{+}=\left(a_1+\frac{3}{2}\right)^2\left(x_{2,0}^{-}\right)^3x_{1,0}^{-}v_1^{+}$.
\end{enumerate}
\end{corollary}

In the following two lemmas, the computations is tedious if we use $\tilde{h}_{i,k}$. We introduce a new series of generators of $H\subseteq \ygg$. Let $$H_{i}(u)=\sum\limits_{k=0}^{\infty} H_{i,k}u^{-k+1}:=ln\big(h_i(u)\big).$$
An explicit computation shows that
\begin{align*}
H_{i}\left(t\right)&=\left(h_{i,0}t^{-1}+h_{i,1}t^{-2}+\cdots\right)-\frac{\left(h_{i,0}t^{-1}+h_{i,1}t^{-2}+\cdots\right)^2}{2}\\
&\qquad +
\frac{\left(h_{i,0}t^{-1}+h_{i,1}t^{-2}+\cdots\right)^3}{3}-\cdots\\
&=\left(h_{i,0}\right)t^{-1}+\left(h_{i,1}-\frac{1}{2}\left(h_{i,0}\right)^2\right)t^{-2}+\left(h_{i,2}-h_{i,0}h_{i,1}+\frac{1}{3}\left(h_{i,0}\right)^3\right)t^{-3}\\
&\qquad +\left(h_{i,3}-h_{i,0}h_{i,2}-\frac{1}{2}\left(h_{i,1}\right)^2+
\left(h_{i,0}\right)^2h_{i,1}-\frac{1}{4}\left(h_{i,0}\right)^4\right)t^{-4}-\ldots
\end{align*}
\begin{lemma}[Corollary 1.5, \cite{Le}]\label{c7lecor15}
\begin{align*}
[H_{i,k}, x_{j,l}^{\pm}]=&\mp 3x_{j,l+k}^{\pm}\\
&\pm\sum_{\substack{0\leq s\leq k-2\\ k+s\ \text{even}}}2^{s-k}(-3)^{k+1-s}\frac{{ k+1\choose s}}{k+1}x_{j,l+s}^{\pm}.
\end{align*}
\end{lemma}
The next corollary follows from Lemma \ref{c7lecor15} immediately.
\begin{corollary}\label{c7c1cor1}\
\begin{enumerate}
  \item $[H_{1,0},x_{2,0}^{-}]=3x_{2,0}^{-}$.
  \item $[H_{1,1},x_{2,0}^{-}]=3x_{2,1}^{-}$.
  \item $[H_{1,2},x_{2,0}^{-}]=3x_{2,2}^{-}+\frac{9}{4}x_{2,0}^{-}$.
  \item $[H_{2,0},x_{1,0}^{-}]=3x_{1,0}^{-}$.
  \item $[H_{2,1},x_{1,0}^{-}]=3x_{1,1}^{-}$.
  \item $[H_{2,2},x_{1,0}^{-}]=3x_{1,2}^{-}+\frac{9}{4}x_{1,0}^{-}$.
  \item $[H_{2,3},x_{1,0}^{-}]=3x_{1,3}^{-}+\frac{27}{4}x_{1,1}^{-}$.
\end{enumerate}
\end{corollary}
\begin{lemma}\label{c7c1l2}
$Y_1\Big(\left(x_{2,0}^{-}\right)^3x_{1,0}^{-}v_1^{+}\Big)\cong W_1\left(\frac{a_1+2}{3}\right)\otimes W_1\left(\frac{a_1+1}{3}\right)$.
\end{lemma}
\begin{proof}Note $H_{1,0}=h_{1,0}$. Since $[h_{1,0}, x_{2,0}^{-}]=3x_{2,0}^{-}$ and $h_{1,0}x_{1,0}^{-}v_1^{+}=-3x_{1,0}^{-}v_1^{+}$,
$$H_{1,0}\left(x_{2,0}^{-}\right)^3x_{1,0}^{-}v_1^{+}=6\left(x_{2,0}^{-}\right)^3x_{1,0}^{-}v_1^{+}.$$
It follows from Remark \ref{c7c1r1} that $H_{1,1}x_{1,0}^{-}v_1^{+}=\left(-3a_1-\frac{9}{2}\right)x_{1,0}^{-}v_1^{+}$ and $H_{1,2}x_{1,0}^{-}v_1^{+}=\left(-3a_1^2-9a_1-9\right)x_{1,0}^{-}v_1^{+}$.
\begin{align*}
H_{1,1}&\left(x_{2,0}^{-}\right)^3x_{1,0}^{-}v_1^{+}\\
&=[H_{1,1},x_{2,0}^{-}]\left(x_{2,0}^{-}\right)^2x_{1,0}^{-}v_1^{+}+x_{2,0}^{-}[H_{1,1},x_{2,0}^{-}]x_{2,0}^{-}x_{1,0}^{-}v_1^{+}\\
&\qquad +\left(x_{2,0}^{-}\right)^2[H_{1,1},x_{2,0}^{-}]x_{1,0}^{-}v_1^{+}+\left(x_{2,0}^{-}\right)^3 H_{1,1}x_{1,0}^{-}v_1^{+}\\
&=3x_{2,1}^{-}\left(x_{2,0}^{-}\right)^2x_{1,0}^{-}v_1^{+}+3x_{2,0}^{-}x_{2,1}^{-}x_{2,0}^{-}x_{1,0}^{-}v_1^{+}\\
&\qquad +3\left(x_{2,0}^{-}\right)^2x_{2,1}^{-}x_{1,0}^{-}v_1^{+}-\left(3a_1+\frac{9}{2}\right)\left(x_{2,0}^{-}\right)^3 x_{1,0}^{-}v_1^{+}\\
&=3\left(\left(a_1-\frac{1}{2}\right)+\left(a_1+\frac{1}{2}\right)+\left(a_1+\frac{3}{2}\right)-\left(a_1+\frac{3}{2}\right)\right)
\left(x_{2,0}^{-}\right)^3x_{1,0}^{-}v_1^{+}\\
&\qquad \text{(this equality follows from Corollary \ref{c7c1cor1})}\\
&=6a_1\left(x_{2,0}^{-}\right)^3x_{1,0}^{-}v_1^{+}.
\end{align*}
\begin{align*}
&H_{1,2}\left(x_{2,0}^{-}\right)^3x_{1,0}^{-}v_1^{+}\\
&=[H_{1,2},x_{2,0}^{-}]\left(x_{2,0}^{-}\right)^2x_{1,0}^{-}v_1^{+}+x_{2,0}^{-}[H_{1,2},x_{2,0}^{-}]x_{2,0}^{-}x_{1,0}^{-}v_1^{+}\\
&\qquad +\left(x_{2,0}^{-}\right)^2[H_{1,2},x_{2,0}^{-}]x_{1,0}^{-}v_1^{+}+\left(x_{2,0}^{-}\right)^3 H_{1,2}x_{1,0}^{-}v_1^{+}\\
&=3x_{2,2}^{-}\left(x_{2,0}^{-}\right)^2x_{1,0}^{-}v_1^{+}+3x_{2,0}^{-}x_{2,2}^{-}x_{2,0}^{-}x_{1,0}^{-}v_1^{+}
+3\left(x_{2,0}^{-}\right)^2x_{2,2}^{-}x_{1,0}^{-}v_1^{+}\\
&\qquad +3\left(\frac{3}{2}\right)^2\left(x_{2,0}^{-}\right)^3x_{1,0}^{-}v_1^{+}-\left(3a_1^2+9a_1+9\right)\left(x_{2,0}^{-}\right)^3 x_{1,0}^{-}v_1^{+}\\
&=3\left(\left(a_1-\frac{1}{2}\right)^2+\left(a_1+\frac{1}{2}\right)^2+\left(a_1+\frac{3}{2}\right)^2+\frac{9}{4}-\left(a_1^2+3a_1+3\right)\right)\\
&\qquad\qquad \left(x_{2,0}^{-}\right)^3x_{1,0}^{-}v_1^{+}\ \text{(this equality follows from Corollary \ref{c7c1cor1})}\\
&=\left(6a_1^2+6\right)\left(x_{2,0}^{-}\right)^3x_{1,0}^{-}v_1^{+}.
\end{align*}

It follows from the above computations that
\begin{align*}
\tilde{h}_{1,0}\left(x_{2,0}^{-}\right)^3x_{1,0}^{-}v_1^{+}&=2\left(x_{2,0}^{-}\right)^3x_{1,0}^{-}v_1^{+};\\ \tilde{h}_{1,1}\left(x_{2,0}^{-}\right)^3x_{1,0}^{-}v_1^{+}&=\left(\frac{2a_1}{3}+2\right)\left(x_{2,0}^{-}\right)^3x_{1,0}^{-}v_1^{+};\\
\tilde{h}_{1,2}\left(x_{2,0}^{-}\right)^3x_{1,0}^{-}v_1^{+}&=\left(2\left(\frac{a_1}{3}\right)^2+4\frac{a_1}{3}+\frac{14}{9}\right)\left(x_{2,0}^{-}\right)^3x_{1,0}^{-}v_1^{+}.
\end{align*}
Thus the degree of the associated polynomial is 2, say $P\left(u\right)=\left(u-a\right)\left(u-b\right).$
Similar to Lemma \ref{l-2is34d}, $a+b=\frac{2a_1}{3}+1$ and $a^2+b^2+a+b=2\left(\frac{a_1}{3}\right)^2+4\frac{a_1}{3}+\frac{14}{9}$, and then $a=\frac{a_1+1}{3}$ and $b=\frac{a_1+2}{3}$.

It follows from representation theory of $\ysl$ that $Y_1\left(\left(x_{2,0}^{-}\right)^3x_{1,0}^{-}v_1^{+}\right)$ is a quotient of the Weyl module $W_1\left(\frac{a_1+2}{3}\right)\otimes W_1\left(\frac{a_1+1}{3}\right)$. Since $W_1\left(\frac{a_1+2}{3}\right)\otimes W_1\left(\frac{a_1+1}{3}\right)$ is irreducible, $$Y_1\left(\left(x_{2,0}^{-}\right)^3x_{1,0}^{-}v_1^{+}\right)\cong W_1\left(\frac{a_1+2}{3}\right)\otimes W_1\left(\frac{a_1+1}{3}\right).$$
\end{proof}
\begin{remark}\label{c7c1r2} Let $a=\frac{a_1+1}{3}$ and $b=\frac{a_1+2}{3}$. Similar to Corollary \ref{w1bw1a}, long computations show that
\begin{enumerate}
\item $\left(x_{1,1}^{-}x_{1,0}^{-}+x_{1,0}^{-}x_{1,1}^{-}\right)\left(x_{2,0}^{-}\right)^3x_{1,0}^{-}v_1^{+}=3(a+b)\left(x_{1,0}^{-}\right)^2\left(x_{2,0}^{-}\right)^3x_{1,0}^{-}v_1^{+}$.
\item $\left(x_{1,2}^{-}x_{1,0}^{-}+x_{1,0}^{-}x_{1,2}^{-}\right)\left(x_{2,0}^{-}\right)^3x_{1,0}^{-}v_1^{+}=3^2(a^2+b^2)\left(x_{1,0}^{-}\right)^2\left(x_{2,0}^{-}\right)^3x_{1,0}^{-}v_1^{+}$.
\item $\left(x_{1,3}^{-}x_{1,0}^{-}+x_{1,0}^{-}x_{1,3}^{-}\right)\left(x_{2,0}^{-}\right)^3x_{1,0}^{-}v_1^{+}=3^3(a^3+b^3)\left(x_{1,0}^{-}\right)^2\left(x_{2,0}^{-}\right)^3x_{1,0}^{-}v_1^{+}$.
\end{enumerate}
\end{remark}
\begin{lemma}\label{c7c1l3}
$Y_2\left(\left(x_{1,0}^{-}\right)^2\left(x_{2,0}^{-}\right)^3x_{1,0}^{-}v_1^{+}\right)$ is a quotient of $W_1\left(a_1+\frac{7}{2}\right)\otimes W_1\left(a_1+\frac{5}{2}\right)$ $\otimes W_1\left(a_1+\frac{3}{2}\right)$.
\end{lemma}
\begin{proof} Note that $$H_{2,0}\left(x_{1,0}^{-}\right)^2\left(x_{2,0}^{-}\right)^3x_{1,0}^{-}v_1^{+}=
h_{2,0}\left(x_{1,0}^{-}\right)^2\left(x_{2,0}^{-}\right)^3x_{1,0}^{-}v_1^{+}=
3\left(x_{1,0}^{-}\right)^2\left(x_{2,0}^{-}\right)^3x_{1,0}^{-}v_1^{+}.$$
\begin{align*}
& H_{2,1}\left(x_{1,0}^{-}\right)^2\left(x_{2,0}^{-}\right)^3x_{1,0}^{-}v_1^{+} \\
&=[H_{2,1},x_{1,0}^{-}]x_{1,0}^{-}\left(x_{2,0}^{-}\right)^3x_{1,0}^{-}v_1^{+}+
x_{1,0}^{-}[H_{2,1},x_{1,0}^{-}]\left(x_{2,0}^{-}\right)^3x_{1,0}^{-}v_1^{+}\\
&\qquad +\left(x_{1,0}^{-}\right)^2 H_{2,1}\left(x_{2,0}^{-}\right)^3x_{1,0}^{-}v_1^{+}\\
&=3x_{1,1}^{-}x_{1,0}^{-}\left(x_{2,0}^{-}\right)^3x_{1,0}^{-}v_1^{+}+3x_{1,0}^{-}x_{1,1}^{-}\left(x_{2,0}^{-}\right)^3x_{1,0}^{-}v_1^{+}\\
&\qquad +\left(-3\left(a_1-\frac{1}{2}\right)-\frac{9}{2}\right)\left(x_{1,0}^{-}\right)^2\left(x_{2,0}^{-}\right)^3x_{1,0}^{-}v_1^{+}\\
&=\big(\left(6a_1+9\right)-\left(3a_1+3\right)\big)\left(x_{1,0}^{-}\right)^2\left(x_{2,0}^{-}\right)^3x_{1,0}^{-}v_1^{+}\text{(by (i) of Remark \ref{c7c1r2})}\\
&=\left(3a_1+6\right)\left(x_{1,0}^{-}\right)^2\left(x_{2,0}^{-}\right)^3x_{1,0}^{-}v_1^{+}.
\end{align*}
\begin{align*}
& H_{2,2}\left(x_{1,0}^{-}\right)^2\left(x_{2,0}^{-}\right)^3x_{1,0}^{-}v_1^{+} \\
&=[H_{2,2},x_{1,0}^{-}]x_{1,0}^{-}\left(x_{2,0}^{-}\right)^3x_{1,0}^{-}v_1^{+}+
x_{1,0}^{-}[H_{2,2},x_{1,0}^{-}]\left(x_{2,0}^{-}\right)^3x_{1,0}^{-}v_1^{+}\\
&\qquad +\left(x_{1,0}^{-}\right)^2 H_{2,2}\left(x_{2,0}^{-}\right)^3x_{1,0}^{-}v_1^{+}\\
&=\left(3x_{1,2}^{-}+\frac{9}{4}x_{1,0}^{-}\right)x_{1,0}^{-}\left(x_{2,0}^{-}\right)^3x_{1,0}^{-}v_1^{+}+x_{1,0}^{-}\left(3x_{1,2}^{-}+\frac{9}{4}x_{1,0}^{-}\right)\left(x_{2,0}^{-}\right)^3x_{1,0}^{-}v_1^{+}\\
&\qquad-\left(3\left(a_1-\frac{1}{2}\right)^2+9\left(a_1-\frac{1}{2}\right)+9\right)\left(x_{1,0}^{-}\right)^2
\left(x_{2,0}^{-}\right)^3x_{1,0}^{-}v_1^{+}\text{(by Remark \ref{c7c1r2})}\\
&=\left(3a_1^2+12a_1+\frac{57}{4}\right)\left(x_{1,0}^{-}\right)^2\left(x_{2,0}^{-}\right)^3x_{1,0}^{-}v_1^{+}.
\end{align*}
\begin{align*}
& H_{2,3}\left(x_{1,0}^{-}\right)^2\left(x_{2,0}^{-}\right)^3x_{1,0}^{-}v_1^{+} \\
&=[H_{2,3},x_{1,0}^{-}]x_{1,0}^{-}\left(x_{2,0}^{-}\right)^3x_{1,0}^{-}v_1^{+}+
x_{1,0}^{-}[H_{2,3},x_{1,0}^{-}]\left(x_{2,0}^{-}\right)^3x_{1,0}^{-}v_1^{+}\\
&\qquad +\left(x_{1,0}^{-}\right)^2 H_{2,3}\left(x_{2,0}^{-}\right)^3x_{1,0}^{-}v_1^{+}\\
&=\left(3x_{1,3}^{-}+\frac{27}{4}x_{1,1}^{-}\right)x_{1,0}^{-}\left(x_{2,0}^{-}\right)^3x_{1,0}^{-}v_1^{+}
+x_{1,0}^{-}\left(3x_{1,3}^{-}+\frac{27}{4}x_{1,1}^{-}\right)\left(x_{2,0}^{-}\right)^3x_{1,0}^{-}v_1^{+}\\
&\qquad-\left(3\left(a_1-\frac{1}{2}\right)^3+\frac{27}{2}\left(a_1-\frac{1}{2}\right)^2+27\left(a_1-\frac{1}{2}\right)
+\frac{81}{4}\right)\\
&\qquad\qquad \left(x_{1,0}^{-}\right)^2\left(x_{2,0}^{-}\right)^3x_{1,0}^{-}v_1^{+}\text{(by Remark \ref{c7c1r2})}\\
&=\left(3a_1^3+18a_1^2+\frac{171}{4}a_1+\frac{75}{2}\right)\left(x_{1,0}^{-}\right)^2\left(x_{2,0}^{-}\right)^3x_{1,0}^{-}v_1^{+}.
\end{align*}
Recall that
\begin{align*}
H_{i}\left(t\right)&=\left(h_{i,0}\right)t^{-1}+\left(h_{i,1}-\frac{1}{2}\left(h_{i,0}\right)^2\right)t^{-2}+\left(h_{i,2}-h_{i,0}h_{i,1}+\frac{1}{3}\left(h_{i,0}\right)^3\right)t^{-3}\\
&\qquad +\left(h_{i,3}-h_{i,0}h_{i,2}-\frac{1}{2}\left(h_{i,1}\right)^2+
\left(h_{i,0}\right)^2h_{i,1}-\frac{1}{4}\left(h_{i,0}\right)^4\right)t^{-4}+\ldots
\end{align*}
It follows from the above calculations that
\begin{align*}
h_{2,0}\left(x_{1,0}^{-}\right)^2\left(x_{2,0}^{-}\right)^3x_{1,0}^{-}v_1^{+}&=
3\left(x_{1,0}^{-}\right)^2\left(x_{2,0}^{-}\right)^3x_{1,0}^{-}v_1^{+};\\
h_{2,1}\left(x_{1,0}^{-}\right)^2\left(x_{2,0}^{-}\right)^3x_{1,0}^{-}v_1^{+}&=
3\left(a_1+\frac{7}{2}\right)\left(x_{1,0}^{-}\right)^2\left(x_{2,0}^{-}\right)^3x_{1,0}^{-}v_1^{+};\\
h_{2,2}\left(x_{1,0}^{-}\right)^2\left(x_{2,0}^{-}\right)^3x_{1,0}^{-}v_1^{+}&=
3\left(a_1+\frac{7}{2}\right)^2\left(x_{1,0}^{-}\right)^2\left(x_{2,0}^{-}\right)^3x_{1,0}^{-}v_1^{+};\\
h_{2,3}\left(x_{1,0}^{-}\right)^2\left(x_{2,0}^{-}\right)^3x_{1,0}^{-}v_1^{+}&=
3\left(a_1+\frac{7}{2}\right)^3\left(x_{1,0}^{-}\right)^2\left(x_{2,0}^{-}\right)^3x_{1,0}^{-}v_1^{+}.\\
\end{align*}
Similarly as in Lemma \ref{l-2is34d}, we omit the proof that the associated polynomial of $Y_2\left(\left(x_{1,0}^{-}\right)^2\left(x_{2,0}^{-}\right)^3x_{1,0}^{-}v_1^{+}\right)$ is $\left(u-\left(a_1+\frac{3}{2}\right)\right)\left(u-\left(a_1+\frac{5}{2}\right)\right)\left(u-\left(a_1+\frac{7}{2}\right)\right)$. By the representation theory of $\ysl$,
 we know $Y_2\left(\left(x_{1,0}^{-}\right)^2\left(x_{2,0}^{-}\right)^3x_{1,0}^{-}v_1^{+}\right)$ is a quotient of $W_1\left(a_1+\frac{7}{2}\right)\otimes W_1\left(a_1+\frac{5}{2}\right)\otimes W_1\left(a_1+\frac{3}{2}\right)$.
\end{proof}
\begin{corollary}In $W_1\left(a_1+\frac{7}{2}\right)\otimes W_1\left(a_1+\frac{5}{2}\right)\otimes W_1\left(a_1+\frac{3}{2}\right)$,
\begin{enumerate}
  \item $x_{2,1}^{-}\left(x_{2,0}^{-}\right)^2\left(x_{1,0}^{-}\right)^2\left(x_{2,0}^{-}\right)^3x_{1,0}^{-}v_1^{+}
      =\left(a_1+\frac{3}{2}\right)\left(x_{2,0}^{-}\right)^3\left(x_{1,0}^{-}\right)^2\left(x_{2,0}^{-}\right)^3x_{1,0}^{-}v_1^{+}$;
  \item $x_{2,0}^{-}x_{2,1}^{-}x_{2,0}^{-}\left(x_{1,0}^{-}\right)^2\left(x_{2,0}^{-}\right)^3x_{1,0}^{-}v_1^{+}
      =\left(a_1+\frac{5}{2}\right)\left(x_{2,0}^{-}\right)^3\left(x_{1,0}^{-}\right)^2\left(x_{2,0}^{-}\right)^3x_{1,0}^{-}v_1^{+}$;
  \item $\left(x_{2,0}^{-}\right)^2 x_{2,1}^{-}\left(x_{1,0}^{-}\right)^2\left(x_{2,0}^{-}\right)^3x_{1,0}^{-}v_1^{+}
      =\left(a_1+\frac{7}{2}\right)\left(x_{2,0}^{-}\right)^3\left(x_{1,0}^{-}\right)^2\left(x_{2,0}^{-}\right)^3x_{1,0}^{-}v_1^{+}$.
\end{enumerate}
\end{corollary}
\begin{lemma}\label{c7c1l4}
$Y_1\left(\left(x_{2,0}^{-}\right)^3\left(x_{1,0}^{-}\right)^2\left(x_{2,0}^{-}\right)^3x_{1,0}^{-}v_1^{+}\right)\cong W_1\left(\frac{a_1}{3}+1\right)$.
\end{lemma}
\begin{proof}
It follows from Proposition \ref{g2fwfr} that $$\tilde{h}_{1,0}\left(x_{2,0}^{-}\right)^3\left(x_{1,0}^{-}\right)^2\left(x_{2,0}^{-}\right)^3x_{1,0}^{-}v_1^{+}=\left(x_{2,0}^{-}\right)^3\left(x_{1,0}^{-}\right)^2\left(x_{2,0}^{-}\right)^3x_{1,0}^{-}v_1^{+}.$$
Thus $Y_1\left(\left(x_{2,0}^{-}\right)^3\left(x_{1,0}^{-}\right)^2\left(x_{2,0}^{-}\right)^3x_{1,0}^{-}v_1^{+}\right)\cong W_1\left(a\right)$.
The eigenvalue of\\ $\left(x_{2,0}^{-}\right)^3\left(x_{1,0}^{-}\right)^2\left(x_{2,0}^{-}\right)^3x_{1,0}^{-}v_1^{+}$ under $\tilde{h}_{1,1}$ will tell the value of $a$.

For the simplicity of the following computations, let $v=\left(x_{2,0}^{-}\right)^3x_{1,0}^{-}v_1^{+}$.
\begin{align*}
   & H_{1,1}\left(x_{2,0}^{-}\right)^3\left(x_{1,0}^{-}\right)^2\left(x_{2,0}^{-}\right)^3x_{1,0}^{-}v_1^{+}\\
   &=[H_{1,1},x_{2,0}^{-}]\left(x_{2,0}^{-}\right)^2\left(x_{1,0}^{-}\right)^2v+ x_{2,0}^{-}[H_{1,1},x_{2,0}^{-}]x_{2,0}^{-}\left(x_{1,0}^{-}\right)^2v\\
   &+ \left(x_{2,0}^{-}\right)^2[H_{1,1},x_{2,0}^{-}]\left(x_{1,0}^{-}\right)^2v+\left(x_{2,0}^{-}\right)^3 H_{1,1}\left(x_{1,0}^{-}\right)^2v\\
   &=3x_{2,1}^{-}\left(x_{2,0}^{-}\right)^2\left(x_{1,0}^{-}\right)^2v+3x_{2,0}^{-}x_{2,1}^{-}x_{2,0}^{-}\left(x_{1,0}^{-}\right)^2v\\
   &+3\left(x_{2,0}^{-}\right)^2 x_{2,1}^{-}\left(x_{1,0}^{-}\right)^2v-\left(6a_1+18\right)\left(x_{2,0}^{-}\right)^3\left(x_{1,0}^{-}\right)^2 v\\
   &=\Bigg(3\left(3a_1+\frac{15}{2}\right)-\left(6a_1+18\right)\Bigg) \left(x_{2,0}^{-}\right)^3\left(x_{1,0}^{-}\right)^2v\\
   &= \left(3a_1+\frac{9}{2}\right) \left(x_{2,0}^{-}\right)^3\left(x_{1,0}^{-}\right)^2\left(x_{2,0}^{-}\right)^3x_{1,0}^{-}v_1^{+}.
\end{align*}
Then $\tilde{h}_{1,1}\left(x_{2,0}^{-}\right)^3\left(x_{1,0}^{-}\right)^2\left(x_{2,0}^{-}\right)^3x_{1,0}^{-}v_1^{+}=\left(\frac{a_1}{3}+1\right) \left(x_{2,0}^{-}\right)^3\left(x_{1,0}^{-}\right)^2\left(x_{2,0}^{-}\right)^3x_{1,0}^{-}v_1^{+}$. $$Y_1\left(\left(x_{2,0}^{-}\right)^3\left(x_{1,0}^{-}\right)^2\left(x_{2,0}^{-}\right)^3x_{1,0}^{-}v_1^{+}\right)\cong W_1\left(\frac{a_1}{3}+1\right).$$
\end{proof}

\subsection{Case 2: $b_1=2$.} \ \\
In this case $$v_1^{-}=x_{2,0}^{-}x_{1,0}^{-}\left(x_{2,0}^{-}\right)^2x_{1,0}^{-}x_{2,0}^{-}v_1^{+}.$$
We summarizing all lemmas in this case in the following proposition. The proof of each lemma is provided.
\begin{proposition}\
\begin{enumerate}
  \item $Y_2\left(v_1^{+}\right)\cong W_1\left(a_1\right)$.
  \item $Y_1\left(x_{2,0}^{-}v_1^{+}\right)\cong W_1\left(\frac{a_1}{3}+\frac{1}{2}\right)$.
  \item $Y_2\left(x_{1,0}^{-}x_{2,0}^{-}v_1^{+}\right)$ is isomorphic to either $W_2\left(a_1+2\right)$ or $W_1\left(a_1+3\right)\otimes W_1\left(a_1+2\right)$.
  \item $Y_1\left(\left(x_{2,0}^{-}\right)^2x_{1,0}^{-}x_{2,0}^{-}v_1^{+}\right)\cong W_1\left(\frac{a_1}{3}+\frac{7}{6}\right)$.
  \item $Y_2\left(x_{1,0}^{-}\left(x_{2,0}^{-}\right)^2x_{1,0}^{-}x_{2,0}^{-}v_1^{+}\right)\cong W_1\left(a_1+5\right)$.
\end{enumerate}
\end{proposition}
\begin{proof}
The item (i) follows from the representations of $\ygg$, so we omit the proof. The second item is proved in Lemma \ref{c7c2l1}. Lemma \ref{c7c2l2} is devoted to prove the third item. The item (iv) will be showed in Lemma \ref{c7c2l3}. The item (v) is proved in Lemma \ref{c7c2l4}.
\end{proof}
\begin{lemma}\label{c7c2l1}
$Y_1\left(x_{2,0}^{-}v_1^{+}\right)\cong W_1\left(\frac{a_1}{3}+\frac{1}{2}\right)$.
\end{lemma}
\begin{proof}
It follows from Proposition \ref{g2fwfr} that $$\tilde{h}_{1,0}x_{2,0}^{-}v_1^{+}=x_{2,0}^{-}v_1^{+}.$$ Thus the associated polynomial $P\left(u\right)$ is of degree 1. Suppose that $$P\left(u\right)=\left(u-a\right).$$
The eigenvalue of $x_{2,0}^{-}v_1^{+}$ under $\tilde{h}_{1,1}$ will tell the values of $a$.
\begin{align*}
   \tilde{h}_{1,1} x_{2,0}^{-}v_1^{+}&= [\tilde{h}_{1,1}, x_{2,0}^{-}]v_1^{+}\\
   &=\left(\frac{1}{3}x_{2,1}^{-}+\frac{1}{2}x_{2,0}^{-}+x_{2,0}^{-}\tilde{h}_{1,0}\right)v_1^{+}\\
   &=\left(\frac{a_1}{3}+\frac{1}{2}\right)x_{2,0}^{-}v_1^{+}.
\end{align*}

By the representation theory of Weyl modules of $\ysl$, we have $$Y_1\left(x_{2,0}^{-}v_1^{+}\right)\cong W_1\left(\frac{a_1}{3}+\frac{1}{2}\right).$$
\end{proof}
\begin{remark}
It follows from Corollary \ref{slw1a} that $\tilde{x}_{1,1}^{-}x_{2,0}^{-}v_1^{+}=\left(\frac{a_1}{3}+\frac{1}{2}\right)x_{1,0}^{-}x_{2,0}^{-}v_1^{+}$. Then $x_{1,1}^{-}x_{2,0}^{-}v_1^{+}=\left(a_1+\frac{3}{2}\right)x_{1,0}^{-}x_{2,0}^{-}v_1^{+}$.
\end{remark}
\begin{lemma}\label{c7c2l2}
$Y_2\left(x_{1,0}^{-}x_{2,0}^{-}v_1^{+}\right)$ is isomorphic to either $W_2\left(a_1+2\right)$ or $W_1\left(a_1+3\right)\otimes W_1\left(a_1+2\right)$.
\end{lemma}
\begin{proof}
It follows from Proposition \ref{g2fwfr} that $$h_{2,0}x_{1,0}^{-}x_{2,0}^{-}v_1^{+}=2x_{1,0}^{-}x_{2,0}^{-}v_1^{+}.$$ Thus the associated polynomial $P\left(u\right)$ is of degree 2. Suppose that $$P\left(u\right)=\left(u-a\right)\left(u-b\right),$$ where $\operatorname{Re}\left(a\right)\leq \operatorname{Re}\left(b\right)$.
The eigenvalues of both $x_{1,0}^{-}x_{2,0}^{-}v_1^{+}$ under $h_{2,1}$ and under $h_{2,2}$ will tell the values of both $a$ and $b$.
\begin{align*}
h_{2,1} &x_{1,0}^{-}x_{2,0}^{-}v_1^{+}\\
   &= [h_{2,1}, x_{1,0}^{-}]x_{2,0}^{-}v_1^{+}+ x_{1,0}^{-}h_{2,1}x_{2,0}^{-}v_1^{+}\\
   &=\left(3x_{1,1}^{-}+\frac{3}{2}\left(3x_{1,0}^{-}+2x_{1,0}^{-}h_{2,0}\right)\right)x_{2,0}^{-}v_1^{+} -a_1x_{1,0}^{-}x_{2,0}^{-}v_1^{+}\\
   &=\Bigg(3\left(a_1+\frac{3}{2}\right)+\frac{3}{2}(3-2)-a_1\Bigg)x_{1,0}^{-}x_{2,0}^{-}v_1^{+}\\
   &= 2\left(a_1+3\right)x_{1,0}^{-}x_{2,0}^{-}v_1^{+}.
\end{align*}
\begin{align*}
h_{2,2}& x_{1,0}^{-}x_{2,0}^{-}v_1^{+}\\
   &= [h_{2,2}, x_{1,0}^{-}]x_{2,0}^{-}v_1^{+}+ x_{1,0}^{-}h_{2,2}x_{2,0}^{-}v_1^{+}\\
   &= [h_{2,2}, x_{1,0}^{-}]x_{2,0}^{-}v_1^{+}-a_1^2x_{1,0}^{-}x_{2,0}^{-}v_1^{+}\\
   &=\left([h_{2,1},x_{1,1}^{-}]+\frac{3}{2}\left(h_{2,1}x_{1,0}^{-}+x_{1,0}^{-}h_{2,1}\right)\right)x_{2,0}^{-}v_1^{+} -a_1^2x_{1,0}^{-}x_{2,0}^{-}v_1^{+}\\
   &=\left(h_{2,1}x_{1,1}^{-}-x_{1,1}^{-}h_{2,1}+\frac{3}{2}\left(h_{2,1}x_{1,0}^{-}+x_{1,0}^{-}h_{2,1}\right)\right)x_{2,0}^{-}v_1^{+} -a_1^2x_{1,0}^{-}x_{2,0}^{-}v_1^{+}\\
      &=\left(h_{2,1}\big(x_{1,1}^{-}+\frac{3}{2}x_{1,0}^{-}\big)-\big(x_{1,1}^{-}-\frac{3}{2}x_{1,0}^{-}\big)h_{2,1}\right)x_{2,0}^{-}v_1^{+} -a_1^2x_{1,0}^{-}x_{2,0}^{-}v_1^{+}\\
   &=\left(a_1+\frac{3}{2}+\frac{3}{2}\right)h_{2,1}x_{1,0}^{-}x_{2,0}^{-}v_1^{+} +a_1\left(a_1+\frac{3}{2}\right)x_{1,0}^{-}x_{2,0}^{-}v_1^{+}\\
   &-\frac{3}{2}a_1x_{1,0}^{-}x_{2,0}^{-}v_1^{+}-a_1^2x_{1,0}^{-}x_{2,0}^{-}v_1^{+}\\
   &=\Big(2\left(a_1+3\right)\left(a_1+3\right)+a^2_1+\frac{3}{2}a_1-\frac{3}{2}a_1-a^2_1\Big)x_{1,0}^{-}x_{2,0}^{-}v_1^{+}\\
   &=2\left(a_1+3\right)^2x_{1,0}^{-}x_{2,0}^{-}v_1^{+}.
\end{align*}
The associated polynomial of $Y_2\left(x_{1,0}^{-}x_{2,0}^{-}v_1^{+}\right)$ is $\big(u-(a_1+2)\big)\big(u-(a_1+3)\big)$ by Lemma \ref{g2puisd3}. By the representation theory of the Weyl modules of $\ysl$, $Y_2\left(x_{1,0}^{-}x_{2,0}^{-}v_1^{+}\right)$ is isomorphic to either $W_2\left(a_1+2\right)$ or $W_1\left(a_1+3\right)\otimes W_1\left(a_1+2\right)$.
\end{proof}
\begin{lemma}\label{c7c2l3}
$Y_1\left(\left(x_{2,0}^{-}\right)^2x_{1,0}^{-}x_{2,0}^{-}v_1^{+}\right)\cong W_1\left(\frac{a_1}{3}+\frac{7}{6}\right)$.
\end{lemma}
\begin{proof}
It follows from Proposition \ref{g2fwfr} that $$\tilde{h}_{1,0}\left(x_{2,0}^{-}\right)^2x_{1,0}^{-}x_{2,0}^{-}v_1^{+}=\left(x_{2,0}^{-}\right)^2x_{1,0}^{-}x_{2,0}^{-}v_1^{+}.$$ Thus the associated polynomial $P\left(u\right)$ is of degree 1. Suppose that $$P\left(u\right)=\left(u-a\right).$$
The eigenvalue of $\left(x_{2,0}^{-}\right)^2x_{1,0}^{-}x_{2,0}^{-}v_1^{+}$ under $\tilde{h}_{1,1}$ will tell the value of $a$.

Note that $wt\left(x_{1,0}^{-} x_{2,0}^{-}v_1^{+}\right)=-\omega_1+2\omega_2$ and $wt\left(x_{2,0}^{-}x_{1,0}^{-} x_{2,0}^{-}v_1^{+}\right)=0$.
\begin{align*}
   &\tilde{h}_{1,1}\left(x_{2,0}^{-}\right)^2x_{1,0}^{-} x_{2,0}^{-}v_1^{+}\\
   &= [\tilde{h}_{1,1}, x_{2,0}^{-}]x_{2,0}^{-}x_{1,0}^{-} x_{2,0}^{-}v_1^{+}+x_{2,0}^{-}[\tilde{h}_{1,1}, x_{2,0}^{-}]x_{1,0}^{-} x_{2,0}^{-}v_1^{+}+\left(x_{2,0}^{-}\right)^2\tilde{h}_{1,1}x_{1,0}^{-} x_{2,0}^{-}v_1^{+}\\
   &=\left(\frac{1}{3}x_{2,1}^{-}+\frac{1}{2}x_{2,0}^{-}+x_{2,0}^{-}\tilde{h}_{1,0}\right)x_{2,0}^{-}x_{1,0}^{-} x_{2,0}^{-}v_1^{+}\\
   &+ x_{2,0}^{-}\left(\frac{1}{3}x_{2,1}^{-}+\frac{1}{2}x_{2,0}^{-}+x_{2,0}^{-}\tilde{h}_{1,0}\right)x_{1,0}^{-} x_{2,0}^{-}v_1^{+}v_1^{+}-\left(\frac{a_1}{3}+\frac{1}{2}\right)\left(x_{2,0}^{-}\right)^2x_{1,0}^{-} x_{2,0}^{-}v_1^{+}\\
   &=\frac{1}{3}\left(x_{2,1}^{-}x_{2,0}^{-}+x_{2,0}^{-}x_{2,1}^{-}\right)x_{1,0}^{-} x_{2,0}^{-}v_1^{+}- \left(\frac{a_1}{3}+\frac{1}{2}\right)\left(x_{2,0}^{-}\right)^2x_{1,0}^{-} x_{2,0}^{-}v_1^{+}\\
   &= \left(\frac{1}{3}\left(a_1+2+a_1+3\right)-\left(\frac{a_1}{3}+\frac{1}{2}\right)\right)\left(x_{2,0}^{-}\right)^2
   x_{1,0}^{-} x_{2,0}^{-}v_1^{+}\ \text{(by Corollary \ref{w1bw1a})}\\
   &= \left(\frac{a_1}{3}+\frac{7}{6}\right)\left(x_{2,0}^{-}\right)^2x_{1,0}^{-} x_{2,0}^{-}v_1^{+}.
\end{align*}
By the representation theory of the Weyl modules of $\ysl$, $$Y_1\left(\left(x_{2,0}^{-}\right)^2x_{1,0}^{-}x_{2,0}^{-}v_1^{+}\right)\cong W_1\left(\frac{a_1}{3}+\frac{7}{6}\right).$$
\end{proof}
\begin{lemma}\label{c7c2l4}
$Y_2\left(x_{1,0}^{-}\left(x_{2,0}^{-}\right)^2x_{1,0}^{-}x_{2,0}^{-}v_1^{+}\right)\cong W_1\left(a_1+5\right)$.
\end{lemma}
\begin{proof}
It follows from Proposition \ref{g2fwfr} that $$h_{2,0}x_{1,0}^{-}\left(x_{2,0}^{-}\right)^2x_{1,0}^{-}x_{2,0}^{-}v_1^{+}=x_{1,0}^{-}\left(x_{2,0}^{-}\right)^2x_{1,0}^{-}x_{2,0}^{-}v_1^{+}.$$ Thus the associated polynomial $P\left(u\right)$ is of degree 1. Suppose that $$P\left(u\right)=\left(u-a\right).$$
The eigenvalue of $x_{1,0}^{-}\left(x_{2,0}^{-}\right)^2x_{1,0}^{-}x_{2,0}^{-}v_1^{+}$ under $h_{2,1}$ will tell the value of $a$.

\begin{align*}
   &h_{2,1}x_{1,0}^{-}\left(x_{2,0}^{-}\right)^2x_{1,0}^{-} x_{2,0}^{-}v_1^{+}\\
   &= [h_{2,1},x_{1,0}^{-}]\left(x_{2,0}^{-}\right)^2x_{1,0}^{-} x_{2,0}^{-}v_1^{+}+ x_{1,0}^{-}h_{2,1}\left(x_{2,0}^{-}\right)^2x_{1,0}^{-} x_{2,0}^{-}v_1^{+}\\
   &= \Big(3x_{1,1}^{-}+\frac{3}{2}\left(3x_{1,0}^{-}+2x_{1,0}^{-}h_{2,0}\right)\Big)\left(x_{2,0}^{-}\right)^2x_{1,0}^{-} x_{2,0}^{-}v_1^{+}\\
   &-2\left(a_1+2\right)x_{1,0}^{-}\left(x_{2,0}^{-}\right)^2x_{1,0}^{-} x_{2,0}^{-}v_1^{+}\\
   &= \Bigg(3\left(a_1+\frac{7}{2}\right)-\frac{3}{2}-2\left(a_1+2\right)\Bigg)x_{1,0}^{-}\left(x_{2,0}^{-}\right)^2x_{1,0}^{-} x_{2,0}^{-}v_1^{+}\\
   &=\Big(a_1+5\Big)x_{1,0}^{-}\left(x_{2,0}^{-}\right)^2x_{1,0}^{-} x_{2,0}^{-}v_1^{+}
\end{align*}
By the representation theory of the Weyl modules of $\ysl$,
\begin{center}
$Y_2\left(x_{1,0}^{-}\left(x_{2,0}^{-}\right)^2x_{1,0}^{-} x_{2,0}^{-}v_1^{+}\right)\cong W_1\left(a_1+5\right).$
\end{center}
\end{proof}
\section{On the local Weyl modules of $\yg$ when $\g$ is of type $G_2$}
Recall $\pi=\big(\pi_1(u),\pi_2(u)\big)$ and $\pi_i\left(u\right)=\prod\limits_{j=1}^{m_i}\left(u-a_{i,j}\right)$ for $i\in I=\{1,2\}$. Let $\lambda=\sum\limits_{i\in I} m_i\omega_i$. In \cite{Na}, the dimension of the local Weyl module $W(\lambda)$ is given.
\begin{proposition}[Corollary 9.5, \cite{Na}]\label{dwmocsp}
Let $\lambda=\sum\limits_{i\in I} m_i\omega_i$. Then
$$\operatorname{Dim}\big(W(\lambda)\big)=\prod\limits_{i\in I} \Big(\operatorname{Dim}\big(W(\omega_i)\big)\Big)^{m_i}.$$
\end{proposition}

It follows from Proposition \ref{g2Lihw} and Corollary \ref{vtv'hwv} that
\begin{theorem}\label{c7mtotsolwm} Let $\pi=\big(\pi_1(u), \pi_{2}(u)\big)$ be a pair of monic polynomials in $u$, and let $\pi_i\left(u\right)=\prod\limits_{j=1}^{m_i} \left(u-a_{i,j}\right)$.  Let $S=\{a_{1,1},\ldots, a_{1,m_1}, a_{2,1}, \ldots, a_{2,m_2}\}$ be a multi-set of roots. Let  $a_1=a_{m,n}$ be one of the numbers in $S$ with the maximal real part, and let $b_1=m$. Similarly let $a_r=a_{s,t}\left(i\geq 2\right)$ be one of the number in $S\setminus\{a_1, \ldots, a_{r-1}\}$ with the maximal real part, and let $b_r=s$.  Let $L=V_{a_1}(\omega_{b_1})\otimes V_{a_2}(\omega_{b_2})\otimes\ldots\otimes V_{a_k}(\omega_{b_k})$, where $k=m_1+m_2$. Then $L$ is a highest weight representation, and its associated polynomials are $\pi$. 
\end{theorem}

\begin{theorem}The local Weyl module $W(\pi)$ of $\ygg$ associated to $\pi$ is isomorphic to $L$ as in Theorem \ref{c7mtotsolwm}.
\end{theorem}
\begin{proof}
On the one hand, $\operatorname{Dim}\big(W(\pi)\big)\leq \operatorname{Dim}\big(W(\lambda)\big)$ by Theorem \ref{ubodowm}; on the other
hand, $\operatorname{Dim}\big(W(\pi)\big)\geq \operatorname{Dim}\left(L\right)$ by the maximality of the local Weyl modules of Yangians. Note that as $G_2$-modules, $W(\omega_i)\cong KR(\omega_i)\cong V_a(\omega_i)$, where the latter isomorphism follows easily from the main theorem of section 2.3 of \cite{ChMo3} and Theorem 6.3 \cite{ChPr4}. Especially, $\operatorname{Dim}\big(W(\omega_i)\big)=\operatorname{Dim}\big(V_{a}(\omega_{i})\big)$ for any $a\in \C$. Therefore,
$\operatorname{Dim}\big(W(\lambda)\big)=\operatorname{Dim}\left(L\right)$, and then this implies that
$\operatorname{Dim}\Big(W\left(\pi\right)\Big)=\operatorname{Dim}\left(L\right)$. Therefore $W\left(\pi\right)\cong L$.
\end{proof}


\end{document}